%% file: 00Main.tex
\documentclass[reqno]{amsart}

\usepackage{amssymb, amsfonts}
\usepackage{verbatim}
\usepackage{fullpage}

\usepackage[hidelinks]{hyperref}
\usepackage{url}
\usepackage{cleveref}
\usepackage{mathrsfs}
\usepackage{color}   

\hypersetup{
    colorlinks=true, 
    linktoc=all,     
    linkcolor=blue,  
}

\usepackage{tikz-cd}

\usepackage{mathtools}
\usepackage{amssymb, amsfonts}
\usepackage[utf8]{inputenc}
\usepackage{amsthm}
\usepackage{amsmath}
\usepackage{amsxtra}
\usepackage{amscd}
\usepackage{amssymb}
\usepackage{color}
\usepackage{mathtools}
\usepackage{xspace}
\usepackage{amsfonts}
\usepackage{mathrsfs}
\usepackage{todonotes}
\usepackage{enumerate}
\usepackage{multirow}
\usepackage{tikz}
\usepackage{hyperref}
\usepackage{cancel}
\usepackage{tikz-cd}
\usepackage{hhline}
\usepackage[section]{placeins}

\usepackage[all]{xy}
\usepackage{amscd}
\usepackage{caption}
\usepackage{epsfig}
\usepackage{graphicx,transparent}
\usepackage{here}
\usepackage{hyperref}
\usepackage{setspace}
\usepackage{wasysym}
\usepackage{xparse}
\usepackage{manfnt}
\usepackage{fullpage}

\usepackage{url}
\usepackage{cleveref}
\usepackage{mathrsfs}

\usepackage{tikz-cd}

\usepackage{mathtools}

\usepackage{amsmath}
\usepackage{tikz}
\usepackage{mathdots}
\usepackage{yhmath}
\usepackage{cancel}
\usepackage{color}
\usepackage{siunitx}
\usepackage{array}
\usepackage{multirow}
\usepackage{amssymb}
\usepackage{tabularx}
\usepackage{extarrows}
\usepackage{booktabs}
\usetikzlibrary{fadings}
\usetikzlibrary{patterns}
\usetikzlibrary{shadows.blur}
\usetikzlibrary{shapes}

\newtheorem{theorem}[equation]{Theorem}
\newtheorem{lemma}[equation]{Lemma}
\newtheorem{proposition}[equation]{Proposition}
\newtheorem{corollary}[equation]{Corollary}
\newtheorem*{conjecture}{Conjecture}

\theoremstyle{definition}
\newtheorem{definition}[equation]{Definition}

\theoremstyle{remark}
\newtheorem{remark}[equation]{Remark}

\numberwithin{equation}{section}

\DeclareMathOperator{\res}{res}

\DeclareMathOperator{\Ind}{Ind}

\DeclareMathOperator{\Hom}{Hom}

\DeclareMathOperator{\Per}{Per}

\DeclareMathOperator{\sgn}{sgn}
\DeclareMathOperator{\I}{I}

\DeclareMathOperator{\III}{III}

\newcommand{\cR}{\mathcal{R}}
\newcommand{\cC}{\mathcal{C}}
\newcommand{\cD}{\mathcal{D}}
\newcommand{\cE}{\mathcal{E}}
\newcommand{\cP}{\mathcal{P}}

\newcommand{\cQ}{\mathcal{Q}}

\newcommand{\cH}{\mathcal{H}}
\newcommand{\cF}{\mathcal{F}}
\newcommand{\ZZ}{\mathbb{Z}}
\newcommand{\PP}{\mathbb{P}}
\newcommand{\CC}{\mathbb{C}}

\newcommand{\RR}{\mathbb{R}}
\newcommand{\fR}{\mathfrak{R}}
\newcommand{\Id}{\mathrm{Id}}

\newcommand{\centered}[1]{\begin{tabular}{l} #1 \end{tabular}}

\begin{document}

\title[]{Connected components of generalized strata of meromorphic differentials with residue conditions}
\author[]{Myeongjae Lee and Yiu Man Wong}

\begin{abstract}
Generalized strata of meromorphic differentials are loci within the usual strata of differentials where certain sets of residues sum to zero. They naturally appear in the boundary of the multi-scale compactification of the usual strata. The classification of generalized
strata is a key step towards understanding the irreducible
components of the boundary strata of the multi-scale compactification. The connected components of generalized strata of residueless differentials were classified in \cite{lee2023connected}. In the present paper, we classify the connected components of generalized strata of meromorphic differentials in full generality. 
\end{abstract}

\thanks{Research of the first author is supported in part by NSF grant DMS-21-01631.}
\thanks{Research of the second author is supported
by the DFG-project MO 1884/3-1 and by the Collaborative Research Centre TRR 326 ``Geometry and Arithmetic of Uniformized Structures".}
\maketitle
\tableofcontents
\input{01Intro}
\input{02PrincipalBoundary}

\input{03BaseCases}
\input{04Hyperelliptic}
\input{05SimpleSaddleConnection}
\input{06Nonhyperelliptic}
\input{07DegenerationDynamics}
\input{08Bsignatures}
\input{09Csignatures}
\input{10Dsignatures}

\bibliography{bib}
\bibliographystyle{siam}

\end{document}

%% file: 01Intro.tex
\section{Introduction} \label{Sec:Intro}
Let $g$ be a nonnegative integer, and let $\mu \coloneqq (a_1,\dots,a_m, -b_1,\dots, -b_n)$ be a partition of $2g-2$, where $a_i, b_j$ are non-negative integers. We denote by $\cH(\mu)$ the moduli space of tuples $(X,\boldsymbol{z} \sqcup \boldsymbol{p}, \omega)$, where $X$ is a genus $g$ compact Riemann surface with $m+n$ marked points labeled by $\boldsymbol{z} \sqcup \boldsymbol{p} = \{z_1,\dots, z_m\} \sqcup \{p_1,\dots,p_n\}$, and $\omega$ is a nonzero meromorphic $1$-form on $X$ such that
\[
\operatorname{div}(\omega) = \sum_{i=1}^m a_i z_i - \sum_{j=1}^n b_j p_j.
\]
That is, $\omega$ has zeros of order $a_i$ at $z_i$ and poles of order $b_j$ at $p_j$. There is a natural $\CC^*$-action on $\cH(\mu)$ given by scaling $\omega$ by a nonzero constant. Let $\cP(\mu)$ denote the projectivized space $\cH(\mu)/\CC^*$.

In this paper, we refer to a tuple $(X,\boldsymbol{z} \sqcup \boldsymbol{p}, \omega)$ as a \emph{flat surface}, since the differential $\omega$ defines a flat metric on $X \setminus \boldsymbol{p}$ with conical singularities at the points in $\boldsymbol{z}$. For simplicity, we refer to these moduli spaces as (projectivized) \emph{strata} of flat surfaces. A stratum $\cH(\mu)$ is called \emph{holomorphic} if $\mu$ contains no negative entries (i.e., $n = 0$); otherwise, we call it \emph{meromorphic}.

Given a tuple $\mu = (a_1,\dots,a_m, -b_1,\dots, -b_n)$ and a partition $\fR \colon \boldsymbol{p} = R_1 \sqcup \dots \sqcup R_r$ of the set of poles $\boldsymbol{p}$, we denote the corresponding data as $(a_1,\dots,a_m \mid \beta_1 \mid \dots \mid \beta_r)$, where the vertical bars separate the parts of $\fR$, and $\beta_i = (-b_j)_{p_j \in R_i}$. We also write this tuple as $\mu^{\fR}$.

Given such a partition $\fR$, we define the linear residue conditions:
\[
\sum_{p \in R_i} \operatorname{res}_{p} \omega = 0,
\]
for each part $R_i$ of $\fR$. Let $\cH(\mu^{\fR})$ denote the subspace of $\cH(\mu)$ consisting of differentials satisfying these residue conditions. We refer to $\cH(\mu^{\fR})$ as a \emph{generalized stratum} of flat surfaces. As the residue condition is invariant under rescaling $\omega$, we denote by $\cP(\mu^{\fR})$ the projectivized generalized stratum. For simplicity, we will also refer to these as strata when no confusion arises.

For each $R_i$, one can find a simple closed curve $\gamma_i \subset X$ enclosing only the poles in $R_i$ such that $[\gamma_i] = 0 \in H_1(X, \mathbb{Z})$. Then the residue condition for $\fR$ is equivalent to
\[
\int_{\gamma_i} \omega = 0
\]
for each $R_i$. Let $K(\fR) \subset H_1(X \setminus \boldsymbol{p}, \boldsymbol{z}; \mathbb{C})$ be the subgroup generated by the $\gamma_i$. In terms of period coordinates, the generalized stratum $\cH(\mu^{\fR})$ is a linear subspace of $H^1(X \setminus \boldsymbol{p}, \boldsymbol{z}; \mathbb{C})$, the dual of $H_1(X \setminus \boldsymbol{p}, \boldsymbol{z}; \mathbb{Z})/K(\fR)$. In particular, the complex dimension of $\cH(\mu^{\fR})$ is $2g + m + n - r - 1$, where $r = |\fR|$ is the number of parts of $\fR$.

In this paper, we classify the connected components of generalized strata $\cH(\mu^{\fR})$. There is a one-to-one correspondence between connected components of $\cH(\mu^{\fR})$ and certain topological invariants.

When the partition $\fR$ is trivial, the connected components of holomorphic and meromorphic strata $\cH(\mu)$ are classified by Kontsevich--Zorich~\cite{kontsevich2003connected} and Boissy~\cite{boissy2015connected}, respectively. In general, the connected components of a stratum are distinguished by topological invariants such as \emph{spin parity} and \emph{hyperellipticity} in the holomorphic case, and additionally by the \emph{rotation number} in the meromorphic case with $g=1$.

The strategy in these works is to show that any connected component of $\cH(\mu)$ can be obtained via a series of flat geometric surgeries called \emph{breaking up a zero} and \emph{bubbling a handle}, starting from a connected stratum of genus $0$. By analyzing the discrete parameters of the surgeries, they prove that the above invariants distinguish all components.

For the opposite extreme, where $\fR$ is the finest possible partition $\boldsymbol{p} = \{p_1\} \sqcup \dots \sqcup \{p_n\}$, the first author classified the connected components of $\cR(\mu^{\fR})$ in~\cite{lee2023connected}. This corresponds to strata of residueless meromorphic differentials. We follow a similar approach.

Our strategy is:
\begin{itemize}
    \item Classify the connected components of (projectivized) strata of dimension one. This step utilizes properties of the multi-scale compactification $\overline{\cP}(\mu^{\fR})$, recalled in \Cref{subsec:ms}.
    \item Show that, with few exceptions, any connected component of a stratum can be obtained from a connected component of a one-dimensional stratum via a series of breaking up a zero and bubbling a handle surgeries.
    \item Refine the classification by constructing explicit deformations between flat surfaces with the same topological invariants.
\end{itemize}

A key difference from earlier works is that we only allow deformations preserving the residue condition. In~\cite{lee2023connected}, to perform the second step, the author uses the $\mathrm{GL}^+(2,\mathbb{R})$-action to shrink a collection of parallel saddle connections and perform local surgeries to approach certain boundary points of the stratum.

The most technically involved part is the first step, covered in Sections~\ref{Sec:Bsignature}--\ref{Sec:Dsignature}. These serve as base cases for an induction on $\dim \cH(\mu^{\fR})$. Unlike~\cite{lee2023connected}, we account for three types of one-dimensional strata $\cP(\mu^{\fR})$. In this case, the multi-scale compactification $\overline{\cP}(\mu^{\fR})$ is a compact Riemann surface whose boundary $\partial \overline{\cP}(\mu^{\fR})$ consists of finitely many points enumerated by combinatorial data. We construct degeneration--undegeneration sequences to connect boundary points with the same topological invariants (analogous to solving a Rubik's cube).

\subsection{Main theorems}
In the present paper, we prove that the connected components of generalized strata $\cP(\mu^\fR)$ are completely classified by topological invariants: hyperellipticity (and ramification profile), spin parity, rotation number, and index. First, we introduce the definitions of these invariants.

\begin{definition}
    A flat surface $(X,\omega)$ is said to be \emph{hyperelliptic} if there exists an involution $\sigma$ of the curve $X$ such that $X/\sigma \cong \mathbb{P}^1$ and $\sigma^{\ast} \omega = -\omega$.
\end{definition}

\begin{definition}
    A connected component of $\cP(\mu^\fR)$ is said to be \emph{hyperelliptic} if all flat surfaces in the component are hyperelliptic.
\end{definition}

In \Cref{Sec:hyperelliptic}, we show that a stratum can have hyperelliptic components only if it has either a unique zero or two zeroes of the same order. For those strata, we define the following invariant.

\begin{definition}
    Assume that $\cP(\mu^\fR)$ has a unique zero or two zeroes of the same order. A \emph{ramification profile} of $\cP(\mu^\fR)$ is an involution $\Sigma$ on the set of poles $\{p_1,\dots,p_n\}$ such that
    \begin{itemize}
        \item $\Sigma$ fixes at most $2g+1$ (resp.\ $2g+2$) poles if the stratum has a unique zero (resp.\ two zeroes). Moreover, $\Sigma$ only fixes even poles.
        \item If $\Sigma(p_i) = p_j$, then $b_i = b_j$. That is, $\Sigma$ preserves the orders of poles.
        \item For each pole $p_i$, either $\{p_i,\Sigma(p_i)\} \in \fR$ or both $\{p_i\},\{\Sigma(p_i)\} \in \fR$.
    \end{itemize}
\end{definition}

A hyperelliptic involution $\sigma$ on a flat surface induces a ramification profile $\Sigma$, which remains invariant within a hyperelliptic component. It turns out that $\Sigma$ completely determines the hyperelliptic components.

\begin{theorem} \label{Thm:mainhyper}
    There exists a one-to-one correspondence between the set of hyperelliptic components of a stratum $\cP(\mu^\fR)$ of positive dimension and the set of ramification profiles of $\cP(\mu^\fR)$.
\end{theorem}

\begin{definition}
    A stratum $\cP(\mu^\fR)$ is said to be of \emph{even type} if every simple pole is paired with another simple pole by $\fR$, and all other poles and zeroes are even.
\end{definition}

For a stratum of even type, there is another topological invariant called the \emph{spin parity}. Recall that the spin parity of any flat surface $(X,\omega)$ with only even poles and zeroes is computed by
\[
\sum_{i=1}^g (\operatorname{Ind}\alpha_i+1)(\operatorname{Ind}\beta_i+1) \pmod{2},
\]
where $\alpha_1,\beta_1,\dots, \alpha_g,\beta_g$ are simple closed curves in $X$ that form a symplectic basis of $H_1(X,\mathbb{Z})$. The index $\operatorname{Ind}\alpha$ of a closed curve $\alpha$ is defined as the degree of the Gauss map $G_\alpha : S^1 \to S^1$ with respect to the flat metric on $X$.

If two simple poles of a flat surface $(X,\omega)$ have opposite residues, then we can construct another flat surface with genus $g+1$ by cutting along the closed geodesics in the two half-infinite cylinders and gluing them, thereby obtaining a finite cylinder. In other words, if $\cP(\mu^\fR)$ is of even type with $k$ pairs of simple poles, then flat surfaces in this stratum inherit spin parity from flat surfaces of genus $g+k$ obtained by gluing $k$ pairs of simple poles.

Assume that $\dim \cP(\mu^\fR) > 0$ and there are $k$ pairs of simple poles in $\fR$. If $g+k>1$, we have the following classification result.

\begin{theorem} \label{Thm:mainhigh}
    Suppose that a stratum $\cP(\mu^\fR)$ has $k$ pairs of simple poles and $g+k>1$. Let $\mu^\fR=(\widetilde{\mu}\mid-1^2\mid\dots\mid-1^2)$, where $\widetilde{\mu}$ is an integral partition of $2(g+k)-2$. Then the stratum $\cP(\mu^\fR)$ of even type, except for a few cases below, has two non-hyperelliptic components. The stratum has a unique non-hyperelliptic component if it is \emph{not} of even type.
    
    The exceptional cases are:
    \begin{itemize}
        \item[(i)] $g+k=2$: $\widetilde{\mu}$ is $(2)$, $(1,1)$, $(4,-2)$, or $(2,2,-2)$;
        \item[(ii)] $g+k=3$: $\widetilde{\mu}$ is $(4)$ or $(2,2)$.
    \end{itemize}
    For these cases, we have the following classification:
    \begin{itemize}
        \item If $\widetilde{\mu}=(2)$ or $(1,1)$, then $\cP(\mu^\fR)$ has \emph{no} non-hyperelliptic component.
        \item If $\widetilde{\mu}$ is $(4,-2)$ or $(2,2,-2)$, then $\cP(\mu^\fR)$ has exactly \emph{one} non-hyperelliptic component of even spin.
        \item If $\widetilde{\mu}$ is $(4)$ or $(2,2)$, then $\cP(\mu^\fR)$ has exactly \emph{one} non-hyperelliptic component of odd spin.
    \end{itemize}
\end{theorem}

If $g \leq 1$ and $k=0$, we have the following classification of non-hyperelliptic components.

\begin{theorem} \label{Thm:main10}
    Suppose that $g=1$ and $\cP(\mu^\fR)$ has no pair of simple poles. Then the non-hyperelliptic components are classified by the rotation number, except for the following cases:
    \begin{itemize}
        \item If $\mu=(r,-r)$, the stratum does not have a component corresponding to rotation number $r$. In addition, if $2\mid r$, then there is no non-hyperelliptic component corresponding to rotation number $r/2$.
        \item If $\mu^\fR$ is of the form $(2n\mid-2\mid\dots\mid-2)$ or $(n,n\mid-2\mid\dots\mid-2)$, then there is \emph{no} non-hyperelliptic component.
        \item If $\mu$ is $(r,r,-2r)$, $(2r,-r,-r)$, or $(r,r,-r,-r)$, then the stratum has \emph{no} non-hyperelliptic component corresponding to rotation number $r$.
        \item If $\mu^\fR=(12\mid-3\mid-3\mid-3\mid-3)$, the stratum has \emph{two} non-hyperelliptic components corresponding to the rotation number $3$.
    \end{itemize}
\end{theorem}

\begin{theorem} \label{Thm:main00}
    Suppose that a nonempty stratum $\cP(\mu^\fR)$ with $g=0$ is non-residueless and $\fR$ has no pair of simple poles. Then it has a unique non-hyperelliptic component, except when $\mu^\fR=(a,a\mid-2\mid\dots\mid-2\mid-(a-(n-2)),-(a-(n-2)))$, where every component is hyperelliptic.
\end{theorem}

In strata of genus zero with one pair of simple poles with opposite residues (i.e., if $g=0$ and $k=1$), we introduce a new topological invariant, called the \emph{index} of a flat surface. By relabeling the poles, we may assume that $p_{n-1}$ and $p_n$ form the pair of simple poles. Denote $\delta \coloneqq \gcd(a_1,\dots,a_m,b_1,\dots,b_{n-2})$. The two simple poles are represented by two half-infinite cylinders in the flat metric. By cutting along closed geodesics around the simple poles and gluing, we obtain a new flat surface $(X',\omega')$ of genus one with a marked closed geodesic $\alpha$, the core curve of the resulting cylinder. Note that a small deformation of $(X,\omega)$ in $\cP(\mu^\fR)$ lifts to a small deformation of $(X',\omega')$ with a marked closed geodesic $\alpha$.

We can extend $\alpha$ to a symplectic basis $\{\alpha,\beta\}$ of $H_1(X')$. Given a class $[\beta]$, $\operatorname{Ind}\beta$ is uniquely determined modulo $\delta$. Suppose that $\{\alpha,\beta'\}$ is another symplectic basis. Then $[\beta']=k[\alpha]+[\beta]$ for some $k\in \mathbb{Z}$. Therefore,
\[
\operatorname{Ind}\beta' \equiv \operatorname{Ind}\beta \pmod{\delta}.
\]
Thus, we define the following invariant.

\begin{definition} \label{Def:index}
    Let $(X,\omega)$ be a flat surface contained in a stratum $\cP(\mu^\fR)$ with $g=0$ and $k=1$. Then the \emph{index} of $(X,\omega)$ is defined as $\operatorname{Ind}(X,\omega)\equiv \operatorname{Ind}\beta \pmod{\delta}$.

    The index of a connected component $\cC$ of $\cP(\mu^\fR)$ is defined as the index of any flat surface in $\cC$.
\end{definition}

The index of a flat surface in $\cP(\mu^\fR)$ is uniquely represented by an integer $1\leq I\leq \delta$. Note that the rotation number of $(X',\omega')$ equals $\gcd(\operatorname{Ind}\alpha,\operatorname{Ind}\beta,\delta)=\gcd(\delta,\operatorname{Ind}(X,\omega))$. In particular, for a fixed rotation number $r\mid \delta$ of $(X',\omega')$, there are $\phi(\delta/r)$ distinct indices for $(X,\omega)$. We now have the following classification result:

\begin{theorem} \label{Thm:main01}
    Suppose that $g=0$ and $\mu^\fR=(\widetilde{\mu}^{\widetilde{\fR}}\mid-1,-1)$ has exactly one pair of simple poles. Then for each $1\leq I\leq \delta$, there exists a unique non-hyperelliptic component of $\cP(\mu^\fR)$ with index $I$, except in the following cases:
    \begin{itemize}
        \item If $\widetilde{\mu}^{\widetilde{\fR}}=(r,-r)$, there exists \emph{no} component with index $r$. Moreover, if $r$ is even, then there is \emph{no} non-hyperelliptic component with index $r/2$.
        \item If $\widetilde{\mu}^{\widetilde{\fR}}$ is of the form $(r,r,-2r)$, $(2r,-r,-r)$, or $(r,r,-r,-r)$, then there is \emph{no} non-hyperelliptic component with index $r$.
    \end{itemize}
\end{theorem}

Finally, although the present paper does not cover the strata of residueless meromorphic differentials of genus $0$, we state the following conjecture:

\begin{conjecture} \label{conj:genuszeroresidueless}
    Any nonempty stratum $\cP(a_1,\dots,a_m\mid -b_1\mid\dots\mid -b_n)$ of genus $0$ is connected.
\end{conjecture}

\subsection{Structure of the paper}

\begin{itemize}
    \item In \Cref{Sec:PrincipalBoudnary}, we describe the principal boundary of generalized strata.
    \item In \Cref{Sec:Base}, we state the classification of the connected components of one-dimensional strata. 
    \item In \Cref{Sec:hyperelliptic}, we classify the hyperelliptic components, assuming the results for one-dimensional strata. 
    \item In \Cref{Sec:ssc}, we prove the existence of a flat surface with a multiplicity-one saddle connection for each non-hyperelliptic connected component, using the statements presented in \Cref{Sec:Base}.
    \item In \Cref{Sec:Nonhyper}, we classify the non-hyperelliptic components, assuming the results presented in \Cref{Sec:Base}. 
    \item In \Cref{Sec:equatorial}, we introduce the projective structure on one-dimensional strata given by period coordinates. 
    \item In Sections~\ref{Sec:Bsignature}--\ref{Sec:Dsignature}, we prove the results stated in \Cref{Sec:Base} using the projective structure introduced in \Cref{Sec:equatorial}.
\end{itemize}

\subsection{Motivation and future work}

The multi-scale compactification $\overline{\cP}(\mu)$ of a stratum of holomorphic differentials is a smooth compactification with normal crossings boundary. Understanding the topology of this compactification may lead to new results about the topology of the stratum $\cP(\mu)$. The boundary strata of $\mathbb{P}\overline{\cH}(\mu)$ are enumerated by level graphs and are isomorphic to products of generalized strata up to finite covers. Therefore, classifying the connected components of generalized strata is the first step in studying boundary strata of the multi-scale compactification. Nonetheless, further work is required to achieve a comprehensive description of the irreducible components of the boundary strata in $\overline{\cP}(\mu) \setminus \cP(\mu)$. In particular, one must also classify the connected components of {\em generalized strata of differentials on disconnected Riemann surfaces}. These arise as subspaces of the projectivized product 
\[
\mathbb{P}\left( \prod_i \mathcal{H}_{g_i}(\mu_i) \right)
\]
defined by residue conditions that involve poles belonging to different components of the product. Such conditions introduce additional complexity in both the geometry and topology of the compactifications, making their classification a necessary step toward fully understanding the boundary complex of $\overline{\cP}(\mu)$.

In \cite{payne2013boundary}, Payne introduced the notion of the {\em boundary complex} $\Delta(D)$ associated with a smooth compactification $\widetilde{X}$ of a smooth complex variety $X$, where the boundary divisor $D = \widetilde{X} \setminus X$ is a normal crossings divisor. The boundary complex $\Delta(D)$ is defined as a $\Delta$-complex whose $k$-cells correspond to the irreducible components of the $(k+1)$-codimensional boundary strata of $D$, with face relations determined by inclusion of these strata. One remarkable feature of this construction is that the (reduced) homology of $\Delta(D)$ encodes the top-weight cohomology of the original variety $X$.

This boundary complex technique extends naturally to the setting of smooth complex Deligne–Mumford (DM) stacks, which is particularly relevant for moduli spaces such as the moduli space of curves with marked points $\mathcal{M}_{g,n}$, and for the (projectivized) strata of differentials $\cP(\mu)$. In these cases, the appropriate compactifications are the Deligne–Mumford compactification $\overline{\mathcal{M}}_{g,n}$ and the compactification of the strata via the moduli space of multi-scale differentials $\overline{\cP}(\mu)$, respectively.

The boundary complexes of $\mathcal{M}_{g,n}$ have been studied in detail in \cite{chan2021tropical} and \cite{chan2022topology}. However, the situation is more intricate for $\overline{\cP}(\mu)$. Unlike the case of $\overline{\mathcal{M}}_{g,n}$, where boundary strata are unions of moduli spaces of curves of lower genus or with fewer marked points, the boundary points of $\overline{\cP}(\mu)$ typically correspond to differentials in strata of the form $\cP(\mu'^{\fR})$, i.e., strata with residue conditions and of strictly smaller dimension. This highlights the necessity of understanding the structure and connected components of strata of differentials with residue constraints in order to analyze the irreducible components of the boundary $\overline{\cP}(\mu) \setminus \cP(\mu)$ and, more generally, to study the topology of the boundary complex.

\subsection*{Acknowledgements} 
The authors would like to thank their advisors, Samuel Grushevsky and Martin Möller, respectively, for valuable discussions and comments. They are also grateful to the Simons Foundation International, LTD. and the Collaborative Research Centre TRR 326 “Geometry and Arithmetic of Uniformized Structures” for the invitations to visit Stony Brook and Frankfurt, respectively, during the preparation of this paper.

%% file: 02PrincipalBoundary.tex
\section{Deformation of flat surfaces and principal boundary of strata} \label{Sec:PrincipalBoudnary}

In this section, we recall the natural $\operatorname{GL}^{+}(2,\RR)$-action on the strata and the flat surfaces with degenerate core. We also recall the principal boundary of strata, which can be approached by shrinking a collection of parallel saddle connections.

\subsection{Period coordinates of the strata of differentials}

Given a flat surface $(X,\omega)$ in $\cH(\mu)$ or $\cH(\mu^\fR)$, 
{\em shrinking saddle connections} refers to deforming the flat surface within the given stratum so that the lengths of the chosen saddle connections decrease. Thus, it is essential to understand the local coordinates of a stratum, especially when residue conditions are involved. Consider a stratum $\cH(\mu^\fR)$. Let $K(\fR)$ be the submodule of $H_1(X \setminus \boldsymbol{p}, \boldsymbol{z}; \ZZ)$ generated by the classes of loops encircling all poles within each part of the residue condition $\fR$. We say that an oriented multicurve $\alpha_1 + \dots + \alpha_k$ on $X$ is {\em $\fR$-homologous to zero} if its class vanishes in $H_1(X \setminus \boldsymbol{p}, \boldsymbol{z}; \ZZ)/K(\fR)$. Geometrically, an oriented multicurve $\alpha_1 + \dots + \alpha_k$ is $\fR$-homologous to zero if it bounds a subsurface whose set of poles is a union of parts of $\fR$ 
(see \Cref{fig:R_hom_zero}). Given two oriented curves $\alpha_1$ and $\alpha_2$, we say they are {\em $\fR$-homologous} if the difference $\alpha_1 - \alpha_2$ is $\fR$-homologous to zero. 

\begin{figure}[h!]
    \centering
    \resizebox{12cm}{4cm}{\includegraphics[]{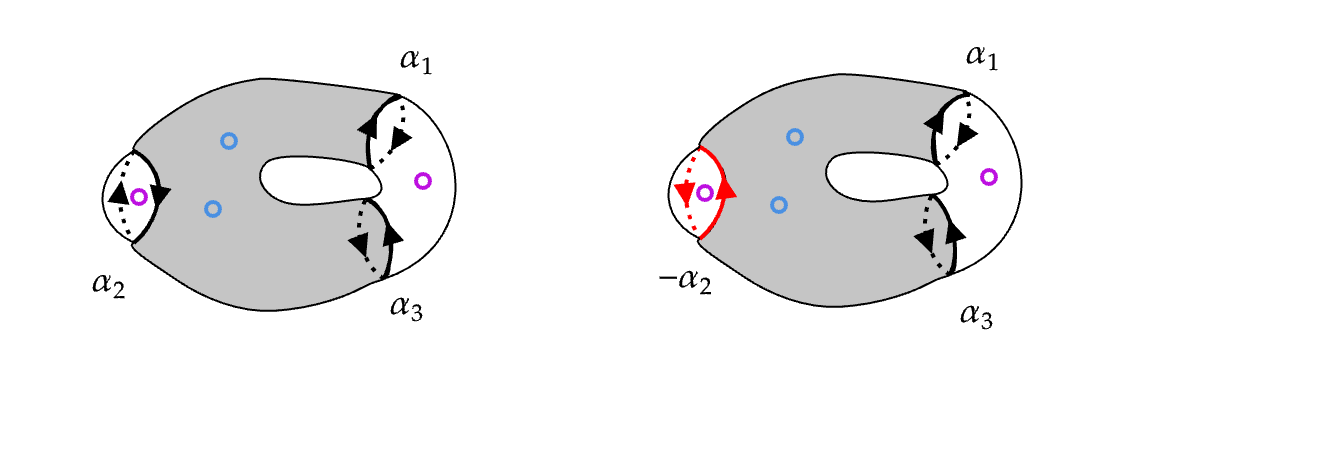}}
    \caption{Example (left) and non-example (right) of multi-curves $\fR$-homologous to zero}
    \label{fig:R_hom_zero}
\end{figure}

A flat surface $(X,\omega) \in \cH(\mu^\fR)$ induces an element of 
\[
\Hom\left(H_1(X \setminus \boldsymbol{p}, \boldsymbol{z}; \ZZ)/K(\fR), \CC\right)
\]
via the map
\[
\alpha \mapsto \int_\alpha \omega.
\]
We fix a basis 
\[
\boldsymbol{\alpha} = (\alpha_1, \dots, \alpha_{2g + n + m - |\fR| - 1})
\]
of the relative homology group $H_1(X \setminus \boldsymbol{p}, \boldsymbol{z}; \ZZ)/K(\fR)$. This basis can be transported to nearby surfaces, and thus defines the {\em period map} on a generalized stratum in a neighborhood of $(X,\omega)$:
\begin{align*}
    \Per_{\boldsymbol{\alpha}} \colon \quad & U_{(X,\omega)} \subset \cH(\mu^\fR) 
    \longrightarrow \CC^{2g + n + m - |\fR| - 1} \\
    & (X',\omega') \longmapsto \left( \int_{\alpha_1} \omega', \dots, \int_{\alpha_{2g + n + m - |\fR| - 1}} \omega' \right),
\end{align*}
which defines a local chart on $U_{(X,\omega)}$.

The natural $\operatorname{GL}^+(2,\RR)$-action on $\cH(\mu)$ is defined as follows: let $s = x + iy$ be a local flat coordinate such that $\omega = ds$. At a zero of order $a-1$, we have a coordinate $s$ such that $\omega = d(s^a)$. For $M=\left(\begin{smallmatrix}
a & b \\
c & d 
\end{smallmatrix}\right)\in \operatorname{GL}^{+}(2,\RR),$ we let $s'\coloneqq(ax+by)+i(cx+dy)$. This gives a new local coordinate $s'$. These new local charts glue together and determine a new complex structure $X'$ on the same topological surface. Moreover, $X'$ has a holomorphic differential $\omega'$ such that $(X',\omega')\in \cH(\mu)$, locally given by $ds'$ (and $d((s')^a)$ at a zero of order $a-1$). The $\operatorname{GL}^+(2,\RR)$-action is then defined by $M\circ (X,\omega)\coloneqq(X',\omega')$. 

\subsection{Flat surfaces with degenerate core} \label{subsec:core}

In this subsection, we introduce flat surfaces with degenerate core, which can be easily described. For a flat surface $(X,\omega)$, a subset $Y\subset X$ is said to be {\em convex} if any straight line joining two points in $Y$ is also contained in $Y$. The {\em convex hull} of a subset $Y$ is the smallest convex subset of $X$ containing $Y$. Recall that the {\em core} $C(X)$ of $(X,\omega)$ is defined to be the convex hull of the set of zeroes $\boldsymbol{z}$. In particular, $C(X)$ contains all zeroes and saddle connections of $(X,\omega)$. In \cite{Tah}, Tahar established the following properties of the core, allowing us to decompose every flat surface into the core and the polar domains:

\begin{proposition}[\cite{Tah}{Lemma 4.5}]  \label{degenerate}
For any flat surface $(X,\omega)\in \cH (\mu)$, $\partial C(X)$ is a finite union of saddle connections. The complement $X\setminus C(X)$ has exactly $n$ connected components, each of which is homeomorphic to a disk containing one pole $p_i$ of $\omega$. 
\end{proposition}

For a pole $p$ of $\omega$, the connected component of $X\setminus C(X)$ containing $p$ is called the {\em polar domain} of $p$.

A flat surface $(X,\omega)$ is said to have {\em degenerate core} if the core $C(X)$ has empty interior. Since $C(X)$ is closed, this is equivalent to saying $C(X)=\partial C(X)$. By \cite[Lemma 5.15]{Tah}, we can construct a flat surface with degenerate core contained in any connected component of a stratum. A consequence of the above proposition is the following:

\begin{proposition} \label{prop:sc_ex}
For any flat surface $(X,\omega)$, there exists a finite collection of saddle connections $\gamma_1,\dots,\gamma_N$ of $(X,\omega)$ such that their homology classes generate $H_1(X\setminus{\boldsymbol{p}}, \boldsymbol{z} ; \ZZ)$. 
\end{proposition}

\subsection{Zero-dimensional strata}

We recall the following two propositions classifying the zero-dimensional strata. In Sections~\ref{Sec:Bsignature}--\ref{Sec:Dsignature}, they will be building blocks of flat surfaces in one-dimensional strata via plumbing construction. 

\begin{proposition}[\cite{chen2019principal}{Proposition 2.3}] \label{zerodim1} 
Let $(\PP^1,\omega)\in \cP(a_1,a_2\mid-b_1\mid\dots\mid-b_n)$. Then it has exactly $n$ saddle connections joining $z_1$ and $z_2$, parallel to each other. Also, $(\PP^1,\omega)$ is determined by the following information:
\begin{itemize}
    \item A cyclic order on $\boldsymbol{p}$ given by $\tau:\{1,\dots, n\}\to \boldsymbol{p}$.
    \item A tuple of integers ${\bf C}= (C_1,\dots,C_n)$ such that $1\leq C_i\leq b_i-1$ for each $i=1,\dots,n$, satisfying $\sum_i C_i = a_1+1$.
\end{itemize}
We denote this flat surface $(\PP^1,\omega)$ by $Z_1(\tau,{\bf C})$. 
\end{proposition}

\begin{proposition}[\cite{chen2019principal}{Proposition 3.8}] \label{zerodim2}
Let $(\PP^1,\omega)\in \cP(a\mid-b_1\mid\dots\mid-b_{n-2}\mid-b_{n-1},-b_n)$. Then it has $n-1$ saddle connections, parallel to each other. Also, $(\PP^1,\omega)$ is determined by the following information:
\begin{itemize}
    \item An order given by $\tau:\{1,\dots, n-2\}\to \boldsymbol{p}$ on the set of $n-2$ residueless poles.
    \item A tuple of integers ${\bf C}= (C_1,\dots,C_{n-2})$ such that $1\leq C_i\leq b_i-1$ for each $i=1,\dots,n-2$.
\end{itemize}
We denote this flat surface $(\PP^1,\omega)$ by $Z_2(\tau,{\bf C},b_{n-1},b_n)$. 
\end{proposition}

Note that a level can be a disconnected flat surface. A zero-dimensional generalized stratum of (possibly) disconnected flat surfaces can be described using the second type mentioned above. Specifically, each connected component is a single-zero flat surface with one pair of non-residueless poles.

Let $(X_1, \omega_1), \dots, (X_k, \omega_k)$ be the connected components, where each $(X_i, \omega_i)$ has a pair of poles $p_i, q_i$ with opposite nonzero residues. Since there are $k$ additional residue conditions, each non-residueless pole must be related to another such pole in a different component. Therefore, by relabeling the components if necessary, the residue conditions can be expressed as
\[
\res_{p_i}(\omega_i) + \res_{q_{i+1}}(\omega_{i+1}) = 0 \quad \text{for each } i = 1, \dots, k-1.
\]
In particular, the saddle connections in these components are all parallel to each other.

\subsection{The multi-scale compactification of strata} \label{subsec:ms}

In this subsection, we will recall the multi-scale compactification $\overline{\cP}(\mu^\fR)$ of strata. This is a smooth compactification, so there is a one-to-one correspondence between the connected components of $\cP(\mu^\fR)$ and its compactification. Moreover, by degenerating to the boundary of $\overline{\cP}(\mu^\fR)$, we can reduce the complexity of the flat surfaces. This will allow us to use induction in Sections~\ref{Sec:hyperelliptic}--\ref{Sec:Nonhyper}. 

In \cite{bainbridge2018compactification} and \cite{bainbridge2019moduli}, Bainbridge,
Chen, Gendron, Grushevsky, and M\"{o}ller constructed the moduli space of multi-scale differentials $\overline{\cH}(\mu)$, which is a smooth orbifold containing $\cH(\mu)$ as a dense subset, and such that the boundary divisor is normal crossing. 

We denote by $\overline{\cH}(\mu^\fR)$ the closure of $\cH(\mu^\fR)$ in $\overline{\cH}(\mu)$. Since the residue around a pole on a flat surface is expressed as a finite sum of periods of saddle connections, the locus $\cH(\mu^\fR)$ forms a linear subvariety of $\cH(\mu)$. 

The boundary behavior of linear subvarieties in $\cH(\mu)$ has been studied in \cite{benirschke2022equations} and \cite{benirschke2022boundary}. In particular, the boundary of $\cH(\mu^\fR)$ is described explicitly in \cite{costantini2022chern}, and we will follow their approach in this paper.

We first recall from \cite{bainbridge2019moduli} the notions of enhanced level graphs and twisted differentials.

\begin{definition}
    An {\em enhanced level graph} $\Gamma$ of type $\mu$ on a marked nodal curve $(X,\boldsymbol{z} \sqcup \boldsymbol{p})$ consists of the following data:
    \begin{itemize}
        \item A dual graph of $(\overline{X}, \boldsymbol{z} \sqcup \boldsymbol{p})$, where each leg $z$ or $p$ in the set of half-edges $H(\Gamma)$ is decorated with the integer prescribed by $\mu$.
        
        \item A level function on the set of vertices:
        \[
        \ell: V(\Gamma) \longrightarrow \{0, -1, -2, \ldots, -N\}
        \]
        which is surjective.
        
        \item An edge $e \in E(\Gamma)$ is called {\em vertical} if it connects vertices of distinct levels; otherwise, $e$ is called {\em horizontal}.
        
        \item Each vertical edge $e$ is assigned a positive integer $\kappa_e$, called its {\em enhancement}.
    \end{itemize}

    For each $v \in V(\Gamma)$, let $g_v$ denote the genus of the irreducible component $X_v$ of $\overline{X}$. Then the following relation holds:
    \[
    2g_v - 2 = \sum_{z_i \mapsto v} a_i 
    - \sum_{p_j \mapsto v} b_j 
    + \sum_{e \in E^+(v)} (\kappa_e - 1)
    - \sum_{e \in E^-(v)} (\kappa_e + 1)
    - \sum_{e \in E^h(v)} 1
    \]
    where:
    \begin{itemize}
        \item[-] The first (second, respectively) sum is over all zeroes (poles) incident to $v$.
        \item[-] $E^+(v)$, $E^-(v)$, and $E^h(v)$ denote the sets of half-edges incident to $v$ going to lower, upper, and same-level vertices, respectively.
    \end{itemize}
\end{definition}

\begin{definition}
    A {\em twisted differential} compatible with an enhanced level graph $\Gamma$ of type $\mu^\fR$ on $(X, \boldsymbol{z} \sqcup \boldsymbol{p})$ is a collection of differentials $\boldsymbol{\eta} = (\eta_v)_{v \in V(\Gamma)}$ satisfying the following conditions:
    \begin{enumerate}
        \item \textbf{Prescribed orders:} The orders of the singularities of $\eta_v$ are determined by the decorations on the half-edges at $v$.
        
        \item \textbf{Balancing of horizontal nodes:} For every horizontal node $e$,
        \[
        \res_{e^+} \boldsymbol{\eta} + \res_{e^-} \boldsymbol{\eta} = 0.
        \]
        
        \item \textbf{$\fR$-Global Residue Conditions (GRCs):} Let $\widetilde{\Gamma}^\fR$ be an auxiliary level graph obtained by adding a vertex $v_R$ at level $\infty$ for each part $R \in \fR$, and connecting each $p_j$ ($j \in R$) to $v_R$. Let $\widetilde{\Gamma}^\fR_{>L}$ denote the subgraph consisting of half-edges and vertices strictly above level $L$.

        For every finite level $L$, any connected component $Y$ of $\widetilde{\Gamma}^\fR_{>L}$ containing all poles of entire parts $R \in \fR$ must satisfy the residue condition:
        \[
        \sum_{e \in E^+(Y, L)} \res_e \boldsymbol{\eta} = 0,
        \]
        where $E^+(Y, L)$ is the set of edges connecting $Y$ to vertices at level $L$.
    \end{enumerate}
    
    For brevity, we refer to the $\fR$-Global Residue Conditions as {\em $\fR$-GRCs}.
\end{definition}

A {\em multi-scale differential} in $\overline{\cH}(\mu^\fR)$ is a tuple
\[
(X, \boldsymbol{z} \sqcup \boldsymbol{p}, \Gamma, \boldsymbol{\eta}, \boldsymbol{\sigma}),
\]
where:
\begin{itemize}
    \item $(X, \boldsymbol{z} \sqcup \boldsymbol{p})$ is a pointed nodal curve;
    
    \item $\Gamma$ is an enhanced level graph of type $\mu^\fR$ on $(X, \boldsymbol{z} \sqcup \boldsymbol{p})$;
    
    \item $\boldsymbol{\eta} = (\eta_v)_v$ is a {\em twisted differential} compatible with the enhanced level graph $\Gamma$;
    
    \item $\boldsymbol{\sigma} = (\sigma_e)_e$ is a collection of cyclic-order-reversing maps identifying the geodesic rays on the two branches at each node. We call $\boldsymbol{\sigma}$ the {\em global prong-matching}, and it can be represented by an element in the {\em prong-matching group}
    \[
    P_\Gamma := \prod_e \mathbb{Z} / \kappa_e \mathbb{Z}.
    \]
\end{itemize}

There is a natural action of the {\em level rescaling group} $\mathbb{C}^L$ on a multi-scale differential, where $-L$ is the lowest level of the graph. Two multi-scale differentials are considered {\em equivalent} if they differ by the action of this level rescaling group.

The {\em level rotation group} $\mathbb{Z}^L \subset \mathbb{C}^L$ acts on the prong-matching group $P_\Gamma$ via:
\[
\cdot : \mathbb{Z}^L \times P_\Gamma \longrightarrow P_\Gamma, \quad 
(\boldsymbol{n}, (u_e)_e) \mapsto \left(u_e + n_{\ell(e^+)} - n_{\ell(e^-)}\right)_e.
\]

Thus, in the coarse moduli space $\cH(\mu^\fR)$, we consider only the prong-matching {\em equivalence classes} of multi-scale differentials, and twisted differentials up to rescaling on levels below $0$.

\subsection{Plumbing construction} \label{sec:Plumb}

In this subsection, we will describe the plumbing construction, which describes the neighborhood of a boundary of $\overline{\cP}(\mu^\fR)$. This allows us to navigate the boundary elements of one-dimensional strata in Sections~\ref{Sec:Bsignature}--\ref{Sec:Dsignature}.

Given a multi-scale differential 
\[
\overline{X} = (X, \boldsymbol{z} \sqcup \boldsymbol{p}, \Gamma, \boldsymbol{\eta}, \boldsymbol{\sigma})
\]
lying in the boundary of $\overline{\cH}(\mu^\fR)$, one can construct a complex family of flat surfaces converging to $\overline{X}$ via a flat geometric surgery called {\em plumbing}. For the details of plumbing, we refer the reader to \cite{bainbridge2019moduli}. Here, we describe only the flat geometric plumbing procedure needed to obtain a single flat surface in a neighborhood of $\overline{X}$.

We refer the reader to \cite{boissy2015connected} for details on representing a non-compact flat surface using {\em basic domains}, which include infinite half-disks and half-infinite cylinders. Note that the basic domain corresponding to a simple pole is a half-infinite cylinder, whose base vector equals the residue. Thus, to plumb a horizontal node, one simply cuts off the two corresponding half-infinite cylinders and identifies their boundaries.

To plumb {\em vertical nodes}, the surgery must be performed $\fR$-GRC-wise rather than one node at a time. Observe that each nodal or marked pole is involved in exactly one $\fR$-GRC. Furthermore, if two nodal poles lie on the same level but belong to different $\fR$-GRCs, then after smoothing the nodes at higher levels, the resulting configuration consists of two disjoint flat surfaces near the two nodes. This implies that those nodes can be smoothed independently.

In this paper, we only require plumbing techniques for the following configurations:
\begin{itemize}
    \item[(i)] A single residueless node, where the nodal polar domain is a polar domain of Type I;
    \item[(ii)] A pair of nodes whose residues sum to zero, with each nodal polar domain being of Type II;
    \item[(iii)] A single non-residueless node, where the nodal polar domain is a polar domain of Type II.
\end{itemize} In \Cref{fig:polar_domain}, we illustrate polar domains of Types I and II. 

\begin{figure}[h!]
    \centering
    \includegraphics[width=0.85\textwidth]{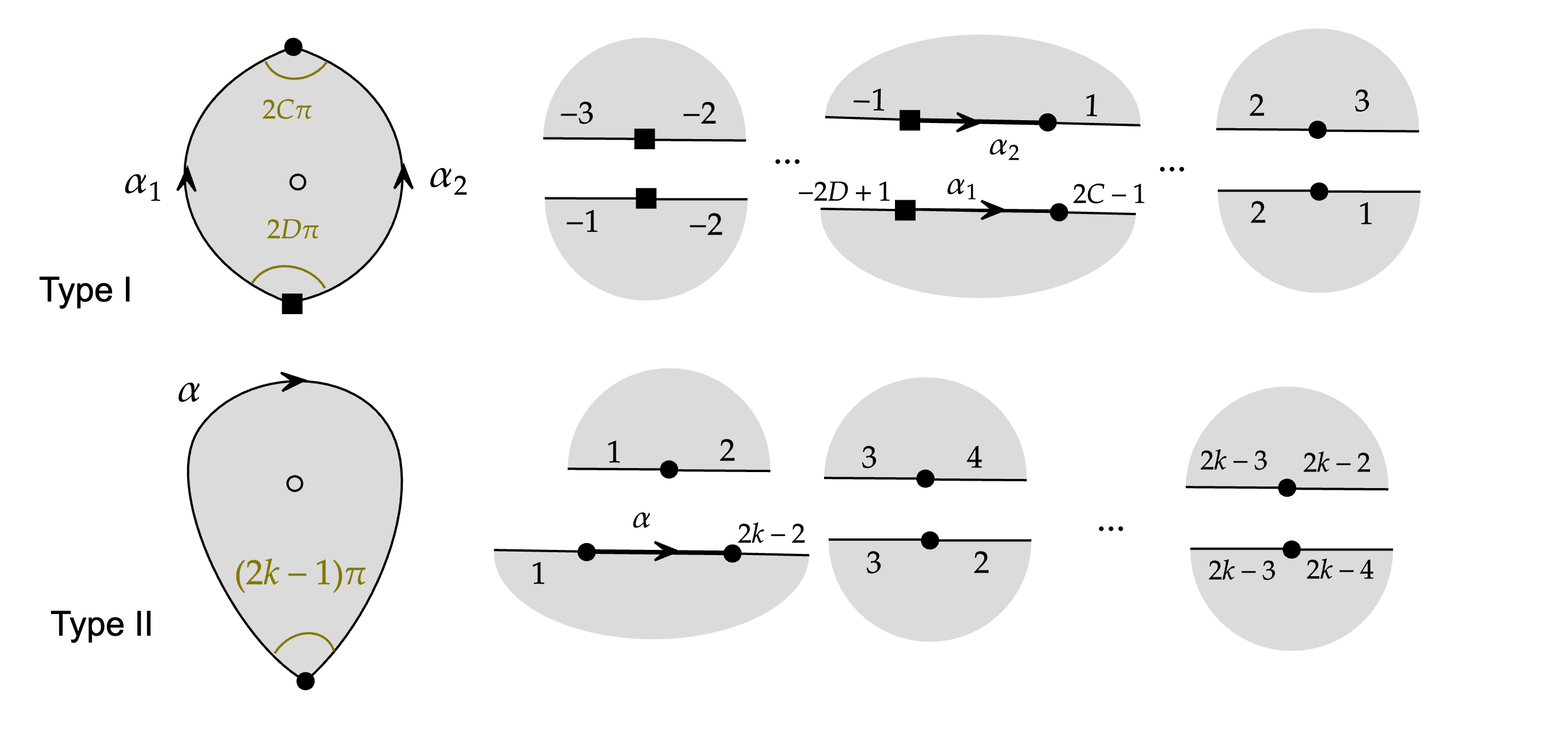}
    \caption{Type I and II polar domains and their representation by basic domains}
    \label{fig:polar_domain}
\end{figure}

The plumbing constructions for cases (i) and (ii) can be described in a fully general manner:

\paragraph{Case (i):} This corresponds to the standard disk surgery known as {\em breaking up a zero}. Given a prong-matching such that the prong $v^+_i$ (respectively $v^+_{i+C-1}$, where the angle between $\alpha_1$ and $\alpha_2$ is $2C\pi$) is the closest outgoing prong in the clockwise direction to the saddle connection $\alpha_1$ (respectively $\alpha_2$) from the bottom level. The resulting flat surface with a conical singularity is shown in \Cref{fig:splittzero}.

\begin{figure}[h!]
    \centering
    \includegraphics[width=0.8\textwidth]{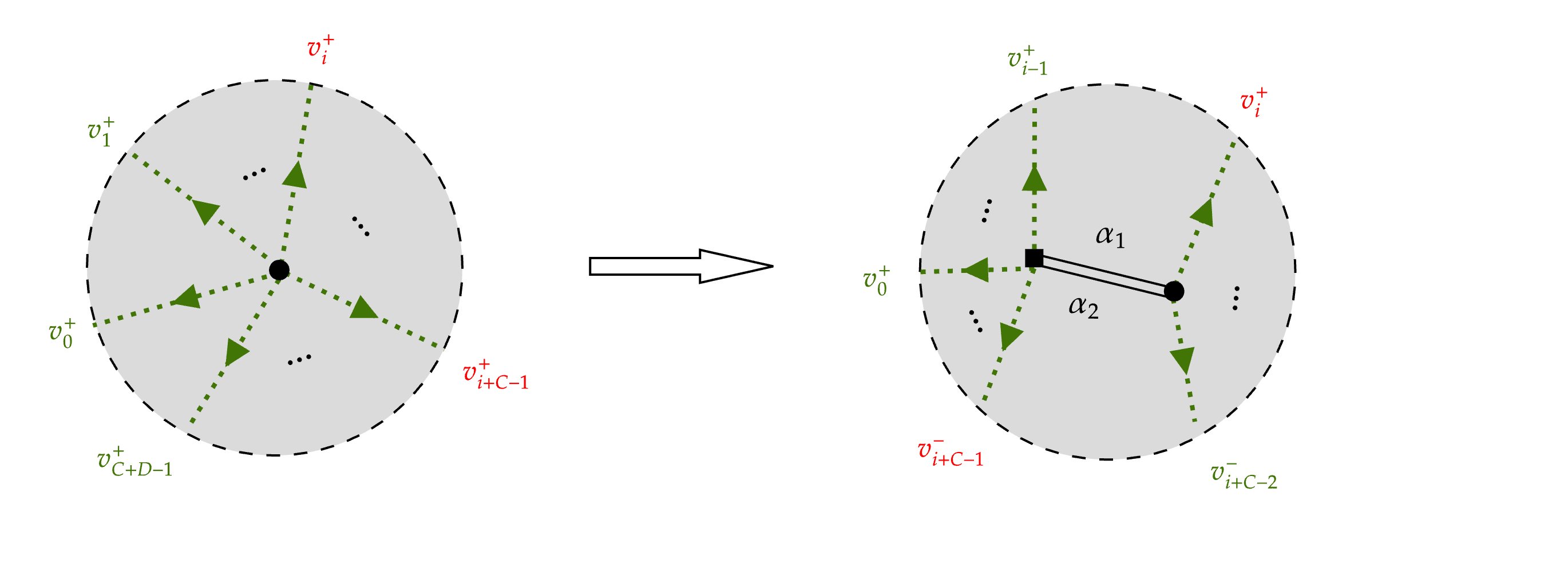}
    \caption{breaking up a zero}
    \label{fig:splittzero}
\end{figure}

\paragraph{Case (ii):} Our construction in this case slightly differs from the plumbing approach in \cite{bainbridge2018compactification}, though the underlying philosophy remains the same. Without loss of generality, we assume that $\alpha_1$ and $\alpha_2$ are not purely imaginary.

Assume a prong-matching such that $v^+_i$ (respectively $w^+_j$) is the closest outgoing (incoming) prong in the clockwise (respectively counter-clockwise) direction to $\alpha_1$ (respectively $\alpha_2$). We begin by constructing a directed path consisting of broken line segments, none of which are parallel to $\alpha_1$ or $\alpha_2$.

The initial segment (respectively the final segment) is sandwiched between $v^+_i, v^+_{i-1}$ (respectively $w^+_j, w^+_{j-1}$). We label the segments such that
\[
\gamma_{\nu,1}, \dots, \gamma_{\nu,n_\nu}
\]
form a maximal sequence of consecutive segments with the same sign of imaginary part. Specifically, $\gamma_{2k+1,\bullet}$ (respectively $\gamma_{2k,\bullet}$) are the segments with positive (respectively negative) imaginary part, and are colored in blue (respectively purple) in \Cref{fig:Plumb2node_1}.

\begin{figure}[h!]
    \centering
    \resizebox{15.8cm}{13.5cm}{\includegraphics[]{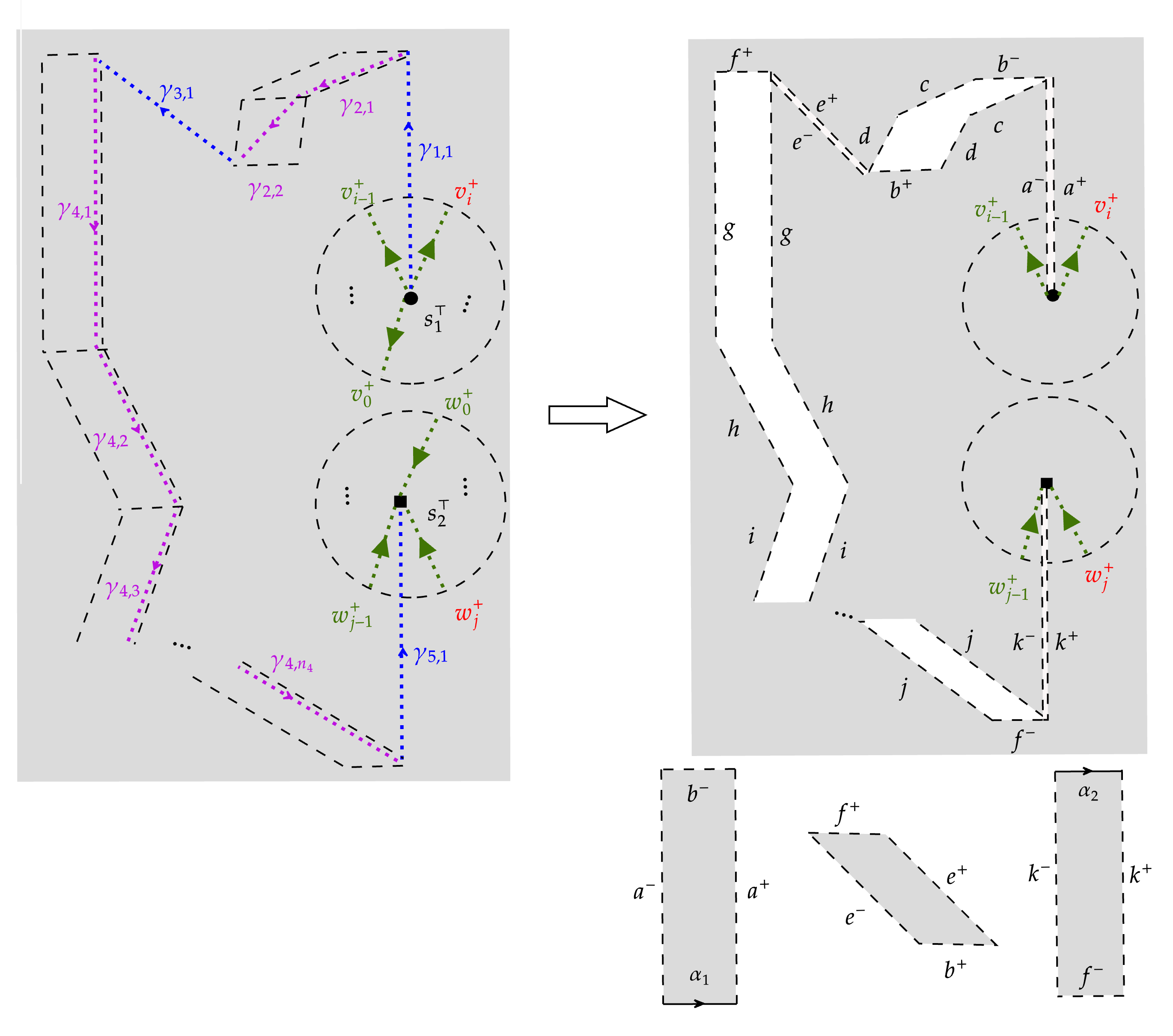}}
    \caption{Plumbing two nodes}
    \label{fig:Plumb2node_1}
\end{figure}

Without loss of generality, we may assume that the segments of the form $\gamma_{2k,n_{2k}}$ are not purely imaginary. We now construct broken strips of width $t = |\alpha_1| = |\alpha_2|$ along paths with negative imaginary part, ensuring that these strips do not intersect the blue segments.

Let
\[
c = \operatorname{sgn}\left( \arg\left( \frac{\gamma_{2k,1}}{\gamma_{2k-1,n_{2k-1}}} \right) \right).
\]
If
\[
\operatorname{sgn}\left( \arg\left( \frac{\gamma_{2k+1,1}}{\gamma_{2k,n_{2k}}} \right) \right) = c,
\]
then for each line $\gamma_{2k,\bullet}$, we take the parallelogram with sides $\gamma_{2k,\bullet}$ and $\gamma_{2k,\bullet} + ct$. Otherwise, for the segment $\gamma_{2k,n_{2k}}$, we take a parallelogram whose diagonal is given by $\gamma_{2k,n_{2k}}$ (see \Cref{fig:Plumb2node_1}). After constructing these strips, we cut along the remaining lines and glue the boundaries using the auxiliary strips. The identification process is depicted in \Cref{fig:Plumb2node_1}.

We can carry out the plumbing more directly in the case where the upper-level flat surface is non-compact. In such situations, the modification of the top-level flat surface becomes more explicit due to its representation via basic domains.

According to the construction in \cite{boissy2015connected}, for each saddle connection in a non-compact flat surface representation, there exists an infinite strip containing it. This allows us to vary the length of such saddle connections freely without disrupting the entire flat surface. Given a prong matching, we modify the upper-level flat surface as follows to preserve the residue condition after plumbing:

\begin{itemize}
    \item[(i)] If the outgoing prongs $v^+_i$ and $v^+_{i+C-2}$ converge to the same pole, no modification is required. Otherwise, we find a directed path $\gamma$, avoiding conical singularities, that connects $v^+_i$ with $v^+_{i+C-2}$. We then increase the length of the saddle connection intersected by $\gamma$ (within the lower basic domains) by $+t$ (see \Cref{fig:modify_case_i}, where the blue saddle connections are adjusted).
    
    \item[(ii)] This case is treated similarly to case (i), except when the nodal zeros $s^\top_1$ and $s^\top_2$ lie on distinct top-level components. In that case, we locate marked poles on both components that belong to the same part $R \in \fR$. We then construct directed paths $\gamma_1$ and $\gamma_2$ connecting $v^+_i$ to one pole of $R$, and another pole of $R$ to $w^+_j$, respectively. We adjust the length of saddle connections intersected with $\gamma_1$ or $\gamma_2$ (within the lower basic domains) by $+t$ (see \Cref{fig:modify_case_ii}).
    
    \item[(iii)] Since the node has non-zero residue, we can again locate marked poles on the top and bottom level components belonging to the same part $R \in \fR$. If the prong $v^+_i$ already converges to a pole in $R$, no modification is needed. Otherwise, we find a directed path $\gamma$ connecting $v^+_i$ to a pole in $R$, and adjust the length of the intersected saddle connection (within the lower basic domains) by $+t$.
\end{itemize}

\begin{figure}[h!]
    \centering
    \resizebox{14cm}{5cm}{\includegraphics[]{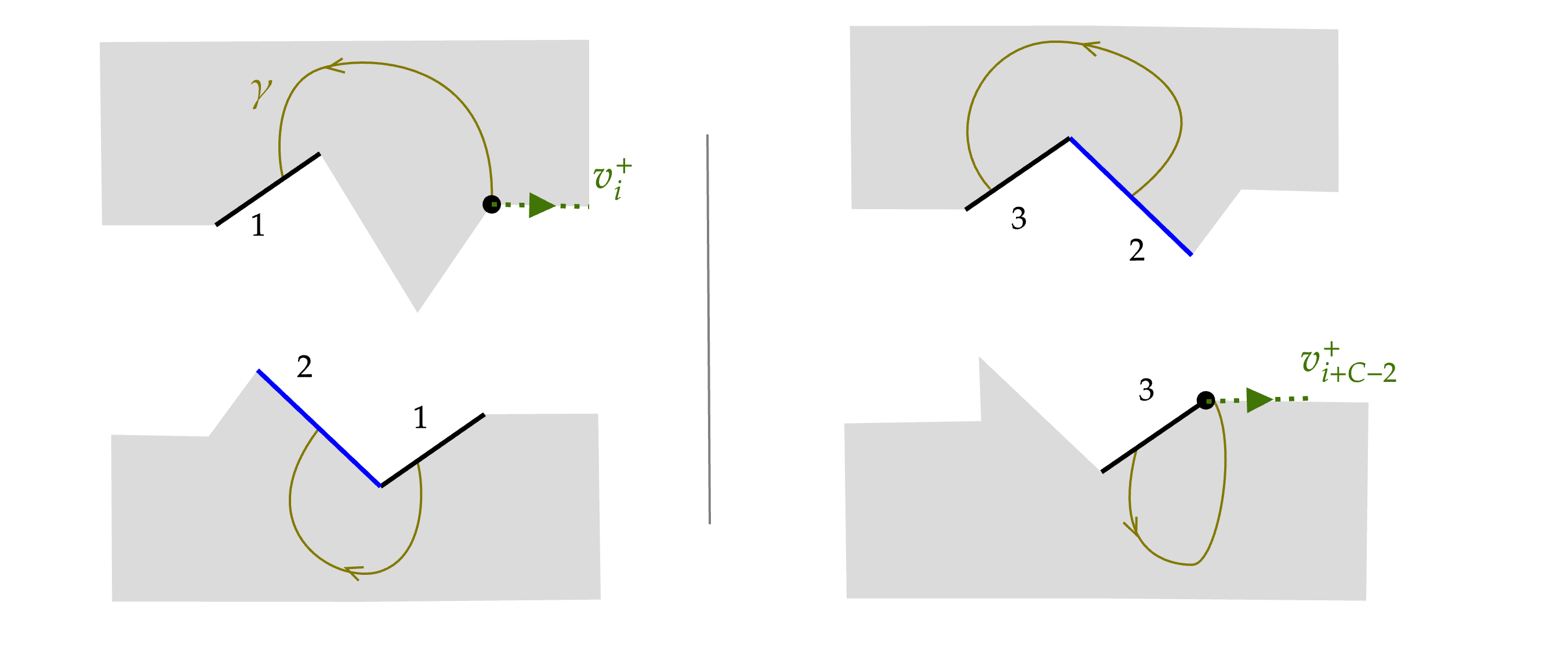}}
    \caption{Modification of top level flat surface case (i)}
    \label{fig:modify_case_i}
\end{figure}

\begin{figure}[h!]
    \centering
    \resizebox{14cm}{5cm}{\includegraphics[]{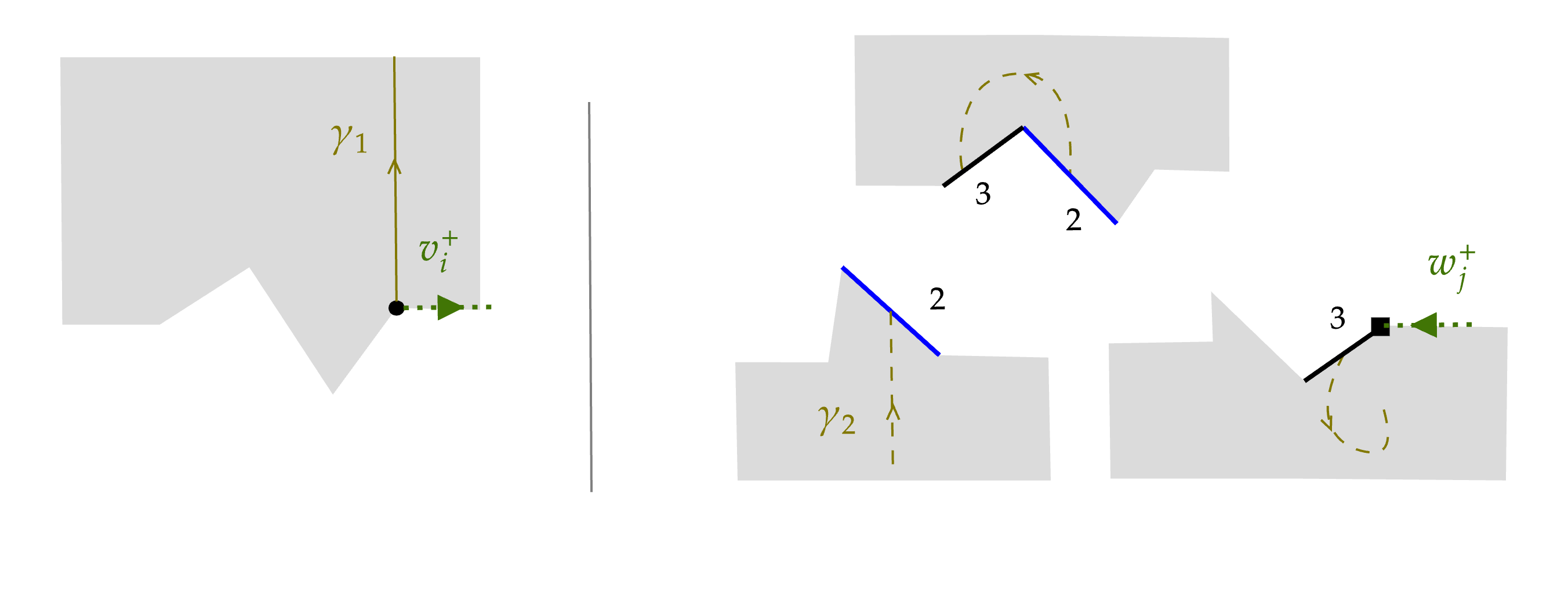}}
    \caption{Modification of top level flat surface case (ii)}
    \label{fig:modify_case_ii}
\end{figure}

After modifying the upper-level flat surface, we append the relevant saddle connections from the bottom level to the basic domains of the differentials on the top level. This final step completes the construction of the plumbed flat surface.

\subsection{Configuration of Saddle Connections and Principal Boundary}

In this subsection, we will recall the definition of principal boundary of strata. This is a boundary divisor obtained by shrinking one collection of parallel saddle connections. In Sections~\ref{Sec:hyperelliptic}--\ref{Sec:Nonhyper}, we will use the principal boundary to apply induction on the dimension of strata. 

The notion of the {\em principal boundary} was first introduced by Eskin, Masur, and Zorich in \cite{eskin2003moduli} for strata of holomorphic differentials. It was defined as the space parameterizing the degenerate limits of flat surfaces in a given stratum obtained by shrinking saddle connections within a fixed homologous class. This boundary is a union of products of lower-dimensional strata of holomorphic differentials, such that the total dimension is one less than that of the original stratum. Although it was called a ``boundary,'' it was not initially regarded as the boundary of any specific compactification of a stratum. In \cite{chen2019principal}, D. Chen and Q. Chen described the limits in the principal boundary using the framework of twisted differentials. In their description, there exists a (projectivized) meromorphic differential that records the relative positions of the collapsed saddle connections. This meromorphic differential lies in a zero-dimensional projectivized stratum.

Similar to the approach in \cite{chen2019principal}, we define in this section the {\em principal boundaries} as specific types of boundary divisors in the multi-scale compactification of a (projectivized) stratum of differentials. 
From now on, we denote by $\overline{\cP}(\mu^\fR)$ the multi-scale compactification of $\cP(\mu^\fR)$, as introduced in \cite{costantini2022chern}.
Before giving the precise definition of principal boundaries, we briefly describe how to approach the (principal) boundary of the multi-scale compactification via degenerations of flat surfaces.

By shrinking saddle connections, we can degenerate a flat surface into simpler flat surfaces. However, in general, it may not be possible to shrink only a single saddle connection. From the period chart of a stratum, it is clear that given a flat surface $(X,\boldsymbol{z} \sqcup \boldsymbol{p},\omega)$ in $\cH(\mu^\fR)$, the $\fR$-homologous saddle connections are {\em entangled}, i.e., these saddle connections are always parallel and have the same length as we deform the flat surface within the stratum.

In this subsection, we explain how to shrink an $\fR$-homologous class of saddle connections so that we approach the principal boundary. Note that an $\fR$-homologous class contains only finitely many saddle connections. We refer to such a collection $\cF = (\alpha_1,\dots,\alpha_k)$ of $k$ saddle connections as a {\em configuration of multiplicity $k$}. Configurations are classified based on the topological properties of their saddle connections. A saddle connection is {\em separating} if cutting along it disconnects the surface. Each saddle connection falls into one of the following categories:
\begin{itemize}
    \item[(I)] Saddle connections with two distinct endpoints.
    \item[(II)] Closed, non-separating saddle connections.
    \item[(III)] Closed, separating saddle connections.
\end{itemize}

To simplify the discussion, we may slightly deform the flat surface so that two saddle connections are parallel if and only if they are $\fR$-homologous.

\begin{proposition} \label{parallel}
For any stratum $\cH(\mu^\fR)$, there exists an open dense subset $W \subset \cH(\mu^\fR)$ such that for any flat surface $(X,\boldsymbol{z} \sqcup \boldsymbol{p},\omega) \in W$, two saddle connections are parallel if and only if they are $\fR$-homologous.
\end{proposition}

This implies that we can use the $\mathrm{GL}^+(2,\mathbb{R})$-action to distinguish the saddle connections in a given configuration $\cF$ from others. More precisely, we define a flow that contracts the saddle connections in a given direction:

\begin{definition}[Contraction Flow]
Let $\alpha, \theta \in S^1$. For $t \in \mathbb{R}^+$, we define the {\em contraction flow} on $\cH(\mu^\fR)$ by:
\[
C_{\alpha,\theta}^t =
\begin{pmatrix}
    \cos\theta & \sin\theta \\
    \cos\alpha & \sin\alpha
\end{pmatrix}
\begin{pmatrix}
    e^{-t} & 0 \\
    0 & 1
\end{pmatrix}
\begin{pmatrix}
    \cos\theta & \sin\theta \\
    \cos\alpha & \sin\alpha
\end{pmatrix}^{-1}
\in \mathrm{GL}^+(2,\mathbb{R}).
\]
The flow $C_{\alpha,\theta}^t$ contracts directions aligned with $\theta$ while preserving lengths in the direction $\alpha$ as $t \to \infty$.
\end{definition}

We apply the contraction flow in the direction of a given configuration $\cF = (\alpha_1, \dots, \alpha_k)$ so that the lengths of these saddle connections become very small compared to those of other saddle connections. It is important to note that the contraction flow does not generally converge, so we stop once the contracted saddle connections are sufficiently short to allow the application of local surgeries on the flat surface.

Cutting along the saddle connections of $\cF$ yields flat subsurfaces with short boundary saddle connections. The decomposition aligns with Sections 9 and 11 of \cite{eskin2003moduli}, adapted to the $\fR$-homologous setting:
\begin{itemize}
    \item[(I)] Each subsurface has two boundary saddle connections, forming slits. Subsurfaces are cyclically ordered and contain entire parts of $\fR$. (See \Cref{fig:config_type_I})
    
    \item[(II)] Each subsurface also has two boundary saddle connections. A unique subsurface $S^*$ contains two disjoint loops; others form slits and are totally ordered. Each contains entire parts of $\fR$. (See \Cref{fig:config_type_II})
    
    \item[(III)] Two subsurfaces have one boundary loop each; others have one slit each. All but the two contain entire parts of $\fR$. A total order exists among the subsurfaces. (See \Cref{fig:config_type_III})
\end{itemize}

Notice that each subsurface of the decomposition will have boundaries that can take one of the following forms: slits, two disjoint closed saddle connections, or a single closed saddle connection. The surgeries required to shrink the boundary saddle connections are simply the reverse of the plumbing construction described in \Cref{sec:Plumb}. 

Now we define the principal boundary in the multi-scale compactification of generalized strata, as it is defined for residueless strata in \cite{lee2023connected}.

\begin{definition}[Principal Boundary]
The codimension-one boundary strata of $\overline{\cH}(\mu^\fR)$ that can be approached by shrinking an $\fR$-homologous configuration $\cF$ of saddle connections on $(X,\boldsymbol{z} \sqcup \boldsymbol{p},\omega) \in \cH(\mu^\fR)$ are called the {\em principal boundary}. It is of Type I, II, or III, corresponding to the type of $\cF$.
\end{definition}

\begin{figure}[h!]
    \centering
    \resizebox{12.5cm}{6.5cm}{\includegraphics[]{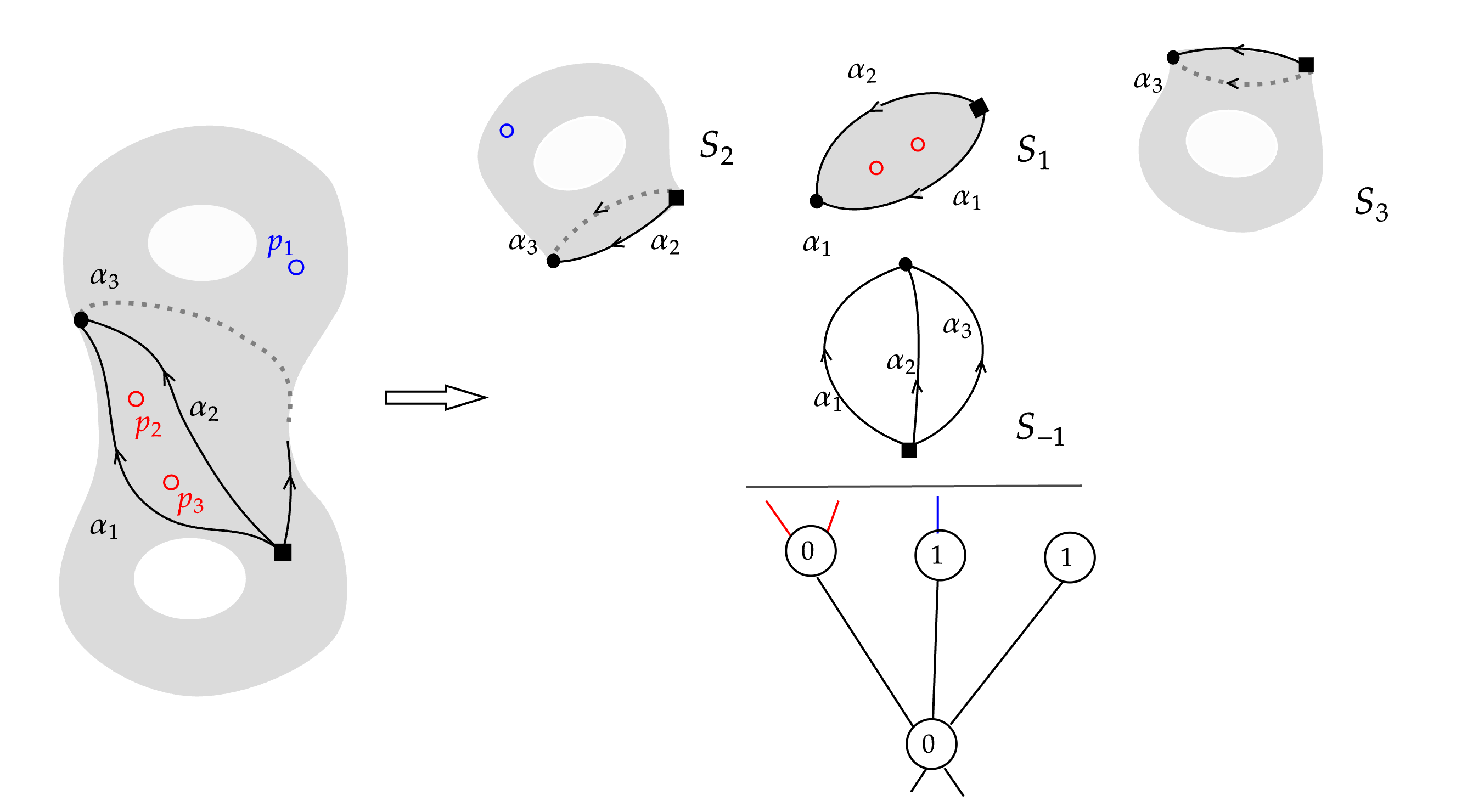}}
    \caption{Example of a Type I configuration of saddle connections}
    \label{fig:config_type_I}
\end{figure}

\begin{figure}[h!]
    \centering
    \resizebox{13.5cm}{7.2cm}{\includegraphics[]{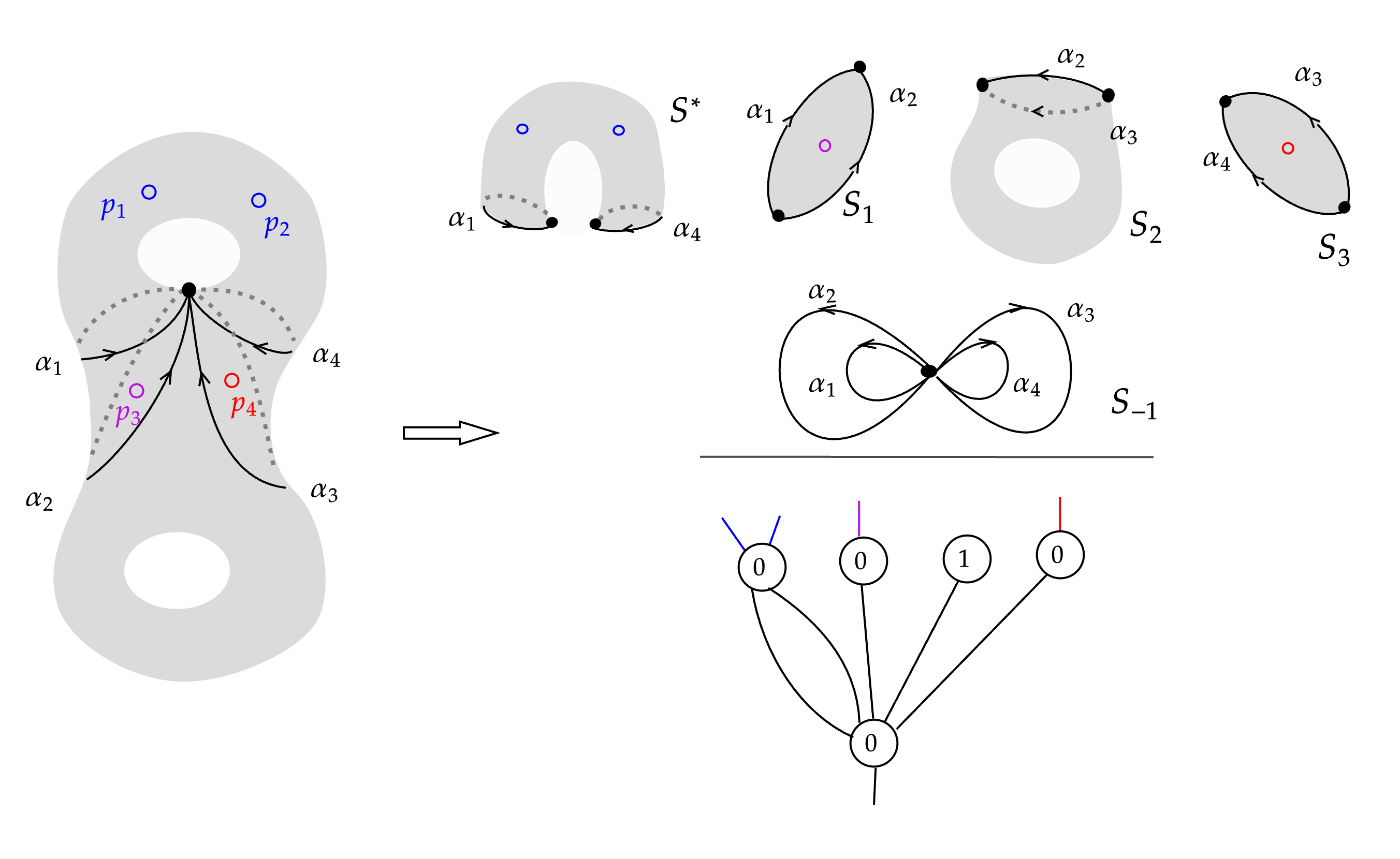}}
    \caption{Example of a Type II configuration of saddle connections}
    \label{fig:config_type_II}
\end{figure}

\begin{figure}[h!]
    \centering
    \resizebox{13.5cm}{7.5cm}{\includegraphics[]{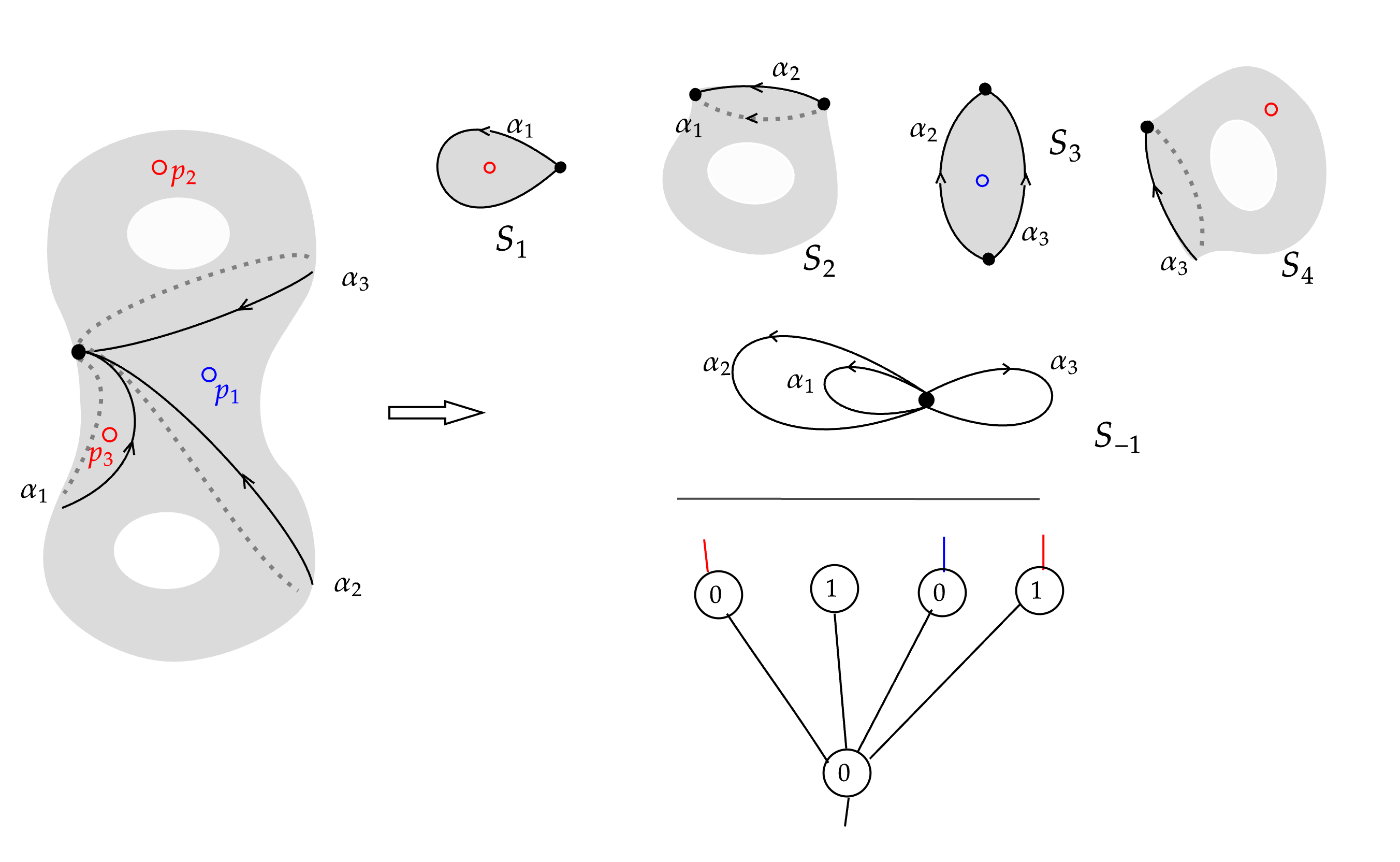}}
    \caption{Example of a Type III configuration of saddle connections}
    \label{fig:config_type_III}
\end{figure}

The configuration forms a ribbon graph $S_{-1}$ (a Type II/III configuration can have an arbitrary number of components) and is always planar. Each cycle corresponds to a boundary of a subsurface. After shrinking the saddle connections to zero, this yields a multi-scale differential
\[
\overline{X} = (X',\boldsymbol{z} \sqcup \boldsymbol{p}, \Gamma, \boldsymbol{\eta}, \boldsymbol{\sigma})
\]
in the boundary of $\overline{\cH}(\mu^\fR)$. The bottom level of $\overline{X}$ consists of differentials of genus zero, as prescribed by $S_{-1}$.

\begin{remark}
If the subsurface $S^*$ in a Type II configuration is a cylinder, we first send its height to infinity. This effectively transforms the configuration into Type III, after which the Type III shrinking process applies.
\end{remark}

After degeneration:
\begin{itemize}
    \item \textbf{Upper level:} Comprises subsurfaces that survive degeneration. Their boundary endpoints become nodal zeros.
    \item \textbf{Lower level:} Formed by filling cycles of $S_{-1}$ with polar domains (Type I or II) or half-infinite cylinders, depending on angles. Each lower-level surface has genus 0.
    \item \textbf{Marked poles:} Polar domains replacing vanishing subsurfaces inherit marked poles; otherwise, poles are nodal.
    \item \textbf{Prong-matching:} Determined by identifying prongs at the zeros of collapsed boundaries with prongs of the filled polar domains of $S_{-1}$.
\end{itemize}

\subsection{Level graphs of principal boundary}

In this subsection, we will describe the level graphs of principal boundary of Type III. The level graphs of Type I and Type II configurations are described in Section 2 of \cite{lee2023connected}. 

In this paper, we first reduce the complexity of flat surfaces by repeatedly shrinking Type I configurations. As a result, we primarily consider Type II or Type III configurations on a single-zero flat surface. The only case we describe principal boundaries of multiple-zero strata is one of the base cases, which is dealt in detail in \Cref{Sec:Bsignature}.

Assume that $\cH(\mu^{\fR})$ is a single-zero stratum with a unique zero $z$ of order $a$. Let $X \in \cH(\mu^{\fR})$ be a general flat surface in the sense of \Cref{parallel}. Suppose that $X$ has a multiplicity-$k$ saddle connection configuration, where the saddle connections $\gamma_1, \dots, \gamma_k$ are labeled in clockwise order around $z$. Assume that the homology class $[\gamma_i]$ is trivial in $H_1(X; \mathbb{Z})$. Then the complement $X \setminus \bigcup_i \gamma_i$ consists of $k+1$ connected regions. Two of these regions are bounded by a single saddle connection, namely $\gamma_1$ and $\gamma_k$, respectively.

We observe that the angle between $\gamma_i$ and $\gamma_{i+1}$ is $2\pi C_i$ on one side and $2\pi D_i$ on the other, for some positive integers $C_i, D_i$. Define ${\bf C} \coloneqq (C_1, \dots, C_k)$ and ${\bf D} \coloneqq (D_1, \dots, D_k)$. The angles at the regions bounded by $\gamma_1$ and $\gamma_k$ are $2\pi Q_1 + \pi$ and $2\pi Q_2 + \pi$, respectively, for some non-negative integers $Q_1, Q_2$. These angles satisfy the relation:$$\sum_i (C_i + D_i) + Q_1 + Q_2 = a.$$ Let $\boldsymbol{p}_i$ denote the set of poles contained in the region between $\gamma_i$ and $\gamma_{i+1}$. Let $\boldsymbol{p}_0$ and $\boldsymbol{p}_k$ denote the sets of poles in the regions bounded by $\gamma_1$ and $\gamma_k$, respectively. Then the decomposition $\boldsymbol{p} = \boldsymbol{p}_0 \sqcup \boldsymbol{p}_1 \sqcup \dots \sqcup \boldsymbol{p}_{k-1} \sqcup \boldsymbol{p}_k$ gives a partition of $\boldsymbol{p}$. For each part $P$ of $\fR$, we must have either $P \subset \boldsymbol{p}_i$ for some $i=1,\dots,k-1$, or $P \subset \boldsymbol{p}_0 \cup \boldsymbol{p}_k$. In other words, $\fR$ is finer than the partition $\boldsymbol{p} = (\boldsymbol{p}_0 \cup \boldsymbol{p}_k) \sqcup \boldsymbol{p}_1 \sqcup \dots \sqcup \boldsymbol{p}_{k-1}. $

For convenience, we denote the full collection of data 
\[
\cF \coloneqq (a, {\bf C}, {\bf D}, Q_1, Q_2, \{ \boldsymbol{p}_i \})
\]
and refer to it as a configuration. We say that the flat surface $X$ {\em has a configuration $\cF$} if there exists a collection of saddle connections on $X$ that forms $\cF$. The combinatorial data of $\cF$ uniquely determines both the level graph of the associated principal boundary and the bottom-level differential obtained by shrinking the configuration.

Given a configuration $\cF$ of Type III, we introduce the {\em configuration graph} $\Gamma(\cF)$ to describe the enhanced level graph of the multi-scale differentials in the principal boundary corresponding to $\cF$.

For each $i = 1, \dots, k - 1$, define the non-negative integer
\[
g_i \coloneqq \frac{1}{2} \left[ C_i + D_i - \sum_{p_j \in \boldsymbol{p}_i} b_j \right].
\]
Also define:
\[
g_0 \coloneqq \frac{1}{2} \left[ Q_1 - \sum_{p_j \in \boldsymbol{p}_0} b_j \right], \qquad
g_k \coloneqq \frac{1}{2} \left[ Q_2 - \sum_{p_j \in \boldsymbol{p}_k} b_j \right].
\]

If $g_i = 0$ and $\boldsymbol{p}_i = \{q_i\}$ for some pole $q_i$, then the corresponding region is isomorphic to the polar domain of $q_i$. If $1 \leq i \leq k - 1$, then $q_i$ must be residueless, and the polar domain is bounded by two saddle connections, $\gamma_i$ and $\gamma_{i+1}$. If $i = 0$ or $i = k$, then $q_i$ must {\em not} be residueless, and the polar domain is bounded by a single saddle connection, $\gamma_1$ or $\gamma_k$, respectively.

Let $J \subset \{0, \dots, k\}$ be the set of indices $i$ such that the region bounded by $\gamma_i$ and $\gamma_{i+1}$ is {\em not} isomorphic to a polar domain. Denote
\[
\boldsymbol{p}_{-1} \coloneqq \{ q_i \mid i \notin J \}.
\]

Suppose that $X$ has a configuration $\cF$ of Type III. All possible local patterns at $z$ resemble those described in the Type II principal boundary, as in \cite[Sec.~3.1]{chen2019principal}. In other words, the regions of $X \setminus \bigcup_i \gamma_i$, listed in clockwise order around $z$, fall into one of the following three types:

\begin{enumerate}[(i)]
  \item A polar domain, followed by surfaces of genus $g_i$ (for each $i \in J$) with figure-eight boundaries, followed by another polar domain.

  \item A polar domain, followed by surfaces of genus $g_i$ (for each $i \in J$) with figure-eight boundaries, followed by a surface of genus $g_k$ with a circular boundary.

  \item A surface of genus $g_0$ with a circular boundary, followed by surfaces of genus $g_i$ (for each $i \in J$) with figure-eight boundaries, followed by a surface of genus $g_k$ with a circular boundary.
\end{enumerate}

We define the {\em configuration graph} $\Gamma(\cF)$ as follows:

\begin{itemize}
    \item The set of vertices is $V(\cF) \coloneqq \{v_{-1}\} \cup \{v_i \mid i \in J\}$.
    \item The set of edges is $E(\cF) \coloneqq \bigcup \{e_i \mid i \in J\}$, where each edge $e_i$ joins $v_{-1}$ and $v_i$.
    \item Each vertex $v_i$ is assigned a collection of half-edges labeled by $\boldsymbol{p}_i$.
    \item To each vertex $v_i$ with $i \in J$, we assign a non-negative integer genus $g_i \geq 0$.
    \item To each edge $e_i$ with $i \in J$, we assign an integer $\kappa_i \coloneqq C_i + D_i - 1 > 0$.
    \item The level function $\ell: V(\cF) \to \{0, -1\}$ is defined by $\ell(v_{-1}) = -1$ and $\ell(v_i) = 0$ for each $i \in J$.
\end{itemize}

The enhanced level graph $\Gamma(\cF)$ is illustrated in \Cref{typeIIIgraph}.

\begin{figure}[ht]
    \centering
    {\includegraphics[width=0.8\textwidth]{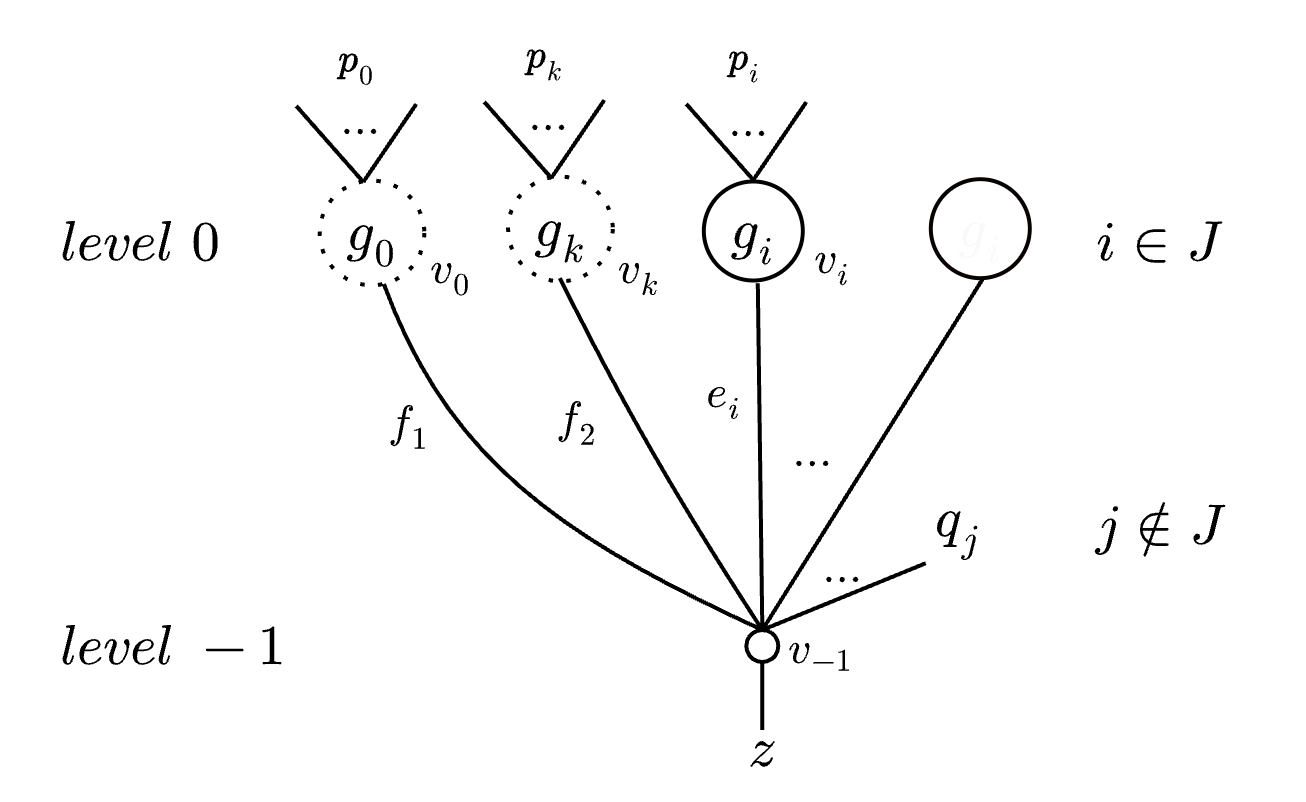}}
    \caption{The graph $\Gamma(\cF)$ for a Type III configuration $\cF$}
    \label{typeIIIgraph}
\end{figure}

The principal boundary $D(\Gamma(\cF))$ of $\overline{\cH}(\mu^{\fR})$ is the subspace of multi-scale differentials compatible with the level graph $\Gamma(\cF)$. For a multi-scale differential $(\overline{X}, \boldsymbol{\eta}) \in D(\Gamma(\cF))$, we denote the irreducible component corresponding to the vertex $v_i$ by $(X_i, \eta_i)$ for each $i \in J \cup \{-1\}$.

\subsection{Existence of principal boundaries of certain Type}

Notice that depending on $\mu^{\fR}$, we can reach the principal boundary of a certain type within any component $\cC$ of a stratum $\overline{\cH}(\mu^{\fR})$. This will be used repeatedly in the induction step to show the accessibility of different components of the stratum via plumbing.

\begin{lemma}\label{lm:univ_prin_bdry}
    Let $\cC$ be any component of a stratum $\cH(\mu^{\fR})$. Then:
    \begin{itemize}
        \item[(I)] If $\cH(\mu^{\fR})$ has more than one zero, then $\overline{\cC}$ contains a Type I principal boundary.
        \item[(II)] If $\cH(\mu^{\fR})$ has a unique zero and $g > 0$, then $\overline{\cC}$ contains a Type II principal boundary.
        \item[(III)] If $\cH(\mu^{\fR})$ any simple pole, then $\overline{\cC}$ contains a Type III principal boundary.
    \end{itemize}
\end{lemma}

\begin{proof}
    By \Cref{prop:sc_ex}, for any flat surface in $\cC$, there exists a saddle connection of Type I in case (I), or a saddle connection of Type II in case (II). Hence, by shrinking the corresponding $\fR$-homologous class of saddle connections, we can approach the desired principal boundary.

    If $\mu$ contains any $-1$, then there exists a flat surface in $\cC$ that has a closed saddle connection bounding a half-infinite cylinder corresponding to the simple pole. Suppose $(X, \omega) \in \cC$ is such a surface, and the residue of the simple pole can only be realized by multiple parallel saddle connections. We can slightly deform $(X,\omega)$ so that these saddle connections are no longer parallel. This is possible because no two of them can be $\fR$-homologous.

    On the deformed surface, there is a unique saddle connection bounding the simple pole, and this saddle connection is contained in a Type III configuration. Thus, we can approach a Type III principal boundary by shrinking this collection.
\end{proof}
    
\begin{lemma} \label{breakintotwo}
    Any connected component $\cC$ of a positive-dimensional stratum $\cP(\mu^{\fR})$ contains a two-level multi-scale differential $\overline{X} \in \partial \overline{\cC}$ satisfying the following:
    \begin{itemize}
        \item Each level contains a unique irreducible component, and the associated level graph is separable.

        \item If $g > 0$ or $m > 1$, then the nodal pole on the bottom level is forced to be residueless; that is, the top-level component contains only entire parts of poles in $\fR$.

        \item If $m > 1$, then for any given pair of zeroes $z_i, z_j$, the bottom-level component $X_{-1}$ can be chosen to contain both $z_i$ and $z_j$.
    \end{itemize}
\end{lemma}

\begin{proof}
    First, assume that $m = 1$ and $g = 0$. We will show that any flat surface in $\cC$ has a saddle connection bounding the polar domain of a non-residueless pole. 

    Any saddle connection $\gamma$ joins the unique zero to itself, and $\mathbb{P}^1 \setminus \gamma$ consists of two regions. If one of these regions contains another saddle connection $\gamma'$, we consider the smaller region bounded by $\gamma'$. By repeating this process and using the fact that any genus zero flat surface has finitely many saddle connections, we eventually obtain a saddle connection $\alpha$ that bounds a region containing no other saddle connections. This region must be the polar domain of some non-residueless pole $p$.

    Let $\cF = (\alpha_1, \dots, \alpha_k)$ denote the collection of parallel saddle connections in that direction. By shrinking them, we obtain a multi-scale differential $\overline{X}$ in the Type III principal boundary. Let $J \subset \{0, \dots, k\}$ denote the set of indices $j$ such that the top-level component $X_j$ exists. Since $\alpha_1$ bounds the polar domain of $p$, we know $0 \notin J$. If $|J| = 1$, then we are done. So assume $|J| > 1$. In particular, there exists some $X_j$ with $1 \leq j < k$. In this case, there are no residue constraints relating the poles in $X_j$ to those in the other irreducible components. Thus, we may further degenerate $\overline{X}$ by sending all components except $X_j$ to a lower level. By plumbing the level transition between level $-1$ and level $-2$, we obtain the desired two-level multi-scale differential $\overline{X}$.

    Now, assume that $g > 0$ or $m > 1$. In this case, the component $\cC$ contains principal boundaries of both Type I and Type II. Thus, we may follow the argument in \cite[Proposition 4.3]{lee2023connected} to construct a two-level multi-scale differential $\overline{X}$ with the desired properties.
\end{proof}

%% file: 03BaseCases.tex
\section{One-dimensional strata} \label{Sec:Base}

In this section, we list all types one-dimensional (projectivized) strata that will serve as base cases for the inductive arguments in the upcoming sections. We state the classification results and the existence of a flat surface with multiplicity one saddle connection. The proofs of these results for each type of one-dimensional strata will be given in detail in Sections~\ref{Sec:Bsignature}--\ref{Sec:Dsignature}. The proofs of the main theorems will be given in \Cref{Sec:ssc} and \Cref{Sec:Nonhyper}, assuming the base cases stated in this section. 

Recall that the dimension of a projectivized stratum $\cP(\mu^\fR)$ is equal to $2g+m+n-|\fR|-2$. We need to sort out all cases such that $2g+m+n-|\fR|=3$.

\begin{proposition}\label{prop:list_one_dim_strat}
    The signatures $\mu^\fR$ of one-dimensional (projectivized) strata can be classified as the following list:

\begin{itemize}
    \item A-signature : The tuple $\mu$ is an integral partition of $-2$. It has three non-negative entries while $\fR$ consists of $n$ singletons, i.e. $\mu^\fR = (a_1,a_2,a_3\mid-b_1\mid\dots\mid-b_n)$ with $a_1\leq a_2\leq a_3$.
    \item B-signature : The tuple $\mu$ is an integral partition of $-2$. It has only two non-negative entries, while $\fR$ consists of one pair and $(n-2)$ singletons, i.e. $\mu^\fR=(a_1,a_2\mid-e_1,-e_2\mid-b_1\mid\dots\mid-b_{n-2})$ with $a_1\leq a_2$ and $e_1\leq e_2$.
    \item C-signature : The tuple $\mu$ is an integral partition of $-2$. It has only one non-negative entries, while $\fR$ consists of one 3-tuple and $(n-3)$ singletons, i.e. $\mu^\fR = (a\mid-e_1,-e_2,-e_3\mid-b_1\mid\dots\mid-b_{n-3})$ with $e_1\leq e_2\leq e_3$.
    \item D-signature : The tuple $\mu$ is an integral partition of $-2$. It has only one non-negative entries, while $\fR$ consists of two pairs and $(n-4)$ singletons, i.e. $\mu^\fR = (a\mid-e_1,-e_2\mid-e_3,-e_4\mid-b_1\mid\dots\mid-b_{n-4})$ with $e_1\leq e_2$ and $e_3\leq e_4$.   
    \item E-signature : The tuple $\mu$ is an integral partition of $0$. It has one non-negative entry while $\fR$ consists of $n$ singletons, i.e. $\mu^\fR = (a\mid-b_1\mid\dots\mid-b_n)$.
\end{itemize}
\end{proposition}

The strata of A-signature may serve as the base cases for proving \Cref{conj:genuszeroresidueless}, connectedness of genus $0$ residueless strata. However, this case will not be considered in the present paper. 

The connected components of strata with E-signature have already been classified in \cite{lee2023connected}, as base cases therein. Therefore, in the present paper, we will take care of the strata of B, C, and D-signatures.

The boundary points of one-dimensional strata correspond either to multi-scale differentials with two levels and no horizontal edge, or with one level and one horizontal edge. In either case, each level is described by a zero-dimensional stratum.

\begin{proposition} \label{Baseboundary}
    Every boundary point of a one-dimensional strata $\cP(\mu^\fR)$ is a principal boundary.
\end{proposition}

\begin{proof}
    Let $\overline{X}$ be a multi-scale differential in the boundary of $\cP(\mu^\fR)$. Then the bottom level is contained in a zero-dimensional stratum. Therefore, there are exactly one collection of parallel saddle connections in the bottom level. Thus we can conclude that $\overline{X}$ is obtained by shrinking this one collection of parallel saddle connections, which means $\overline{X}$ is contained in the principal boundary. 
\end{proof}

We will classify by combinatorial data the plumbed flat surfaces near different principal boundaries and show the deformation from one to another. The deformations are combinations of prong-matching rotations and rearrangements of principal level structure. 

\subsection{Strata of B-signature} 

Now we present the classification results of connected components of strata of B-signature. The proofs will be presented in \Cref{Sec:Bsignature}. Since there are only two zeroes and $\fR$ only has parts with one or two elements, $\mu^{\fR}$ may have ramification profiles. First, we give the classification of hyperelliptic components of strata of B-signature. The proof will be given in \Cref{subsec:Dhyper}.

\begin{proposition} \label{Prop:Bhyper}
    For each ramification profile $\Sigma$ of $\mu^{\fR}$ of B-signature, there exists a unique hyperelliptic component $\cC_{\Sigma}$. 
\end{proposition}

For some special cases, we have the following

\begin{proposition} \label{Prop:Bspecial}
    Suppose that $a_1=a_2$ and $e_1=e_2$. The stratum $\cP(\mu^{\fR})$ of B-signature satisfies the following:
    
\end{proposition}
The proofs of the two propositions above can be found in \Cref{sec:Bhyp}. For the general cases, we have the following two propositions on the classification of non-hyperelliptic components, corresponding to the cases $(e_1,e_2)=(1,1)$ or $(e_1,e_2)\neq(1,1)$.

\begin{proposition} \label{Prop:Bpair}
    Assume $(e_1,e_2)=(1,1)$. For each $1\leq I\leq \delta\coloneqq \gcd(a_1,a_2,\{b_i\})$. The stratum $\cP(\mu^{\fR})$ of B-signature has a unique non-hyperelliptic component with index $I$, except the following special cases:
    \begin{itemize}
        \item $\cP(n-2,n-2\mid -2\mid\dots\mid-2\mid-1^2)$ has {\em no} non-hyperelliptic component. 
        \item $\cP(b,b\mid -2b\mid -1^2)$ has {\em no} non-hyperelliptic component with index $b$.
        \item $\cP(b,b\mid -b\mid -b|-1^2)$ has {\em no} non-hyperelliptic component with index $b$. 
    \end{itemize}
\end{proposition}

\begin{proof}
    By \Cref{prop:Bpair_standard}, we can always find Type IIIa boundary point of certain form on a non-hyperelliptic (connected) component corresponding to index $I$. Then by \Cref{cor:IIIb_perfect_1st_tail_empty_2nd_tail}, we can conclude that all the boundaries of such form are path-connected.
\end{proof}

\begin{proposition} \label{Prop:Bnonhyper}
    Suppose that $(e_1,e_2)\neq (1,1)$. The stratum $\cP(\mu^{\fR})$ of B-signature has a unique non-hyperelliptic component, except the special stratum $\cP(n-3-e,n-3-e\mid -2\mid \dots\mid -2\mid -e^2)$ that has {\em no} non-hyperelliptic component. 
\end{proposition}
\begin{proof}
    This is a consequence of \Cref{prop:e2_high_standard_boundary}.
\end{proof}

As corollaries of the above results, we have the following three statements. 

\begin{corollary} \label{Cor:Bresidueless}
    If $e_1,e_2>1$ and $a_1,a_2\geq n-1$, then any non-hyperelliptic connected component $\cC$ of the stratum $\cP(\mu^{\fR})$ of B-signature contains a residueless flat surface. Furthermore, if there exists some pole of order higher than two, then $\cC$ contains a non-hyperelliptic residueless flat surface. 
\end{corollary}

\begin{proof}
    The residueless locus $\mathbb{P}\cR(\mu)\subset \cP(\mu^{\fR})$ is nonempty if and only if $e_1,e_2>1$ and $a_1,a_2\geq n-1$. If $\cP(\mu^{\fR})$ has a ramification profile $\Sigma$, then $a_1=a_2$ and $e_1=e_2$. We can easily construct a residueless flat surface with hyperelliptic involution $\sigma$ by breaking up a zero from a hyperelliptic flat surface in $\cR(a_1+a_2,-b_1,\dots, -b_n)$. 

    If $\cP(\mu^{\fR})$ has a (unique) non-hyperelliptic component, then $a_1\neq a_2$, $e_1\neq e_2$ or $b_i\neq 2$ for some $i$. Since $a_1,a_2\geq n-1$, there exists a residueless flat surface in $\mu^{\fR}$. If there exists no ramification profile of $\mu^{\fR}$ and the flat surface is automatically contained in the non-hyperelliptic component. So assume that there is a ramification profile $\Sigma$. In particular, $a_1=a_2$, $e_1=e_2$ and $b_i>2$ for some $i$. If $p_i$ is fixed by $\Sigma$, then $b_i$ is even and thus $b_i\geq 4$.  
\end{proof}

\begin{corollary} \label{Cor:Bssc}
    Let $\cC$ be a non-hyperelliptic component of $\cP(\mu^{\fR})$ of B-signature. Then $\cC$ contains a flat surface with a multiplicity one saddle connection joining two distinct zeroes, except in the following cases:
    
    \begin{itemize}
        \item $e_1=e_2$, $a_1=a_2$ and all $b_i=2$.
        \item $(e_1,e_2)=(1,1)$ and all $b_i=2$ and the index of $\cC$ is $n-1\pmod{2}$. 
    \end{itemize}

    Furthermore, if $\mu^{\fR}\neq (a_1,a_2\mid-2\mid\dots\mid-2\mid-b^2)$, then this flat surface can be chosen so that by shrinking this saddle connection, we obtain $\overline{X}$ with non-hyperelliptic top level component $X_0$. 
\end{corollary} 
\begin{proof}
    First, assume that $(e_1,e_2)=(1,1)$ and $b_j>2$ for some $j$. By \Cref{Prop:Bpair}, $\cC$ is determined by its index $1\leq I\leq d$. Consider a two-level multi-scale differential $\overline{X}$ such that $X_{-1}\in \cP(a_1,a_2,-(a_1+a_2+2))$ and $X_0\in \cP(a_1+a_2\mid-b_1\mid\dots\mid-b_{n-2}\mid-1^2)$. This flat surface is determined by the angle $1\leq C_i\leq b_i-1$ and the permutation $\tau\in Sym_{n-2}$. The index is equal to $\sum_i C_i \pmod{\delta}$. Choose any $C_i$ for $i\neq j$. Then we can choose $C_j$ such that $\sum_i C_i \equiv I \pmod{\delta} $, unless $b_j=d$ and $\sum_{i\neq j} C_i \equiv 0 \mod d$. In particular, $b_i>2$ for all $i$. Therefore we can find another choice of $C_i$ for $i\neq j$ such that $\sum_{i\neq j} C_i\not\equiv 0 \pmod{ \delta}$. So $X_0$ can be chosen to have index $I$. By \Cref{Prop:Bpair}, $\overline{X}$ is contained in the boundary of $\cC$. By plumbing the level transition, we obtain the desired flat surface. Now suppose $b_i=2$ for all $i$. Then $\delta=2$ and the only possible choice for $X_0$ is so that all $C_i=1$. Then the index of $X_0$ is equal to $n-2 \pmod 2$. So we cannot have the desired flat surface only if the index is $n-1\pmod 2$. 

    Now we assume that $(e_1,e_2)\neq (1,1)$. By \Cref{Prop:Bnonhyper}, $\cC$ is a unique non-hyperelliptic component. We can construct $\overline{X}$ as in the previous paragraph. If $a_1\neq a_2$, then the bottom level component $X_{-1}$ is non-hyperelliptic. If $e_1\neq e_2$ or $b_i> 2$ for some $i$, then $X_0$ can chosen to be non-hyperelliptic. So $\overline{X}$ is contained in the boundary of $\cC$ and by plumbing the level transition, we obtain the desired flat surface.
\end{proof}

\begin{corollary} \label{Cor:Bhyperssc}
    Let $\cC$ be a hyperelliptic component of a stratum $\cP(\mu^{\fR})$ of B-signature. Then
    \begin{itemize}
        \item If $\cC$ has no residueless pole fixed by hyperelliptic involution, then there exists a flat surface in $\cC$ with a multiplicity one saddle connection joining $z_1$ and $z_2$.
        \item  If a residueless pole $p$ is fixed, then there exists a flat surface in $\cC$ with a pair of parallel multiplicity two saddle connections bounding the polar domain of $p$.
    \end{itemize}
       
\end{corollary} 
\begin{proof}
    If $\cC$ has no residueless fixed pole, then consider a two-level multi-scale differential $\overline{X}$ such that $X_{-1}\in \cP(a_1,a_2,-(a_1+a_2+2))$ and a hyperelliptic $X_0\in \cP(a_1+a_2\mid-b_1\mid\dots\mid-b_{n-2}\mid-e_1,-e_2)$ with the same ramification profile with $\cC$. Then by \Cref{Prop:Bhyper}, $\overline{X}$ is contained in the boundary of $\cC$. By plumbing the level transition, we obtain the desired flat surface. 

    If $\cC$ has a fixed residueless pole $p$, say $p_1$, then consider a two-level multi-scale differential $\overline{X}$ such that $X_{-1}\in \cP(a_1,a_2\mid-b_1\mid-(a_1+a_2+2-b_2))$ and $X_0\in \cP(a_1+a_2-b_1\mid-b_2\mid\dots\mid-b_{n-2}\mid-e_1,-e_2)$ and both components have involutions induced by the ramification profile of $\cC$. Then by \Cref{Prop:Bhyper}, $\overline{X}$ is contained in the boundary of $\cC$. By plumbing the level transition, we obtain the desired flat surface. 
\end{proof}

\subsection{Strata of C-signature}

For the strata of C-signature, we have the following connectedness result.

\begin{proposition} \label{Prop:Cconnected}
    Any stratum $\cP(\mu^\fR)$ of C-signature is connected.
\end{proposition}
\begin{proof}
    This is a consequence of \Cref{prop:C_standard}.
\end{proof}
As a corollary of the above result, we have the following

\begin{corollary} \label{Cor:Cssc}
    Let $\cP(\mu^\fR)$ be a stratum of C-signature and $p$ be one of the three non-residueless poles. This stratum contains a flat surface with a multiplicity one saddle connection that bounds the polar domain of $p$. 
    
    By shrinking this saddle connection, we obtain a principal boundary $\overline{X}$. The top level component $X_0$ can be made non-hyperelliptic if we further assume that $\mu^{\fR}\neq (a\mid-2\mid\dots\mid-2\mid-b,-c,-c)$ where $p$ is the order $b$ pole.  
\end{corollary}

\subsection{Strata of D-signature}

For the strata of D-signatures, we have the following connectedness result. The proofs will be presented in \Cref{Sec:Dsignature}. Since there is a unique zero and $\fR$ only has parts with one or two elements, $\mu^{\fR}$ may have ramification profiles. First, we give the classification of hyperelliptic components of strata of D-signature. The proof will be given in \Cref{subsec:Dhyper}.

\begin{proposition} \label{Prop:Dhyper}
    For each ramification profile $\Sigma$ of $\mu^{\fR}$ of D-signature, there exists a unique hyperelliptic component $\cC_{\Sigma}$. 
\end{proposition}

We have the following two proposition on classification of non-hyperelliptic components, each corresponding to the cases $(e_1,e_2)=(e_3,e_4)=(1,1)$, $(e_1,e_2)=(1,1)$ and $(e_3,e_4)\neq (1,1)$, or $(e_1,e_2),(e_3,e_4)\neq (1,1)$.

\begin{proposition} \label{Prop:Dspin}
    If $(e_1,e_2)=(e_3,e_4)=(1,1)$, then the stratum $\cP(\mu^\fR)$ of D-signature has two non-hyperelliptic component distinguished by spin parity if all $b_i$ are even, except for 
    \begin{itemize}
        \item if $\mu^\fR=(2\mid-1^2\mid-1^2)$, there is no non-hyperelliptic component;
        \item if $\mu^\fR=(4\mid-2\mid-1^2\mid-1^2)$, there is no non-hyperelliptic component of even spin.
    \end{itemize}
    Otherwise there is a unique non-hyperelliptic component. 
\end{proposition}
\begin{proof}
    This is a consequence of \Cref{prop:TypeD_spin} and the following discussion therein.
\end{proof}

\begin{proposition} \label{Prop:Dindex}
    If $(e_1,e_2)=(1,1)$ and $(e_3,e_4)\neq (1,1)$, then the stratum $\cP(\mu^\fR)$ of D-signature has unique non-hyperelliptic component with index $1\leq I\leq \delta\coloneqq \gcd(\{b_i\},e_3,e_4)$, except for $\mu^\fR=(2b\mid-b^2\mid-1^2)$ where the stratum has no non-hyperelliptic component corresponding to index $b$.
\end{proposition}

\begin{proof}
    This is a consequence of \Cref{prop:TypeD_index} and the following discussion therein. 
\end{proof}

\begin{proposition} \label{Prop:Dnonhyper}
    If $(e_1,e_2),(e_3,e_4)\neq (1,1)$, then the stratum $\cP(\mu^\fR)$ of D-signature has a unique non-hyperelliptic component. 
\end{proposition}
\begin{proof}
    This is a consequence of \Cref{prop:TypeD_nonhyp_no_sim_pair}.
\end{proof}

As a corollary of the above results, we have the following

\begin{corollary} \label{Cor:Dssc}
    Assume that $\cC$ is a non-hyperelliptic component of a stratum $\cP(\mu^\fR)$ of D-signature. 
    \begin{itemize}
        \item If $(e_1,e_2)=(1,1)$, then $\cC$ contains a flat surface with a pair of multiplicity two saddle connections $\gamma_1$ and $\gamma_2$, each bounding the polar domain of $q_1$ and $q_2$, respectively. 
        \item If $e_2>1$, then $\cC$ contains a flat surface with a multiplicity one saddle connection $\gamma$ bounding the polar domain of $q_1$.  
    \end{itemize}

In each case, by shrinking the collection of saddle connections given above, we obtain a boundary with non-hyperelliptic top level component except $e_3=e_4$, $b_i=2$ for each $i$ and $e_2\leq 2$.  
\end{corollary}

\begin{corollary} \label{Cor:Dhyperssc}
    Let $\cC$ be a hyperelliptic component of a stratum $\cP(\mu^\fR)$ of D-signature. There exists a flat surface in $\cC$ having a pair of multiplicity two saddle connections, each bounding the polar domain of $q_1$ and $q_2$. 
\end{corollary}

\subsection{Strata of E-signature}
The projectivized strata of E-signature are just single-zero strata of residueless differentials on elliptic curves and they already have been investigated in \cite{lee2023connected}. We will recall the results here.

\begin{theorem} [\cite{lee2023connected}{Theorem 1.7}]\label{Enonhyper}
    Let $\cR(\mu)$ be a genus one residueless stratum and $\delta\coloneqq \gcd(\{b_i\}_{i=1}^n)$. For each positive integer $r|\delta$, there exists a unique non-hyperelliptic component of $\cR(\mu)$ with rotation number $r$, except: 
    \begin{itemize}
        \item $\mu=(2n,-2^n)$, there is {\em no} non-hyperelliptic component of $\cR(\mu)$.
        \item $\mu=(2r, -r,-r)$ or $\mu=(2r,-2r)$, there is no non-hyperelliptic component of rotation number $r$.
        \item $\mu=(12,-3^4)$, there are exactly {\em two} non-hyperelliptic components with rotation number 3. 
    \end{itemize}
\end{theorem}

%% file: 04Hyperelliptic.tex
\section{Hyperelliptic components}\label{Sec:hyperelliptic}

In this section, we introduce the properties of hyperelliptic connected components and prove \Cref{Thm:mainhyper}. Recall that a hyperelliptic involution induces an involution $\Sigma$ on $\boldsymbol{p}$, which is invariant in a connected component. In case of hyperelliptic components, we first show that this involution $\Sigma$ is indeed a ramification profile. 

\begin{proposition}
    If $\cC$ is a hyperelliptic component of $\cP(\mu^\fR)$, then $\cP(\mu^\fR)$ has a unique zero or two zeroes of the same order. Moreover, the involution $\Sigma$ induced by the hyperelliptic involution is a ramification profile of $\cP(\mu^\fR)$.
\end{proposition}

\begin{proof}
Let $\Sigma$ be the involution on $\boldsymbol{p}$ induced by the hyperelliptic involution $\sigma$ of flat surfaces $(X,\omega)$ in $\cC$. By reordering the poles if necessary, we may assume that $\Sigma$ fixes $p_1,\dots,p_r$ for some and $\Sigma(p_i)=p_{i+1}$ for each $i=r+1,r+3,\dots,n-1$. 

Since $\sigma^{\ast}\omega=-\omega$, $\sigma$ preserves the order of zeroes and poles. That is, $b_i=b_{i+1}$ for each $i=r+1,r+3,\dots,n-1$. Moreover, $\res_{p_i}\omega = \res_{\sigma(p_i)} \sigma^{\ast}\omega = -\res_{\sigma(p_i)}\omega$. So $\{p_i,\sigma(p_i)\}$ or both $\{p_i\}$ and $\{\sigma(p_i)\}$ are parts of $\fR$ for each $i$. Therefore each part of $\fR$ is either a pair or a singleton. If $\sigma$ fixes a pole $p_i$, then $p_i$ is residueless. Moreover, if $z$ is a local coordinate at $p_i$ such that $\omega=\frac{dz}{z^{b_i}}$, then $\sigma(z)=-z$ and $-\omega =\sigma^{\ast}\omega=(-1)^{b_i-1}\frac{dz}{z^{b_i}}=(-1)^{b_i-1}\omega$. Thus $p_i$ is an even pole. Therefore, $\Sigma$ satisfies all conditions to be a ramification profile. 

It remains to show that $\cP(\mu^\fR)$ has a unique zero or two zeroes of the same order. Suppose that $\sigma$ fixes $m_1$ zeroes and interchanges $m_2$ pairs of zeroes. It is sufficient to show that $m_1+m_2=1$. By taking quotient by $\sigma$, $\pi:X\to \mathbb{P}^1$ is a double cover ramified at the fixed points of $\sigma$. We also obtain a half-translation structure on $\mathbb{P}^1$ defined by some quadratic differential $q$ such that $\pi^{\ast}q=\omega^2$. The differential $q$ has $m_1+m_2$ zeroes. It has a pole of order $b_i-1$ at the image of a fixed pole $p_i$, and a pole of order $2b_i$ at the image of interchanged pair $p_i,p_{i+1}$. So this surface $(\PP^1,q)$ is contained in a stratum of quadratic differentials $\cQ(\nu)$, where $\nu=(a'_1,\dots,a'_{m_1+m_2},-1^{2g+2-m_1-r},-b_1-1,\dots,-b_r-1,-2b_{r+1},\dots,-2b_{n-1})$ for some $a'_1,\dots,a'_{n_1+n_2}$. Note that the simple poles are the images of regular points fixed by $\sigma$. For convenience, we assume that those simple poles are not labeled. The 2-residue of $q$ at the pole of order $2b_i$ is equal to $(\operatorname{res}_{p_i}\omega)^2.$. So it is equal to zero if and only if $p_i$ is a residueless pole. Also, any odd order poles has automatically zero 2-residue. Let $\cQ_0(\nu)\subset \cQ(\nu)$ the subset of quadratic differentials satisfying the above 2-residue conditions. Conversely, for any quadratic differential $q\in \cQ_0(\nu)$, we can consider a ramified double cover $\pi:X\to \mathbb{P}^1$ such that $\pi^{\ast}q=\omega^2$ for some $(X,\omega)\in \cP(\mu^\fR)$. The differential $\omega$ is unique up to multiplication by $-1$. Thus we have an injection $\phi: \PP\cC \to \PP\cQ_0(\nu)$, which is a locally isomorphism. In particular, we have $\dim \PP\cC=\dim \cQ_0(\nu)$. 

Let $t_1, t_2$ be the number of singletons and pairs of $\fR$. Then $n=t_1+2t_2$ and $|\fR|=t_1+t_2$. Recall that $r$ of those singletons are corresponding to the fixed poles. The subvariety $\cQ_0(\nu)$ is obtained by imposing $\frac{t_1}{2}-r$ 2-residue conditions. So $$\dim \cQ_0(\nu) +\frac{t_1}{2}-r=\dim \cQ(\nu)=(m_1+m_2)+(2g+2-m_1-r)+\frac{t_1}{2}+t_2-2$$ and thus $\dim \cQ_0(\nu) =2g+m_2+t_2$. However, we also have $\dim \cQ_0(\nu)=\dim \cC = 2g+m_1+2m_2+n-1-|\fR|=2g+m_1+2m_2+t_2-1$. So we have $m_1+m_2=1$, as desired. 
\end{proof}

Note that from the proof above, we can also see that for any ramification profile $\Sigma$ of $\mu^{\fR}$, we can associate a corresponding stratum $\cQ(\nu)$ of quadratic differentials and construct at least one hyperelliptic component of $\cP(\mu^{\fR})$ that induces $\Sigma$. As a consequence, we have the following property of hyperelliptic flat surfaces. 

\begin{corollary} \label{Cor:hyperextension}
    Suppose that $(X,\omega)\in \cP(\mu^{\fR})$ is a hyperelliptic flat surface whose hyperelliptic involution induces a ramification profile $\Sigma$. That is, $(X,\omega)$ has a unique zero or two zeroes interchanged by the hyperelliptic involution. Then $(X,\omega)$ is contained in a hyperelliptic component with ramification profile $\Sigma$. 
\end{corollary}

\begin{remark}
    For single-zero strata, the above corollary states that a hyperelliptic flat surfaces are only contained in the hyperelliptic components. However, for double-zero strata, it does not prevent to have a hyperelliptic component whose hyperelliptic involution fixes two distinct zeroes. 
\end{remark}

In order to prove \Cref{Thm:mainhyper}, it remains to prove the following uniqueness statement. 

\begin{proposition} \label{Prop:hyperunique}
    For each ramification profile $\Sigma$ of $\mu^\fR$, there exists a unique hyperelliptic connected component $\cC_{\Sigma}$ of $\cP(\mu^\fR)$. 
\end{proposition}

In order to prove uniqueness, we need the following proposition regarding the existence of multiplicity one saddle connections. 

\begin{proposition} \label{Prop:ssc_he}
Assume that $\cC$ is a hyperelliptic component with ramification profile $\Sigma$. Then $\cC$ contains a flat surface $(X,\omega)$ with a multiplicity one saddle connection if and only if $\Sigma$ fixes less than $2g+2$ marked points. 

If $\Sigma$ fixes $2g+2$ marked points and $p$ is a fixed pole, then $\cC$ contains a flat surface $(X,\omega)$ with a pair of multiplicity two saddle connections bounding the polar domain of $p$.
\end{proposition} 

We will use induction on $\dim\cP(\mu^{\fR})>0$ to prove \Cref{Prop:ssc_he} and \Cref{Prop:hyperunique}. For the base cases of $\dim\cP(\mu^{\fR})=1$, above two propositions will be proved in Sections~\ref{Sec:Bsignature}-\ref{Sec:Dsignature}. For now, we will assume these base cases. 

\begin{proof}
    We use induction on $\dim \cP(\mu^{\fR})>0$. Suppose that the proposition holds for the one-dimensional base cases and assume that $\dim \cP(\mu^{\fR})>1$. By \Cref{breakintotwo}, we can degenerate a flat surface in $\cC$ into a boundary $\overline{X}$ where each irreducible component is hyperelliptic. By applying \Cref{breakintotwo} repeatedly, we may assume that $X_{-1}$ is contained in a zero-dimensional stratum. In particular, $X_{-1}$ is a rational curve. Note that each irreducible component has a hyperelliptic involution and the node is fixed by both involutions. By \Cref{Cor:hyperextension}, each irreducible component is contained in a hyperelliptic component of a smaller dimension. 

    Suppose that $\Sigma$ fixes less than $2g+2$ marked points. If the top level component $X_0$ has less than $2g+2$ fixed marked points, by induction hypothesis, $X_0$ can be deformed to contain a multiplicity one saddle connection. So assume that $X_0$ has $2g+2$ fixed marked points. That is, $X_{-1}$ contains two zeroes and no fixed marked poles. Again by induction hypothesis, $X_0$ can be deformed to contain a pair of parallel saddle connections bounding the polar domain of $p$. By shrinking this pair of saddle connections, we obtain a three-level multi-scale differential. By merging two levels other than the top level component, we obtain a rational component $Y_{-1}$ containing $p$. Also $Y_{-1}$ intersect the top component $Y_0$ at two nodes interchanged by the hyperelliptic involution. In particular, $Y_{-1}$ only contains one fixed marked point. So $Y_{-1}$ can be deformed so that it contains a multiplicity one saddle connection. By plumbing, this saddle connection becomes a multiplicity one saddle connection of a flat surface in $\cC$.

    Now suppose that $\Sigma$ fixes $2g+2$ marked points. Then $X_0$ must have $2g+2$ fixed marked points. If $X_0$ contains $p$, then by induction hypothesis, $X_0$ can be deformed to contain a pair of saddle connection bounding the polar domain of $p$. So we assume that $p$ is contained in $X_{-1}$. Since $X_{-1}$ has two fixed poles($p$ and the node), any other fixed pole $p'$ is contained in $X_0$. By induction hypothesis, $X_0$ can be deformed to contain a pair of parallel saddle connections bounding the polar domain of $p'$. By shrinking this pair of saddle connections, we obtain a three-level multi-scale differential. By merging two levels other than the top level component, we obtain a rational component $Y_{-1}$ containing two fixed poles $p$ and $p'$. By induction hypothesis, $Y_{-1}$ can be deformed so that it contains a pair of parallel saddle connections bounding the polar domain of $p$. By plumbing, this pair of saddle connections becomes saddle connections of a flat surface in $\cC$, still bounding the polar domain of $p$. 
\end{proof}

Now we can prove \Cref{Prop:hyperunique}. 

\begin{proof}[Proof of \Cref{Prop:hyperunique}]
    We use induction on $\dim \cP(\mu^\fR)>0$. Suppose that the proposition holds for the base cases and assume that $\dim \cP(\mu^{\fR})>1$. The induction step can be proven identically to the proof of \cite[Prop. 6.18]{lee2023connected}, together with \Cref{Prop:ssc_he}. 
\end{proof}

%% file: 05SimpleSaddleConnection.tex
\section{Existence of multiplicity one saddle connections} \label{Sec:ssc}

In this section, we prove the existence of a flat surface with a multiplicity one saddle connection in every non-hyperelliptic component of $\cP(\mu^\fR)$ of dimension higher than one, except for some special cases. For hyperelliptic components, a similar statement was already proven in \Cref{Prop:ssc_he}. By shrinking the multiplicity one saddle connection, we prove that any non-hyperelliptic component of higher dimension can be obtained by breaking up a zero and bubbling a handle from a non-hyperelliptic component of a one-dimensional stratum (again with a few exceptions).

For single-zero strata with $g=0$, note that $H_1(\mathbb{P}^1;\mathbb{Z})=0$ and each saddle connection of a flat surface is a separating closed curve. In this case, we have the following:

\begin{proposition}\label{Prop:ssc0}
    Assume that $\cC$ is a non-hyperelliptic component of a single-zero stratum $\cP(\mu^\fR)$ with $g=0$ of nonzero dimension. Let $p$ be a non-residueless pole with the smallest order. By relabeling the poles, we may assume that $p=q_1$.
    
    \begin{itemize}
        \item If $\{q_1,q_2\}\in \fR$ for some simple pole $q_2$, then $\cC$ contains a flat surface with a pair of multiplicity two saddle connections $\gamma_1$ and $\gamma_2$, each bounding the polar domain of $q_1$ and $q_2$, respectively.
        \item Otherwise, $\cC$ contains a flat surface with a multiplicity one saddle connection $\gamma$ bounding the polar domain of $p$.
    \end{itemize}
    
In each case, by shrinking the collection of saddle connections given above, we obtain a boundary with a non-hyperelliptic top-level component except in the following cases:
\begin{itemize}
    \item $\cP(\mu^\fR)=\cP(a\mid -2\mid\dots\mid-2\mid-e^2\mid-1,-c)$ of dimension one,
    \item $\cP(\mu^\fR)=\cP(3+c\mid -1^2\mid-1^2\mid-1,-c)$ of dimension two,
\end{itemize}
for $c=1,2$, where $p$ is the simple pole paired with the pole of order $c$.
\end{proposition}

Note that the exceptional cases can be predicted because the top-level components are contained in hyperelliptic strata $\cP(a\mid-2\mid\dots\mid-2\mid-e^2)$, $\cP(2\mid-1^2\mid-1^2)$, or $\cP(4\mid-2\mid-1^2\mid-1^2)$. See \Cref{Prop:Dspin}.

\begin{proof}
    We use induction on $\dim \cP(\mu^{\fR})\geq 1$. The one-dimensional cases are stated in \Cref{Sec:Base}: \Cref{Cor:Bssc} for B-signature, \Cref{Cor:Cssc} for C-signature, and \Cref{Cor:Dssc} for D-signature. The first exceptional case comes from \Cref{Cor:Dssc}.
    
    Now suppose $\dim \cP(\mu^\fR)>1$. Let $\boldsymbol{r}\in\fR$ be the part containing $p$. By \Cref{breakintotwo}, we obtain a two-level boundary $\overline{X}$ with two irreducible components intersecting at one node $s$. By applying \Cref{breakintotwo} repeatedly to the bottom-level component $X_{-1}$, we may assume that $X_{-1}$ is contained in a zero-dimensional stratum. That is, $X_{-1}$ contains exactly two non-residueless poles, say $q$ and $q'$. If both are marked poles, then $\{q,q'\}\in \fR$. Otherwise, one of them, say $q'$, is at the node $s$. On the other hand, the top-level component $X_0$ is contained in a single-zero stratum of dimension $\dim \cP(\mu^{\fR})-1>0$. We apply the induction hypothesis on $X_0$.
    
    First, assume that $p\in X_0$. If $X_0$ contains another pole in $\boldsymbol{r}$, then $p$ is a non-residueless pole in $X_0$. By the induction hypothesis, $X_0$ can be deformed to have a saddle connection bounding the polar domain of $p$. By plumbing the level transition and shrinking this saddle connection, we can reduce to the case $p\in X_{-1}$. If $X_0$ does not contain any other pole in $\boldsymbol{r}$, then $p$ is residueless in $X_0$. Since $X_0$ is contained in a single-zero stratum of dimension $\dim \cP(\mu^{\fR})-1>0$, we can choose a non-residueless pole $x\in X_0$ of the smallest order in another part $\boldsymbol{r'}\in \fR$. By the induction hypothesis, $X_0$ can be deformed to have a multiplicity one saddle connection $\gamma$ bounding $x$, unless $\boldsymbol{r'}=\{x,y\}$ for some $y$ and $x,y$ are both simple, or $X_0$ is hyperelliptic. In these cases, by \Cref{Prop:ssc_he} and the induction hypothesis, $X_0$ can still be deformed to have a pair $\gamma_1$ and $\gamma_2$ of multiplicity two saddle connections, each bounding $x$ and $y$. By plumbing the level transition and shrinking $\gamma$ (or the pair $\gamma_1, \gamma_2$), we obtain a two-level boundary $\overline{Y}$ whose bottom-level component $Y_{-1}$ contains no poles in $\boldsymbol{r}$. By the argument above, we can again reduce to the case $p\in X_{-1}$.
    
    Now assume that $p=q\in X_{-1}$. Let $x$ be a non-residueless pole in $X_0$ with the smallest order. By the induction hypothesis, $X_0$ can be deformed to have a multiplicity one saddle connection $\gamma$ bounding $x$, unless $X_0$ is hyperelliptic, or $\{x,y\}\in \fR$ or $\{x,y,q\}\in \fR$ where both $x,y$ are simple. In these exceptional cases, $X_0$ can still be deformed to have a pair of multiplicity two saddle connections bounding $x$ and $y$. In any case, by shrinking $\gamma$ (or the pair $\gamma_1, \gamma_2$) and plumbing the level transition between level $-1$ and $-2$, the new bottom-level component containing $p$ is contained in a non-hyperelliptic component of a one-dimensional stratum. 
    
    Finally, we have to deal with the exceptional cases in \Cref{Cor:Dssc}. That is, the new bottom-level component is contained in a stratum of D-signature such that the orders of $x,y$ are equal, say $b$, and all $b_i=2$. In this case, $\{p,q'\}\in \fR$ and the new top-level component has a unique zero of order $0$, so it is indeed in $\cP(0,-1^2)$. That is, we can reduce to the case when $X_0$ is contained in a stratum of D-signature $\cP(a\mid -2\mid\dots\mid-2\mid -b^2\mid -1^2)$ and $X_{-1}$ contains no marked residueless poles. Since $x$ is the smallest residueless pole in $X_0$, we must have $b=1$. 
    
    If we have $k$ residueless double poles in $X_0$, we can deform $X_0$ to have a pair of multiplicity two saddle connections bounding the polar domains of $x,y$. By shrinking them and plumbing the level transition between levels $-1$ and $-2$, the new bottom-level component is contained in $\cP(a\mid -2k-2\mid-1^2\mid-1,-c)$ where $c=1,2$ is the order of $q'$. If $-2k-2\neq -2$, then we are done by \Cref{Cor:Dssc}. So $k=0$ and we have $\cP(\mu^\fR)=\cP(3+c\mid -1^2\mid-1^2\mid-1,-c)$ for $c=1,2$, which is excluded.
\end{proof}

For non-hyperelliptic components of strata with $m>1$ zeroes or genus $g>0$, we have:

\begin{theorem} \label{Thm:ssc}
Let $\cP(\mu^\fR)$ be a stratum with $g>0$, or a non-residueless stratum with $g=0$ and $m>1$ zeroes. Suppose that $\dim \cP(\mu^\fR)>0$ and $\mu^\fR\neq (a_1,2(n-2)-a_1\mid-2\mid\dots\mid-2\mid-1^2)$ for any even $a_1$. For any non-hyperelliptic component $\cC$ of $\cP(\mu^\fR)$, the following holds:
\begin{itemize}
    \item If $\cP(\mu^\fR)$ has $m>1$ zeroes, then for any pair $z_i,z_j$ of distinct zeroes, $\cC$ contains a flat surface $(X,\omega)$ with a multiplicity one saddle connection $\gamma$ joining $z_i$ and $z_j$. 
    \item If $g>0$ and $\mu$ has a unique zero, then $\cC$ contains a flat surface $(X,\omega)$ with a multiplicity one saddle connection $\gamma$ such that $[\gamma]\in H_1(X,\mathbb{Z})$ is nontrivial.
\end{itemize}
\end{theorem}

Recall that in the exceptional case of a stratum $\cP(a_1,2(n-2)-a_1\mid-2\mid\dots\mid-2\mid-1^2)$, a non-hyperelliptic component $\cC$ with index $n+1\pmod 2$ contains a flat surface with a pair of multiplicity two saddle connections joining two distinct zeroes by \Cref{Cor:Bssc}.

\begin{proof}
We use induction on $\dim \cP(\mu^{\fR})>0$. The induction argument is identical to the proof of the residueless version \cite[Thm~6.1]{lee2023connected}, except that we must take care of strata of the form $\cP(a_1,2(n-2)-a_1\mid-2\mid\dots\mid-2\mid-1^2)$ with even $a_1$. By \Cref{Cor:Bssc}, a non-hyperelliptic component of this stratum with index $n+1\pmod 2$ does not satisfy the theorem. We call such a component {\em special} in this proof.

First, assume that $g>0$ and $m=1$. By \Cref{breakintotwo}, there exists a two-level boundary $\overline{X}\in \partial\overline{\cC}$ with two irreducible components intersecting at one node $s$. If both components are hyperelliptic, then $\cC$ is hyperelliptic, which is excluded. Thus, one of the irreducible components must be non-hyperelliptic. If this non-hyperelliptic component is contained in a stratum of positive genus, then the induction hypothesis applies, and it can be deformed to have a multiplicity one saddle connection with nontrivial homology class. Therefore, we may assume that the non-hyperelliptic component is contained in a zero-dimensional stratum, and the other component of genus $g$ is hyperelliptic.

Suppose that the top level component $X_0$ is hyperelliptic. By \Cref{Prop:ssc_he}, $X_0$ can be deformed to have either a multiplicity one saddle connection with nontrivial homology class or a pair of multiplicity two saddle connections bounding a fixed marked pole. By shrinking these saddle connections, $X_0$ degenerates into a two-level multi-scale differential $\overline{Y}$ with two irreducible components intersecting at two nodes $s_1$ and $s_2$. By plumbing the level transition between level $-1$ and $-2$, we obtain a two-level boundary with a new non-hyperelliptic bottom level component of genus zero. This component still has two non-simple poles at the nodes $s_1$ and $s_2$ with opposite residues. By \Cref{Prop:ssc0}, this bottom level component can be deformed to have a multiplicity one saddle connection $\gamma$ bounding the polar domain of either $s_1$ or $s_2$, whichever has smaller order. After plumbing the level transition, $\gamma$ becomes a multiplicity one saddle connection with nontrivial $[\gamma]\in H_1(X,\mathbb{Z})$.

Suppose now that the bottom level component $X_{-1}$ is hyperelliptic. As before, by \Cref{Prop:ssc_he}, $X_{-1}$ degenerates into a two-level multi-scale differential $\overline{Y}$ with two irreducible components intersecting at two nodes $s_1$ and $s_2$. By plumbing the level transition between level $0$ and $-1$, we obtain a two-level boundary with a new non-hyperelliptic top level component of genus $g-1$. If this top level component were contained in a special component, then $X_0$ would be contained in a stratum of the form $\cP(2k\mid-2\mid\dots\mid-2\mid-1^2)$ of genus zero, which is a contradiction since $X_0$ is in a non-hyperelliptic component. Thus, by the induction hypothesis, this top level component can be deformed to have a multiplicity one saddle connection $\gamma$ joining $s_1$ and $s_2$. By plumbing the level transition, $\gamma$ becomes a multiplicity one saddle connection with nontrivial $[\gamma]\in H_1(X,\mathbb{Z})$.

Now assume that $m\geq 2$. By \Cref{breakintotwo}, we obtain a two-level boundary $\overline{X}$ with two irreducible components intersecting at one node $s$, and the bottom level component $X_{-1}$ contains $z_1$ and $z_2$. If $X_{-1}$ is contained in a stratum of positive dimension, we can repeatedly apply \Cref{breakintotwo} to the bottom level component. Thus, we may assume that $X_{-1}$ is contained in a zero-dimensional stratum; that is, $X_{-1}$ contains exactly two zeroes $z_1$ and $z_2$, and only residueless poles. From this point onward, we can follow the same argument as in the proof of \cite[Thm~6.1]{lee2023connected}, together with \Cref{Prop:ssc_he}.
\end{proof}

\subsection{Breaking up a zero and merging zeroes}

Suppose that there exists a multiplicity one saddle connection on a flat surface $(X,\omega)$ joining two distinct zeroes $z_1,z_2$. Then we can merge $z_1$ and $z_2$ by shrinking this saddle connection. This procedure is the inverse of the surgery called {\em breaking up a zero} first introduced in \cite{kontsevich2003connected}. In \cite{chengendron2022towards}, breaking up a zero is interpreted as a smoothing of a certain multi-scale differential. By repeatedly merging pairs of zeroes, we can eventually merge all zeroes and obtain a single-zero flat surface. For a given $\mu$, let $a=a_1+\dots+a_m$ and let $\mu'$ be the corresponding single-zero signature with a unique zero of order $a$. We say that a component $\cC$ of $\cP(\mu^\fR)$ is {\em adjacent} to a component $\cD$ of $\cP(\mu'^{\fR})$ if $\cC$ can be obtained by breaking up the zero from $\cD$. 

\begin{lemma} \label{lm:mergezeroes}
    Let $\cP(\mu^\fR)$ be any stratum with $g>0$, or a non-residueless stratum with $g=0$. Every non-hyperelliptic component of $\cP(\mu^\fR)$ is adjacent to some non-hyperelliptic component of the corresponding single-zero stratum $\cP(\mu'^{\fR})$, except in the following cases:
    \begin{itemize}
        \item $(a\mid-2\mid\dots\mid-2\mid-b^2)$ with genus zero;
        \item $(2\mid-1^2\mid-1^2)$ with genus zero;
        \item $(4\mid-2\mid-1^2\mid-1^2)$ of even spin with genus zero;
        \item $(2n\mid-2\mid\dots\mid-2)$ with genus one.
    \end{itemize}
\end{lemma}

Remark that by \Cref{zerodim2}, \Cref{Prop:Dspin}, and \cite[Thm~1.7]{lee2023connected}, for the exceptional cases, $\cP(\mu'^{\fR})$ has {\em no} non-hyperelliptic components. 

\begin{proof}
    For residueless strata with $g>1$, this lemma is proven in \cite[Prop~9.1]{lee2023connected}. So assume that $\cP(\mu^\fR)$ is a non-residueless stratum. Let $\cC$ be a non-hyperelliptic component of $\cP(\mu^\fR)$. By \Cref{Thm:ssc}, $\cC$ contains a flat surface with a multiplicity one saddle connection joining two distinct zeroes $z_1$ and $z_2$. Thus, $\cC$ is adjacent to some connected component $\cD$ of a stratum with $m-1$ zeroes. In order to repeat merging zeroes, we need to prove that $\cD$ can be chosen to be non-hyperelliptic. Since hyperelliptic components have at most two zeroes, this is obvious if $m>3$. Assume the contrary, that $\cD$ is always hyperelliptic and $m\leq 3$.
    
    First, suppose $m=3$ and that $\cD$ is a double-zero hyperelliptic component. In particular, two zeroes have the same order. If the first zero was obtained by merging $z_1$ and $z_2$ of $\cC$, then we have $a_1+a_2=a_3$. By \Cref{Thm:ssc}, $\cC$ contains a flat surface with a multiplicity one saddle connection joining $z_2$ and $z_3$. Thus, we can merge $z_2$ and $z_3$ instead and obtain a flat surface with two zeroes of orders $a_1$ and $a_2+a_3$. Since they have distinct orders, this flat surface is not hyperelliptic. That is, $\cD$ can be chosen to be non-hyperelliptic.
    
    Now suppose $m=2$ and that $\cD$ is a single-zero hyperelliptic component. If $a_1=a_2$, then $\cC$ is also a hyperelliptic component, a contradiction. So we have $a_1 < a_2$. If $\dim \cP(\mu^\fR)=1$, then it is a one-dimensional stratum of B-signature. These cases are proven in \Cref{Cor:Bssc}, except when $\mu'=(a\mid-2\mid\dots\mid-2\mid-b^2)$. If $\dim \cP(\mu^\fR)>1$, then $\dim \cP(\mu'^\fR)>0$, and by \Cref{breakintotwo}, there exists a two-level multi-scale differential $\overline{X}\in \partial\overline{\cD}$ with two irreducible components intersecting at one node. By breaking up the zero of the bottom level component $X_{-1}$, we can obtain an element $\overline{Y'}$ of $\partial \overline{\cC}$. Now the bottom level component $Y'_{-1}$ is contained in a non-hyperelliptic component with positive dimension. If the induction hypothesis applies to $Y'_{-1}$, then we can merge these zeroes to obtain a non-hyperelliptic single-zero flat surface, completing the proof. The induction hypothesis does not apply only if the bottom level component is contained in $\cP(a_1,a_2\mid-2\mid\dots\mid-2\mid-b^2)$ with genus zero. In particular, $\mu^{\fR}$ has a pair of poles of order $b$ and the ramification profile of $\cD$ interchanges these poles. By \Cref{Thm:mainhyper}, there exists a two-level multi-scale differential $\overline{W}\in \partial\overline{\cD}$ with the bottom level component $W_{-1}\in \cP(a_1+a_2\mid-(a_1+a_2-2b+2)\mid-b^2)$. Thus, we reduce to the stratum $\cP(a_1,a_2\mid-(a_1+a_2-2b+2)\mid-b^2)$ of B-signature. This stratum satisfies the lemma unless $a_1+a_2-2b+2=2$. That is, $W_0$ has a zero of order zero at the node and therefore $W_0\in \cP(0,-1,-1)$. Thus, $\cD$ is the hyperelliptic component of $\cP(a_1+a_2\mid-1^2\mid-b^2)$. Again, by \Cref{Thm:mainhyper}, there exists $\overline{W'}\in \partial\overline{\cD}$ with the bottom level component $W'_{-1}\in \cP(a_1+a_2\mid-2b\mid-1^2)$. This stratum satisfies the lemma unless $b=1$ and $a_1=a_2=1$. However, $\cP(1,1\mid-1^2\mid-1^2)$ is a hyperelliptic stratum and thus $\cC$ is hyperelliptic, a contradiction.
\end{proof}

In other words, the breaking up the zero map $$
B:\{\text{non-hyperelliptic components of }\cP(\mu')\} \to \{\text{non-hyperelliptic components of }\cP(\mu)\}
$$ is surjective for any stratum with $g>0$, or a non-residueless stratum with $g=0$ except some special cases. 

\subsection{Bubbling and unbubbling a handle}

Bubbling a handle, the second surgery introduced in \cite{kontsevich2003connected}, is also interpreted as a smoothing of a certain multi-scale differential in \cite{chengendron2022towards}. 

Let $\cP(\mu^{\fR})$ be a single-zero stratum with $\dim\cP(\mu^{\fR})>1$. Consider $\mu_0=(a-2, -b_1,\dots,-b_n)$. Then $\cP(\mu_0^{\fR})$ is a stratum with genus $g-1$ and $\dim\cP(\mu_0^{\fR})=\dim\cP(\mu^{\fR})-2$. Let $\cD$ be a connected component of $\cP(\mu_0^{\fR})$ and $X_0\in \cD$. Consider another stratum $\cP(a\mid-a\mid-1^2)$ with genus one. An element of this stratum is determined by the angle $2\pi s$, $1\leq s\leq a-1$, between two half-infinite cylinders corresponding to two simple poles. Let $X_{-1}\in \cP(a\mid-a\mid-1,-1)$ be the flat surface with angle $2\pi s$. By identifying the unique zero of $X_0$ and the residueless pole of $X_{-1}$, we obtain a two-level multi-scale differential $\overline{X}$. By plumbing the level transition and the pair of simple poles, we obtain a flat surface in $\cP(\mu^{\fR})$. The connected component $\cC$ containing this flat surface is denoted by $\cD \oplus s$. We also say that $\cC=\cD \oplus s$ is obtained by {\em bubbling a handle} from $\cD$. 

The inverse of this surgery, unbubbling a handle, can also be realized as shrinking two multiplicity one saddle connections simultaneously. 

\begin{lemma}\label{lm:unbubble}
    Let $\cP(\mu^{\fR})$ be a single-zero stratum of genus $g>0$ with $\dim \cP(\mu^{\fR})>1$. Then every non-hyperelliptic component $\cC$ of $\cP(\mu^\fR)$ can be obtained by bubbling a handle from some connected component $\cD$ of a stratum with genus $g-1$. That is, $\cC=\cD\oplus s$ for some $1\leq s\leq a-1$. Furthermore, $\cD$ can be chosen to be also non-hyperelliptic except for the following cases:
    \begin{itemize}
        \item $(2n+2\mid-2\mid\dots\mid-2)$ with genus two;
        \item $(4\mid-1^2\mid-1^2)$ with genus one;
        \item $(a\mid-2\mid\dots\mid-2\mid-b^2)$ with genus one.
    \end{itemize}
\end{lemma}

Remark that for the exceptional cases, $\cP(\mu_0^{\fR})$ has {\em no} non-hyperelliptic component.

\begin{proof}
    If $\cP(\mu^\fR)$ is a residueless stratum, then this lemma is proven in \cite[Lemma~10.1]{lee2023connected}. So assume that $\cP(\mu^\fR)$ is non-residueless. By \Cref{Thm:ssc}, there exists a flat surface in $\cC$ with a multiplicity one saddle connection $\gamma$ with nontrivial homology class $[\gamma]$. By shrinking $\gamma$, we degenerate to a two-level multi-scale differential $\overline{X}$ in the principal boundary of Type II with two irreducible components intersecting at two nodes $s_1$ and $s_2$. 
    
    If the top level component $X_0$ is non-hyperelliptic, then by \Cref{lm:mergezeroes}, we can merge two zeroes at $s_1$ and $s_2$ to obtain a flat surface contained in some connected component $\cD$ of a genus $g-1$ stratum. Thus we can deduce that $\cC=\cD\oplus s$ for some $1\leq s\leq a-1$. Moreover, if $X_0$ is not contained in a stratum $\cP(a_1,a_2\mid-2\mid\dots\mid-2\mid-b^2)$ of D-signature, $\cD$ can be chosen to be also non-hyperelliptic. If $X_0$ is in this stratum, then $\cC$ is a component of $\cP(a\mid-2\mid\dots\mid-2\mid-b^2)$ with genus one, which is an exception.

    Now assume that $X_0$ is contained in a hyperelliptic component. In particular, $s_1$ and $s_2$ are zeroes of the same order $\frac{a-2}{2}$ in $X_0$. Suppose first that $X_0$ has fewer than $2(g-1)+2$ fixed marked points. By \Cref{Prop:ssc_he}, $X_0$ can be deformed to have a multiplicity one saddle connection joining $s_1$ and $s_2$. By shrinking this saddle connection and plumbing the level transition between levels $-1$ and $0$, we obtain a two-level multi-scale differential $\overline{X}$ with the non-hyperelliptic bottom level component $X_{-1}\in\cP(a,-a)$. By \cite[Prop~7.15]{lee2023connected}, $X_{-1}$ can degenerate to a two-level multi-scale differential with two irreducible components intersecting at two nodes, such that the number of prongs at the two nodes are distinct. By plumbing the level transition between level $0$ and $-1$, we now have another two-level multi-scale differential $\overline{X'}$ so that $X'_0$ has two zeroes of distinct order at the nodes. In other words, we may assume that $X_0$ is not hyperelliptic, as in the previous paragraph. 

    Finally, suppose that $X_0$ has $2(g-1)+2$ fixed marked points. Then by \Cref{Prop:ssc_he}, $X_0$ can be deformed to have a pair of multiplicity two saddle connections bounding the polar domain of a fixed marked pole, say $p$, of order $b$. By shrinking these saddle connections and plumbing the level transition between levels $-1$ and $0$, we obtain a two-level multi-scale differential $\overline{X}$ with the non-hyperelliptic bottom level component $X_{-1}\in \cP(a,-b,-(a-b))$. Again, since \cite[Prop~7.15]{lee2023connected} applies to $X_{-1}$, we may assume that $X_0$ is non-hyperelliptic.
\end{proof}

By applying the above lemma repeatedly, we deduce that any non-hyperelliptic component $\cC$ of a single-zero stratum $\cP(\mu^{\fR})$ with genus $g>0$ and $\dim\cP(\mu^{\fR})>2g-1$ can be written as
\[
\cC_0\oplus s_1\oplus \dots\oplus s_g,
\]
for some connected component $\cC_0$ of a stratum with genus zero and some integers $1\leq s_i\leq a-2(g-i)-1$. Moreover, $\cC_0$ can be chosen to be non-hyperelliptic except in the following cases:
\begin{itemize}
    \item $(2n+2g-4\mid-2\mid\dots\mid-2)$;
    \item $(2g+2\mid-1^2\mid-1^2)$;
    \item $(a\mid-2\mid\dots\mid-2\mid-b^2)$,
\end{itemize}
where we reach the exceptional cases in the above lemma.

%% file: 06Nonhyperelliptic.tex
\section{Non-hyperelliptic components} \label{Sec:Nonhyper}

In this section, we classify non-hyperelliptic components of any stratum $\cP(\mu^{\fR})$, proving \Cref{Thm:mainhigh}--\ref{Thm:main00} and \Cref{Thm:main01}.

\subsection{Strata of genus zero}

In order to prove \Cref{Thm:main00} in general, we need the following setting. Assume that $\cP(\mu^{\fR})$ is a single-zero stratum. Then, by relabeling the poles if necessary, there is a part $\{p_1,\dots,p_{\ell}\}$ of $\fR$ with cardinality $\ell>1$. Since $\fR$ does not have a pair of simple poles, we may also assume that $p_{\ell}$ is not a simple pole.

We can consider another stratum $\cP(\mu_{-p_1}^{\fR})$, where $\mu_{-p_1}$ is obtained by omitting $b_1$ and replacing $a_1$ by $a_1-b_1$ from $\mu$. By abuse of notation, we denote by $\fR$ the induced residue condition on $\mu_{-p_1}$. More precisely, it consists of $\{p_2,\dots,p_{\ell}\}$ and other parts of $\fR$. Since $p_{\ell}$ is not simple, $\cP(\mu_{-p_1}^{\fR})$ is a nonempty stratum of dimension $\dim \cP(\mu^{\fR})-1$.

\begin{proof}[Proof of \Cref{Thm:main00} for single-zero strata]
    We use induction on $\dim \cP(\mu^{\fR})$. If $\dim \cP(\mu^{\fR})=1$, then this is a base case of C- or D-signatures. By \Cref{Prop:Cconnected} and \Cref{Prop:Dnonhyper}, it has a unique non-hyperelliptic component.
    
    Assume that $\dim \cP(\mu^{\fR})>1$ and $\cP(\mu^{\fR})$ is \emph{not} residueless. Let $\cC$ be any non-hyperelliptic component. If $m>1$, by \Cref{lm:mergezeroes}, $\cC$ is adjacent to a non-hyperelliptic component $\cD$ of another stratum with $m-1$ zeroes. By the induction hypothesis, $\cD$ is unique, and therefore $\cC$ is also unique. If $m=1$, by \Cref{Prop:ssc0}, $\partial \overline{\cC}$ contains a two-level boundary $\overline{X}$ with a non-hyperelliptic top level component. By the induction hypothesis, the top level component is contained in a unique non-hyperelliptic component. Also, the bottom level component is contained in $\cP(a,-b_1,-a-2+b_1)$, which is connected. Therefore, $\cC$ is unique.
\end{proof}

\subsection{Strata without a simple pole}

First, we consider the strata \emph{without} any simple poles. We prove the following.

\begin{theorem}\label{allcontainresidueless}
    Suppose all $b_i>1$, $g>0$, and $\mu\neq (2n,-2^n)$ or $(n,n,-2^n)$. Then any non-hyperelliptic component of $\cP(\mu^{\fR})$ contains a non-hyperelliptic residueless flat surface.
\end{theorem}

If $\mu=(2n,-2^n)$ or $(n,n,-2^n)$, then the residueless strata $\cR(\mu)$ are hyperelliptic, so this theorem is obviously false. We have the following immediate

\begin{corollary}\label{reductiontoresidueless}
    Suppose all $b_i>1$, $g>0$, and $\mu\neq (2n,-2^n)$ or $(n,n,-2^n)$. The number of non-hyperelliptic components of a non-residueless stratum $\cP(\mu^{\fR})$ is equal to that of a residueless stratum $\cR(\mu)$.
\end{corollary}

\begin{proof}
    Let $N$ be the number of non-hyperelliptic components of $\cP(\mu^{\fR})$ and $N_0$ be that of $\cR(\mu)$. Suppose that $\mu^{\fR}\neq (2n,-2^n)$ or $(n,n,-2^n)$. By \Cref{allcontainresidueless}, we have $N\leq N_0$. By the main theorems of \cite{lee2023connected}, $N_0$ is determined by the number of possible topological invariants (spin parity if $g>1$, rotation number if $g=1$). On the other hand, those topological invariants extend to $\cP(\mu^{\fR})$ and give a lower bound for $N$. That is, we also have $N_0\leq N$, and thus $N=N_0$.
\end{proof}

The cases $\mu= (2n,-2^n)$ or $(n,n,-2^n)$, where \Cref{allcontainresidueless} does not apply, will be dealt with separately in the proof of \Cref{Thm:main10}. Now we prove \Cref{allcontainresidueless} by induction on $\dim \cP(\mu^{\fR})$.

\begin{proof}[Proof of \Cref{allcontainresidueless}]
    Assume that $\cC$ is a non-hyperelliptic component of a non-residueless stratum $\cP(\mu^{\fR})$. We use induction on $\dim \cP(\mu^{\fR})>1$ to find a non-hyperelliptic residueless flat surface in $\cC$.
    
    First, suppose that $m>1$. Since $g>0$, by \Cref{lm:mergezeroes}, $\cC$ can be obtained by breaking up a zero from a non-hyperelliptic component $\cC'$ of a single-zero stratum. This component $\cC'$ has a non-hyperelliptic residueless flat surface by the induction hypothesis. By breaking up the zero, we can obtain a non-hyperelliptic residueless flat surface in $\cC$.
    
    Now suppose that $m=1$. Assume that $g>1$. By \Cref{lm:unbubble}, $\cC$ can be obtained by bubbling a handle from a non-hyperelliptic component $\cC'$ of a stratum of genus $g-1>0$. Note that the exceptional case is a residueless stratum, which is excluded. By the induction hypothesis, there exists a non-hyperelliptic flat surface in $\cC'$. By bubbling a handle, we obtain a non-hyperelliptic residueless flat surface in $\cC$. 
    
    Now assume that $g=1$. By \Cref{lm:mergezeroes}, $\cC$ contains a flat surface with a multiplicity one saddle connection with nontrivial homology class. By shrinking this saddle connection, we obtain a multi-scale differential $\overline{X}$ in the principal boundary of Type II. The top level component $X_0$ is contained in a stratum of genus $g=0$ with two zeroes at the nodes. Let $\kappa_1\leq \kappa_2$ be the numbers of prongs at the nodes. Then $X_0$ has two zeroes of orders $\kappa_1-1$ and $\kappa_2-1$.
    
    If $X_0$ is hyperelliptic with ramification profile $\Sigma_0$, then $\cP(\mu)$ contains a flat surface $X'_0$ with ramification profile $\Sigma_0$ by \cite{lee2023connected}. By \Cref{Thm:mainhyper}, $X_0$ and $X'_0$ are contained in the same connected component. So we can deform $X_0$ continuously to $X'_0$. Since $\cC$ is not hyperelliptic, the prong-matching of $\overline{X}$ is not compatible with hyperelliptic involutions on each irreducible component. By plumbing the level transition, we obtain a non-hyperelliptic residueless flat surface in $\cC$.
    
    If $X_0$ is non-hyperelliptic, then let $\cP(\mu_0^{\fR})$ be the stratum containing $X_0$. Since all $b_i>1$, $X_0$ is contained in a unique non-hyperelliptic component $\cC_0$ by \Cref{Thm:main00}. By adding more residue conditions, we can choose some stratum of B-signature contained in $\cP(\mu_0^{\fR})$. Again by \Cref{Thm:main00}, it contains a unique non-hyperelliptic component $\cC'_0$, and this is contained in $\cC_0$. That is, $X_0$ can be deformed to a flat surface in $\cC'_0$.
    
    Now it is reduced to the case when $\cP(\mu_0^{\fR})$ is a stratum of B-signature. By \Cref{Cor:Bresidueless}, this stratum contains a non-hyperelliptic residueless flat surface if and only if $n\leq \kappa_1$. In particular, if $\kappa_2-\kappa_1 \leq b_j$ for some $j$, then $\kappa_1 \geq \frac{1}{2} (\sum_{i\neq j} b_i -2) > n-2$. 
    
    This completes the proof under this assumption. We find $\overline{X'}\in \partial\overline{\cC}$ satisfying this. Consider a two-level multi-scale differential $\overline{Y}$ as follows: the bottom level component $Y_{-1}$ is non-hyperelliptic and contained in $\cP(\kappa_1-1,\kappa_2-1, -b_1,-(a-b_1))$, and $Y_0$ is contained in $\cP(a-b_1-2\mid-b_2\mid\dots\mid-b_{n-2}\mid-e_1,-e_2)$. The two components $Y_0$ and $Y_{-1}$ intersect at the unique node $s$. This is contained in the boundary of the unique non-hyperelliptic component $\cC_0$, so $X_0$ can degenerate continuously to $\overline{Y}$. 
    
    By plumbing the level transition between level $-1$ and level $-2$, the new bottom level component is a genus one non-hyperelliptic residueless flat surface with two poles, $p_1$ and the node $s$. By \cite{lee2023connected}, this flat surface is contained in a unique non-hyperelliptic component of $\cR(a\mid-b_1\mid-(a-b_1))$ with rotation number $r\leq b_1$. Therefore, by \cite[Prop.~7.17]{lee2023connected}, it can be deformed to a two-level differential $\overline{Z}$, where the top level component $Z_0$ contains all marked poles and two zeroes of orders $a'_1,a'_2$ such that $a'_2-a'_1\leq r\leq b_1$. By plumbing the level transition between level $0$ and level $-1$, we obtain $\overline{X'}$ satisfying $\kappa'_2-\kappa'_1\leq b_1$.
\end{proof}

\subsection{Unbubbling a pair of simple poles}

In this subsection, we deal with single-zero strata containing $k$ pairs of simple poles in the residue condition $\fR$. Let $\mu=(a,-b_1,\dots,-b_{n-2k},-1^{2k})$ and let $\fR$ be a residue condition with $k$ pairs of simple poles as its parts. We will relate the number of connected components of $\cP(\mu^{\fR})$ with genus $g$ to that of another stratum $\cP(\widetilde{\mu}^{\widetilde{\fR}})$ with genus $g+k$, where $\mu=(\widetilde{\mu}\mid-1^2\mid\dots\mid-1^2)$ and $\widetilde{\fR}$ is the residue condition on $\widetilde{\mu}$ induced by $\fR$. 

\begin{theorem}
    Suppose that $g+k>1$. Then there exists a one-to-one correspondence between non-hyperelliptic components of a stratum $\cP(\mu^{\fR})$ with $g>0$ and those of the stratum $\cP(\widetilde{\mu}^{\widetilde{\fR}})$.
\end{theorem}

In order to prove the above main theorem of this subsection, we need to show that any non-hyperelliptic component of $\cP(\mu^{\fR})$ can be obtained by a surgery called {\em bubbling a pair of simple poles}, which we describe below.

Let $\mu=(\mu_0\mid-1^2)$ and let $\fR_0$ be the residue condition on $\mu_0$ induced by $\fR$. Then $\cP(\mu_0^{\fR_0})$ is a stratum with $k-1$ pairs of simple poles. Also, $\dim \cP(\mu_0^{\fR_0}) = \dim \cP(\mu^{\fR}) - 1$. Let $\cC_0$ be a connected component of $\cP(\mu_0^{\fR_0})$ and let $X_0\in \cC_0$. Consider another flat surface $X_{-1}\in \cP(a\mid-a\mid-1,-1)$. This flat surface is completely determined by the angle $2\pi t$, $1\leq t\leq a-1$, between two half-infinite cylinders corresponding to the two simple poles. By identifying the unique zero of $X_0$ and the residueless pole of $X_{-1}$, we obtain a two-level multi-scale differential $\overline{X}\in \overline{\cP}(\mu^{\fR})$. Let $\cC$ be the connected component of $\cP(\mu^{\fR})$ containing $\overline{X}$ in the boundary. In this case, we denote $\cC=\cC_0 \bar{\oplus} t$, and we say that $\cC$ is obtained by {\em bubbling a pair of simple poles} from $\cC_0$.

Let $\cC$ be a non-hyperelliptic component of a single-zero stratum $\cP(\mu^{\fR})$ of dimension higher than one. Suppose that $\fR$ has a pair of simple poles, say $\{p_{n-1},p_n\}$. By \Cref{Prop:ssc0}, $\cC$ contains a flat surface with a pair of multiplicity two saddle connections each bounding the polar domain of $p_{n-1}$ and $p_n$. By shrinking them, we obtain a two-level multi-scale differential $\overline{X}$. The top-level component $X_0$ is contained in $\cP(\mu_0^{\fR_0})$, and the bottom-level component is $X_{-1}\in \cP(a\mid-a\mid-1^2)$. So $\cC = \cC_0 \bar{\oplus} t$ for some component $\cC_0$ of $\cP(\mu_0^{\fR_0})$ and $1\leq t\leq a-1$. Moreover, the component $\cC_0$ can be chosen to be non-hyperelliptic except in the following cases:
\begin{itemize}
    \item $\mu_0^{\fR_0} = (a-2\mid-2\mid\dots\mid-2\mid-b^2)$ with genus zero.
    \item $\mu_0^{\fR_0} = (2\mid-1^2\mid-1^2)$.
\end{itemize}

By applying this repeatedly with $k$ pairs of simple poles, we have the following:

\begin{lemma} \label{lm:unbubble_pair}
    Let $\cC$ be a non-hyperelliptic component of a single-zero stratum $\cP(\mu^{\fR})$ with $k$ pairs of simple poles. Suppose that $\dim \cP(\mu^{\fR})>k$. Then we can write
    \[
    \cC = \widetilde{\cC} \bar{\oplus} t_1 \bar{\oplus} \dots \bar{\oplus} t_k
    \]
    for some integers $1\leq t_i \leq a-2(k-i)-1$, where $\widetilde{\cC}$ is a connected component of a stratum $\cP(\widetilde{\mu}^{\widetilde{\fR}})$ with no pair of simple poles, and $\mu=(\widetilde{\mu}\mid-1^2\mid\dots\mid-1^2)$. Moreover, if $\cP(\widetilde{\mu}^{\widetilde{\fR}})$ is not one of the exceptional cases discussed above, then $\widetilde{\cC}$ can be chosen to be non-hyperelliptic.
\end{lemma}

Similarly to \cite[Prop.~6.2]{boissy2015connected} for bubbling handles, we have useful formulae for the connected components obtained by bubbling multiple pairs of simple poles.

\begin{lemma} \label{lm:unbubblespin}
    Let $\cC$ be a connected component of a stratum with a unique zero of order $a$. Suppose that $(t_1,t_2),(s_1,s_2)\neq \left(\frac{a+2}{2},\frac{a+4}{2}\right)$. Then we have
    \[
    \cC \bar{\oplus} t_1 \bar{\oplus} t_2 = \cC \bar{\oplus} s_1 \bar{\oplus} s_2
    \]
    if and only if $t_1+t_2\equiv s_1+s_2 \pmod{2}$.
\end{lemma}

\begin{proof}
    Consider the three-level multi-scale differential $\overline{X}$ in the boundary of $\cC\bar{\oplus}t_1\bar{\oplus}t_2$ that is used for the bubbling pairs construction. That is, $X_0\in \cC$, $X_{-1}\in \cP(a+2\mid-(a+2)\mid-1^2)$, and $X_{-2}\in \cP(a+4\mid-(a+4)\mid-1^2)$. By plumbing the level transition between level $-1$ and $-2$, we obtain a new bottom-level component $X'_{-1}$ in a stratum $\cP(a+4\mid-(a+2)\mid-1^2\mid-1^2)$ of D-signature. Note that $X'_{-1}$ is non-hyperelliptic since $(t_1,t_2)\neq \left(\frac{a+2}{2},\frac{a+4}{2}\right)$. Since $a+2>2$, the stratum $\cP(a+4\mid-(a+2)\mid-1^2\mid-1^2)$ has two non-hyperelliptic components distinguished by spin parity. The spin parity of $X'_{-1}$ is equal to $t_1+t_2\pmod{2}$. Therefore, $\cC\bar{\oplus}t_1\bar{\oplus}t_2 = \cC\bar{\oplus}s_1\bar{\oplus}s_2$ if and only if $t_1+t_2\equiv s_1+s_2\pmod{2}$.
\end{proof}

\subsection{General cases: single zero strata} \label{Subsec:singlezero}

In this subsection, we prove \Cref{Thm:main10}, \Cref{Thm:main01}, and \Cref{Thm:mainhigh} for any single-zero stratum $\cP(\mu^{\fR})$, using the results in previous subsections.

\begin{proof}[Proof of \Cref{Thm:main01} for single-zero strata]
    If $\dim \cP(\mu^{\fR})=1$, then this is a stratum of D-signature with one pair of simple poles. This case is proven in \Cref{Prop:Dindex}. Assume that $\dim \cP(\mu^{\fR})>1$ and let $\cC$ be any non-hyperelliptic component. By \Cref{lm:unbubble_pair}, $\cC=\cC_0\bar{\oplus}t$ for some non-hyperelliptic $\cC_0$ and $1\leq t\leq a-1$. Also, by \Cref{Thm:main00}, $\cC_0$ is the unique non-hyperelliptic component. Since the index $1\leq I\leq \delta$ of $\cC_0\bar{\oplus}t$ satisfies $t\equiv I\pmod{\delta}$, it suffices to prove that $\cC_0\bar{\oplus}t=\cC_0\bar{\oplus}I$.

    If $b_i=1$ for some $i$, then $\delta=1$. Also, $p_i$ is {\em not} paired with another simple pole by assumption. So by \Cref{Prop:ssc0}, there exists a flat surface in $\cC$ with a multiplicity one saddle connection bounding the polar domain of $p_i$. By shrinking this saddle connection, we obtain a two-level multi-scale differential $\overline{X}$ with the bottom level component $X_{-1}\in \cP(a-2,-1,-(a-1))$. Thus, $X_{-1}\bar{\oplus}t$ is contained in a stratum of D-signature with one pair of simple poles. Since $a-1\neq 1$ and $\gcd(1,a-1)=1$, by \Cref{Prop:Dindex}, this stratum is connected with a unique non-hyperelliptic component $\cD$, containing $X_{-1}\bar{\oplus}t$ for any $t$. Thus, $\cC$ is unique and $\cC=\cC_0\bar{\oplus}t$ for each $t$.

    So we assume that $b_i>1$ for each $i=1,\dots,n-2$. That is, $\cC_0$ has no simple poles. We claim that for any $i=1,\dots,n-2$, $\cC_0\bar{\oplus}t_1 = \cC_0\bar{\oplus}t_2$ if $t_1 - t_2 \equiv 0 \pmod{b_i}$. Suppose that $p_i$ is a non-residueless pole of $\cC_0$. By uniqueness of $\cC_0$, we can construct a two-level boundary $\overline{X}\in \partial\overline{\cC_0}$ with $X_{-1}\in \cP(a-2,-b_i,-(a-b_i))$. By the same argument as above, we have $\cC_0\bar{\oplus}t_1 = \cC_0\bar{\oplus}t_2$ if $t_1-t_2\equiv 0\pmod{\gcd(a,b_i)}$. Now suppose that $p_i$ is residueless. Then we can choose a non-residueless pole $p_j$. Again by uniqueness of $\cC_0$, we can construct a two-level boundary $\overline{X}\in \partial\overline{\cC_0}$ with $X_{-1}\in \cP(a-2,-b_i,-b_j,-(a-b_j-b_i))$. By a similar argument, we obtain $\cC_0\bar{\oplus}t_1 = \cC_0\bar{\oplus}t_2$ if $t_1-t_2\equiv 0\pmod{\gcd(a,b_i,b_j)}$.

    Consequently, if $t_1-t_2\equiv 0\pmod{\delta}$, then $t_1=t_2+\sum_{i=1}^{n-2}k_i b_i$ for some integers $k_i$. Thus, we can conclude that $\cC_0\bar{\oplus}t_1 = \cC_0\bar{\oplus}t_2$. Therefore, $\cC=\cC_0\bar{\oplus}I$ where $1\leq I\leq \delta$ is the index of $\cC$.
\end{proof}

\begin{proof}[Proof of \Cref{Thm:main10} for single-zero strata]
    If $b_i>1$ for all $i$ and $\mu\neq (2n,-2^n)$, then it follows immediately from \Cref{reductiontoresidueless}. So first we assume that $b_1=1$. That is, $p_1$ is a simple pole {\em not} paired with another simple pole. Then $\delta=1$ and the only possible rotation number is one. It suffices to prove that $\cP(\mu^{\fR})$ has a unique non-hyperelliptic component in this case. By \Cref{lm:unbubble}, $\cC=\cD\oplus s$ for some non-hyperelliptic component $\cD$ and some $1\leq s\leq a-1$. Note that $\cD\bar{\oplus}s$ is a non-hyperelliptic component of a genus zero stratum with one pair of simple poles. Since $\delta=1$, $\cD\bar{\oplus}s$ is the unique non-hyperelliptic component of this stratum by \Cref{Thm:main01} for single-zero strata. This implies that $\cC=\cD\oplus s$ is also unique.

    Now assume that $\mu=(2n,-2^n)$. If $n=1$ or $n=2$, it is a usual stratum dealt with in \cite{boissy2015connected}. So assume that $n>2$. There exists no non-hyperelliptic component for the residueless stratum by \Cref{Enonhyper}. Suppose that there exists a non-residueless pole. That is, $\dim \cP(\mu^{\fR})>1$. By \Cref{lm:unbubble}, any non-hyperelliptic component $\cC$ can be written as $\cC=\cD\oplus s$ for some $\cD$ and some $1\leq s<\frac{n}{2}$. By the same argument as above, using $\cD\bar{\oplus}s$, we conclude that $\cP(\mu^{\fR})$ has two non-hyperelliptic components distinguished by rotation number.
\end{proof}

\begin{proof}[Proof of \Cref{Thm:mainhigh} for single-zero strata]
    Assume that $\cC$ is a non-hyperelliptic component of $\cP(\mu^{\fR})$. We first prove that $\cC$ can be written as one of two bubbling handle/pair-of-simple-poles constructions, which may be distinguished by spin parity.

    First, assume that $g=0$ and $k\geq 2$. If $\mu^{\fR}=(2\mid-1^2\mid-1^2)$, then there exists {\em no} non-hyperelliptic component. If $\mu^{\fR}=(4\mid-2\mid-1^2\mid-1^2)$, then there exists a unique non-hyperelliptic component $\cP(2\mid-2\mid-1^2)\bar{\oplus}1$. Assume that $\dim \cP(\mu^{\fR})=1$ and $\widetilde{\mu}\neq (2),(4,-2)$; then $\cP(\mu^{\fR})$ is a stratum of D-signature with $k=2$. This case is proven in \Cref{Prop:Dspin}. In particular, $\cC=\cD\bar{\oplus}1$ or $\cD\bar{\oplus}2$ for some fixed component $\cD$, and they are equal to each other if and only if $\cP(\mu^{\fR})$ is {\em not} of even type.

    Similarly, if $\dim \cP(\mu^{\fR})=k-1$ with $k>2$, then $\widetilde{\mu}\neq (2),(4,-2)$ automatically. By \Cref{lm:unbubble_pair}, we have
    \[
    \cC=\cD\bar{\oplus}t_1\bar{\oplus}\dots\bar{\oplus}t_{k-2}
    \]
    for some non-hyperelliptic component $\cD$ of a stratum of D-signature with two pairs of simple poles. Also, by \Cref{Prop:Dspin}, $\cD=\cD'\bar{\oplus}1$ or $\cD'\bar{\oplus}2$ for some fixed component $\cD'$. 
    
    Therefore, by \Cref{lm:unbubblespin}, we can conclude that $\cC$ can be written as one of the following two components:
    \[
    (\cD'\bar{\oplus}1) \bar{\oplus}(1\bar{\oplus}\dots\bar{\oplus}1) \quad \text{or} \quad
    (\cD'\bar{\oplus}2) \bar{\oplus}(1\bar{\oplus}\dots\bar{\oplus}1).
    \]
    They are equal to each other if and only if $\cP(\mu^{\fR})$ is {\em not} of even type.

    Suppose that $g=1$ and $k\geq 1$. If $\mu^{\fR}=(4\mid-2\mid-1^2)$, then there exists a unique non-hyperelliptic component $\cP(2,-2)\bar{\oplus}1$. If $\mu^{\fR}=(2\mid-1^2)$, then there exists {\em no} non-hyperelliptic component. Assume that $\dim \cP(\mu^{\fR})=2$ and $\widetilde{\mu}\neq (2),(4,-2)$; then $k=1$, and $\cC$ can be written as $\cD\bar{\oplus}t$ for some $1\leq t<\frac{a}{2}$ and some non-hyperelliptic component $\cD$ of a stratum of E-signature. Therefore, by an immediate corollary of \cite[Thm~1.7]{lee2023connected}, $\cD$ contains an irreducible horizontal boundary divisor, that is, a flat surface of genus zero with one pair of simple poles. In this case, $\cD\bar{\oplus}t$ contains a stratum of D-signature with two pairs of simple poles, except the simple poles in one pair are not labeled. Since the spin parity is defined independently of the labeling of simple poles, we can conclude that there are two non-hyperelliptic components distinguished by spin parity by \Cref{Prop:Dspin}. In particular, $\cC=\cD\bar{\oplus}1$ or $\cD\bar{\oplus}2$ where $\cD$ is some fixed component of the stratum of E-signature. They are equal to each other if and only if $\cP(\mu^{\fR})$ is {\em not} of even type.

    Similarly, if $\dim \cP(\mu^{\fR})=1+k$ for $k>1$, then $\widetilde{\mu}\neq (2),(4,-2)$. By \Cref{lm:unbubble_pair}, $\cC$ can be written as
    \[
    \cC=\cD\bar{\oplus}t_1\bar{\oplus}\dots\bar{\oplus}t_k
    \]
    for some non-hyperelliptic component $\cD$ of a stratum of E-signature. Therefore, we can conclude that $\cC$ can be written as one of the following two components:
    \[
    (\cD\bar{\oplus}2)\bar{\oplus}(1\bar{\oplus}\dots\bar{\oplus}1)\quad \text{or} \quad(\cD\bar{\oplus}1)\bar{\oplus}(1\bar{\oplus}\dots\bar{\oplus}1).
    \]
    They are equal to each other if and only if $\widetilde{\mu}$ is {\em not} of even type.

    Suppose that $g\geq 2$ and $k\geq 0$. If $\dim \cP(\mu^{\fR})=2g+k-1$, then by \Cref{lm:unbubble_pair} we can write
    \[
    \cC=\cD\bar{\oplus}(t_1\bar{\oplus}\dots\bar{\oplus}t_k)
    \]
    where $\cD$ is some non-hyperelliptic component of a residueless stratum with genus $g$. By \cite{lee2023connected}, $\cD=\cD'\oplus1$ or $\cD'\oplus2$ for some fixed component $\cD'$ with genus $g-1>1$. Therefore, by \Cref{lm:unbubblespin}, we can conclude that $\cC$ can be written as one of the following two components:
    \[
    (\cD'\bar{\oplus}2)\bar{\oplus}(1\bar{\oplus}\dots\bar{\oplus}1)\quad \text{or} \quad(\cD'\bar{\oplus}1)\bar{\oplus}(1\bar{\oplus}\dots\bar{\oplus}1).
    \]
    They are equal to each other if and only if $\widetilde{\mu}$ is {\em not} of even type.

    Finally, if $\dim \cP(\mu^{\fR})>2g+k-1$, then combining \Cref{lm:unbubble} and \Cref{lm:unbubble_pair}, $\cC$ can be written as
    \[
    \cC=\cD\oplus(s_1\oplus\dots\oplus s_{g-1})\bar{\oplus}(t_1\bar{\oplus}\dots\bar{\oplus}t_k)
    \]
    for some non-hyperelliptic component $\cD$ of a two-dimensional stratum with genus one and no pair of simple poles. By \Cref{Thm:main10} for single-zero strata, $\cD$ is determined by its rotation number $r|\delta$, so $\cD=\cD'\oplus r$ for some fixed $\cD'$.
    
    Since $g+k\geq2$ and $r\leq d<\frac{a}{2}-k-g$, by \Cref{lm:unbubblespin}, $\cC$ is equal to one of the following components:
    \[
    (\cD'\oplus1)\oplus(1\oplus\dots\oplus1)\bar{\oplus}(1\bar{\oplus}\dots\bar{\oplus}1)\quad \text{or} \quad(\cD'\oplus2)\oplus(1\oplus\dots\oplus1)\bar{\oplus}(1\bar{\oplus}\dots\bar{\oplus}1).
    \]
    
    Therefore, it suffices to determine whether $\cD'\bar{\oplus}2\bar{\oplus}1$ and $\cD'\bar{\oplus}1\bar{\oplus}1$ are distinct. So it is reduced to the case when $(g,k)=(0,2)$ and $\dim \cP(\mu^{\fR})>1$. If $\widetilde{\mu}$ is of even type, then they are distinguished by spin parity. If not, then $\delta$ is odd. Also, by \Cref{Thm:main01} for single-zero strata, $\cD'\bar{\oplus}1=\cD'\bar{\oplus}(1+d)$. Therefore,
    \[
    \cD'\bar{\oplus}1\bar{\oplus}1=\cD'\bar{\oplus}(1+d)\bar{\oplus}1=\cD'\bar{\oplus}2\bar{\oplus}1
    \]
    by \Cref{lm:unbubblespin}.
\end{proof}

\subsection{Multiple-zero strata}

We finish the section by proving the main theorems in general. In order to complete the classification of non-hyperelliptic components, it remains to prove \Cref{Thm:main00}, \Cref{Thm:main01}, and \Cref{Thm:mainhigh} for multiple-zero strata.

\begin{proof}[Proof of \Cref{Thm:main10}]
    Suppose that $\cP(\mu^{\fR})$ has no simple poles, and $\mu\neq (n,n,-2^n)$. Then we already have the classification of non-hyperelliptic components by \Cref{reductiontoresidueless}.  

    Suppose that $\cP(\mu^{\fR})$ has a simple pole but {\em no} pair of simple poles. In the previous subsection, we proved that the corresponding single-zero stratum $\cP(\mu'^{\fR})$ obtained by merging zeroes has a unique non-hyperelliptic component $\cD$. By \Cref{lm:mergezeroes}, any non-hyperelliptic component $\cC$ of $\cP(\mu^{\fR})$ is adjacent to $\cD$, thus unique.

    Finally, let $\mu=(n,n,-2^n)$. We already know that there is no non-hyperelliptic component for the residueless strata. Assume that there exists a non-residueless pole. If $n=1$ or $n=2$, then it is a usual stratum dealt with in \cite{boissy2015connected}. Assume that $n>2$. By merging zeroes, $\cC$ is adjacent to some non-hyperelliptic component $\cD$ of a single-zero stratum $\cP(\mu'^{\fR})$. In the previous subsection, we proved that there are two non-hyperelliptic components distinguished by rotation number. If $n$ is even, after breaking up the zero, we still obtain two non-hyperelliptic components distinguished by rotation number. Note that two non-hyperelliptic components are written as $\cD'\oplus 1$ and $\cD'\oplus (n+1)$ since $n+1$ is even. Consider $\cD'\bar{\oplus}s$. By breaking up the zero in the bottom level component, we have a non-hyperelliptic flat surface in a stratum $\cP(n,n\mid -2n\mid -1^2)$ with index $I\equiv s\pmod n$. Since $1\equiv n+1 \pmod n$, $B(\cD'\bar{\oplus}1)=B(\cD'\bar{\oplus}(n+1))$, and therefore $\cP(\mu^{\fR})$ has a unique non-hyperelliptic component. 
\end{proof}

\begin{proof}[Proof of \Cref{Thm:main00}]
    First, suppose that $\dim \cP(\mu^{\fR})>m-1$. Then $\dim \cP(\mu'^{\fR})>0$, and thus this single-zero stratum has a unique non-hyperelliptic component $\cD$. Any non-hyperelliptic component of $\cP(\mu^{\fR})$ is obtained by breaking up the zero from $\cD$ by \Cref{lm:mergezeroes}. Therefore, $\cC$ is also unique. 

    Now assume that $\dim \cP(\mu^{\fR})=m-1$ and $m>1$. We use induction on $m$. If $m=2$, then $\cP(\mu^{\fR})$ is a stratum of B-signature. So by \Cref{Prop:Bnonhyper}, it has a unique non-hyperelliptic component unless $\mu^{\fR}=(a,a\mid-2\mid\dots\mid-2\mid-(a-n+2),-(a-n+2))$, where there exists no non-hyperelliptic component. If $m>2$, then by \Cref{lm:mergezeroes}, merging two zeroes with the highest orders, any non-hyperelliptic component $\cC$ is adjacent to a non-hyperelliptic component $\cD$ of a stratum with $m-1$ zeroes. By the induction hypothesis, $\cD$ is unique and thus $\cC$ is unique.
\end{proof}

In order to prove \Cref{Thm:main01}, let $\cP(\mu^{\fR})$ be a multiple-zero stratum with genus $g=0$, $n>1$ zeroes, and exactly one pair of simple poles. We use induction on $n\geq 1$. Let $\cP(\mu'^{\fR})$ be the corresponding single-zero stratum. Then $\dim \cP(\mu'^{\fR})=n-|\fR|-1$, and it is equal to zero if and only if $\fR$ consists of one pair and $n-2$ singletons. Therefore, in this case, we use one-dimensional strata of B-signature as a base case instead of the single-zero strata.

First, we prove \Cref{Thm:main01} for double-zero strata $\cP(\mu^{\fR})$ with dimension larger than one. That is, the corresponding single-zero stratum $\cP(\mu'^{\fR})$ has dimension larger than zero.

\begin{proposition} \label{Prop:main01base1}
    Let $\cP(\mu^{\fR})$ be a stratum with $g=0$, two zeroes, and exactly one pair of simple poles. Suppose that $\dim \cP(\mu^{\fR})>1$. Denote $\delta\coloneqq\gcd(a_1,a_2,b_1,\dots,b_{n-1})$. For each $1\leq I\leq d$, there exists a unique non-hyperelliptic component with index $I$. 
\end{proposition}

\begin{proof}
    Let $\cC$ be any non-hyperelliptic component with index $I$. By \Cref{Thm:ssc}, there exists a flat surface in $\cC$ with a multiplicity one saddle connection joining two distinct zeroes. By merging zeroes, we obtain a component of a single-zero stratum $\cD$. That is, $\cC=B(\cD)$. Let $\Delta\coloneqq \gcd(b_1,\dots,b_{n-2})$.

    First, assume that $\cD$ is hyperelliptic. Since $\cC$ is not hyperelliptic, we must have $a_1\neq a_2$. Assume that $\cD$ has a fixed pole $p$ of order $b$. Then by \Cref{Prop:ssc_he}, there exists a pair of multiplicity two saddle connections bounding the polar domain of $p$. By shrinking them, we obtain a two-level boundary $\overline{X}$ with the bottom level component $X_{-1}\in \cP(a\mid-b\mid-\frac{a-b+2}{2},-\frac{a-b+2}{2})$. By breaking up the zero, we obtain a non-hyperelliptic flat surface in $\cP(a_1,a_2\mid-b\mid-\frac{a-b+2}{2},-\frac{a-b+2}{2})$. By \Cref{Cor:Bssc}, this is adjacent to a non-hyperelliptic component. Thus, $\cC=B(\cD')$ for a non-hyperelliptic component. That is, we may assume that $\cD$ is non-hyperelliptic.
    
    By \Cref{Thm:main01} for single-zero strata, $\cD$ is the unique non-hyperelliptic component $\cD_J$ with index $1\leq J\leq \Delta$. We want to prove that $B(\cD_{J_1})=B(\cD_{J_2})$ if $J_1\equiv J_2 \pmod \delta$. Then for a fixed non-hyperelliptic $\cC_0$, the component $\cC_J$ is equal to $\cC_0\bar{\oplus} t$ for any $1\leq t\leq a-1$ satisfying $t\equiv J\pmod \Delta$.
    
    Consider a flat surface with index $t$ in $\cP(a\mid-a\mid-1,-1)$, which appears in bubbling a pair of simple poles surgery. By breaking up the zero, we obtain a non-hyperelliptic component of a stratum $\cP(a_1,a_2\mid-a\mid-1,-1)$. This component is determined by the index given by $1\leq I\leq \gcd(a_1,a_2)$ such that $t\equiv I \pmod{\gcd(a_1,a_2)}$. That is, we have $B(\cC_0\bar{\oplus} t_1)=B(\cC_0\bar{\oplus} t_2)$ if $t_1\equiv t_2\pmod{\gcd(a_1,a_2)}$. Thus we can conclude that $B(\cC_0\bar{\oplus} t_1)=B(\cC_0\bar{\oplus} t_2)$ if $t_1\equiv t_2 \pmod{\delta}$, where $\delta=\gcd(\Delta,a_1,a_2)$.
\end{proof}

Now we prove \Cref{Thm:main01} for triple-zero strata $\cP(\mu^{\fR})$ with dimension two. That is, the corresponding single-zero stratum has dimension zero.

\begin{proposition} \label{main1-2base2}
    Let $\cP(\mu^{\fR})$ be a stratum with $g=0$, two zeroes, and exactly one pair of simple poles. Suppose that $\dim \cP(\mu^{\fR})=2$. Denote $\delta\coloneqq\gcd(a_1,a_2,a_3,\{b_j\})$. For each $1\leq I\leq \delta$, there exists a unique non-hyperelliptic component with index $I$.
\end{proposition}

\begin{proof}
    Assume that $a_1\leq a_2\leq a_3$, and let $\cC$ be a non-hyperelliptic component with index $I$. Denote $\Delta\coloneqq\gcd(a_1,a_2+a_3,b_1,\dots,b_{n-2})$.
    
    By \Cref{Thm:ssc}, we can merge the two zeroes $z_2$ and $z_3$, and $\cC$ is adjacent to the unique non-hyperelliptic component $\cD_J$ with index $1\leq J\leq \Delta$ of a stratum of B-signature with two zeroes of orders $a_1$ and $a_2+a_3$, such that $J\equiv I\pmod{\delta}$.
    
    Suppose that $J_1\equiv J_2\pmod{\delta}$. We need to show that $B(\cD_{J_1})=B(\cD_{J_2})$. Then $\cC=B(\cD_J)=B(\cD_I)$, proving the uniqueness of $\cC$. Consider Type IIIc boundaries $\overline{X}$ and $\overline{Y}$ of $\cD_{J_1}$ and $\cD_{J_2}$, respectively, such that the top level components $X_0$ and $Y_0$ are equal. Note that the bottom level components $X_{-1}$ and $Y_{-1}$ also contain the pair of simple poles, and their indices are denoted by $K_1$ and $K_2$, respectively. They satisfy $K_i\equiv J_i \pmod{\Delta}$. Thus $J_1-J_2\equiv K_1-K_2\pmod{\delta}$. 
    
    By breaking up the zero of the bottom level components, we obtain flat surfaces $X'_{-1}$ and $Y'_{-1}$ in another stratum of B-signature with two zeroes of orders $a_2$ and $a_3$. They are contained in the same connected component since $K_1\equiv K_2\pmod{\delta}$, where $\delta=\gcd(a_2,a_3,\Delta)$. Thus, we have $\cD_{J_1}=\cD_{J_2}$.
\end{proof}

We are now ready to prove \Cref{Thm:main01} for multiple-zero strata.

\begin{proof}[Proof of \Cref{Thm:main01}]
    Let $\cC$ be a non-hyperelliptic component with index $I$. First, assume that $\dim \cP(\mu^{\fR})>n-1$. Then we can merge all zeroes, obtaining a non-hyperelliptic component $\cD_J$ of a single-zero stratum of dimension larger than zero. That is, $\cC=B(\cD_J)$. Denote $\Delta=\gcd(a_1+\dots+a_m,b_1,\dots,b_{n-2})$. By \Cref{Thm:main01} for single-zero strata, $\cD_J$ is uniquely determined by its index $1\leq J\leq \Delta$. We have $J\equiv I\pmod{\delta}$. It suffices to show that $\cC=B(\cD_I)$.
    
    By breaking up the zero into two zeroes $z_i$ of order $a_i$ and $w_i$ of order $a_1+\dots+\widehat{a_i}+\dots+a_m$, we obtain a component $\cD_J^i$ of a double-zero stratum, denoted by $B_i(\cD_J)=\cD_J^i$. By breaking up the zero $w_i$ further into $n-1$ zeroes, we can recover $\cC$. The index $J_i$ of $\cD_J^i$ is determined by $J_i\equiv J\pmod{\gcd(\Delta,a_i)}$. Also, by \Cref{Prop:main01base1}, $\cD_J^i$ is uniquely determined by its index. Thus $B_i(\cD_J)=B_i(\cD_{J'})$ for any $J'$ such that $J'\equiv J\pmod{\gcd(\Delta,a_i)}$. In particular, $\cC=B(\cD_{J'})$ for any such $J'$. 
    
    Since this holds for each $i=1,\dots,n$, we can conclude that $\cC=B(\cD_{J'})$ for any $J'$ such that $J'\equiv J\pmod{\delta}$, where $\delta=\gcd(\Delta,a_1,\dots,a_m)$. Therefore, we have $\cC=B(\cD_I)$, and $\cC$ is uniquely determined by its index $I$.
    
    Now assume that $\dim \cP(\mu^{\fR})=m-1$. By merging all zeroes, we obtain a non-hyperelliptic component $\cC_0$ of a single-zero stratum of dimension zero. Therefore, it is not uniquely determined by its index. 
    
    Instead, consider a non-hyperelliptic component $\cE_J=B_1(\cC_0)$ of index $J$ of a double-zero stratum. Break up the zero $w_1$ into two zeroes $z_i$ of order $a_i$ $(i>1)$ and $w'_i$ of order $a_2+\dots+\widehat{a_i}+\dots+a_m$, obtaining a component $\cE_J^i$ of a triple-zero stratum, denoted by $B^i(\cD^1)=\cE^i$. The index of $\cE^i$ is determined by $J^i\equiv J\pmod{\gcd(\Delta,a_1,a_i)}$. Also, by \Cref{main1-2base2}, $\cE^i$ is uniquely determined by its index. Therefore, $B^i(\cE_J)=B^i(\cE_{J'})$ for any $J'$ such that $J'\equiv J\pmod{\gcd(\Delta,a_1,a_i)}$. In particular, $\cC$ is adjacent to $\cE_{J'}$ for any such $J'$. 
    
    Since this holds for each $i=2,\dots,n$, we conclude that $\cC$ is adjacent to $\cE_I$. Thus, $\cC$ is uniquely determined by its index $I$.
\end{proof}

Finally, we prove \Cref{Thm:mainhigh} for a multiple-zero stratum. Let $\cP(\mu'^{\fR})$ be the corresponding single-zero stratum. If both $\cP(\mu^{\fR})$ and $\cP(\mu'^{\fR})$ are \emph{not} of even type, then \Cref{Thm:mainhigh} follows immediately from the single-zero version and \Cref{lm:mergezeroes}, as any non-hyperelliptic component of $\cP(\mu^{\fR})$ is adjacent to the unique non-hyperelliptic component of $\cP(\mu'^{\fR})$. Similarly, if both are of even type, \Cref{Thm:mainhigh} also follows by the same argument. Therefore, it remains to deal with the case when $\cP(\mu'^{\fR})$ is of even type but $\cP(\mu^{\fR})$ is not.

\begin{proof}[Proof of \Cref{Thm:mainhigh} for multiple-zero strata]
    We first take care of exceptional cases. Let $\cC$ be a non-hyperelliptic component of $\cP(\mu^{\fR})$. By \Cref{lm:mergezeroes}, $\cC$ is adjacent to some connected component $\cD$ of $\cP(\mu'^{\fR})$. 
    
    Assume that $g+k=2$ and $\widetilde{\mu}=(1,1)$. Then $\widetilde{\mu'}=(2)$, and therefore $\cD$ is hyperelliptic. Then $\cC$ is also hyperelliptic, a contradiction. Thus, there is no non-hyperelliptic component.
    
    Now assume that $g+k=3$ and $\widetilde{\mu}=(2,2,-2)$ or $(2,2)$. In this case, $\cP(\mu'^{\fR})$ has a unique non-hyperelliptic component $\cD$, and therefore $\cC$ is unique.
    
    From now on, assume that $\widetilde{\mu}\neq(1,1),(2,2,-2)$, or $(2,2)$. By the discussion before this proof, we may assume that $\cP(\mu'^{\fR})$ is of even type and $\cP(\mu^{\fR})$ is not. In particular, the sum $a_1+\dots+a_m$ is even, but at least one $a_i$, say $a_1$, is odd. We need to prove that $\cC$ is a unique non-hyperelliptic component. By merging all zeroes but $z_1$, we deduce that $\cC$ is adjacent to a non-hyperelliptic component $\cD$ of a double-zero stratum with two odd zeroes. 
    
    If $\cD$ is unique, then $\cC$ is also unique because it is obtained by breaking up a zero from $\cD$. Thus, it suffices to prove it for double-zero strata. Let $\cD_{\mathrm{even}}$ and $\cD_{\mathrm{odd}}$ be two non-hyperelliptic components of $\cP(\mu'^{\fR})$ distinguished by spin parity. Denote the breaking up of the zero into $z_1,z_2$ by $B$. We want to show that $B(\cD_{\mathrm{even}})=B(\cD_{\mathrm{odd}})$.
    
    As in the proof for single-zero strata, we can write $\cD_{\mathrm{odd}}=\cC_0\bar{\oplus}1\bar{\oplus}2$ and $\cD_{\mathrm{even}}=\cC_0\bar{\oplus}1\bar{\oplus}1$ (or $\oplus$ instead of some of the $\bar{\oplus}$) for some non-hyperelliptic component $\cC_0$.
    
    Consider the flat surfaces $Y_I$ of the stratum $\cP(a\mid -a\mid -1^2)$ with index $1\leq I\leq a-1$. They appear in the bubbling-a-pair construction $\bar{\oplus}I$. By breaking up the zero, we obtain non-hyperelliptic flat surfaces in a stratum of B-signature with one pair of simple poles. The connected component is determined by the index $1\leq J\leq\gcd(a_1,a)$ such that $J\equiv I\pmod{\gcd(a_1,a)}$. In particular, $B(Y_1)=B(Y_{1+\gcd(a_1,a)})$. That is, $B(\cD_{\mathrm{even}})=B(\cC_0\bar{\oplus}1\bar{\oplus}(1+\gcd(a_1,a)))=B(\cC_0\bar{\oplus}1\bar{\oplus}2)=B(\cD_{\mathrm{odd}})$, since the spin parity of $\cC_0\bar{\oplus}t_1\bar{\oplus}t_2$ is determined by the parity of $t_1+t_2\pmod{2}$.
    
    This completes the proof.
\end{proof}

%% file: 07DegenerationDynamics.tex
\section{Equatorial arcs in one-dimensional strata} \label{Sec:equatorial}

The base cases of the induction are strata $\cP(\mu^\fR)$ of dimension one. In \cite[Sec.~4]{lee2023one}, A cellular decomposition of one-dimensional residueless strata is introduced using the projective structure induced by the period coordinates. This decomposition can be applicable to any one-dimensional generalized stratum.

The $1$-skeleton of this decomposition is called the {\em equatorial net}, whose connected components reflect the connected components of the stratum itself. In this section, we briefly explain how the degeneration moves and the associated combinatorial data interact. These descriptions of degeneration will be used in later sections to prove classification results for one-dimensional projectivized strata.

We now outline the cellular decomposition. For one-dimensional strata $\cP(\mu^{\fR})$, we can choose two curves $\alpha,\beta$ on the Riemann surface $X\setminus\boldsymbol{p}$ that form a basis of $\fR$-homology $H_1(X\setminus\boldsymbol{p},\boldsymbol{z};\mathbb{Z})/K(\fR)$. Recall that the local period coordinates are given by
\[
p: \cR(\mu^\fR) \longrightarrow \mathbb{C}^2, 
\]
\[
(X,\omega)\mapsto \left( \int_\alpha \omega , \int_\beta \omega \right).
\]

By projectivization, we have 
$$p: \cP(\mu^{\fR})\longrightarrow \mathbb{C},$$
$$(X,\omega)\mapsto \int_\alpha \omega \left/ \int_\beta \omega \right.$$

Let $\alpha',\beta'$ form another basis of $H_1(X\setminus\boldsymbol{p},\boldsymbol{z};\mathbb{Z})/K(\fR)$. Then there is a 2 by 2 integer matrix $\left(\begin{smallmatrix}
    a & b \\ c& d
\end{smallmatrix}\right)$ such that $$\left(\begin{smallmatrix}
    [\alpha'] \\ [\beta']
\end{smallmatrix}\right) = \left(\begin{smallmatrix}
     a & b \\ c& d
\end{smallmatrix} \right)\left(\begin{smallmatrix}
    [\alpha] \\ [\beta]
\end{smallmatrix} \right).$$

That is, the preimage of the real line under this map $p$ is independent on the choice of the basis. So it locally define a real one-dimensional submanifold of $\cP(\mu^{\fR})$. This is the locus of flat surfaces whose saddle connections are all in the same direction. 

The map $p$ extends to
\[
\overline{p}: \overline{\cP}(\mu^\fR) \longrightarrow \mathbb{C}\mathbb{P}^1.
\]

The {\em equatorial net} of $\overline{\cP}(\mu^\fR)$, denoted by $\mathcal{A}(\mu^\fR)$, is defined by the preimage of $\mathbb{R}\mathbb{P}^1 \subset \mathbb{C}\mathbb{P}^1$ under $\overline{p}$. This is no longer a submanifold and have singularities at the boundary points of $\overline{\cP}(\mu^\fR)$.

The equatorial net decomposes into $1$-cells, called {\em equatorial arcs}, and $0$-cells. The equatorial net defines a ribbon graph on the compactification $\overline{\cP}(\mu^\fR)$ of a one-dimensional stratum, as the coarse moduli space of this stratum is a compact Riemann surface. So there is a one-to-one correspondence between the connected components of $\overline{\cP}(\mu^\fR)$ and the connected components of its equatorial net $\mathcal{A}(\mu^\fR)$. 

The half-edges of the underlying ribbon graph will be referred to as {\em equatorial half-arcs}. To understand how to navigate from one boundary element to another within the same connected component, one can simply follow the equatorial arcs.

\subsection{Equatorial arcs and their representation}

The boundary of a one-dimensional stratum $\overline{\cP}(\mu^\fR)$ consists of two-level multi-scale differentials $\overline{X}$, where each level is a (possible disconnected) flat surface contained in the zero-dimensional generalized stratum. That is, in each level, there exists only one collection of saddle connection all parallel to each other. If $\alpha$, $\beta$ are saddle connections in the level 0 and -1, respectively, then in the plumbing neighborhood of $\overline{X}$, the ratio between the periods of these two collections is determined by the plumbing parameter $$t=\int_\beta \omega \left/ \int_\alpha \omega \right.$$ so the flat surface after plumbing is contained in the equatorial arc if and only if $t\in \mathbb{R}$. As $t\to 0$, the flat surface converges to the boundary $\overline{X}$. Therefore there are two equatorial arcs, $t\in \mathbb{R}_+$ and $t\in \mathbb{R}_-$ emanating from $\overline{X}$. Recall that we need the discrete parameter--prong-matching to determine the flat surface after plumbing. Therefore, we can conclude that for each boundary $\overline{X}$, the number of equatorial arcs adjacent to $\overline{X}$ is twice the number of prong-matchings. 

The level rotation group $\mathbb{C}$ on level -1 acts on a plumbing neighborhood of $\overline{X} \in \partial\overline{\cP}(\mu^\fR)$ by multiplication on the plumbing parameter $t$ by $e^{2\pi z}$ for $z\in \mathbb{C}$. The subgroup $\mathbb{Z}[\tfrac{1}{2}] \subset \mathbb{C}$ preserves the equatorial net near $\overline{X}$, specifically by rotating the equatorial half-arcs around it. We denote by $R$ the rotation of an equatorial half-arc to the next one in the clockwise direction. In addition to level rotation, one can transform an equatorial half-arc around a boundary element into its conjugate half-arc, which starts from the opposite endpoint. Let $U$ denote this transformation. The transformations $U$ and $R$ are depicted in \Cref{fig:transformations_arcs}. Note that the action generated by $R$ and $U$ on the set of equatorial half-arcs is transitive in the given connected component. Thus by applying $R$ and $U$, we can connected any two $0$-cells in the same connected component by describing an explicit path joining them. Our strategy for classifying the connected components of one-dimensional strata is to show that for a fixed topological invariant, we can reach a boundary point of a form fixed for the topological invariant.

\begin{figure}[h!]
    \centering
    \includegraphics[width=\textwidth]{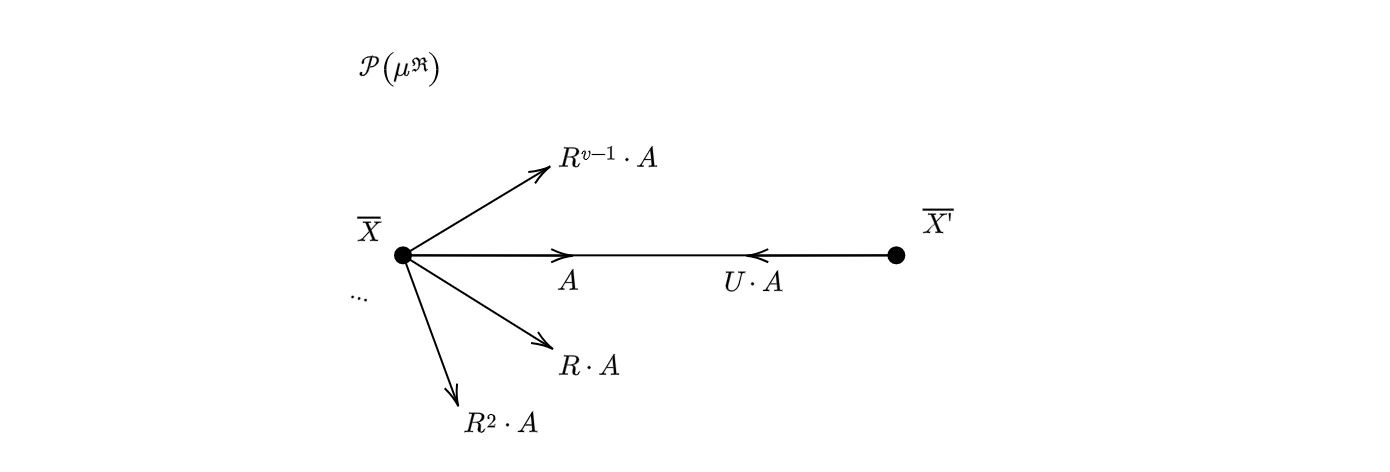}
    \caption{Transformations $R$ and $U$}
    \label{fig:transformations_arcs}
\end{figure}

In order to visualize the multi-scale differentials or flat surfaces on equatorial arcs in a certain one-dimensional projectivized stratum, we use diagrams called {\em separatrix diagrams}. They are defined as follows:

\begin{definition}
    A {\em separatrix diagram} of a flat surface with only parallel saddle connections is a ribbon graph whose vertices and edges correspond to zeros and saddle connections, respectively. The cycles of the ribbon graph are in one-to-one correspondence with the poles. In addition, a number $\theta$ in $\mathbb{N}[\frac{1}{2}]$ is assigned to each pair of adjacent half-edges to record the angles between the saddle connections (i.e., the actual angle is $2\theta\pi$).
\end{definition}

%% file: 08Bsignatures.tex
\section{Strata of B-signature} \label{Sec:Bsignature}

In this section, we prove \Cref{Prop:Bhyper}--\Cref{Prop:Bnonhyper}. Let $\cP(\mu^{\fR})$ be a stratum of B-signature. We denote $\mu^\fR = (a_1, a_2\mid-b_1 \mid \dots \mid -b_{n-2}\mid  -e_1, -e_2)$ with $a_1+a_2-(e_1+e_2+\sum_i b_i)=-2$ throughout this section. Assume also that $0 < a_1 \leq a_2$, $0 < e_1 \leq e_2$, and $2 \leq b_1 \leq \dots \leq b_{n-2}$. Let $q_1$ and $q_2$ denote the non-residueless poles of orders $e_1$ and $e_2$, respectively. 

\subsection{Boundaries of strata of B-signatures}

We first describe the boundaries of $\cP(\mu^{\fR})$. They are Type I or Type III principal boundary. 

\subsubsection{Type I boundary}

By \Cref{lm:univ_prin_bdry}, any connected component of $\cC$ of $\cP(\mu^\fR)$ contains a Type I boundary. Recall that a Type I boundary is obtained by shrinking a saddle connection joining $z_1$ and $z_2$. A Type I boundary of a stratum $\cP(\mu^{\fR})$ of B-signature is given by the following combinatorial data:

\begin{itemize}
    \item An integer $0 < \ell \leq n-2$ indicating the number of residueless poles contained in the top-level component.
    \item A permutation $\tau \in \operatorname{Sym}_{n-2}$ on the residueless poles.
    \item A tuple of integers ${\bf C} = (C_1, \dots, C_{n-2})$ where $1 \leq C_i \leq b_i - 1$ for each $i$, and positive integers $C,D$ such that:
    \begin{align*}
        a_1+1&=C + \sum_{i=\ell+1}^{n-2} C_{\tau(i)}, \\
        a_2+1&=D +\sum_{i=\ell+1}^{n-2} (b_{\tau(i)} - C_{\tau(i)}).
    \end{align*}
\end{itemize}

Let $\tau_1=\tau|_{\{1,\dots, \ell\}}$ and ${\bf C_1}=(C_{\tau(i)})_{i=1,\dots, \ell}$. Then the top level component is isomorphic to $Z_2(\tau_1,{\bf C_1},e_1,e_2)$. Note that the bottom level component has $n-1-\ell$ residueless pole, including one pole of order $C+D$ at the node $s^\bot$. Let $\tau_2(i)=\tau(i+\ell)$ for $i=1,\dots, n-2-\ell$ and $\tau_2(n-1-\ell)=s^\bot$. Also let ${\bf C_2}=(C_{\tau(\ell+1)},\dots,C_{n-2},C)$. Then the bottom level component is isomorphic to $Z_1(\tau_2,{\bf C_2})$. We denote this Type I boundary by 
\[
X^B_{\I} (\ell,\tau,{\bf C}).
\]

We now present the level graphs of various principal boundaries, along with the separatrix diagrams corresponding to these boundaries in \Cref{fig:B_type_I_graph}. The enhancement of the level graph is denoted by $\kappa \coloneqq C + D - 1$.
\begin{figure}[h!]
        \centering
        \resizebox{13.5cm}{6.5cm}{\includegraphics[]{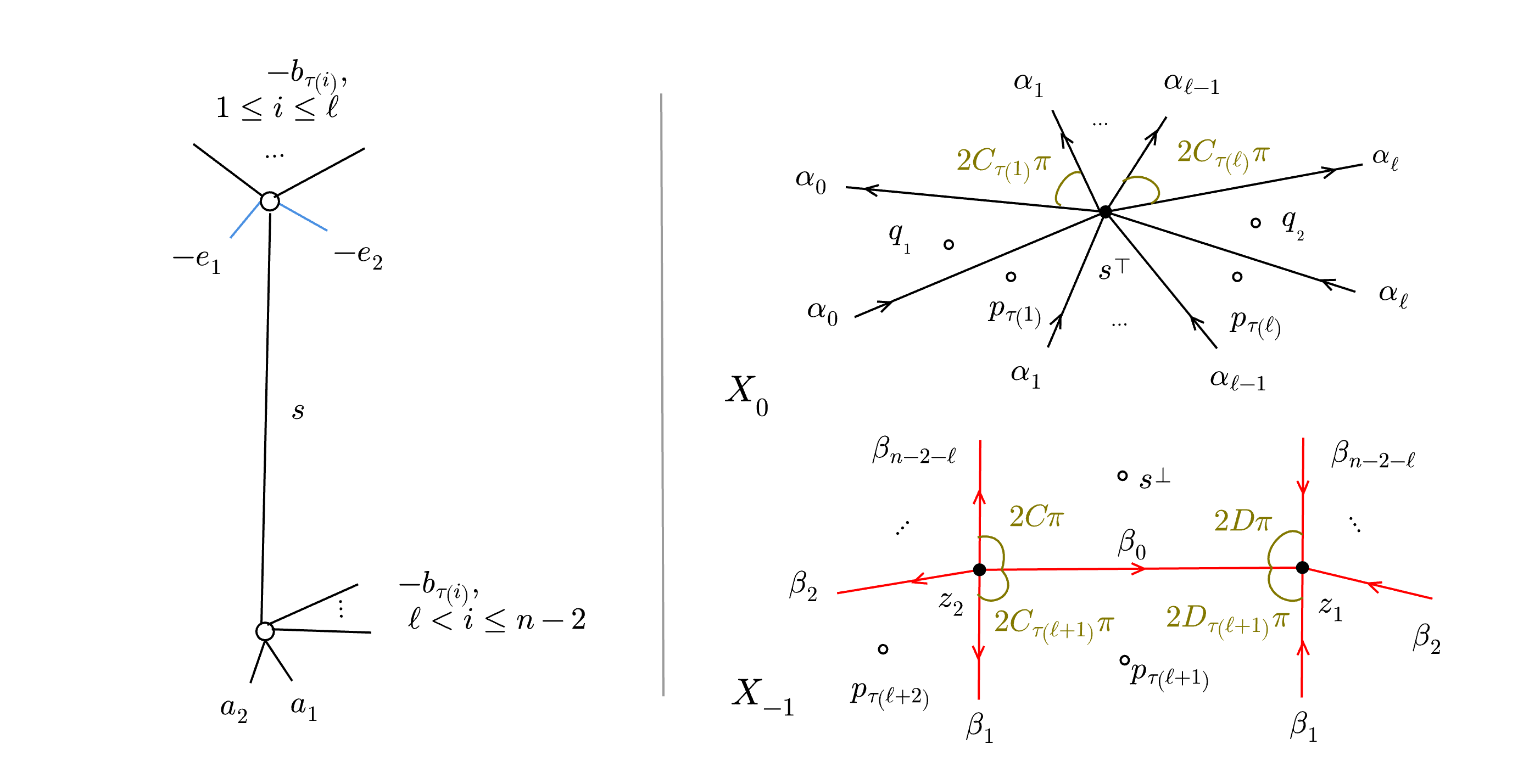}}
        \caption{Level graph and separatrix diagram of Type I boundary of strata of B-signatures}
        \label{fig:B_type_I_graph}
\end{figure}

In \Cref{fig:B_typeI_pm}, the labeling of prongs on each irreducible component is depicted. On the top level, we label the outgoing prongs using the cyclic group $\mathbb{Z}/\kappa\mathbb{Z}$; that is, $v^+_{u_1} = v^+_{u_2}$ if $u_1 - u_2 \equiv 0 \pmod{\kappa}$. The outgoing prongs $v^+_{c_0}, v^+_{c_1}, \dots, v^+_{c_\ell}$ coincide with saddle connections in the same direction, while the outgoing prongs $v^+_{-d_0}, v^+_{-d_1}, \dots, v^+_{-d_\ell}$ are those immediately counter-clockwise to each corresponding saddle connection. Here,

\begin{align*}
    c_i = \sum_{j=1}^i C_{\tau(j)}, \quad
    d_i = e_1 + \sum_{j=1}^i D_{\tau(j)}.
\end{align*}

\begin{figure}[h!]
        \centering
        \resizebox{15.5cm}{9cm}{\includegraphics[]{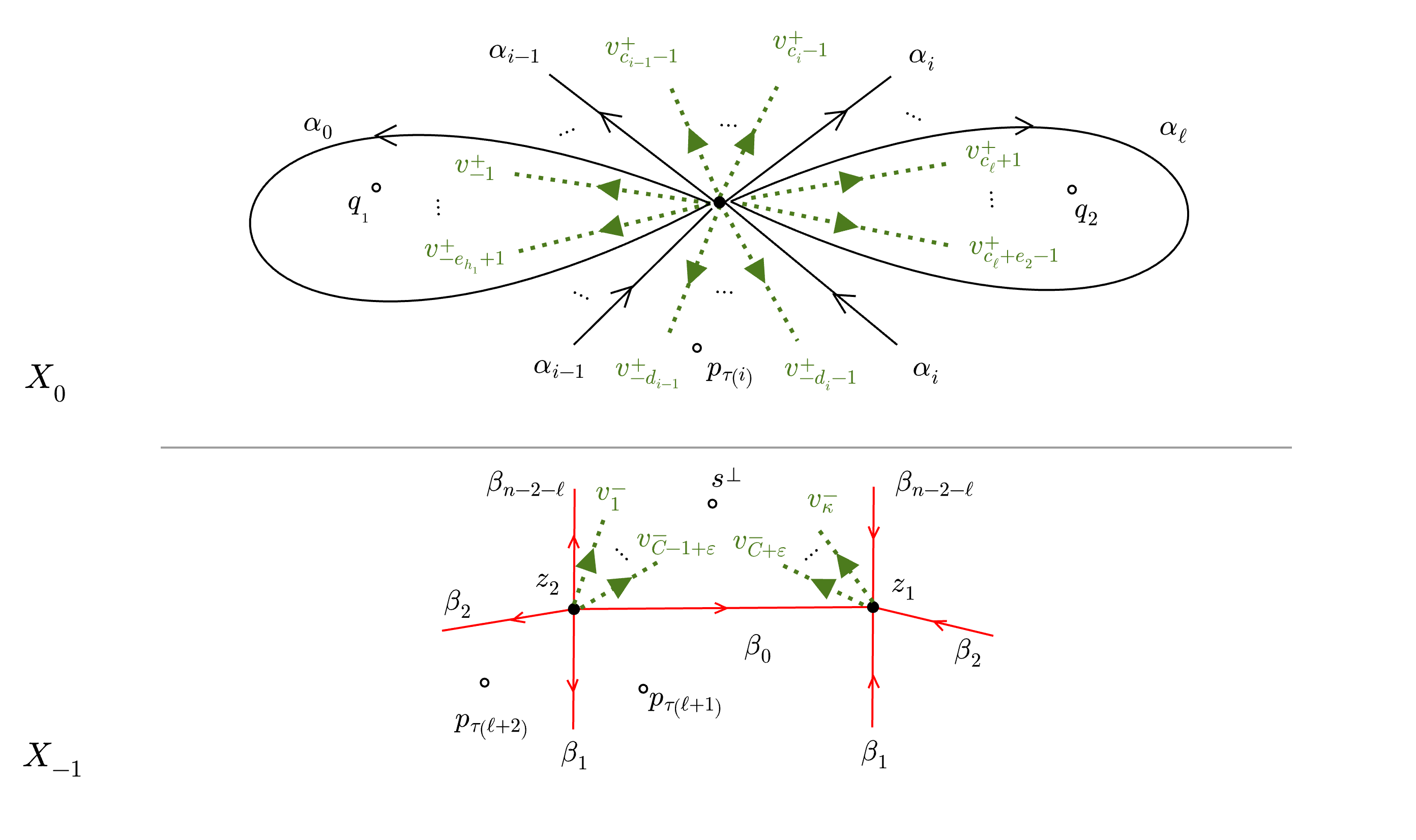}}
        \caption{The marking of prongs for Type I boundary of strata of B-signatures}
        \label{fig:B_typeI_pm}
\end{figure}
Note that 
\[
e_1 + e_2 - 1 + \sum_{j=1}^\ell \left(C_{\tau(j)} + D_{\tau(j)}\right) = \kappa,
\]
hence we have the identity $v^+_{-d_\ell} = v^+_{c_\ell + e_2 - 1}$.

On the bottom level, the incoming prongs at $s^\bot$ correspond to the outgoing prongs at $z_2$ and $z_1$ pointing into the polar domain of $s^\bot$. We label these prongs $v^-_1, \dots, v^-_{\kappa}$ in clockwise order around $s^\bot$. The prong $v^-_1$ is defined as the outgoing prong at $z_2$ directed toward $s^\bot$ that lies immediately clockwise from the saddle connection $\beta_{n-2-\ell}$ at $z_1$.

Depending on the rescaling of the bottom level, there are $C - 1 + \varepsilon$ outgoing prongs at $z_1$ directed toward $s^\bot$, where
\[
\varepsilon = 
\begin{cases}
    1 & \text{if } \arg(\alpha_0/\beta_0) \in (0, \pi], \\
    0 & \text{if } \arg(\alpha_0/\beta_0) \in (\pi, 2\pi].
\end{cases}
\]

Then a prong-matching $\boldsymbol{\sigma}$ is uniquely determined by the image $v^+_u$ of $v^-_{\kappa}$. In this case, we identify the prong-matching with $u \in \mathbb{Z}/\kappa\mathbb{Z}$, i.e., $\boldsymbol{\sigma} = u$.

\subsubsection{Type III boundary}

By shrinking any saddle connection that joins a zero to itself, a flat surface in $\cP(\mu^\cF)$ of B-signature degenerates to a Type III boundary. There are three types of Type III boundaries, distinguished by the structure of the level graphs. These can be categorized as follows:

\begin{itemize}
    \item Type IIIa: when $q_1$ and $q_2$ lie on different components at the bottom level;
    \item Type IIIb: when only one of the non-residueless poles lies on the bottom level, while the other remaining on the top level is not a simple pole;
    \item Type IIIc: when $q_1$ and $q_2$ lie on the same component at the bottom level.
\end{itemize}

For each type, the positions of the zeros and poles can be interchanged, resulting in a total of 8 different boundary types distinguished by the level graph configuration. Note that, for a given stratum, some of these boundary types may not occur due to numerical constraints on $a_1, a_2, e_1,$ and $e_2$.

If $e_1 = e_2 = 1$, then boundaries of Type IIIb do not exist, whereas boundaries of Type IIIa and Type IIIc appear in every connected component. On the other hand, if $e_2 > 1$, then Type IIIc with $z_1$ and $q_2$ on the top level always exists, since we assume $a_2 \leq e_1$.

We will navigate these different boundary types to classify the connected components. The three types of Type III boundaries are characterized by the combinatorial data described in \Cref{tab:TypeBIIIabc_comb}, with the corresponding level graphs and separatrix diagrams shown in \Cref{tab:TypeB_IIIabc}.

\begin{table}[h!]
    \centering
    \begin{tabular}{|c|p{12cm}|}
    \hline
        \textbf{Type} & \textbf{Combinatorial Data} \\ \hline
         IIIa & 
         A tuple $\underline{h} = (h_1, h_2)$ indicating the zeros $z_{h_1}$ and $z_{h_2}$ contained in the irreducible components with $q_1$ and $q_2$, respectively. \\ \cline{2-2}
         & Integers $\ell_1, \ell_2$ with $0 \leq \ell_1 < \ell_2 \leq n-2$, such that $\ell_1$ (resp. $n-2 - \ell_2$) is the number of residueless poles on the component $X_{-1}^1$ (resp. $X_{-1}^2$). \\ \cline{2-2}
         & A permutation $\tau \in \mathrm{Sym}_{n-2}$ indexing the residueless poles. \\ \cline{2-2}
         & A tuple ${\bf C} = (C_1, \dots, C_{n-2})$ with $1 \leq C_i \leq b_i - 1$. \\ \cline{2-2}
         & A prong-matching equivalence class $[(u,v)] \in \mathbb{Z}/\kappa_1\mathbb{Z} \times \mathbb{Z}/\kappa_2\mathbb{Z}$ modulo diagonal rotation. \\ \hline

        IIIb & 
        A tuple $\underline{h} = (h_1, h_2, h_3, h_4) \in \{(1,2,2,1), (1,1,2,2), (2,1,1,2), (2,2,1,1)\}$ indicating the zero $z_{h_2}$ and the pole $q_{h_4}$ on the bottom level, while $z_{h_1}$ and $q_{h_3}$ lie on the top level. \\ \cline{2-2}
        & An integer $\ell$, where $0 \leq \ell \leq n-2$, representing the number of residueless poles on the bottom level. \\ \cline{2-2}
        & A permutation $\tau \in \mathrm{Sym}_{n-2}$ indexing the residueless poles. \\ \cline{2-2}
        & An integer $C$ with $1 \leq C \leq e_{h_2} - 1$. \\ \cline{2-2}
        & A tuple ${\bf C} = (C_1, \dots, C_{n-2})$ with $1 \leq C_i \leq b_i - 1$. \\ \hline

        IIIc & 
        A tuple $\underline{h} = (h_1, h_2) \in \{(1,2), (2,1)\}$ indicating the zero $z_{h_1}$ on the bottom level and $z_{h_2}$ on the top level. \\ \cline{2-2}
        & Integers $\ell_1, \ell_2$ with $0 \leq \ell_1 < \ell_2 \leq n-2$, where $\ell_1$ (resp. $n-2 - \ell_2$) denotes the number of residueless poles between $q_1$ and $s^\bot$ (resp. $q_2$ and $s^\bot$) on $Z_{-1}$. \\ \cline{2-2}
        & A permutation $\tau \in \mathrm{Sym}_{n-2}$ indexing the residueless poles. \\ \cline{2-2}
        & An integer $C$ with $1 \leq C \leq \kappa$. \\ \cline{2-2}
        & A tuple ${\bf C} = (C_1, \dots, C_{n-2})$ with $1 \leq C_i \leq b_i - 1$. \\ \hline
    \end{tabular}
    \caption{Combinatorial data of Type III boundaries for strata with B-signatures}
    \label{tab:TypeBIIIabc_comb}
\end{table}

\begin{table}[h!]
    \centering
    \begin{tabular}{|c|c|}
    \hline
      Type   & Level graph/ Separatrix diagram \\ \hline
       IIIa  & \centered{  
\resizebox{14cm}{6.6cm}{\includegraphics[]{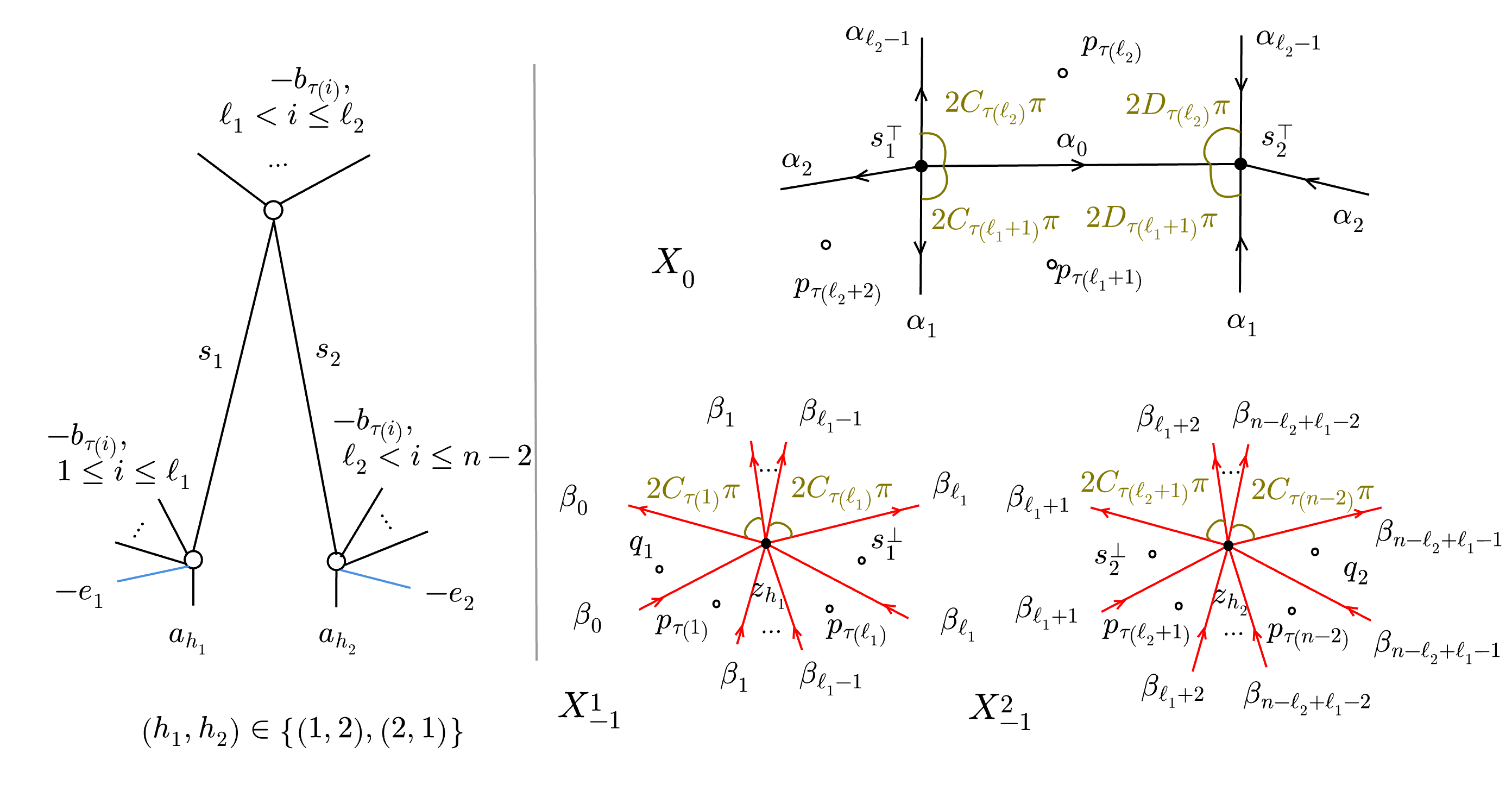}}
 }  \\ \hline

 IIIb & \centered{  
\resizebox{14cm}{6.6cm}{\includegraphics[]{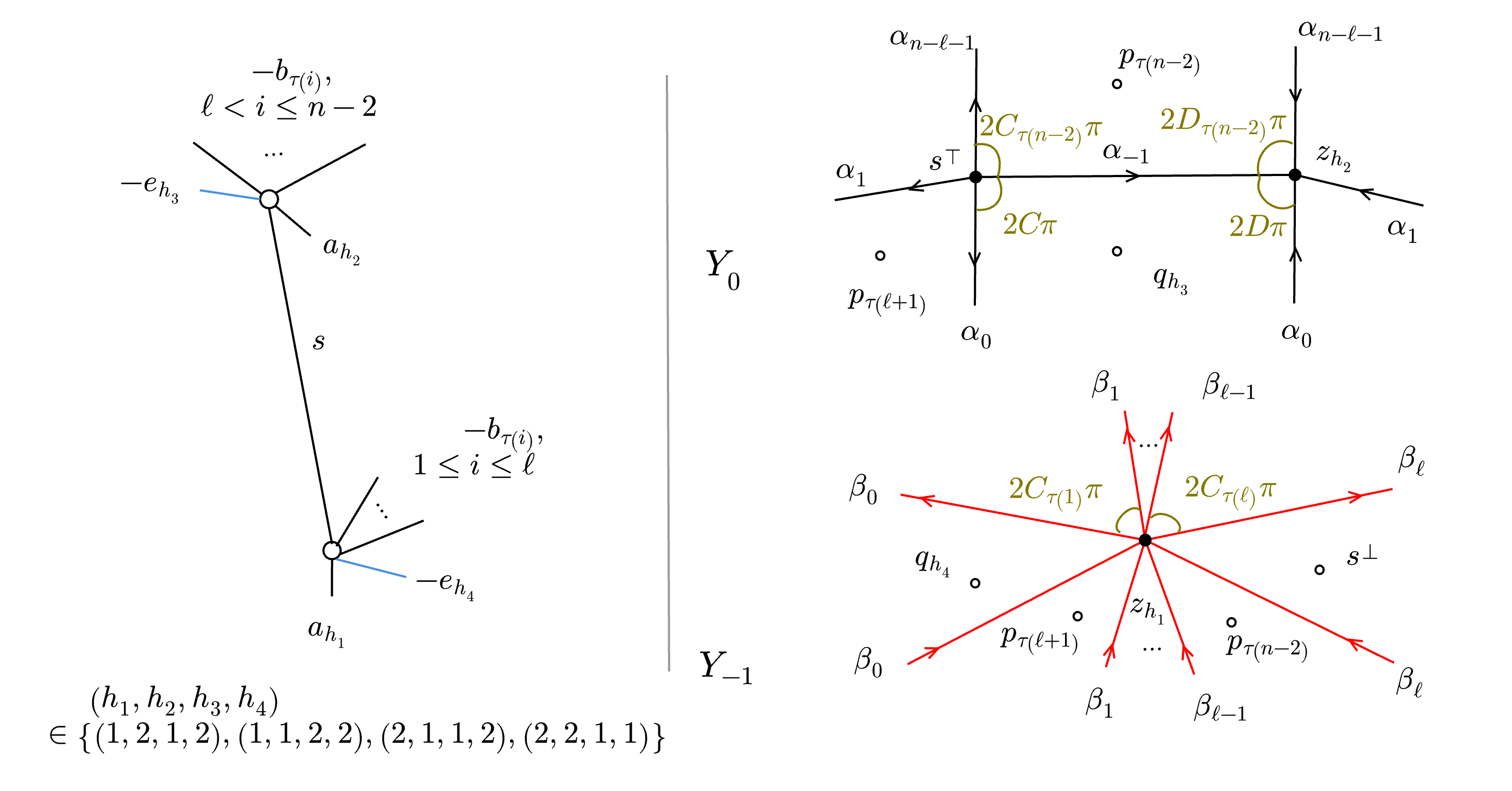}}
 }  \\ \hline

 IIIc & \centered{  
\resizebox{14cm}{6.6cm}{\includegraphics[]{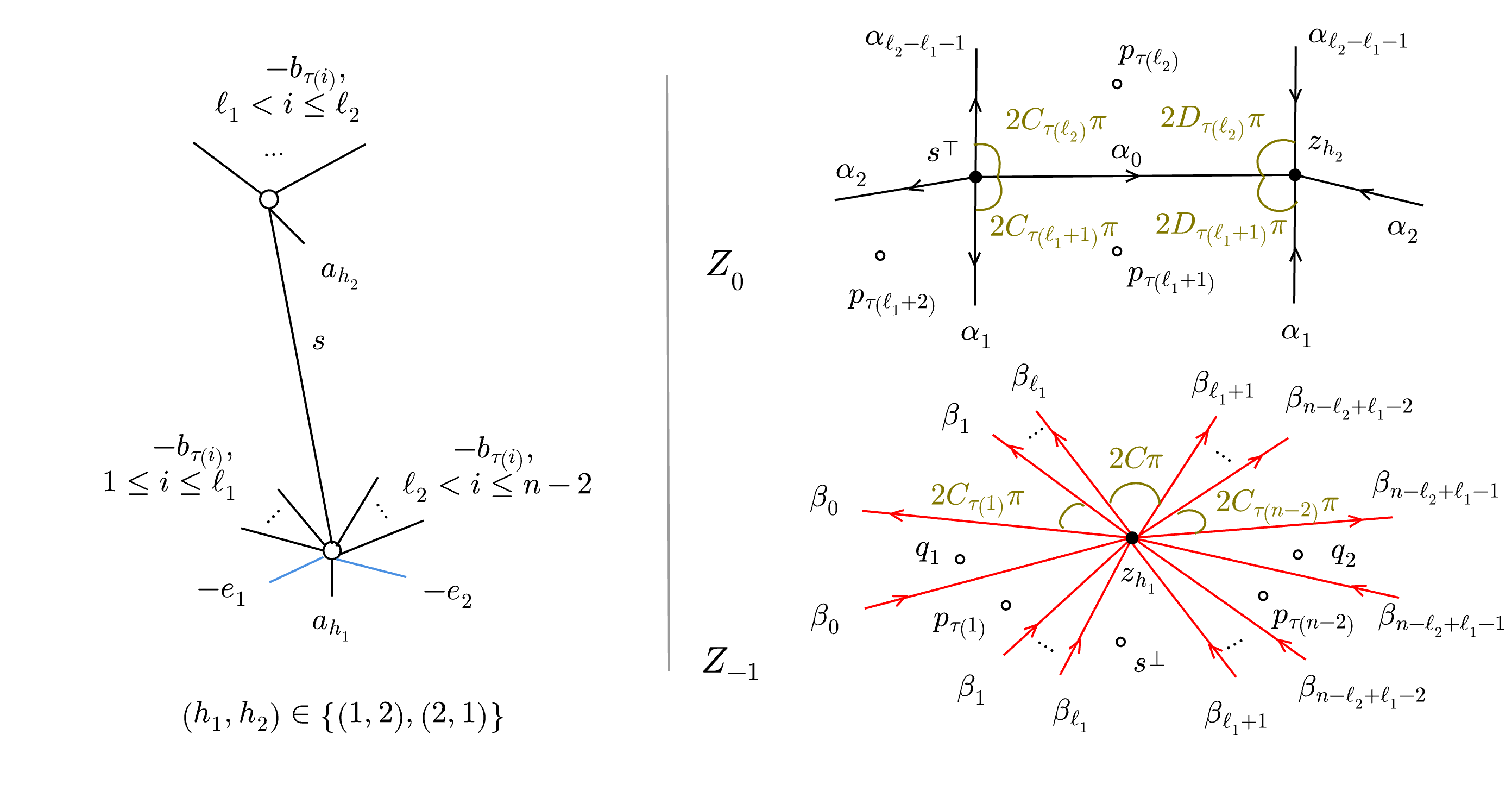}}
 }  \\ \hline
       
    \end{tabular}
    \caption{Type III Principal boundaries of Strata of B-signatures}
    \label{tab:TypeB_IIIabc}
\end{table}

For a Type IIIa boundary, let $\tau_1=\tau|_{\{1,\dots, \ell_1\}}$, $\tau_2(i)=\tau(i+\ell_1)$ for $i=1,\dots, \ell_2-\ell_1$ and $\tau_3(i)=\tau(i+\ell_2)$ for $i=1,\dots,n-2-\ell_2$. Also, let ${\bf C_1}=(C_{\tau(i)})_{i=1,\dots, \ell_1}$, ${\bf C_2}=(C_{\tau(i)})_{i=\ell_1+1,\dots, \ell_2}$ and ${\bf C_3}=(C_{\tau(i)})_{i=\ell_2+1,\dots, n-2}$. The top level component is isomorphic to $Z_1(\tau_2,{\bf C_2})$. There are two bottom level components, isomorphic to $Z_2(\tau_1,{\bf C_1},e_1,\kappa_1+1)$ and $Z_2(\tau_3,{\bf C_3},\kappa_2+1,e_2)$, where $\kappa_1=a_{h_1}-e_1+1-\sum_{i=1}^{\ell_1} b_{\tau(i)}$ and $\kappa_2=a_{h_2}-e_2+1-\sum_{i=\ell_2+1}^{n-2}$. We denote this boundary by 
\[
X^B_{\mathrm{IIIa}} (\underline{h}, \ell_1, \ell_2, \tau, \mathbf{C}, \mathrm{Pr}).
\]

For a Type IIIb boundary, the top level component contains $n-1-\ell$ residueless poles, one of which is $q_{h_3}$. Let $\tau_1(i)=\tau(i+\ell)$ for $i=1,\dots, n-2-\ell$ and $\tau_1(n-1-\ell)=q_{h_3}$. Also, let ${\bf C_1}=(C_{\tau(1)},\dots, C_{\tau(\ell)},C)$. Then the top level component is isomorphic to $Z_1(\tau_1,{\bf C_1})$. Let $\tau_2(i)=\tau(i)$ for $i=1,\dots, \ell$ and ${\bf C_2}=(C_{\tau(i)})_{i=1,\dots,\ell}$. Then the bottom level component is isomorphic to $Z_2(\tau_2,{\bf C_3},e_{h_4},\kappa+1)$ where $\kappa=a_{h_1}-e_{h_3}+1-\sum_{i=1}^{\ell} b_{\tau(i)}$. We denote this boundary by 
\[
X^B_{\mathrm{IIIb}} (\underline{h}, \ell, \tau, C, \mathbf{C})
\]

For a Type IIIc boundary, let $\tau_1(i)=\tau(i+\ell_1)$ for $i=1,\dots, \ell_2-\ell_1$ for $i=1,\dots,\ell_1$. Also, let ${\bf C_1}=(C_{\tau(\ell_1)},\dots, C_{\tau(\ell_2)})$. Then the top level component is isomorphic to $Z_1(\tau_1,{\bf C_1})$. The bottom level component contains $n-1-(\ell_2-\ell_1)$ residueless poles, one of which is the pole of order $\kappa+1\coloneqq \sum_{i=\ell-1+1}^{\ell_2}b_{\tau(i)} -a_{h_2}$ at the node $s^\bot$. Let $\tau_2(i)=\tau(i)$ for $i=1,\dots, \ell_1$, $\tau_2(i)=\tau(i+\ell_2-\ell_1-1)$ for $i=\ell_1+2,\dots, n-1-(\ell_2-\ell_1)$ and $\tau_2(\ell_1+1)=s^\bot$. Also let ${\bf C_2}=(C_{\tau(1)},\dots, C_{\tau(\ell_1)},C,C_{\tau(\ell_2+1)},\dots,C_{\tau(n-2)})$. Then the bottom level component is isomorphic to $Z_2(\tau_2,{\bf C_2},e_1,e_2)$. We denote this boundary by 
\[
X^B_{\mathrm{IIIc}} (\underline{h}, \ell_1, \ell_2, \tau, C, \mathbf{C}).
\]
We present the corresponding level graphs and separatrix diagrams for the principal boundaries in \Cref{tab:TypeB_IIIabc}. The choices of $h_1, h_2, h_3, h_4$ already incorporate all relevant permutations. We use the letter $s$ to denote the node, and the symbols $\top$ and $\bot$ to label the associated nodal poles and zeros. The enhancements of the level graphs are denoted by $\kappa$.

Note that for each variant of Type III, the conical angles around the zeros impose numerical constraints on the combinatorial data:
\begin{itemize}
    \item[(IIIa)] $a_{h_2} + 1 = e_2 + \sum_{i=\ell_2+1}^{n-2} b_{\tau(i)} + \sum_{i=\ell_1+1}^{\ell_2} D_{\tau(i)}$;
    \item[(IIIb)] $a_{h_2} + 1 = D + \sum_{i=\ell+1}^{n-2} D_{\tau(i)}$;
    \item[(IIIc)] $a_{h_2} + 1 = \sum_{i=\ell_1+1}^{\ell_2} D_{\tau(i)}$.
\end{itemize} 
On the top level of a Type IIIa boundary, we label the outgoing prongs at $s_1^\top$ by $\mathbb{Z}/\kappa_1\mathbb{Z}$ clockwise, and the incoming prongs at $s_2^\top$ by $\mathbb{Z}/\kappa_2\mathbb{Z}$ counterclockwise. The prongs $v^+_{c_0}, v^+_{c_1}, \dots, v^+_{c_{\ell_2-\ell_1-1}}$ at $s_1^\top$ and $w^+_{c_0}, w^+_{c_1}, \dots, w^+_{c_{\ell_2-\ell_1-1}}$ at $s_2^\top$ overlap with the saddle connections, where

\[
    c_i = \sum_{j=1}^i C_{\tau(\ell_1 + j)}, \quad
    d_i = \sum_{j=1}^i D_{\tau(\ell_1 + j)}.
\]

On the bottom level, the incoming (respectively, outgoing) prongs at $s_1^\bot$ (resp. $s_2^\bot$) are represented by outgoing (resp. incoming) prongs at $z_{h_1}$ (resp. $z_{h_2}$) pointing toward the nodal polar domains. The outgoing prongs at $z_{h_1}$ are labeled $v^-_1, \dots, v^-_{\kappa_1}$ clockwise, while the incoming prongs at $z_{h_2}$ are labeled $w^-_1, \dots, w^-_{\kappa_2}$ counterclockwise. The prong $v^-_1$ (resp. $w^-_{\kappa_2}$) is the outgoing (resp. incoming) prong closest to the saddle connection $\beta_{\ell_1}$ (resp. $\beta_{\ell_1+1}$) at $z_{h_1}$ (resp. $z_{h_2}$) in the clockwise direction. A prong-matching $\boldsymbol{\sigma}$ is an orientation-preserving map sending outgoing prongs at $z_{h_1}$ to those at $s_1^\top$, and incoming prongs at $z_{h_2}$ to those at $s_2^\top$. It is uniquely determined by the image $v^+_u$ of $v^-_{\kappa_1}$ and the image $w^+_v$ of $w^-_1$, so we identify $\boldsymbol{\sigma} = (u, v) \in \mathbb{Z}/\kappa_1\mathbb{Z} \times \mathbb{Z}/\kappa_2\mathbb{Z}$.

On the top level of a Type IIIb boundary, we label the outgoing prongs at $s^\top$ clockwise by $\mathbb{Z}/\kappa\mathbb{Z}$. The prongs $v^+_{c_{-1}}, v^+_{c_0}, \dots, v^+_{c_{n-\ell-3}}$ at $s^\top$ overlap with the saddle connections, where $c_{-1} = 0$ and 
\[
c_i = C + \sum_{j=1}^i C_{\tau(\ell + j)} \quad \text{for } i \geq 0.
\]

On the bottom level, the incoming prongs at the nodal poles $s^\bot$ are represented by outgoing prongs at $z_{h_1}$ toward the nodal polar domain, labeled clockwise by $v^-_1, \dots, v^-_{\kappa}$. The prong $v^-_1$ is the outgoing prong at $z_{h_1}$ closest to the saddle connection $\beta_0$ in the clockwise direction. A prong-matching $\boldsymbol{\sigma}$ is a cyclic order-preserving map identifying these prongs with those at $s^\top$, uniquely determined by the image $v^+_u$ of $v^-_{\kappa}$. Thus, we identify $\boldsymbol{\sigma} = u \in \mathbb{Z}/\kappa\mathbb{Z}$.

On the top level of a Type IIIc boundary, the labeling of prongs at $s^\top$ is the same as in the Type IIIb case. The outgoing prongs $v^+_{c_0}, \dots, v^+_{c_{\ell_2 - \ell_1 - 1}}$ at $s^\top$ overlap with the saddle connections, where
\[
c_i = \sum_{j=1}^i C_{\tau(\ell_1 + j)}.
\]

On the bottom level, the incoming prongs at the nodal poles $s^\bot$ are represented by outgoing prongs at $z_{h_1}$ pointing toward the nodal polar domain, labeled clockwise by $v^-_1, \dots, v^-_{\kappa}$. The prong $v^-_1$ is the one immediately following $\beta_{\ell_1}$ clockwise. Depending on the rescaling of the bottom level, there are $C - 1 + \varepsilon$ outgoing prongs at $z_{h_1}$ between $\beta_{\ell_1}$ and $\beta_{\ell_1+1}$ in the clockwise direction, where
\[
\varepsilon = 
\begin{cases}
    1 & \text{if } \arg(\alpha_0 / \beta_0) \in (0, \pi], \\
    0 & \text{if } \arg(\alpha_0 / \beta_0) \in (\pi, 2\pi].
\end{cases}
\] A prong-matching $\boldsymbol{\sigma}$ is uniquely determined by the image $v^+_u$ of $v^-_{\kappa}$, and hence we identify $\boldsymbol{\sigma} = u \in \mathbb{Z}/\kappa\mathbb{Z}$.

\subsection{Plumbing construction and equatorial half-arcs}

In this subsection, we describe the plumbing constructions for different types of principal boundaries such that all saddle connections are parallel. Specifically, an equatorial half-arc can be characterized by specifying the associated plumbing configuration.

Let $\overline{W} = X^B_{\mathrm{I}}(\ell, \tau, \mathbf{C})$ be a Type~I boundary. By rescaling the bottom level so that 
\[
\arg\left(\frac{\alpha_0}{\beta_0}\right) = \pi \quad \text{or} \quad 2\pi,
\]
we obtain plumbed surfaces $\overline{W}_t\left(u + \frac{\varepsilon}{2}\right)$ where all saddle connections are parallel. Here:
\begin{itemize}
    \item $t \in \mathbb{R}_+$ is the plumbing parameter;
    \item $u \in \mathbb{Z}/\kappa\mathbb{Z}$ determines the prong-matching;
    \item $\varepsilon = 1$ if $\arg\left(\frac{\alpha_0}{\beta_0}\right) = \pi$, and $\varepsilon = 0$ if $\arg\left(\frac{\alpha_0}{\beta_0}\right) = 2\pi$.
\end{itemize} The separatrix diagrams of the plumbed surfaces $\overline{W}_t\left(u + \frac{\varepsilon}{2}\right)$ with $\varepsilon = 0$ are shown in \Cref{tab:TypeBIPlumb}. In these diagrams, the blue saddle connections indicate those modified by the plumbing deformation $+t$.

Note that there is a natural cyclic ordering of the prongs at $s^\top$. We use the symbols $<$ and $\leq$ to denote this cyclic order. For instance, writing $0 < u_1 < u_2 \pmod{\kappa}$ means that there exist integers $x, y, z$ such that
\[
x \equiv 0 \pmod{\kappa}, \quad z \equiv u_1 \pmod{\kappa}, \quad y \equiv u_2 \pmod{\kappa}, \quad \text{and} \quad x < y < z \leq x + \kappa.
\]

We denote by $\overline{W}\left(u + \frac{\varepsilon}{2}\right)$ the equatorial half-arc of the plumbed surface $\overline{W}_t\left(u + \frac{\varepsilon}{2}\right)$ emanating from the boundary $\overline{W}$. With this notation, the rotation operation $R$ on equatorial half-arcs satisfies
\[
R^k \cdot \overline{W}\left(u + \frac{\varepsilon}{2}\right) = \overline{W}\left(u + \frac{\varepsilon}{2} + \frac{k}{2}\right).
\]

\begin{table}[h!]
    \centering
    \begin{tabular}{|c|c|c|}
    \hline
 \centered{$ -d_{j}\leq u < -d_{j-1}\pmod{\kappa}$ \\$c_{i-1}< u+C\leq c_i\pmod{\kappa}$} &

  \centered{$c_{i-1}\leq u < c_{i}\pmod{\kappa}$\\ $ c_{j-1}< u+C\leq c_{j}\pmod{\kappa}$} &

  \centered{$c_{\ell}\leq u < c_{\ell}+e_2\pmod{\kappa}$ \\ $ c_{i-1}< u+C\leq c_i\pmod{\kappa}$}

 \\ \hline

\centered{   
\resizebox{4.5cm}{2.9cm}{\includegraphics[]{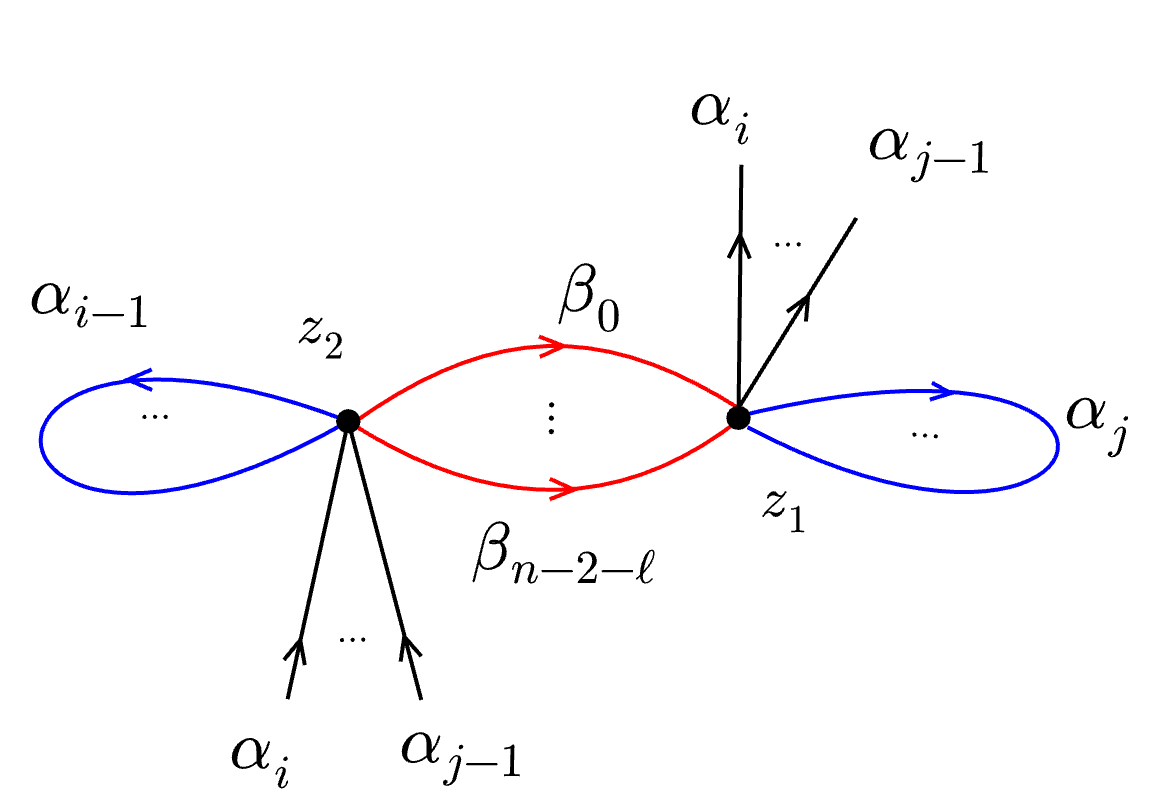}}}  & 

\centered{ 
\resizebox{4.5cm}{3.3cm}{\includegraphics[]{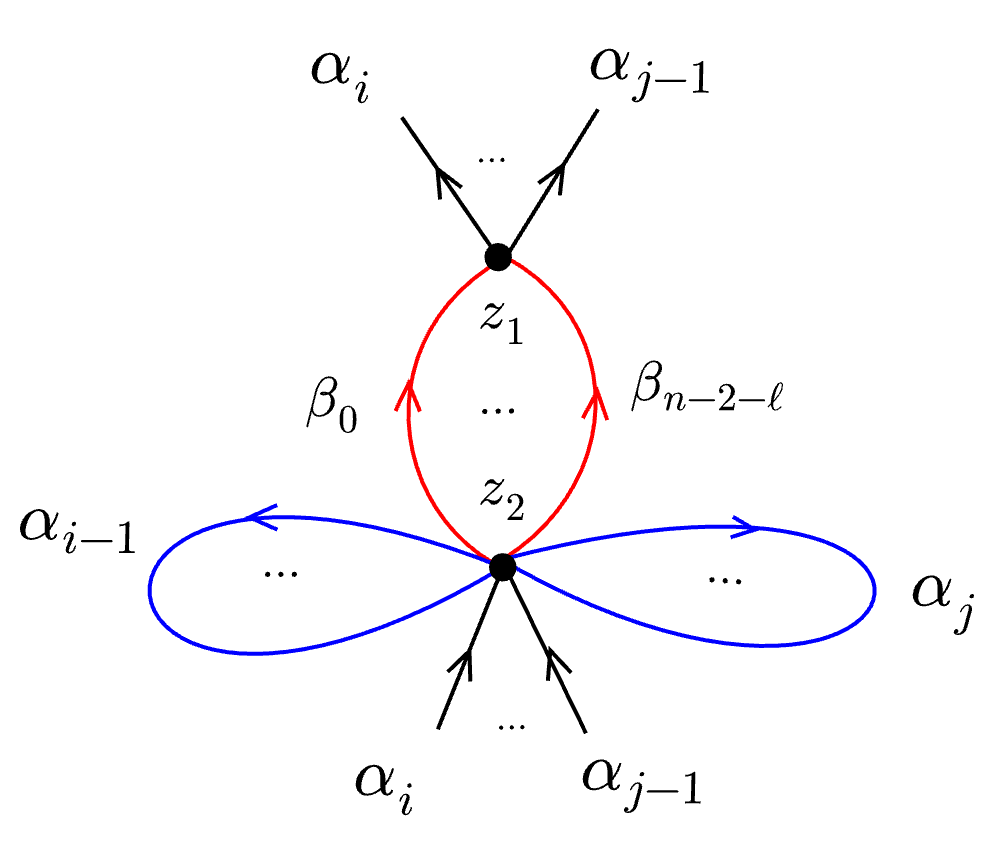} } 
}  &

\centered{ 
\resizebox{4.5cm}{3.3cm}{\includegraphics[]{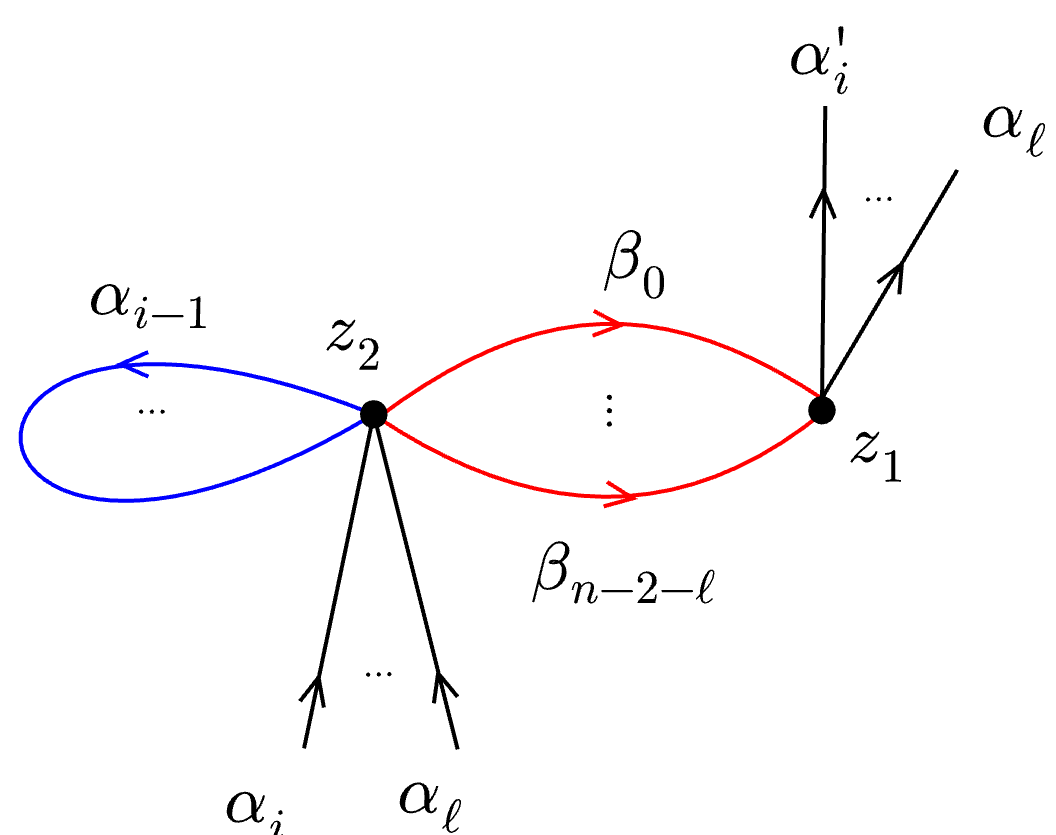} } 
}

\\ \hline
    \centered{$-d_i\leq u < -d_{i-1}\pmod{\kappa}$\\ $c_{i-1}<u+C\leq c_i\pmod{\kappa}$}
 &
  $ c_\ell<u, u+C \leq c_\ell+e_1\pmod{\kappa}$ 
  &
  $c_{i-1}<u,u+C  \leq c_{i} \pmod{\kappa}$
  \\ \hline

\centered{  
\resizebox{4.5cm}{2.9cm}{\includegraphics[]{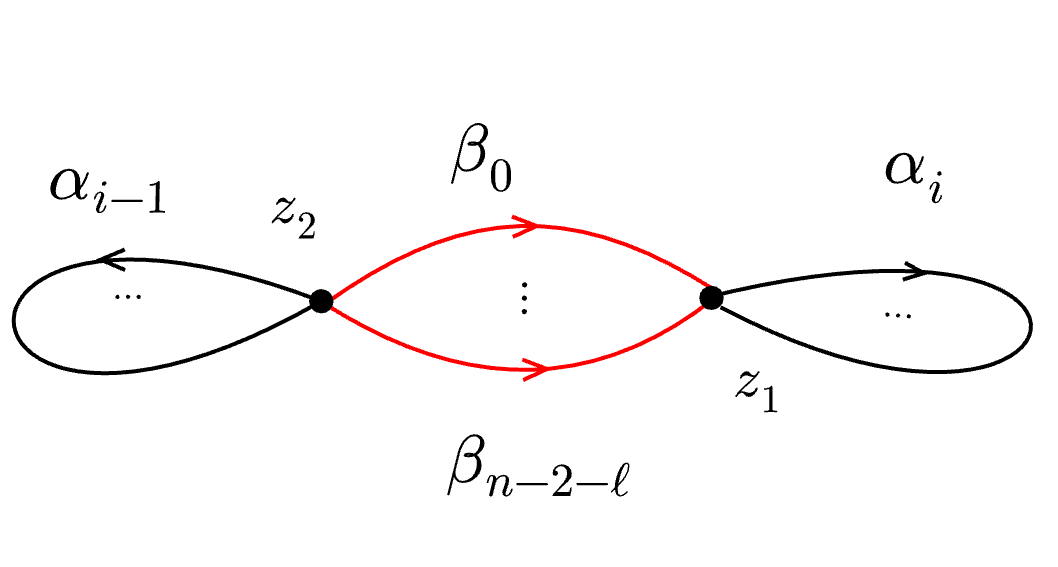}}
 } &

\centered{
\resizebox{4.5cm}{3cm}{\includegraphics[]{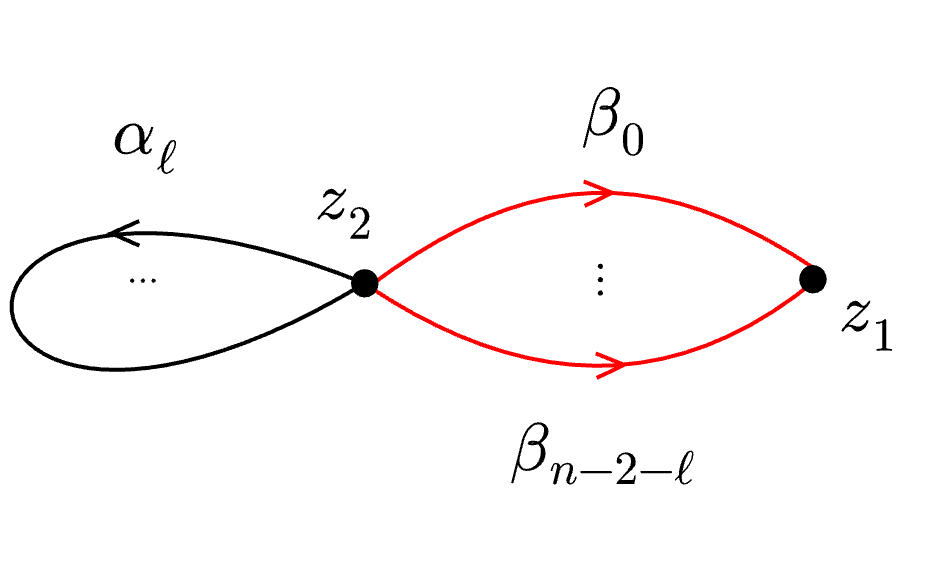} } }&

\centered{
\resizebox{4.5cm}{3.4cm}{\includegraphics[]{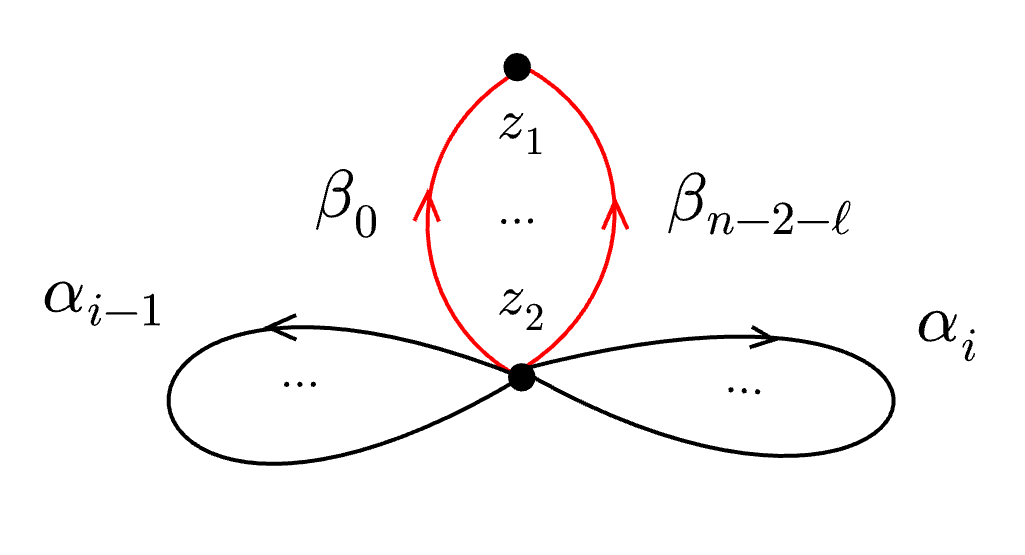} } }
 \\ \hline
 \centered{$c_\ell\leq u < c_{\ell}+e_2 \pmod{\kappa}$ \\ $ -e_1< u+C\leq 0 \pmod{\kappa}$ }
    
  &
   
  &
  \\ \hline

\centered{  
\resizebox{4.3cm}{3.4cm}{\includegraphics[]{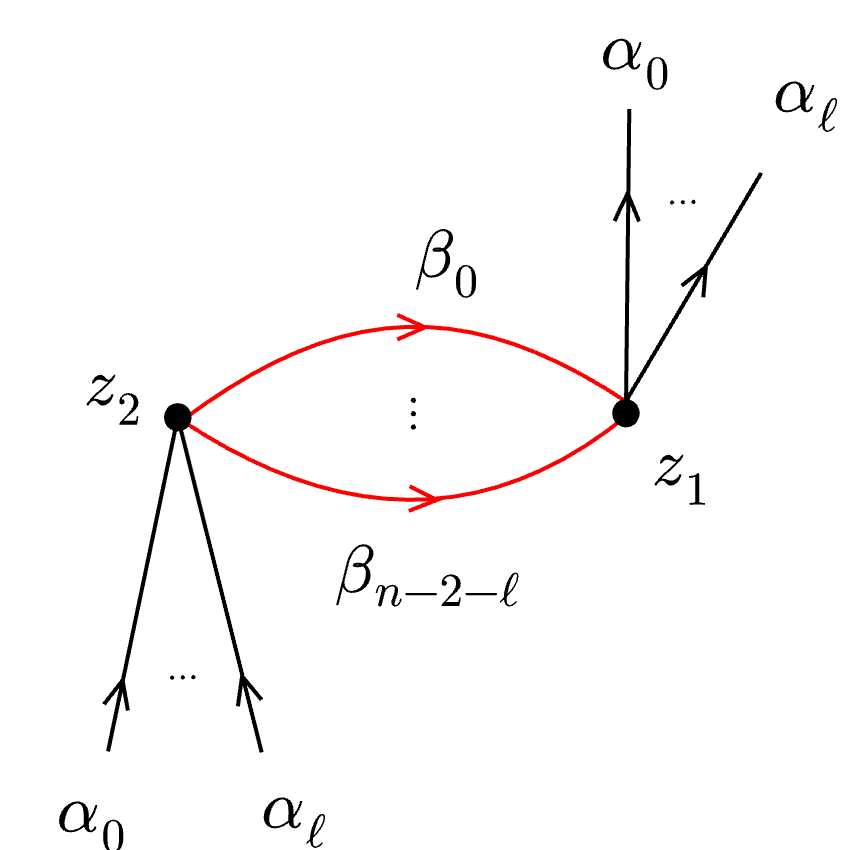}}
 } &

&

 \\ \hline
    \end{tabular}
    \caption{Flat pictures of plumbed surfaces $\overline{W}_t(u)$ where $u$ is integer}
    \label{tab:TypeBIPlumb}
\end{table}

Next, we describe the plumbed surfaces corresponding to the Type~III boundaries. For simplicity, we denote the Type~III multi-scale differentials by
\[
\overline{X} = X^D_{\mathrm{IIIa}}(\underline{h}, \ell_1, \ell_2, \tau, \mathbf{C}, \mathrm{Pr}), \quad 
\overline{Y} = X^D_{\mathrm{IIIb}}(\underline{h}, \ell_1, \ell_2, \tau, C, \mathbf{C}), \quad \text{and} \quad
\overline{Z} = X^D_{\mathrm{IIIc}}(\underline{h}, \ell_1, \ell_2, \tau, C, \mathbf{C}).
\]
For each $t \in \mathbb{R}_+$, and by fixing the prong-matching, we obtain families of plumbed surfaces with all saddle connections parallel. As in the Type~I case, we introduce notation for equatorial half-arcs. Again, we define $\varepsilon = 1$ if $\arg\left(\frac{\alpha_0}{\beta_0}\right) = \pi$, and $\varepsilon = 0$ if $\arg\left(\frac{\alpha_0}{\beta_0}\right) = 2\pi$.

\begin{itemize}
    \item[(IIIa)] The prong-matching of a Type~IIIa boundary $\overline{X}$ is uniquely determined by a tuple $(u, v) \in \mathbb{Z}/\kappa_1\mathbb{Z} \times \mathbb{Z}/\kappa_2\mathbb{Z}$. The corresponding family of plumbed surfaces with parallel saddle connections is denoted by 
    \[
    \overline{X}_t\left(u + \frac{\varepsilon}{2}, v - \frac{\varepsilon}{2}\right).
    \]
    The separatrix diagrams for $\varepsilon = 0$ are shown in \Cref{tab:TypeBIIIaPlumb}. We denote by $\overline{X}\left(u + \frac{\varepsilon}{2}, v - \frac{\varepsilon}{2}\right)$ the associated equatorial half-arc starting from $\overline{X}$.
    
    \item[(IIIb)] The prong-matching of a Type~IIIb boundary $\overline{Y}$ is uniquely determined by a single value $u \in \mathbb{Z}/\kappa\mathbb{Z}$. The corresponding family of plumbed surfaces is denoted by 
    \[
    \overline{Y}_t\left(u + \frac{\varepsilon}{2}\right).
    \]
    The separatrix diagrams for $\varepsilon = 0$ are shown in \Cref{tab:TypeBIIIbPlumb}. The associated equatorial half-arcs are denoted by $\overline{Y}\left(u + \frac{\varepsilon}{2}\right)$.
    
    \item[(IIIc)] Similarly, the prong-matching of a Type~IIIc boundary $\overline{Z}$ is determined by a value $u \in \mathbb{Z}/\kappa\mathbb{Z}$. The plumbed surfaces are denoted by 
    \[
    \overline{Z}_t\left(u + \frac{\varepsilon}{2}\right),
    \]
    and the equatorial half-arcs by $\overline{Z}\left(u + \frac{\varepsilon}{2}\right)$. The separatrix diagrams for $\varepsilon = 0$ are shown in \Cref{tab:TypeBIIIcPlumb}.
\end{itemize}

\begin{table}[h!]
    \centering
    \begin{tabular}{|c|c|}
    \hline
    \centered{$ c_{i-1}\leq u < c_{i} \pmod{\kappa_1}$\\  $d_{j-1}< v\leq d_j \pmod{\kappa_2}$\\
  $i\neq j $}
  &
    \centered{$ c_{i-1}\leq u < c_{i} \pmod{\kappa_1}$\\  $d_{i-1}< v\leq d_i \pmod{\kappa_2}$}

 \\ \hline

\centered{   
\resizebox{7cm}{4.5cm}{\includegraphics[]{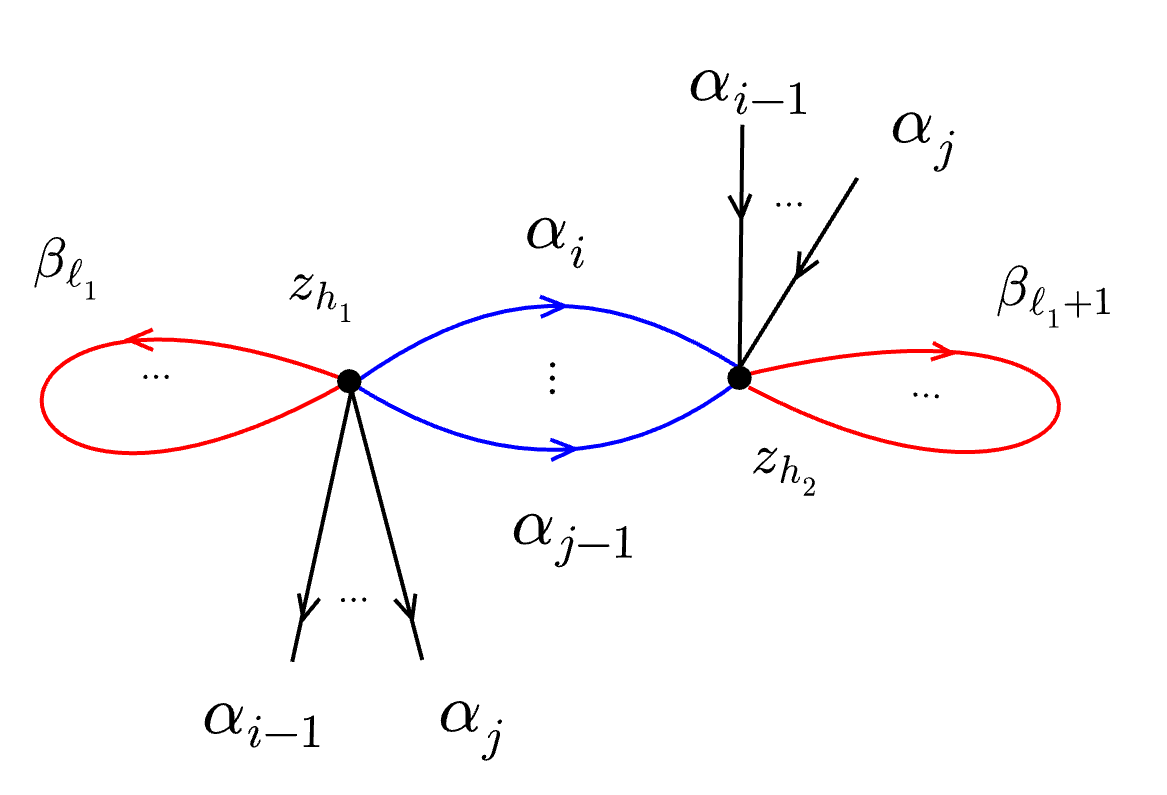}}}  & 

\centered{ 
\resizebox{7cm}{3.5cm}{\includegraphics[]{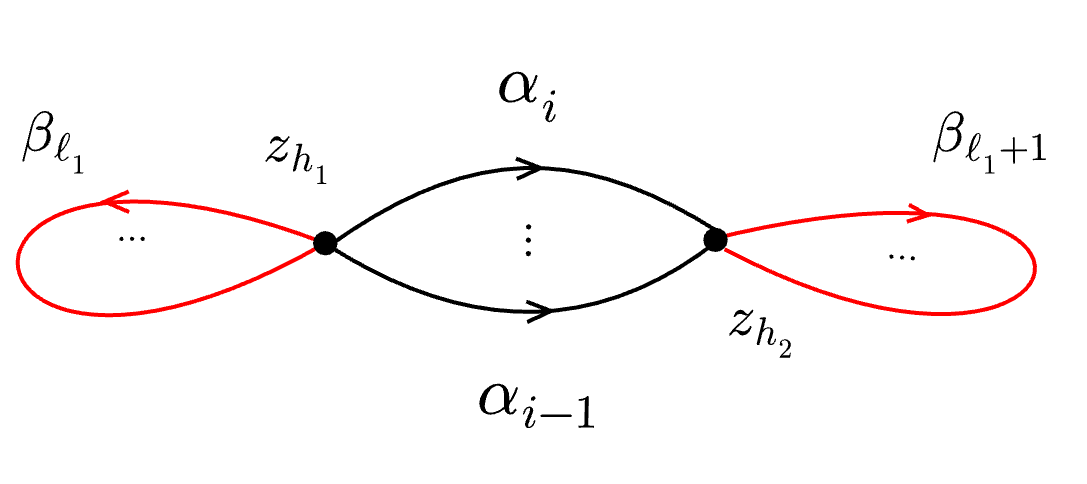} } 
} 

\\ \hline

    \end{tabular}
    \caption{Flat pictures of plumbed surfaces $\overline{X}_t(u,v)$ with $u,v$ integers}
    \label{tab:TypeBIIIaPlumb}
\end{table}

\begin{table}[h!]
    \centering
    \begin{tabular}{|c|c|}
    \hline
 $ c_{i-1}\leq u < c_{i}\pmod{\kappa}, \quad  i>0$ &

  $0\leq u < c_{0}=D \pmod{\kappa}$

 \\ \hline

\centered{   
\resizebox{7cm}{4.5cm}{\includegraphics[]{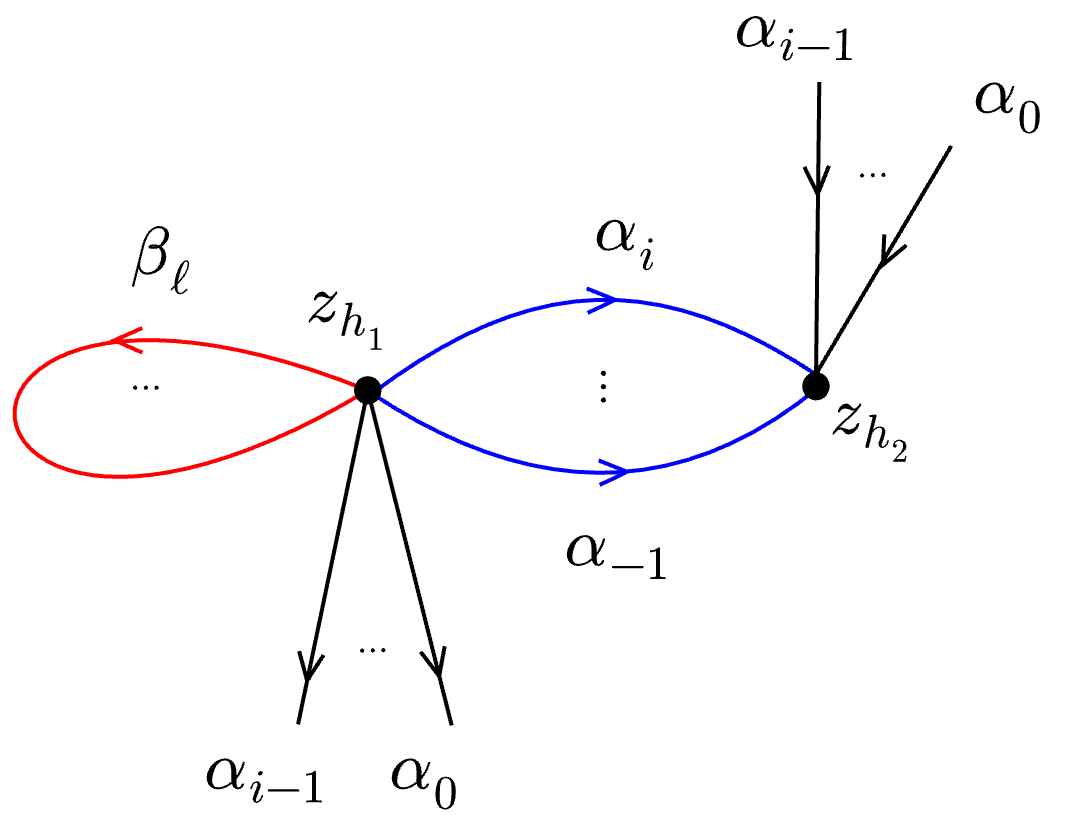}}}  & 

\centered{ 
\resizebox{7cm}{3.5cm}{\includegraphics[]{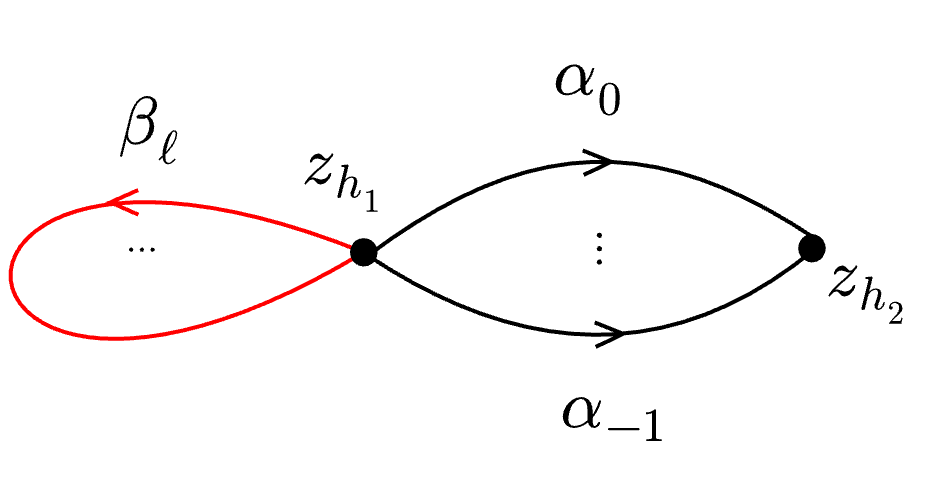} } 
} 

\\ \hline

    \end{tabular}
    \caption{Flat pictures of plumbed surfaces $\overline{Y}_t(u)$ with $u$ an integer}
    \label{tab:TypeBIIIbPlumb}
\end{table}

\begin{table}[h!]
    \centering
    \begin{tabular}{|c|c|}
    \hline
    \centered{$ c_{i-1}\leq u < c_{i}\pmod{\kappa}$\\ $  c_{j-1}< u+C\leq c_{j} \pmod{\kappa}$\\ $ i\neq j$ }
 &
    \centered{$c_{i-1}\leq u < c_{i}\pmod{\kappa} $\\ $c_{i-1}<u+C\leq c_{i} \pmod{\kappa}$}

 \\ \hline

\centered{   
\resizebox{7cm}{4.5cm}{\includegraphics[]{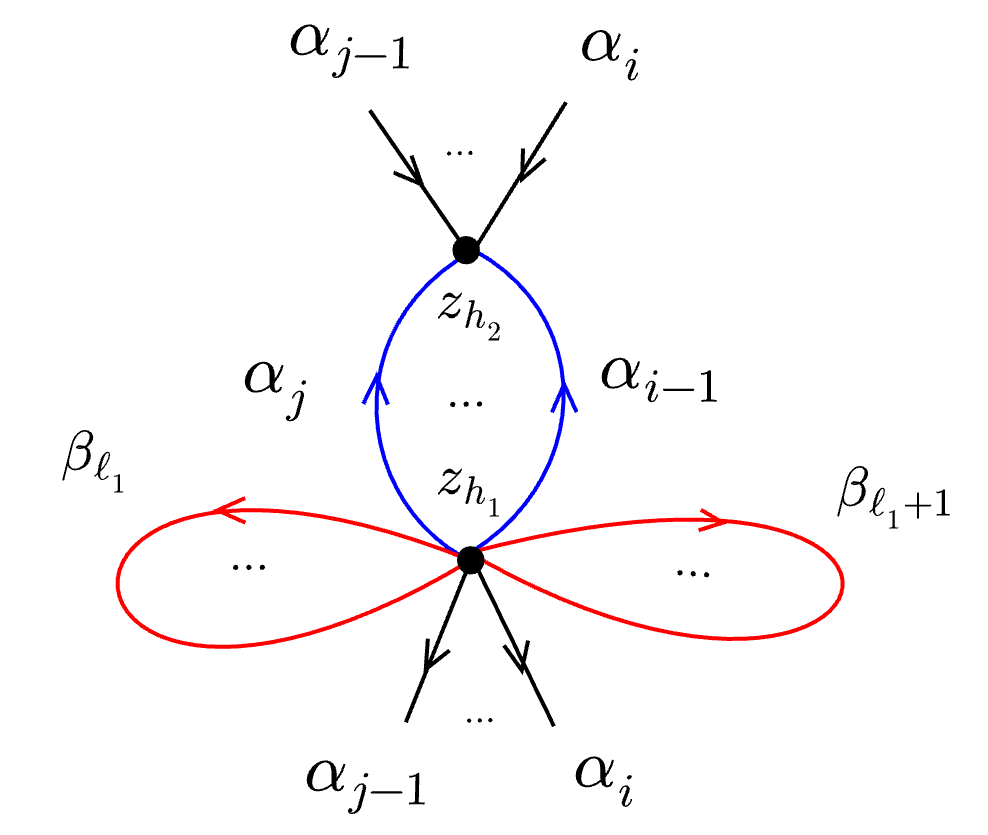}}}  & 

\centered{ 
\resizebox{7cm}{3.5cm}{\includegraphics[]{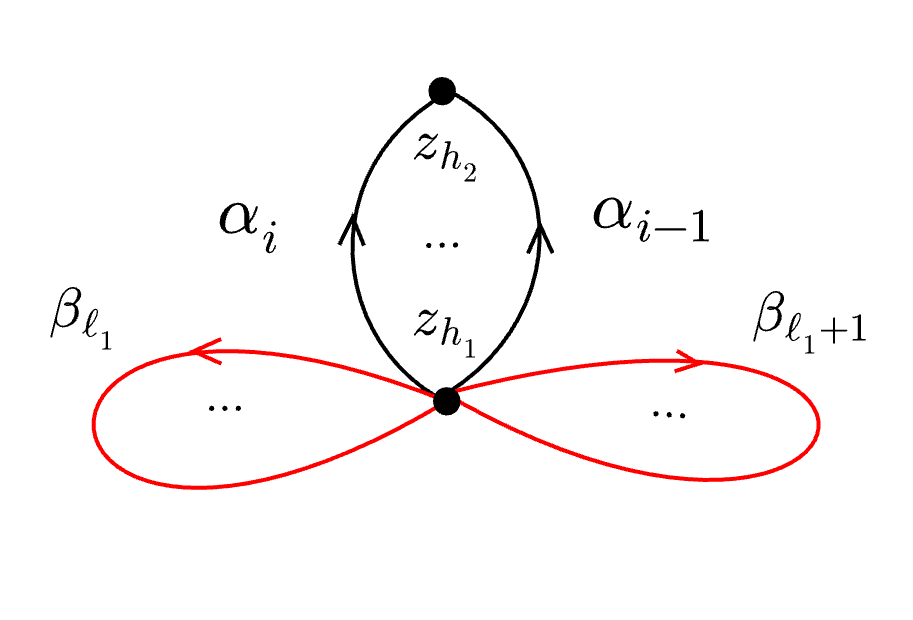} } 
} 

\\ \hline

    \end{tabular}
    \caption{Flat pictures of plumbed surfaces $\overline{Z}_t(u)$ with $u$ an integer}
    \label{tab:TypeBIIIcPlumb}
\end{table}

The rotation operator $R$ on equatorial half-arcs originating from a Type~III boundary acts as follows:
\begin{align*}
    R^{k} \cdot \overline{X}\left(u + \frac{\varepsilon}{2}, v - \frac{\varepsilon}{2}\right) 
        &= \overline{X}\left(u + \frac{\varepsilon}{2} + \frac{k}{2}, \, v - \frac{\varepsilon}{2} - \frac{k}{2}\right), \\
    R^{k} \cdot \overline{Y}\left(u + \frac{\varepsilon}{2}\right) 
        &= \overline{Y}\left(u + \frac{\varepsilon}{2} + \frac{k}{2}\right), \\
    R^{k} \cdot \overline{Z}\left(u + \frac{\varepsilon}{2}\right) 
        &= \overline{Z}\left(u + \frac{\varepsilon}{2} + \frac{k}{2}\right).
\end{align*}

\subsection{Transformation $U$ on equatorial half-arcs}

On every flat surface in a stratum of B-signature, there exists a Type I saddle connection. Consequently, every equatorial half-arc originating from a Type III boundary terminates at a Type I boundary. In this subsection, we provide explicit descriptions of the terminal Type I boundaries for equatorial half-arcs emanating from Type III boundaries. The following propositions describe the correspondence between the combinatorial data at the endpoints of such arcs. These results follow from direct observations.

\begin{proposition} \label{Prop:BIIIamove}
Let $\overline{X} = X^B_{\III a} \big((2,1), \ell_1, \ell_2, \Id, {\bf C}, [(u,v)]\big)$ be a Type IIIa boundary. Assume that $0 \leq u < c_1 \pmod{\kappa_1}$ and $d_{j-1} \leq v < d_j \pmod{\kappa_2}$ for some $j \neq 1$. Then
\[
U \cdot \overline{X}(u,v) = \overline{W}(u'),
\]
where $\overline{W} = X^B_{\I}(\ell, \tau, {\bf C}')$ is a Type I boundary, with the following data:
\begin{itemize}
    \item An integer $\ell = \ell_1 + j + (n - 2 - \ell_2)$;
    \item A permutation $\tau$ defined by
    \[
    \tau(i) = \begin{cases}
        i & \text{if } 1 \leq i \leq \ell_1 + j, \\
        i + (\ell_2 - \ell_1 - j) & \text{if } \ell_1 + j + 1 \leq i \leq \ell, \\
        i - (n - 2 - \ell_2) & \text{if } \ell + 1 \leq i \leq n - 2;
    \end{cases}
    \]
    \item A tuple ${\bf C}'$ defined by
    \[
    C'_i = \begin{cases}
        c_1 - u & \text{if } i = \ell_1 + 1, \\
        C_{\ell_1 + j} + d_j - v & \text{if } i = \ell_1 + j, \\
        C_i & \text{otherwise};
    \end{cases}
    \]
    \item An index $u' = -\left(e_1 + \sum_{k=1}^{\ell_1} D_k + u\right)$.
\end{itemize}
\end{proposition}

\begin{proposition} \label{Prop:BIIIbmove}
Let $\overline{X} = X^B_{\III b}\big((2,1,2,1), \ell, \Id, C, {\bf C}\big)$ be a Type IIIb boundary, and suppose $u \in \mathbb{Z}/\kappa\mathbb{Z}$. Then
\[
U \cdot \overline{X}(u) = \overline{W}(u'),
\]
where $\overline{W} = X^B_{\I}(\ell', \tau, {\bf C}')$ is a Type I boundary, given by the following cases:

\textbf{Case 1:} If $c_{j-1} \leq u < c_j \pmod{\kappa}$ for some $j \geq 1$:
\begin{itemize}
    \item $\ell' = n - 2 - j + 1$;
    \item $\tau$ defined by
    \[
    \tau(i) = \begin{cases}
        i & \text{if } 1 \leq i \leq \ell, \\
        i + j & \text{if } \ell + 1 \leq i \leq n - 2 - j, \\
        i - (n - 2 - (\ell + j)) & \text{if } n - 1 - j \leq i \leq n - 2;
    \end{cases}
    \]
    \item ${\bf C}'$ defined by
    \[
    C'_i = \begin{cases}
        c_j - u & \text{if } i = \ell + j, \\
        C_i & \text{otherwise};
    \end{cases}
    \]
    \item $u' = -\left(e_1 + \sum_{k=1}^{\ell} D_k + (u-c_{j-1})\right)$.
\end{itemize}

\textbf{Case 2:} If $0 \leq u < C \pmod{\kappa}$:
\begin{itemize}
    \item $\ell' = \ell$;
    \item $\tau = \Id$;
    \item ${\bf C}'$ defined by $C'_i = C_i$ for all $i$;
    \item $u' = -\left(e_1 + \sum_{k=1}^{\ell} D_k + u \right)$.
\end{itemize}
\end{proposition}

\begin{proposition} \label{Prop:BIIIcmove}
Let $\overline{X} = X^B_{\III c}\big((2,1), \ell_1, \ell_2, \Id, C, {\bf C}\big)$ be a Type IIIc boundary. Suppose that $0 \leq u < c_1 \pmod{\kappa}$ and $c_{j-1} < u + C \leq c_j \pmod{\kappa}$ for some $j \neq 1$. Then
\[
U \cdot \overline{X}(u) = \overline{W}(u'),
\]
where $\overline{W} = X^B_{\I}(\ell, \tau, {\bf C}')$ is a Type I boundary, with:
\begin{itemize}
    \item $\ell = \ell_1 + j + (n - 2 - \ell_2)$;
    \item $\tau$ defined by
    \[
    \tau(i) = \begin{cases}
        i & \text{if } 1 \leq i \leq \ell_1 + j, \\
        i + j & \text{if } \ell_1 + j < i \leq \ell, \\
        (i - \ell) + (\ell_1 + j) & \text{if } \ell + 1 \leq i \leq n - 2;
    \end{cases}
    \]
    \item ${\bf C}'$ defined by
    \[
    C'_i = \begin{cases}
        c_1 - u & \text{if } i = \ell_1 + 1, \\
        u + C - c_{j-1} & \text{if } i = \ell_1 + j, \\
        C_i & \text{otherwise};
    \end{cases}
    \]
    \item $u' = -\left(e_1 + \sum_{k=1}^{\ell_1} D_k + u\right)$.
\end{itemize}
\end{proposition}

We also consider the transformation $U$ acting on equatorial arcs that connect two Type I boundaries. In the following proposition, we give an explicit description of this transformation.

\begin{proposition} \label{Prop:BImove}
Let $\overline{W} = X^B_{\I}(\ell, \Id, {\bf C})$ be a Type I boundary. Suppose that $c_{i-1} \leq u < c_i \pmod{\kappa}$ and $-d_j < u + C \leq -d_{j-1} \pmod{\kappa}$ for some $j > i$. Then
\[
U \cdot \overline{W}\left(u + \frac{1}{2}\right) = \overline{W'}(u'),
\]
where $\overline{W'} = X^B_{\I}(\ell', \tau, {\bf C'})$, with the following data:
\begin{itemize}
    \item An integer $\ell' = n - 2 - (j - i - 1)$;
    
    \item A permutation $\tau$ defined by
    \[
    \tau(k) = \begin{cases}
        k & \text{if } 1 \leq k \leq i, \\
        n - 1 - (k - i) & \text{if } i < k \leq i + (n - 2 - \ell), \\
        k - (i + n - 2 - \ell) + (j - 1) & \text{if } i + (n - 2 - \ell) < k \leq \ell', \\
        n - 1 - (k - i) & \text{if } \ell' < k \leq n - 2;
    \end{cases}
    \]
    
    \item A tuple ${\bf C'}$ defined by
    \[
    C'_k = \begin{cases}
        u - c_{i-1} + 1 & \text{if } k = i, \\
        C_j + (-d_{j-1} - (u + C)) & \text{if } k = j, \\
        C_k & \text{otherwise};
    \end{cases}
    \]
    
    \item An index
    \[
    u' = C'_i + \sum_{k=1}^{i-1} C_k + \sum_{k=\ell+1}^{n-2} D_k + \left(-d_{j-1} - (u + C)\right).
    \]
\end{itemize}
\end{proposition}

\subsection{Hyperelliptic components}\label{sec:Bhyp}

If $\cC$ is a hyperelliptic component of $\cP(\mu^{\fR})$, then any boundary of $\overline{\cC}$ admits an involution that is a degeneration of the hyperelliptic involution. We can characterize the boundaries of hyperelliptic components using the combinatorial data introduced earlier. It is clear that boundaries of Type IIIb and Type IIIc do not admit such an involution, as their level graphs are not symmetric. 

Since the saddle connections of a flat surface generate the relative homology $H_1(X \setminus \boldsymbol{p}, \boldsymbol{z}; \mathbb{Z})$, every flat surface must contain a saddle connection joining $z_1$ and $z_2$. Therefore, the boundary of $\overline{\cC}$ consists only of Type I and Type IIIa boundaries. The following proposition characterizes the Type I boundaries that appear in the boundary of hyperelliptic components:

\begin{proposition} \label{TypeBTypeIhyper}
A Type I boundary $X^B_{\I} (\ell, \Id, {\bf C})$ lies in the boundary of a hyperelliptic component if and only if the following conditions hold:
\begin{itemize}
    \item $C_{\tau(i)} = D_{\tau(\ell + 1 - i)}$ for $i = 1, \dots, \ell$,
    \item $C_{\tau(\ell + j)} = D_{\tau(n + 1 - j)}$ for $j = 1, \dots, n - \ell$.
\end{itemize}
\end{proposition}

\begin{proof}
Note that the involution exchanges $p_{\tau(i)}$ with $p_{\tau(\ell + 1 - i)}$, and $p_{\tau(\ell + j)}$ with $p_{\tau(n + 1 - j)}$.
\end{proof}

We can use \Cref{TypeBTypeIhyper} to show that certain special strata do not contain any non-hyperelliptic components.

\begin{proof}[Special cases of \Cref{Prop:Bpair} and \Cref{Prop:Bnonhyper}]
First, assume that $b_i = 2$ for all $i$. Let $\cC$ be any connected component of $\cP(\mu^{\fR})$. Suppose that $X^B_{\I}(n, \tau, {\bf C}) \in \partial\cC$. By relabeling the poles if necessary, we may assume that $\tau = \Id$. Since $b_i = 2$ for all $i$, we must have $C_i = D_i = 1$. By \Cref{TypeBTypeIhyper}, it follows that $\cC$ is a hyperelliptic component.

Now assume that $e_1 = e_2 = 1$ and $n = 1$. Let $\cC$ be a connected component with index $\frac{b_1}{2}$. Suppose that $X^B_{\I}(1, \Id, {\bf C}) \in \partial\cC$. Again, by relabeling poles, we may assume that $\tau = \Id$. Then $\delta = b_1$ and the index is equal to $C_1 \pmod\delta $. Therefore, $C_1 = D_1 = \frac{b_1}{2}$, and by \Cref{TypeBTypeIhyper}, $\cC$ is hyperelliptic.
\end{proof}

\begin{proof}[Proof of \Cref{Prop:Bhyper}]
Let $\overline{W} \coloneqq X^B_{\I} (\ell, \tau, {\bf C}) \in \overline{\cC}$ be a hyperelliptic surface with ramification profile $\Sigma$. By relabeling the poles, we can assume that $\Sigma$ is one of the following:
\begin{itemize}
    \item For even $n$, $\Sigma = (1, n-2)(2, n-3) \dots (\frac{n-2}{2}, \frac{n}{2})$ or $(1, n-4)(2, n-5) \dots (\frac{n-4}{2}, \frac{n-2}{2})(n-3)(n-2)$. 
    \item For odd $n$, $\Sigma = (1, n-3)(2, n-4) \dots (\frac{n-3}{2}, \frac{n-1}{2})(n-2)$.
\end{itemize}
We will show that we can:
\begin{itemize}
    \item[(1)] move the fixed pole $p_{n-2}$ and all non-fixed poles to the top level;
    \item[(2)] adjust the angle assignments $C_i$ without changing the positions of the marked poles;
    \item[(3)] interchange the positions of any pair of conjugate poles; 
    \item[(4)] interchange the positions of two pairs of conjugate poles.
\end{itemize}

Before we prove the above items, we first show that we can swap the sets of residueless poles on the top and bottom levels. Due to the symmetries in the combinatorial data, by \Cref{TypeBTypeIhyper}, for the prong-matching $u \equiv -1 \pmod{\kappa}$, we have $u + C \equiv c_\ell + e_2 - 1 \pmod{\kappa}$. Then, we have $U \cdot \overline{W}(-1) = \overline{W'}(u')$, where the top-level component $W'_0$ contains the residueless poles from the bottom level $W_{-1}$ of $\overline{W}$.

We now prove the first item. Suppose we start with $\overline{W}$ such that no level contains the pole $p_{n-2}$ and all non-fixed poles. If this condition is not met, we can use the previous move to swap the residueless marked poles on two levels. Without loss of generality, let’s assume we start with a Type I multi-scale differential $\overline{W}$ such that $p_{n-2}$ is on the bottom level. If there is no fixed pole on the top level, let $i_0 = \ell/2$ and $j_0 = \ell/2 + 1$; otherwise, let $i_0 = (\ell-1)/2$ and $j_0 = (\ell+3)/2$. Consider the prong-matching $u \equiv c_{i_0} - 1 \pmod{\kappa}$. By symmetry, we have $u + C \equiv -d_{j_0-1} \pmod{\kappa}$. Then the equatorial half-arc $U \cdot \overline{W}(u + \frac{1}{2}) = \overline{W'}(u')$ for some $\overline{W'}$ such that $p_{n-2}$ and all the residueless marked poles lie on the top level.

Next, we prove the second item by showing that we can set $C_{\tau(i)}$ and $D_{\tau(\ell+1-i)}$ to 1. Consider the prong-matching $u \equiv c_i - 1 \pmod{\kappa}$. By symmetry, we have $u + C \equiv -d_{\ell-i} \pmod{\kappa}$. Then, the equatorial half-arc $UR^{2C_{\tau(i)} - 2} U \cdot \overline{W}(u + \frac{1}{2}) = \overline{W'}(u')$ results in a surface $\overline{W'}$ where $C'_{\tau(i)} = D'_{\tau(\ell+1-i)} = 1$, and $\overline{W'}$ differs from $\overline{W}$ only in the angles of the polar domains of $p_{\tau(i)}$ and $p_{\tau(\ell+1-i)}$. 

Now we prove item (3). For any non-fixed pole $p_{\tau(i)}$ on the top level, by item (2), we can assume that $D_{\tau(i)} = 1$. Let $j = \ell + 1 - i$ and, without loss of generality, assume $i < j$. Let $\overline{W'} = X^B_{\I} (\ell, \tau', \ell, {\bf C}')$ such that $\tau' = \tau \circ (ij)$ and $C'_{\tau'(k)} = C_{\tau(k)}$ for $k \in \underline{n-2}$. We consider the prong-matchings:
$$ u \equiv c_i - 1 \pmod{\kappa} \quad \text{on} \quad \overline{W}, \quad u' \equiv -d_{j-1} \pmod{\kappa} \quad \text{on} \quad \overline{W'}. $$ 
By symmetry, we have $u + C \equiv -d_{j-1} \pmod{\kappa}$ while $u' + C \equiv c_i \pmod{\kappa}$. One can check that $UR^2U \cdot \overline{W}(u + \frac{1}{2}) = \overline{W''}(c''_i - \frac{1}{2})$ and $UR^2U \cdot \overline{W'}(u') = \overline{W''}(-d''_{j-2})$. Here, $p_{\tau(i)}$ and $p_{\tau(j)}$ are on the bottom level of $\overline{W''}$.

For item (4), let $p_{\tau(i-1)}$ and $p_{\tau(i)}$ be non-fixed poles on the top level. Let $j = \ell + 1 - i$ and, without loss of generality, assume $i < j$. By item (2), we can assume that $D_{\tau(i-1)} = D_{\tau(i)} = 1$. Let $\overline{W'} = X^B_{\I} (\ell, \tau', \ell, {\bf C})$ such that $\tau' = \tau \circ (i-1, i) \circ (j+1, j)$. We consider the prong-matchings:
$$ u \equiv c_{i-1} - 1 \pmod{\kappa} \quad \text{on} \quad \overline{W}, \quad u' \equiv c'_{i} - 1 \pmod{\kappa} \quad \text{on} \quad \overline{W'}. $$ 
One can check that $UR^2U \cdot \overline{W}(u + \frac{1}{2}) = \overline{W''}(c''_{i-1} - \frac{1}{2})$ and $UR^2U \cdot \overline{W'}(u'+\frac{1}{2}) = \overline{W''}(c''_{i} - \frac{1}{2})$. Here, $p_{\tau(i-1)}$ and $p_{\tau(j+1)}$ are on the bottom level of $\overline{W''}$.
\end{proof}

\subsection{Non-hyperelliptic components}\label{subsec:Bhyper}

In this subsection, we will establish lemmas to provide degeneration techniques for non-hyperelliptic components of strata of B-signature. The following lemma characterizes the Type I boundaries that are ubiquitous in non-hyperelliptic components, so that we can later initiate the degeneration from these Type I boundaries.

\begin{lemma}\label{lm:typeB_nonhyp_char}
    Let $\cC$ be a {\em{non-hyperelliptic}} component of $\cP(\mu^{\fR})$. Then there exists a multi-scale differential $\overline{W}=X^B_{\I} (\ell,\tau,{\bf C}) \in \partial\cC$ such that  
    \begin{equation*}
        C < e_2 + \max\left(\sum_{i=1}^{\ell} C_{\tau(i)}, \sum_{i=1}^{\ell} D_{\tau(i)}\right).
    \end{equation*}
\end{lemma}

\begin{proof}
    Let $\overline{W}=X^B_{\I} (\ell,\tau,{\bf C})$ be a Type I boundary of $\overline{\cC}$. First, note that if 
    \[
    \sum_{i=\ell+1}^{n-2} C_{\tau(i)} = \sum_{i=\ell+1}^{n-2} D_{\tau(i)},
    \]
    then by the assumption that $a_1 \leq a_2$, we will have $C \leq D$. In particular, if there is no marked residueless pole on the bottom level, then automatically $C \leq D$ and thus $C = \min(C, D)$.
    
    We now show that we can always start with some $\overline{W}$ such that 
    \begin{enumerate}
        \item[(1)] $\ell = n - 2$, i.e., there is no marked residueless pole on the bottom level; or
        \item[(2)] there is exactly one marked residueless pole on the bottom level and $C_{\tau(\ell+1)} = D_{\tau(\ell+1)}$.
    \end{enumerate}
  The properties listed above imply that $C \leq D$. We may assume that $\ell > 0$. Otherwise, we have 
    \[
    C + D - 1 = \kappa = e_1 + e_2 - 1 \quad (\text{i.e., } C + D = e_1 + e_2).
    \]
    Consider the prong-matching 
    \[
    u \equiv \max(e_2 - C, 0) \pmod{\kappa},
    \]
    then $e_2 \leq u + C \leq 0 \pmod{\kappa}$ and the saddle connection $\alpha_0$ will be a saddle connection of multiplicity one connecting $z_1$ and $z_2$ on a plumbed surface $\overline{W}_t(u)$. This means that $U \cdot \overline{W}(u) = \overline{W'}(u')$ where $\overline{W'}$ satisfies (1). Now we may assume $C \geq e_2 + \max(C_{\tau(\ell)}, D_{\tau(\ell)})$ (and $D \geq e_1 + \max(C_{\tau(1)}, D_{\tau(1)})$) because otherwise, by considering the same (or similar) prong matching as the case $\ell = 0$, we can reach some Type I boundary satisfying (1). Let $k$ be the largest number such that 
    \[
    C \leq e_2 + \sum_{i=k}^{\ell} b_{\tau(i)}.
    \]
    Consider 
    \[
    u \equiv c_k - 1 - \min(C_{\tau(k)}, D_{\tau(k)}) \pmod{\kappa}.
    \]
    Then we have the following possibilities for $u$ and $u + C$:
    \begin{itemize}
        \item[(i)] $c_{k-1} \leq u < c_k \pmod{\kappa}$ and $-d_{k-1} < u + C \leq -d_{k-2} \pmod{\kappa}$;
        \item[(ii)] $c_k \leq u < c_{k+1} \pmod{\kappa}$ and $-d_k < u + C \leq -d_{k-1} \pmod{\kappa}$;
        \item[(iii)] $u \equiv c_{k-1} - 1 \pmod{\kappa}$ and $u + C \equiv -d_k \pmod{\kappa}$.        
    \end{itemize}
    Then $U \cdot \overline{W}(u + \frac{1}{2}) = \overline{W'}(u')$ where $\overline{W'}$ is some Type I boundary satisfying (1) (for cases (i) and (ii)) or (2) (for case (iii)). 
      
   Now we will show the inequality in the lemma. Notice that the sum of angles on the top level (the nodal zero order plus one) is equal to $C + D - 1$ (the nodal pole order minus one), which is  
    \[
    C + D - 1 = (e_2 - \frac{1}{2}) + (e_1 - \frac{1}{2}) + \sum_{i=1}^{\ell} (C_{\tau(i)} + D_{\tau(i)}).
    \]
    Hence,  
    \begin{align*}
        C \leq \frac{1}{2}(C + D) &= \frac{1}{2}(e_1 + e_2) + \frac{1}{2} \sum_{i=1}^{\ell} (C_{\tau(i)} + D_{\tau(i)}) \\
        &\leq e_2 + \max\left( \sum_{i=1}^{\ell} C_{\tau(i)}, \sum_{i=1}^{\ell} D_{\tau(i)} \right).
    \end{align*}
    Equality is attained if and only if the following holds:
    \begin{align*}
        e_2 &= e_1, \\
        C &= D, \\ 
        \sum_{i=1}^{\ell} C_{\tau(i)} &= \sum_{i=1}^{\ell} D_{\tau(i)}.
    \end{align*}

    Now we will show how to avoid attaining equality. By our assumption on $\overline{W}$, equality implies that $a_1 = a_2$ and the bottom level is hyperelliptic. The top level cannot be hyperelliptic, and thus there is some index $i \leq \frac{\ell}{2}$ where $(C_{\tau(i)}, D_{\tau(i)}) \neq (D_{\tau(\ell+1-i)}, C_{\tau(\ell+1-i)})$. Otherwise, $\overline{W}$ will be hyperelliptic. Now let $k$ be the smallest index satisfying $(C_{\tau(k)}, D_{\tau(k)}) \neq (D_{\tau(\ell+1-k)}, C_{\tau(\ell+1-k)})$. Without loss of generality, we may assume $D_{\tau(k)} < C_{\tau(\ell+1-k)}$. We consider the prong-matching 
    \[
    u \equiv c_{k-1} + 1 \pmod{\kappa}.
    \]
    Then $u + C \equiv -d_{\ell+1-k} \pmod{\kappa}$. Then the equatorial half-arc 
    \[
    UR^{2D_{\tau(k)}} U \cdot \overline{W}(u + \frac{1}{2}) = \overline{W'}(u' + \frac{1}{2}),
    \]
    where $\overline{W'} = X^B_{\I} (\ell-1, \tau', {\bf C}')$ such that we have sent $p_{\tau(k)}$ to the bottom level. In addition, 
    \[
    \sum_{i=1}^{\ell-1} D_{\tau'(i)}' = \sum_{i=1}^{\ell} D_{\tau(i)}
    \]
    and $C' = C - C_{\tau(k)}$. This means that $C' < C = e_2 + \sum_{i=1}^{\ell-1} D_{\tau'(i)}'$, and we are done.

\end{proof}

Combining \Cref{lm:typeB_nonhyp_char} with the following lemma, we can establish the existence of certain Type IIIb and Type IIIc boundaries in a non-hyperelliptic component, such that the residueless poles between $q_1$ and $s^\bot$ satisfy any predetermined permutation. The following lemma enables us to move a residueless pole that does not initially lie between $q_1$ and $s^\bot$ to the bottom level next to $s^\bot$ (and hence between $q_1$ and $s^\bot$).

\begin{lemma}\label{lm:IIIc_push_pole_standard_to_bot}
    Let $\cC$ be a non-hyperelliptic component of $\cP(\mu^{\fR})$, and assume there exists a Type IIIc boundary $\overline{Z} = X^B_{\III c}((2,1),\ell_1,\ell_2,\tau,C,{\bf C})$. Given $p_{\tau(j)}$ where $j > \ell_1$, we can find a boundary $$\overline{Z'} = X^B_{\III c}((2,1),\ell_1',\ell_2',\tau',C',{\bf C}') \in \partial\cC$$ of one of the following forms:
    \begin{enumerate}
        \item $\ell_1' = \ell_1$ and $\ell_2' = n-2$, i.e., all residueless poles between the nodal pole and $q_2$ have been moved onto the top level; $\tau'(i) = \tau(i)$ and $C'_{\tau'(i)} = C_{\tau(i)}$ for all $i \leq \ell_1$.
        \item $\ell_1' = \ell_1 + 1$ with $\tau'(\ell_1 + 1) = \tau(j)$ and $C'_{\tau(j)} = b_{\tau(j)} - 1$, i.e., the pole $p_{\tau(j)}$ has been moved to lie between the nodal pole and $q_1$; $\tau'(i) = \tau(i)$ and $C'_{\tau'(i)} = C_{\tau(i)}$ for all $i \leq \ell_1$.
    \end{enumerate}

    Moreover, if the angle sum of residueless poles on the bottom level satisfies $D + \sum_{i \leq \ell_1} D_{\tau(i)} + \sum_{i > \ell_2} D_{\tau(i)} < n - 2$, then we have
    \[
    C' + \sum_{i \leq \ell_1'} C'_{\tau'(i)} + \sum_{i > \ell_2'} C'_{\tau'(i)} = C + \sum_{i \leq \ell_1} C_{\tau(i)} + \sum_{i > \ell_2} C_{\tau(i)}.
    \]
\end{lemma}

\begin{proof}
    There are two possibilities for the position of $p_{\tau(j)}$:
    \begin{enumerate}
        \item[(a)] $p_{\tau(j)}$ lies on the top level.
        \item[(b)] $p_{\tau(j)}$ lies on the bottom level, i.e., $j > \ell_2$.
    \end{enumerate}
    In case (b), consider the prong-matching $u \equiv 0 \pmod{\kappa}$. Suppose $c_{i-1} < u + C \leq c_i \pmod{\kappa}$. The angle from $\beta_{j - \ell_2}$ clockwise to $\alpha_i$ on the plumbed surface $\overline{Z}(0)$ is
    \[
    \theta = c_i - (u + C) + \frac{1}{2} + \sum_{k = \ell_2 + 1}^{j - 1} D_{\tau(k)}.
    \]
    Then
    \[
    UR^{-2\theta - 1}U \cdot \overline{Z}(0) = \overline{Z'}(u'),
    \]
    where $\overline{Z'} = X^B_{\III c}((2,1),\ell_1',\ell_2',\tau',C',{\bf C}')$ with $\ell_1' \geq \ell_1$ and $p_{\tau(j)}$ now on the top level with $C'_{\tau(j)} = b_{\tau(j)} - 1$. Since $D'_{\tau(j)} = 1$, the top level cannot contain only one pole (as $a_{h_2} > 0$).

    Assume $\tau'(\ell_1'+1) = \tau(j)$. To fully reduce case (b) to case (a), we must decrease $\ell_1'$ back to $\ell_1$ by removing any extra poles between $q_1$ and the nodal pole that are not among $p_{\tau(1)}, \ldots, p_{\tau(\ell_1)}$, while keeping $p_{\tau(j)}$ on the top level. Consider the prong-matching $u' \equiv 0 \pmod{\kappa'}$ on $\overline{Z'}$. If $0 < C' \leq c'_1 \pmod{\kappa'}$, then
    \[
    UR^{2}U \cdot \overline{Z'}(u') = \overline{Z''}(u''),
    \]
    where $\overline{Z''}$ is a Type IIIc boundary with $\ell_1'' = \ell_1' - 1$ and $p_{\tau(j)}$ still on the top level with $D''_{\tau(j)} = 1$.

    Assume now $C' \leq b_{\tau(j)} - 1$. Let $\theta$ be the angle from $\beta_{\ell_1'+1}$ clockwise to $\alpha_1$ on the plumbed surface $\overline{Z'}_t(1)$. Then
    \[
    \theta = b_{\tau(j)} - 1 - C + e_2 + \sum_{i = \ell_2' + 1}^{n - 2} b_{\tau'(i)}.
    \]
    Applying $UR^{-2\theta - 1}U$ gives
    \[
    \overline{Z''}(u'' + \tfrac{1}{2}) = UR^{-2\theta - 1}U \cdot \overline{Z'}(u'),
    \]
    where $\overline{Z''} = X^B_{\III c}((2,1),\ell_1',\ell_2',\tau',C'',{\bf C}')$ with
    \[
    C'' = C' - a_{h_2} + \sum_{\substack{\ell_1' < i \leq \ell_2'\\ i \neq \ell_1' + j'}} b_{\tau(i)} > C'.
    \]
    Repeating this operation as needed, we may assume initially $C' > b_{\tau(j)} - 1$, enabling us to reduce $\ell_1'$ by one.

    In case (a), assume $\ell_2 \neq n - 2$, otherwise $\overline{Z}$ already satisfies condition (ii). If $D_{\tau(j)} = 1$, then using the previous operation to increase $C$ ensures $C > b_{\tau(j)} - 1$, allowing us to move $p_{\tau(j)}$ between $q_1$ and the nodal pole.

    Finally, to reduce $D_{\tau(j)}$ to one, let $u \equiv c_{j-1} + \max(C_{\tau(j)} - C,\:0) \pmod{\kappa}$. Then $c_{j-1} \leq u < c_j \pmod{\kappa}$ and $c_{i-1} < u + C \leq c_i \pmod{\kappa}$ for some $i \neq j$. Since we assume there is at least one residueless pole between $s^\bot$ and $q_2$, the operation
    \[
    UR^{-2}U \cdot \overline{Z}(u) = \overline{Z'}(u')
    \]
    yields a boundary where $D'_{\tau(j)} = D_{\tau(j)} - 1$.

    The last claim follows from the fact that $UR^{\pm2}U$ does not change the angle sum of residueless poles on the bottom level. If this operation is needed to increase $C$, then $C \leq b_{\tau(j)} - 1$, implying $D$ is at least the number of poles on the top level. Thus,
    \[
    D + \sum_{i \leq \ell_1} D_{\tau(i)} + \sum_{i > \ell_2} D_{\tau(i)} \geq n - 2,
    \]
    contradicting the hypothesis.
\end{proof}

The following proposition shows that certain types of boundaries in a non-hyperelliptic component of a stratum can be obtained by applying \Cref{lm:typeB_nonhyp_char} and \Cref{lm:IIIc_push_pole_standard_to_bot}.

\begin{proposition}\label{cor:IIIb_perfect_1st_tail_empty_2nd_tail}
Let $\cC$ be a non-hyperelliptic component of $\cP(\mu^{\fR})$. Given any $\tau^* \in \mathrm{Sym}_{n-2}$, we have the following:
\begin{itemize}
    \item If $e_2 > 1$, then there exists a boundary $\overline{Y} = X^B_{\III b}((2,1,2,1), \ell, \tau, C, {\bf C}) \in \partial\cC,$
    
    \item If $e_2 = 1$, then there exists a boundary $\overline{Z} = X^B_{\III c}((2,1), \ell, n-2, \tau, C, {\bf C}) \in \partial\cC,$
    
\end{itemize}
such that $\tau(i) = \tau^*(i)$ and $C_{\tau(i)} = b_{\tau(i)} - 1$ for all $i \leq \ell$.
\end{proposition}

\begin{proof}
By \Cref{lm:typeB_nonhyp_char}, without loss of generality, we can find a Type I boundary $
    \overline{W} = X^B_{\I}(\ell, \tau, {\bf C}) \in \partial\cC
$
such that $
    C < e_2 + \sum_{i=1}^\ell C_{\tau(i)}.
$
Then, the equatorial half-arc $U \cdot \overline{W}(0)$ corresponds to either:
\begin{itemize}
    \item a Type IIIb boundary 
    $
        \overline{Y} = X^B_{\III b}((2,1,2,1), 0, \tau', C', {\bf C}'),
    $
    where no marked residueless pole lies on the bottom level—thus satisfying the desired condition for Type IIIb; or
    \item a Type IIIc boundary 
    $
        \overline{Y} = X^B_{\III c}((2,1), 0, \ell_2, \tau'', C'', {\bf C}''),
    $
    where no marked residueless pole lies on the bottom level between $q_1$ and the nodal pole.
\end{itemize}
Note that if $e_2 = 1$, then the inequality $C \leq \sum_{i=1}^\ell C_{\tau(i)}$ ensures that only a Type IIIc boundary can occur.

In the Type IIIc case, we can apply \Cref{lm:IIIc_push_pole_standard_to_bot} repeatedly to move the marked residueless poles to the bottom level, between $q_1$ and the nodal pole, while preserving the prescribed order and assigning angles as $C_{\tau(i)} = b_{\tau(i)} - 1$. Furthermore, if $e_2 > 1$, we can continue using the operations from the proof of \Cref{lm:IIIc_push_pole_standard_to_bot} until we obtain a Type IIIb boundary.
\end{proof}

\begin{definition}
    A Type IIIa boundary $\overline{X} = X^B_{\III a}((h_1, h_2), \ell_1, \ell_2, \tau, {\bf C}, Pr)$ is said to have \emph{the hyperelliptic top level} if the top-level component $X_0$ of $\overline{X}$ admits a hyperelliptic involution $\sigma$ that interchanges the two nodes, and the prong-matching equivalence class $Pr$ is compatible with $\sigma$. Otherwise, $\overline{X}$ is said to have \emph{the non-hyperelliptic top level}.
\end{definition}

\noindent
Note that $\overline{X}$ may have the non-hyperelliptic top level even if its top-level component $X_0$ is hyperelliptic, since the prong-matching equivalence class $Pr$ may not be compatible with the hyperelliptic involution of $X_0$.

\begin{proposition} \label{prop:IIIa_hyper}
    A Type IIIa boundary $\overline{X} = X^B_{\III a}((h_1, h_2), \ell_1, \ell_2, \tau, {\bf C}, Pr)$ has the hyperelliptic top level if and only if the following conditions hold:
    \begin{itemize}
        \item $\kappa_1 = \kappa_2$;
        \item If $(u, v) \in Pr$ and $u = c_i$ for some $i$, then $v = d_j$ for some $j$.
    \end{itemize}
    In particular, if $\overline{X}$ has the non-hyperelliptic top level, then:
    \begin{enumerate}
        \item If $\kappa_1 \geq \kappa_2$, there exists a prong-matching $(u, 0) \in Pr$ such that $u \neq c_i$ for any $i$;
        \item If $\kappa_2 \geq \kappa_1$, there exists a prong-matching $(0, v) \in Pr$ such that $v \neq d_j$ for any $j$.
    \end{enumerate}
\end{proposition}

\begin{proof}
    The first condition is necessary for the existence of an involution, as it requires the same number of prongs at both nodes. The second condition ensures a one-to-one correspondence between the $c_i$'s and the $d_j$'s such that $(c_i, d_j) \in Pr$, which determines the involution.

    If $\kappa_1 \neq \kappa_2$, then the existence of such unmatched prong-matching is a direct consequence of \cite[Lemma 7.14]{lee2023connected}. If $\kappa_1 = \kappa_2$, then the conclusion follows from our characterization of the hyperelliptic top level.
\end{proof}

\noindent
Note that if $b_i = 2$ for each $i = 1, \dots, \ell_1$, then $\overline{X}$ automatically satisfies the above conditions, since $C_i = 1$ for all $i$ and $\kappa_1 = \kappa_2 = \ell_1$. Thus, $\overline{X}$ has the hyperelliptic top level.

\medskip

In order to construct equatorial paths connecting Type IIIa boundaries that can be generalized within the same stratum or even to other strata, we consider only those equatorial arcs adjacent to Type IIIa boundaries with specific properties. More precisely, let $\cP(\mu^\fR)$ be a stratum of B-signature, $(h_1, h_2) \in \{(1,2), (2,1)\}$, and let $\mathcal{I}_1$ and $\mathcal{I}_2$ be two disjoint ordered subsets (possibly empty) of the index set of residueless marked poles:
\begin{align*}
    \mathcal{I}_1 &= (i_1(1), i_1(2), \dots, i_1(L_1)) \subset \underline{n-2}, \\
    \mathcal{I}_2 &= (i_2(1), i_2(2), \dots, i_2(L_2)) \subset \underline{n-2},
\end{align*}
such that $a_{h_1} > b_{i_1(1)} + \dots + b_{i_1(L_1)}$ and $a_{h_2} > b_{i_2(1)} + \dots + b_{i_2(L_2)}$. 

We denote by $\mathcal{A}_{h_1,h_2}(\mu^\fR, \mathcal{I}_1, \mathcal{I}_2)$ the truncated equatorial net of $\mathcal{A}(\mu^\fR)$ consisting of equatorial arcs adjacent to Type IIIa boundaries
\[
X^B_{\III a}((h_1, h_2), \ell_1, \ell_2, \tau, {\bf C}, Pr),
\]
such that:
\begin{itemize}
    \item $\ell_1 \geq L_1$ and $n - 2 - \ell_2 \geq L_2$;
    \item $\tau(k) = i_1(k)$ for $k = 1, \dots, L_1$ and $\tau(n - 1 - k) = i_2(k)$ for $k = 1, \dots, L_2$.
\end{itemize}

In $\mathcal{A}_{h_1,h_2}(\mu^\fR, \mathcal{I}_1, \mathcal{I}_2)$, two Type IIIa boundaries are connected only by equatorial arcs corresponding to degenerations that do not modify the polar domains of $q_1$, $q_2$, or $p_i$ for $i \in \mathcal{I}_1 \cup \mathcal{I}_2$. These degenerations allow us to "forget" poles that are "untouched."

\medskip

Given subsets $\mathcal{I}_1' = (i_1(1), \dots, i_1(L_1'))$ and $\mathcal{I}_2' = (i_2(1), \dots, i_2(L_2'))$, we have the inclusion:
\[
\mathcal{A}_{h_1,h_2}(\mu^\fR, \mathcal{I}_1', \mathcal{I}_2') \subset \mathcal{A}_{h_1,h_2}(\mu^\fR, \mathcal{I}_1, \mathcal{I}_2).
\]
Hence, to show the connectedness of two points in $\mathcal{A}_{h_1,h_2}(\mu^\fR, \mathcal{I}_1, \mathcal{I}_2)$, it suffices to show that both can be connected to the same component of $\mathcal{A}_{h_1,h_2}(\mu^\fR, \mathcal{I}_1', \mathcal{I}_2')$.

\medskip

Let $\cP(\mu'^{\fR'})$ be a stratum of B-signature such that $n' = n - L_1 - L_2$. Suppose $\rho: \underline{n'-2} \to \underline{n-2} \setminus (\mathcal{I}_1 \cup \mathcal{I}_2)$ is an order-preserving map such that $b'_i = b_{\rho(i)}$ and
\[
a'_{h_1'} - e_1' = a_{h_1} - e_1 - \sum_{i \in \mathcal{I}_1} b_i, \quad 
a'_{h_2'} - e_2' = a_{h_2} - e_2 - \sum_{i \in \mathcal{I}_2} b_i.
\]
Then there is a covering map
\[
\rho^*: \mathcal{A}_{h_1,h_2}(\mu^\fR, \mathcal{I}_1, \mathcal{I}_2) \rightarrow \mathcal{A}_{h_1', h_2'}(\mu'^{\fR'}, \emptyset, \emptyset),
\]
where $(h_1', h_2') = (h_1, h_2)$ if and only if $\sgn(a'_{h_1'} - a'_{h_2'}) = \sgn(a_{h_1} - a_{h_2})$.

The map sends a Type IIIa boundary in $\cP(\mu^{\fR})$ to the Type IIIa boundary in $\cP(\mu'^{\fR'})$ obtained by replacing the polar domains of $q_1$ and $p_{i_1(1)}, \dots, p_{i_1(L_1)}$ (and respectively $q_2$ and $p_{i_2(1)}, \dots, p_{i_2(L_2)}$) by a single Type II polar domain of pole order $e'_1$ (respectively $e'_2$). Alternatively, these domains can be replaced by infinite half-cylinders. In particular, one can always construct a covering map
\[
\rho^*: \mathcal{A}_{h_1,h_2}(\mu^\fR, \mathcal{I}_1, \mathcal{I}_2) \rightarrow \mathcal{A}_{h_1,h_2}(\mu'^{\fR'}, \emptyset, \emptyset)
\]
where $\cP(\mu'^{\fR'})$ contains a pair of simple poles and
\[
a'_{h_1} = a_{h_1} + 1 - e_1 - \sum_{i \in \mathcal{I}_1} b_i, \quad 
a'_{h_2} = a_{h_2} + 1 - e_2 - \sum_{i \in \mathcal{I}_2} b_i.
\] This covering map $\rho^*$ implies that an equatorial path in $\mathcal{A}_{h_1,h_2}(\mu'^{\fR'}, \emptyset, \emptyset)$ can be lifted to one in $\mathcal{A}_{h_1,h_2}(\mu^\fR, \mathcal{I}_1, \mathcal{I}_2)$, providing a reduction technique useful in proving connectedness between Type IIIa boundaries.

\medskip

On a stratum $\cP(\mu^\fR)$ of B-signature with $e_1 = e_2 = 1$, the \emph{index} (as introduced in the Introduction) is a local invariant. For a Type IIIc boundary $\overline{Z} = X^B_{\III c}((h_1, h_2), \ell_1, \ell_2, \tau, C, {\bf C})$, one can directly compute from the plumbed surface that:
\[
\operatorname{Ind}(\overline{Z}) \equiv C + \left( \sum_{i \leq \ell_1} + \sum_{i > \ell_2} \right) C_{\tau(i)} \pmod{\delta},
\]
where $\delta \coloneqq \gcd(a_1, \{b_i\}_{i \in \underline{n-2}})$. Similarly, for a Type IIIa boundary $\overline{X} = X^B_{\III a}((h_1, h_2), \ell_1, \ell_2, \tau, {\bf C}, Pr)$, and for any $(u, v) \in Pr$, we have:
\[
\operatorname{Ind}(\overline{X}) \equiv -u - v + \left( \sum_{i \leq \ell_1} + \sum_{i > \ell_2} \right) C_{\tau(i)} \pmod{\delta}.
\]

From our discussion on sub ribbon graphs above, the index defined on $\mathcal{A}_{h_1,h_2}(\mu'^{\fR'},\emptyset,\emptyset)$ will induce a local invariant on $\mathcal{A}_{h_1,h_2}(\mu^\fR,\mathcal{I}_1,\mathcal{I}_2)$. From the formula of index of Type IIIa boundary on a B-signature stratum with a pair of simple poles, we can define the index of a Type IIIa boundary $\overline{X}$ in $\mathcal{A}_{h_1,h_2}(\mu^\fR,\mathcal{I}_1,\mathcal{I}_2)$ as $$\operatorname{Ind}(\overline{X},\mathcal{I}_1,\mathcal{I}_2)\equiv -u-v+(\sum_{i\leq L_1}+\sum_{i\geq n-1-L_2})C_{\tau(i)}\pmod{\delta(\mathcal{I}_1,\mathcal{I}_2)},$$ where $$\delta(\mathcal{I}_1,\mathcal{I}_2)=\gcd(a_{H_1}+1-e_1-\sum_{i\in\mathcal{I}_1}b_i,\: \{b_i\}_{\substack{i\in\underline{n-2}\\i\notin\mathcal{I}_1\cup\mathcal{I}_2}}).$$

\begin{remark}\label{rm:trunc_eq_net}
    By the definition of a Type IIIa boundary in the truncated equatorial net $\mathcal{A}_{h_1,h_2}(\mu^\fR,\mathcal{I}_1,\mathcal{I}_2)$ and the fact that the index is a local invariant, two elements 
    \[
        \overline{X}=X^B_{\III a}((h_1,h_2),\ell_1,\ell_2,\tau,{\bf C},Pr) \quad \text{and} \quad \overline{X'}=X^B_{\III a}((h_1,h_2),\ell_1',\ell_2',\tau',{\bf C}',Pr')
    \]
    cannot be connected within $\mathcal{A}_{h_1,h_2}(\mu^\fR,\mathcal{I}_1,\mathcal{I}_2)$ unless $C_i = C'_i$ for all $i \in \mathcal{I}_1 \cup \mathcal{I}_2$ and $\operatorname{Ind}(\overline{X},\mathcal{I}_1,\mathcal{I}_2) = \operatorname{Ind}(\overline{X'},\mathcal{I}_1,\mathcal{I}_2)$.
\end{remark}

 \Cref{lm:IIIa_move} illustrates the resulting Type IIIa boundary under various conditions. Before presenting the lemma, we include a remark regarding the notation for prong matchings of Type IIIa boundaries. We denote by $\mathcal{A}(\overline{X})$, where $\overline{X}=X^B_{\III a}((h_1,h_2),\ell_1,\ell_2,\tau,{\bf C},Pr)$, the connected component of the truncated equatorial net $\mathcal{A}_{h_1,h_2}(\mu^\fR,\mathcal{I}_1,\mathcal{I}_2)$ containing $\overline{X}$, where 
\[
    \mathcal{I}_1 = \big(\tau(1),\tau(2),\dots, \tau(\ell_1)\big), \quad 
    \mathcal{I}_2 = \big(\tau(n-2),\tau(n-3),\dots, \tau(\ell_2+1)\big).
\]

\begin{remark}
    The ordering of residueless poles on the top level of a Type IIIa boundary is cyclic. Thus, applying a cyclic permutation does not change the top level but only our reference, namely the number $\tau(1)$ (if $\ell_1=0$). The prong matching on a Type IIIa boundary depends on the choice of $\tau(1)$. Changing $\tau(1)$ modifies the tuple $(u,v)$ and potentially its equivalence class in $\mathbb{Z}/\kappa_1\mathbb{Z} \times \mathbb{Z}/\kappa_2\mathbb{Z}$. 

    To write proofs conveniently, we represent a prong matching independently of the choice of $\tau(1)$. We denote by $(u,v)_k$ the prong matching with respect to the reference pole $p_k$. We also denote by $c_{k,i}$ (resp.\ $d_{k,i}$) the labels of prongs for overlapping saddle connections. By default, we write $(u,v)$ (resp.\ $c_i$ and $d_i$) to mean $(u,v)_{\tau(1)}$ (resp.\ $c_{\tau(1),i}$ and $d_{\tau(1),i}$). The transformation of the tuple under change of reference is:
    \[
        (u,v) = (u - d_{k-1},\: v - c_{k-1})_{\tau(k)}.
    \]
    In particular, if $(u,v) = (u',v')_{\tau(k)}$, then $u+v$ and $u'+v'$ differ by a sum of $b_{\tau(i)}$ over poles on the top level.
\end{remark}

\begin{lemma} \label{lm:IIIa_move}
    Let $\overline{X}=X^B_{\III a}((h_1,h_2),0,n-2,\Id,{\bf C},[(u,v)])$ be a boundary in a stratum $\cP(\mu^\fR)$ of B-signature. Assume $c_{i-1} \leq u < c_i \pmod{a_{h_1}}$ and $d_{j-1} < v \leq d_j \pmod{a_{h_2}}$ with $i \neq j$. Then $UR^{-2}U \cdot \overline{X}(u,v)$ (or $UR^2U \cdot \overline{X}(u+\tfrac{1}{2}, v-\tfrac{1}{2})$ if $j \equiv i+1 \pmod{\ell}$ and $C_i = D_j = 1$) defines a path in $\mathcal{A}(\overline{X})$ connecting $\overline{X}$ to a boundary $\overline{X'}$ given as follows:
    
    \begin{itemize}
        \item[(1)] If $D_i, C_j > 1$, then $\overline{X'} = X^B_{\III a}((h_1,h_2),0,n-2,\Id,{\bf C'},[(u,v)])$ with 
        \[
            C'_i = C_i + 1,\quad C'_j = C_j - 1,\quad C'_k = C_k \text{ for } k \neq i,j.
        \]

        \item[(2)] If $D_i > 1$ and $C_j = 1$, then $\overline{X'} = X^B_{\III a}((h_1,h_2),0,n-3,\tau^{(2)},{\bf C'},[(u',v')])$, where
        \[
            (u',v') = \big(u - c_{i-1} + 1,\: d'_{i,j-i}\big)_i,
        \]
        with 
        \[
            C'_i = C_i + 1,\quad C'_j = C_j + d_j - v,\quad C'_k = C_k \text{ for } k \neq i,j,
        \]
        and
        \[
            \tau^{(2)}(k) = 
            \begin{cases}
                k & \text{if } k < j, \\
                k + 1 & \text{if } j \leq k < n-2, \\
                j & \text{if } k = n-2.
            \end{cases}
        \]

        \item[(3)] If $D_i = 1$ and $C_j > 1$, then $\overline{X'} = X^B_{\III a}((h_1,h_2),1,n-2,\tau^{(3)},{\bf C'},[(u',v')])$, where
        \[
            (u',v') = \big(c'_{j,i-j},\: v - d_{j-1}\big)_j,
        \]
        with
        \[
            C'_i = c_i - u,\quad C'_j = C_j - 1,\quad C'_k = C_k \text{ for } k \neq i,j,
        \]
        and
        \[
            \tau^{(3)}(k) = 
            \begin{cases}
                k - 1 & \text{if } 1 < k < i, \\
                k & \text{if } i \leq k \leq n-2, \\
                i & \text{if } k = 1.
            \end{cases}
        \]

        \item[(4)] If $C_i = D_j = 1$, then $\overline{X'} = X^B_{\III a}((h_1,h_2),1,n-3,\tau^{(4)},{\bf C'},[(u',v')])$ with
        \[
            (u',v') = 
            \begin{cases}
                (0,\, d'_{i+1,j-i-1})_{i+1} & \text{if } j \not\equiv i+1 \pmod{\ell}, \\
                (c'_{j+1,i-j-1} - 1,\, 1)_{j+1} & \text{if } j \equiv i+1 \pmod{\ell},
            \end{cases}
        \]
        and
        \[
            C'_i = 
            \begin{cases}
                c_i - u & \text{if } j \not\equiv i+1 \pmod{\ell}, \\
                D_i + c_i - u - 1 & \text{otherwise},
            \end{cases}
            \quad
            C'_j = 
            \begin{cases}
                C_j + d_j - v & \text{if } j \not\equiv i+1 \pmod{\ell}, \\
                d_j - v + 1 & \text{otherwise},
            \end{cases}
        \]
        and $C'_k = C_k$ for all $k \neq i,j$, with
        \[
            \tau^{(4)}(k) = 
            \begin{cases}
                k - 1 & \text{if } 1 < k < \min(i,j), \\
                k & \text{if } \min(i,j) \leq k < \max(i,j), \\
                k + 1 & \text{if } \max(i,j) \leq k < n-2, \\
                j & \text{if } k = n-2, \\
                i & \text{if } k = 1.
            \end{cases}
        \]
    \end{itemize}
\end{lemma}

\begin{proof}
    This follows directly from the description of the actions $R$ and $U$ given in \Cref{Prop:BIIIamove}, namely that $R^2U \cdot \overline{X} = U \cdot \overline{X'}$.
\end{proof}

\Cref{lm:IIIa_nonhyp_top_level_modification} plays a crucial role in this section. It allows us to relate Type IIIa boundaries with different combinatorial data and will be used several times. In addition, it gives a frequently used degeneration sending poles to the bottom level.

\begin{lemma}\label{lm:IIIa_nonhyp_top_level_modification}
Let $\overline{X}=X^B_{\III a}((2,1),\ell_1,\ell_2,\tau,{\bf C},[(0,v)])$ be a boundary in a stratum $\cP(\mu^\fR) $ of B-signature which has the non-hyperelliptic top level. Then we have the following:
    \begin{enumerate}
        \item Given that $\ell_2-\ell_1>1$ and we are not in case (a1), (a2) or (c). By cyclic permutation of the top level, we may assume that $b_{\tau(\ell_1+1)}=\min_{\ell_1<i\leq \ell_2}(b_{\tau(i)})$ if $\kappa_1\geq \kappa_2$, otherwise $b_{\tau(\ell_2)}=\min_{\ell_1<i\leq \ell_2}(b_{\tau(i)})$. Then there exists some $\overline{X'}=X^B_{\III a}((2,1),\ell_1',\ell_2',\tau',{\bf C'},[(0,v')])$ in $\mathcal{A}(\overline{X})$ such that $\ell_1'=\ell_1+1$ if $\kappa_1\geq \kappa_2$, otherwise $\ell_2'=\ell_2-1$. In addition, $\tau'(\ell_1+1)=\tau(\ell_1+1)$ and $\tau'(\ell_2)=\tau(\ell_2)$. 
        
        \item Given that we are not in case (a3), (b) or (c). Then, for every $\tau',\bf C'$ and $v'$ such that $\tau'(i)=\tau(i)$ resp. $C_{\tau(i)}=C'_{\tau(i)}$ for $\ell_1<i\leq \ell_2$ and $v'\equiv v\pmod{\delta(\mathcal{I}_1,\mathcal{I}_2})$, where $\mathcal{I}_1,\mathcal{I}_2$ are the sets of bottom level residueless poles. Then there exists $\overline{X'}=X^B_{\III a}((2,1),\ell_1,\ell_2,\tau',{\bf C'},[(0,v')])$ in $\mathcal{A}(\overline{X})$.  
    \end{enumerate}  

    The aforementioned exceptional cases are as follows:
    
    \begin{itemize}
        \item[(a1)] $\ell_1=2$,  $\kappa_1=\kappa_2=b_{\tau(1)}=b_{\tau(2)}$.
        \item[(a2)] $\ell_1=2$, $\kappa_1$ even while $\kappa_2=b_{\tau(2)}=2$ and $v\equiv 1 \pmod{2}$.
        \item[(a3)] $\ell_1=2$,  $\kappa_1=\kappa_2=b_{\tau(1)}=b_{\tau(2)}$ and $v\equiv \pm\lfloor \kappa_1/2\rfloor \pmod{\kappa_1}$.
        \item[(b)] $\ell_1=3$, $\{\kappa_1,\kappa_2\}=\{3,6\}$ and $b_{\tau(1)}=b_{\tau(2)}=b_{\tau(3)}=3$.
        \item[(c)] $\ell_1=4$, $\kappa_1=\kappa_2=6$ and $b_{\tau(1)}=b_{\tau(2)}=b_{\tau(3)}=b_{\tau(4)}=3$ with $v\equiv 0,3 \pmod{6}$. 
    \end{itemize}
\end{lemma}

We will give the proof of the above lemma later. The proof will not depend on the data of bottom level components.

Before proving \Cref{lm:IIIa_nonhyp_top_level_modification}, we first address the exceptional cases $\ell_1 = 2$ (i.e., exactly two poles on the top level) and $\ell_1 = 4$ (i.e., exactly two poles on the top level and the pseudo-index is $1$ or $2$). For simplicity, we prove the statement on a stratum with the minimal number of poles. The result generalizes to more complicated strata by our discussion on the truncated equatorial net.

\begin{lemma}\label{lm:TypeB_IIIa_2pole_top_level_C2=1}
   Let $\overline{X} = X^B_{\III a}((2,1),0,2,\Id,{\bf C},[(0,v)])$ be a Type IIIa boundary point in a stratum of B-signature $\mu^\fR = (a_1,a_2 \mid -b_1 \mid -b_2 \mid -1^2)$ with the non-hyperelliptic top level. Then there exists $\overline{X'} = X^B_{\III a}((2,1),0,2,\Id,{\bf C}',[(0,v')])$ in the truncated equatorial net $\mathcal{A}(\overline{X})$ such that:
   \begin{itemize}
       \item If $a_1 = a_2 = b_1 = b_2$ and $v \equiv \pm\lfloor b_1/2 \rfloor \pmod{a_2}$ with $C_2 > C_1$, then $C'_1 = \lceil b_1/2 \rceil - 1$ and $C'_2 = \lfloor b_1/2 \rfloor + 1$;
       \item Otherwise, $C'_1 = b_1 - 1$ and $C'_2 = a_2 - b_1 + 1$.
   \end{itemize}
\end{lemma}

\begin{proof}
We begin by proving that in non-exceptional cases, we can always modify $D_1$ to one (i.e., $C_1$ to $b_1 - 1$). Assume $b_1 > 2$ and $a_1 > 2$, since otherwise the statement is trivial.

By the implicit assumption $a_2 \geq a_1$ (and $b_2 \geq b_1$), we have:
\[
C_1 + C_2 \geq D_1 + D_2 \quad \text{and} \quad D_2 + C_2 \geq D_1 + C_1,
\]
which together imply:
\[
C_2 - D_1 \geq |C_1 - D_2| \geq 0.
\]

Assume there exists a prong-matching $(u_0, v_0) \in [(0,v)]$ such that $0 \leq u_0 < C_1 \pmod{a_2}$ and $D_1 < v_0 \leq 0 \pmod{a_1}$ (guaranteed if $C_1 > D_1$ by prong-counting). Since $C_2 - D_1 \geq 0$, the operation $UR^{-2D_1+2}U \cdot \overline{X}(u_0,v_0) = \overline{X'}(u',v')$ defines a valid move in $\mathcal{A}(\overline{X})$, where $(u',v') \in [(0,v)]$ and $D'_1 = 1$ (and $C'_2 = C_2 - D_1 + 1$) on $\overline{X'}$.

Now assume $C_1 \leq D_1$ and $D_1 \geq 2$. Suppose every prong-matching $(u_0, v_0) \in [(0,v)]$ with $0 \leq u_0 < C_1 \pmod{a_2}$ satisfies $0 < v_0 \leq D_1 \pmod{a_1}$. If we can find $(u_0, v_0)$ with $C_1 + 1 \leq u_0 < 0 \pmod{a_2}$ and $0 < v_0 \leq D_1 - 1 \pmod{a_1}$, applying $UR^2U$ to $\overline{X}(u_0,v_0)$ simultaneously decreases $D_1$ and $C_2$ by one. Repeating this until $C_1 > D_1$ completes the argument.

This process halts only if every $0 < v_0 \leq D_1 - 1 \pmod{a_1}$ corresponds to $0 \leq u_0 \leq C_1 \pmod{a_2}$, implying for any $(u_0, D_1-1) \in [(0,v)]$:
\[
0 \leq u_0 < u_0+1 < \dots < u_0 + D_1 - 2 \leq C_1 \pmod{a_2}.
\]

Changing the prong-matching by $\gcd(a_1, a_2)$, we can assume $D_1 - 2 < \gcd(a_1, a_2)$. Otherwise, the integers of the form $u_0 + k_1 + k_2 \gcd(a_1,a_2)$ for $0 \leq k_1 \leq D_1 - 2$ would exhaust $\mathbb{Z}/a_2\mathbb{Z}$. Hence, $\gcd(a_1,a_2) - (D_1 - 2) \geq C_2$, i.e.,
\[
C_1 + 2 - D_1 \geq a_2 - \gcd(a_1, a_2).
\]

Now, we consider the problematic cases where $D_1 \geq C_1$:
\begin{enumerate}
    \item $a_1 = a_2$ and $C_1 = D_1$, $D_1 - 1$, or $D_1 - 2$;
    \item $a_2 - \gcd(a_1, a_2) = 2$ and $C_1 = D_1$;
    \item $a_2 - \gcd(a_1, a_2) = 1$ and $C_1 = D_1$ or $D_1 - 1$.
\end{enumerate}

Cases (ii) and (iii) imply $(a_1,a_2) = (2,4), (1,3), (1,2)$, contradicting $a_1 > 2$. In case (i), the equality $C_1 = D_1$ forces hyperellipticity at the top level with fixed poles under the involution, which we exclude.

Remaining subcases:
\begin{itemize}
    \item[(ia)] $a_1 = a_2 = b_1 = b_2$, and $C_1 = D_1 - 2 = \lceil b_1/2 \rceil - 1$;
    \item[(ia')] $a_1 = a_2 = b_1 = b_2$, and $C_1 = D_1 - 1 = \lceil b_1/2 \rceil - 1$;
    \item[(ib)] $a_1 = a_2$, $b_2 > b_1$ (i.e., $D_2 > C_1$, $C_2 > D_1$), and $C_1 = D_1 - 2$;
    \item[(ib')] same as above with $C_1 = D_1 - 1$.
\end{itemize}

We identify two persistent issues:
\begin{enumerate}
    \item[(1)] Every $(u_0,v_0)\in [(0,v)]$ with $0 \leq u_0 < C_1$ has $0 < v_0 \leq D_1$;
    \item[(2)] Every $v_0 \leq D_1 - 1$ corresponds to $u_0 \leq C_1$.
\end{enumerate}

In (ia), since $D_1 - C_1 = 2$, we cannot have $(1, D_1 - 1) \in [(0,v)]$. This would produce a prong $(C_1 + 1, 1)$ contradicting (2). As $D_1 - 1 = \lceil b_1 / 2 \rceil$, $(0, D_1 - 1)$ is the prong stated in the lemma. Case (ia') yields $(0, \pm \lfloor b_1 / 2 \rfloor)$.

For (ib) or (ib'), consider $(u_0,v_0) = (-1, D_1)$ and $(C_1 + 1, D_1 - C_1 - 1)$. Then $UR^{-2C_1}U \cdot \overline{X}(u_0, v_0) = \overline{X'}(u_0, 0)$, with:
\[
\overline{X'} = X^B_{\III a}((2,1), 0, 1, (12), (C_1', C_2 + C_1), [(u_0, 0)]).
\]

On $\overline{X'}$, level rotation changes $u_0$ by $\gcd(a_2, D_2 - C_1) = \gcd(D_1, C_1) \leq D_1 - C_1 \in \{1, 2\}$. Since $C_2 > D_1 > C_1$, we can adjust $u_0$ to obtain $(D_1, D_1 - C_1)$ or $(C_1 - 1, 0)$, solving (2) or (1) respectively.

This completes the proof.
\end{proof}

\begin{lemma}\label{lm:TypeB_IIIa_4pole_top_level_ind=1or2}
   Let $\overline{X} = X^B_{\mathrm{III}a}((2,1),0,4,\tau,\mathbf{C},[(0,v)])$ be a Type IIIa boundary stratum with B-signature $\mu^\fR = (a_1,a_2\mid -3\mid -3\mid -3\mid -3\mid -1^2)$, having the non-hyperelliptic top level. Assume $v \not\equiv 0,3 \pmod{6}$. Then there exists $\overline{X'} = X^B_{\mathrm{III}a}((2,1),0,4,\Id,\mathbf{C}',[(0,v')])$ in $\mathcal{A}(\overline{X})$ such that:
   \begin{itemize}
       \item If $v \equiv 1,4 \pmod{6}$, then $v' \equiv 1 \pmod{6}$ and $\mathbf{C}' = (1,2,2,1)$;
       \item If $v \equiv 2,5 \pmod{6}$, then $v' \equiv 2 \pmod{6}$ and $\mathbf{C}' = (1,1,2,2)$.
   \end{itemize}
\end{lemma}

\begin{proof}
    By changing the reference saddle connection, we can adjust the prong-matching by $3$, hence we may assume $(u_0, v_0) \in [(0,v)]$ is such that:
    \begin{itemize}
        \item $u_0 \equiv c_{j-1} \pmod{6}$ and $v_0 \equiv d_{j-1} + 1 \pmod{6}$, or
        \item $u_0 \equiv c_j - 1 \pmod{6}$ and $v_0 \equiv d_j \pmod{6}$,
    \end{itemize}
    for some $1 \leq j \leq 4$ depending on whether we have $v \equiv 1,4 \pmod{6}$ or $v \equiv 2,5 \pmod{6}$. By symmetry, it suffices to prove the case $v \equiv 1,4 \pmod{6}$.

    First, we show we can transform the angle assignments to a standard form relative to our reference pole. If $C_{\tau(j)} = D_{\tau(j-1)}$, then a level rotation of $C_{\tau(j)} - D_{\tau(j)} = \pm 1$ yields a prong-matching $(u_1, v_1)$ such that either:
    \[
        u_1 \equiv c_{j-1} \pm 1 \pmod{6}, \quad v_1 \equiv d_{j-1} \text{ or } d_j \pmod{6}.
    \]
    Then, $UR^{-2}U \cdot \overline{X}(u_1, v_1) = \overline{X'}(u', v')$, where $\overline{X'} = X^B_{\mathrm{III}a}((2,1),1,3,\tau',\mathbf{C}',[(u',v')])$ has the hyperelliptic top level with $\{\tau'(1), \tau'(4)\} = \{\tau(j), \tau(j-1)\}$. This allows us to exchange the positions of two top-level poles and adjust the angles. Hence, we may assume:
    \[
        (C_{\tau(j)}, C_{\tau(j-1)}, C_{\tau(j-2)}, C_{\tau(j-3)}) \in \{(1,2,2,1), (2,1,1,2)\}.
    \]
    If initially $(2,1,1,2)$ holds, then $u_0 + 3 \equiv c_{j-3} \pmod{6}$ and $v_0 - 3 \equiv d_{j-3} + 1 \pmod{6}$, so we can reset the reference pole to $p_{\tau(j-2)}$ and reduce to the case $(1,2,2,1)$.

    If $C_{\tau(j)} \ne D_{\tau(j-1)}$, then necessarily $C_{\tau(j)} = C_{\tau(j-1)}$ and $C_{\tau(j-2)} = C_{\tau(j-3)}$. Thus, the angle tuple is either $(2,2,1,1)$ or $(1,1,2,2)$. In the first case, define $(u', v') = (u_0 + 4, v_0 - 4)$. Then:
    \[
        UR^2U \cdot \overline{X}(u', v')
    \]
    swaps $C_{\tau(j-1)}$ and $C_{\tau(j-3)}$, turning $(2,2,1,1)$ into $(2,1,1,2)$. Similarly, $(1,1,2,2)$ becomes $(1,2,2,1)$. Thus, up to relabeling the poles, we can assume $\mathbf{C} = (1,2,2,1)$ and $v = 1$. We now demonstrate how to permute the poles on the top level.

    From earlier, we can swap $p_{\tau(1)} \leftrightarrow p_{\tau(2)}$ and $p_{\tau(3)} \leftrightarrow p_{\tau(4)}$ (with angle adjustments). If we change $\mathbf{C}$ to $(1,2,1,2)$, then:
    \[
        UR^{-2}U \cdot \overline{X}(u_0 + 2, v_0 - 2) = \overline{X'}(u', v'),
    \]
    where pole $p_4$ moves to the bottom level in $\overline{X'}$. Since $\gcd(6,6-3)=3$, we can rotate the levels and reset the reference prong-matching. Therefore, we can permute the top-level poles via the $3$-cycle $(134)$. By composing $(12)$ with $(134)$, we obtain the $4$-cycle $(1234)$, thus all permutations of the four poles are realizable. This completes the proof.
\end{proof}

For the inductive argument in the proof of \Cref{lm:IIIa_nonhyp_top_level_modification}, we need the following lemma to reduce the number of residueless poles on the top level.

\begin{lemma} \label{lm:IIIa_nonhyp_induction_reduce_ell}
Let $\overline{X}=X^B_{\III a}((2,1),0,n-2,\tau,{\bf C},Pr)$ be a Type~IIIa boundary stratum of B-signature 
\[
\mu^\fR = (a_1,a_2\mid -b_1 \mid \dots \mid -b_{n-2} \mid -1^2),
\]
the with non-hyperelliptic top level and $n-2 > 2$. Assume $\tau(1)=1$. If there exists a prong-matching $(u,v)\in Pr$ such that $0 \leq u < c_1$ and $d_{k-1} < v \leq d_k$ for some $k\neq 1$, then there exists $\overline{X'}$ in $\mathcal{A}(\overline{X})$ as described in case (i) of \Cref{lm:IIIa_nonhyp_top_level_modification}. In particular, if $p_1$ is not a double pole and $C_1 \leq D_1$, then such a prong-matching exists, and hence so does such $\overline{X'}$.

Otherwise, there exists a Type~IIIa boundary $\overline{X'} \in \mathcal{A}(\overline{X})$ of one of the following forms:
\begin{enumerate}
    \item $X^B_{\III a}((2,1),0,n-3,\tau',{\bf C'},Pr')$, meaning some residueless pole $p_k$ is sent to the second bottom-level component.
    \item $X^B_{\III a}((2,1),1,n-2,\tau',{\bf C'},Pr')$ with $\tau'(1)=k$ for some $k \neq 1$, meaning a residueless pole $p_k$ (not $p_1$) is sent to the first bottom-level component.
\end{enumerate}
\end{lemma}

\begin{proof}
We first show that we can always send exactly one pole to the bottom level, which will specialize to our claims in \Cref{lm:IIIa_nonhyp_top_level_modification}. Since the top level is non-hyperelliptic, there exists a prong-matching $(u,v)\in Pr$ such that $c_{k-1} \leq u < c_k$ and $d_{k'-1} < v \leq d_{k'}$ for some $k \neq k'$. 

If $D_{\tau(k)} \neq C_{\tau(k')}$, then we can send one of the poles $p_{\tau(k)}$ or $p_{\tau(k')}$ to the bottom level using $UR^2U \cdot \overline{X}(u,v)$ (or $UR^{-2}U \cdot \overline{X}(u+\tfrac{1}{2},v-\tfrac{1}{2})$), repeating this operation $\min(D_{\tau(k)}, C_{\tau(k')})$ times as described in \Cref{lm:IIIa_move}.

Now suppose $D_{\tau(k)} = C_{\tau(k')} = 1$. Applying $UR^2U \cdot \overline{X}(u,v)$ yields 
\[
\overline{X'} = X^B_{\III a}((2,1),1,n-3,\tau',{\bf C'},Pr'),
\]
where either $\tau'(1) = \tau(k)$ or $\tau'(n-3) = \tau(k')$. If $\overline{X'}$ has the non-hyperelliptic top level, then we can find a prong-matching $(u',v')\in Pr'$ such that $u' \equiv c_i' \pmod{a_2 - b_{\tau(k)}}$ (or $v' \equiv d_j' \pmod{a_1 - b_{\tau(k')}}$) but the other prong is not congruent to any $d_i$ (or $c_j$). Applying $UR^2U$ to this configuration sends one pole back to the top level, completing the step.

If the top level of $\overline{X'}$ is hyperelliptic, then on $\overline{X}$, we must have $u - c_{k-1} \neq d_{k'} - v$ or $c_k - u \neq v - d_{k'-1}$. Without loss of generality, assume $c_k - u > v - d_{k'-1}$; that is, $c_{k-1} < u + v - d_{k'-1} < c_k \pmod{a_2}$. By the proof of \Cref{Prop:Bhyper}, we can permute conjugate poles and adjust angle assignments. So we may assume $k' > k+1$. Then apply
\[
UR^{-2}U \cdot \overline{X}(u + v - d_{k'-1}, d_{k'-1}) = \overline{X''}(u'',v'').
\]
If $\overline{X''}$ has the hyperelliptic top level, then we deduce that initially $(C_i,D_i) = (1,1)$ for all $i\neq \tau(k)$. Hence $C_{\tau(k)} > 1$. If $p_{\tau(k)} = p_1$, we again reach a contradiction. Otherwise, find a prong-matching $(u_0,v_0) \in Pr$ such that $v_0 \equiv d_{k-1} \pmod{a_1}$ and $u_0 \in (c_{i-1},c_i] \pmod{a_2}$ for some $i \neq k$, which works since $D_{\tau(i)} \neq C_{\tau(k)}$.

We now proceed by induction on $n-2$. Assume there exists $(u,v)\in Pr$ such that $0 \leq u < d_1 \pmod{a_2}$ and $d_{k-1} < v \leq d_k \pmod{a_1}$ for some $k \neq 1$. Then either $p_1$ is sent to the bottom level with $q_1$, or $p_{\tau(k)}$ is sent to the level with $q_2$. In the latter case, apply the covering map 
\[
\rho^*: \mathcal{A}_{2,1}(\mu^\fR, \emptyset, (\tau(k))) \to \mathcal{A}_{2,1}(\mu'^{\fR'}, \emptyset, \emptyset)
\]
with 
\[
\mu'^{\fR'} = (a_1 - b_{\tau(k)}, a_2 \mid -b_1 \mid \dots \mid\hat{-b_{\tau(k)}} \mid \dots \mid -1^2).
\]
If $n-2 = 3$, then $\mu'^{\fR'}$ has $a_1' < a_2'$, so by \Cref{lm:TypeB_IIIa_2pole_top_level_C2=1}, we know $C_2 - D_1 > 0$. Thus, by induction, we can send $p_1$ to the bottom level. Lifting this back yields
\[
\overline{X''} = X^B_{\III a}((2,1),1,n-3,\tau'',{\bf C''},Pr'')
\]
with $\tau''(1) = 1$ and $\tau''(n-3) = \tau(k)$. Since $\kappa_1'' \geq \kappa_2''$, by \Cref{prop:IIIa_hyper}, there exists $(u'',v'') \in Pr''$ such that $c_{i-1} < u'' < c_i \pmod{\kappa_1}$ and $v'' \equiv d_j \pmod{\kappa_2''}$. Applying $UR^2U$ again returns $p_{\tau(k)}$ to the top level, completing the proof.

Finally, if $C_1 > D_1$, then counting prongs guarantees a prong-matching $(u,v)$ with $u < c_1$ and $d_{k-1}<v \leq d_k\pmod{a_1}$ for some $k \neq 1$. If $C_1 = D_1$ and the image prongs are symmetric about $p_1$, then there exists $(a, D_1 - a) \in Pr$ with $a \equiv c_{k'-1} \pmod{a_2}$ and $D_1 - a \equiv d_{k''} \pmod{a_1}$. Assume $k'' > k'$ and $C_{\tau(k'')} > D_{\tau(k')} = 1$. Then applying $UR^{-2}U \cdot \overline{X}(a, D_1 - a)$ gives
\[
\overline{X'} = X^B_{\III a}((2,1),1,n-2,\tau',{\bf C'},Pr'), \quad \text{with } \tau'(1) = \tau(k').
\]
Then, since $(0,D_1 - 1) \in Pr'$, applying $UR^2U$ gives $\overline{X''}$ with $D_1'' = D_1 - 1$. Replacing $\overline{X}$ with $\overline{X''}$ completes the argument.
\end{proof}

The above lemma actually works not very well with the case that $p_{1}$ is a double pole. This is because if we cannot find a proper prong-matching and send some other pole to the bottom level, we cannot adjust the angle assignments such that $C_{1}>D_{1}$. The next conditional lemma will be to deal with the case where we do not have a proper prong-matching and $b_1=2$, given that item (ii) of \Cref{lm:IIIa_nonhyp_top_level_modification} holds for $\overline{X'}$.

\begin{lemma}\label{lm:IIIa_proper_prong_matching_for_double_pole}
    Let $\overline{X}=X^B_{\III a}((2,1),0,n-2,\tau,{\bf C},Pr)$ be a Type IIIa boundary of a stratum of B-signature $\mu^\fR=(a_1,a_2\mid -b_1\mid\dots -b_{n-2}\mid-1^2)$ with $n-2>2$, $b_{1}=2$ and having the non-hyperelliptic top level. Assume that $\tau(1)=1$. Let $\overline{X'}=X^B_{\III a}((2,1),1,n-2,\tau',{\bf C'},Pr')$ in $\mathcal{A}(\overline{X})$ be a boundary as described in (ii) of \Cref{lm:IIIa_nonhyp_induction_reduce_ell}. If item (ii) of \Cref{lm:IIIa_nonhyp_top_level_modification} holds for $\overline{X'}$, then item (i) of \Cref{lm:IIIa_nonhyp_top_level_modification} holds for $\overline{X}$.
\end{lemma}

\begin{proof}
    Notice that we have assumed that we cannot find proper prong-matching on $\overline{X}$ such that the first image prong is inside the polar domain of $p_{1}$ (which is a double pole) while the other is not, i.e. we must have $(0,1)\in Pr$ and $a_1$ dividing $a_2$. Let $k=\tau'(1)$. We can assume that $D'_{k}=1$ and $(a,-a)\in Pr'$ for some integer $a$, i.e. $(2,-2)\in Pr'$. This is because we can consider $(a,-a+1)\in Pr$ with $|a|$ minimal so that $c_{k'-1}\leq a< c_{k'} \pmod{a_2}$ and $d_{j-1}< -a+1\leq d_j \pmod{a_1}$ (where $i\neq j$) with $D_{\tau(k)}<C_{\tau(j)}$. Then by applying $UR^{-2}UR^{-2D_{\tau(k)}+2}UR^{-2D_{\tau(k)}+2}U$ on $\overline{X}(a,-a+1)$ will reach the desired $\overline{X'}$. Notice that since $p_{1}$ is a double pole, the assumption of item (ii) for $\overline{X'}$ implies we can change the prong-matching by $$\delta^\top\coloneqq\delta((k),\emptyset)=\gcd(a_1,\{b_{\tau(i)}\}_{\substack{i=1,\dots,\ell_1\\i\neq k}})\leq 2.$$  If $a_1>2+\delta^\top$ and we can by item (ii) of \Cref{lm:IIIa_nonhyp_top_level_modification} adjust the positions and angle assignments on the top level of $\overline{X'}$ such that $d'_{j-1}< -1-\delta^\top< d'_{j}$ for some $j\neq 1$, then $UR^{2}U\cdot \overline{X'}(1,-1-\delta^\top)=\overline{X''}(1,-\delta^\top)$, where $$\overline{X''}=X^B_{\III a}((2,1),0,n-2,\tau'',{\bf C}'',Pr'').$$ Moreover, the prong-matching $(0,-\delta^\top+1)\in Pr''$ is the proper prong-matching in \Cref{lm:IIIa_nonhyp_induction_reduce_ell} and we are done. Hence, we will divide the proof into cases depending on whether $a_1>2+\delta^\top$ and how to modify the top level of $\overline{X'}$. Notice that $a_1\geq n-2$ (otherwise the stratum itself is empty) and $n-2>2$, so the following cases cover the case where $a_1>2+\delta^\top$.
    \begin{enumerate}
        \item[(1a)] $a_1\geq n-2+\delta^\top$ and $n-2\geq 3+\delta^\top$; 
        \item[(1b)] $a_1\geq n-2+\delta^\top$ and $n-2=4$ (where $\delta^\top=2)$;
        \item[(1c)] $a_1\geq n-2+\delta^\top$ and $n-2=3$;        
        \item[(2a)] $a_1= (n-2)+\delta^\top-1$ and $n-2\geq 3+\delta^\top$;
        \item[(2b)] $a_1= (n-2)+\delta^\top-1$ and $n-2=3+\delta^\top-1$ (where $\delta^\top=2$);
        \item[(3)] $a_1= (n-2)+\delta^\top-2$ and $n-2\geq 3+\delta^\top$ (where $\delta^\top=2$).
    \end{enumerate}
    
Let $p_{k_1}$ be a largest pole on the top level of $\overline{X'}$. Before we start, we want to remark that in any listed case, $b_{k_1}\geq 2+\delta^\top$. This is because $b_{k_1}>2$ and if $b_{k-1}=1+\delta^\top$, then $\delta^\top$ is forced to be $1$. 

    In case (1a), if $b_{k_1}\geq 3+\delta^\top $, then we can modify $\overline{X'}$ such that $\tau'(n-2)=k_1$ and $D'_{k_1}\geq 2+\delta^\top$. Then $-D'_{k_1}< -1-\delta^\top< 0\pmod{a_1}$, and by our aforementioned operations we are done. Assume now  $b_{k_1}=2+\delta^\top$. Let $p_{k_2}$ be a second largest pole on the top level of $\overline{X'}$. Then $b_{k_2}= 1+\delta^\top$ or $2+\delta^\top$, otherwise $a_1\leq n-3+\delta^\top$ which contradicts $a_1\geq n-2+\delta^\top$. We modify $\overline{X'}$ such that $\tau'(n-2)=k_2$ and $\tau'(n-3)=k_1$ with $D'_{k_1}=1+\delta^\top$. If $D'_{k_2}\leq \delta^\top$, then $-D'_{k_1}-D'_{k_2}< -1-\delta^\top <-D'_{k_2}\pmod{a_1}$, otherwise, $D'_{k'}=1+ \delta^\top$ and we have similarly $-D'_{k_1}-D'_{k_2}< -1-2\delta^\top < -D'_{k_2}\pmod{a_1}$. Then by the aforementioned operation, we are also done. Notice that the treatment for the situation $b_{k_1}\geq 3+\delta^\top$ works for any listed cases if one also has $b_{k_1}\geq 3+\delta^\top$. In addition, the treatment for the situation $b_{k_1}=2+\delta^\top$ actually works for any listed cases if one also has $n-2\geq 4$, $b_{k_1}=2+\delta^\top$ (even if $b_{k_2}=\delta^\top=2$).
    
    In case (1b), we can use the exact same argument from (1a). Similarly, in case (1c), as $a_1\geq 3+\delta$ and the top level of $\overline{X'}$ only contains $p_1$ (a double pole) and $p_{k_1}$, we must have $b_{k_1}\geq 3+\delta^\top$. And it can use the same argument in (1a).

    In case (2a) and (2b), we have $n-2\geq 4$. The argument will be the same as (1a). 
        
    Now for case (3), we should have $n-2\geq 5$ and $a_1=n-2$. We can modify the top level of $\overline{X'}$ such that $\tau'(n-4)=k_1$ with $D'_{k_1}=2$. Then $ d'_{n-6}<-1-\delta^\top<d'_{n-5}\pmod{a_1}$ and we are done by the aforementioned operation. 

    Now we may consider the situation $a_1\leq 2+\delta^\top$. Taking into account that $a_1\geq n-2$ and $n-2>2$, it can be divided into the following:
    \begin{enumerate}
        \item[(4a)] $a_1=3$, $n-2=3$ and $\delta^\top=1$;
        \item[(4b)] $a_1=4$, $n-2=3$ and $\delta^\top=2$.
    \end{enumerate}
    
    Notice that since $a_1=2+\delta^\top$ in the cases (4a) and (4b), we can modify $\overline{X'}$ such that $(0,2)\in Pr'$ on $\overline{X'}$. Since $d'_1\equiv 1 \pmod{a_1}$, we have $UR^2U\cdot \overline{X'}(0,2)=\overline{X''}(-b_k+1,2)$ where $\overline{X''}=X^B_{\III a}((2,1),0,n-2,\tau'',{\bf C}'',Pr'')$ such that $\tau''(1)=1$ and $\tau''(3)=k$. If $-b_k+3\neq 1 \pmod{a_1}$, then $(0,-b_k+3)\in Pr''$ a proper prong-matching in \Cref{lm:IIIa_nonhyp_induction_reduce_ell} and we are done. Otherwise, $b_k\equiv 2\pmod{a_1}$, i.e. either $b_k=2$; or $b_k\geq 5$ (or $b_k\geq 6$) if $\delta^\top=1$ (resp. $\delta^\top=2)$. 
    
    We may first assume $a_1<a_2$ (i.e. $a_2\geq 2a_1$). Then $C''_{k}\geq a_1$ if $b_k\neq 2$, otherwise $C''_{\tau''(2)}= a_2-2\equiv -2 \pmod{a_1}$. We consider either $(-b_k+2,1)\in Pr''$ or $(a_2-2,1)\in Pr''$. For the former situation, we have $UR^{-3}U\cdot \overline{X''}(-a_1,1)=\overline{X'''}(2+\frac{1}{2},\frac{1}{2})$, where $\overline{X'''}=X^B_{\III a}((2,1),0,n-2,\tau''\circ (23),{\bf C}'',Pr''')$. Notice that $(0,3)\in Pr'''$ is a proper prong-matching in \Cref{lm:IIIa_nonhyp_induction_reduce_ell}. For the latter situation, we have $UR^{2D''_{\tau''(2)}+1}U\cdot \overline{X''}(a_2-2+\frac{1}{2},\frac{1}{2})=\overline{X'''}(-D''_{\tau''(2)},1)$, where $\overline{X'''}=X^B_{\III a}((2,1),0,n-2,\tau''\circ (23),{\bf C}'',Pr''')$. In this case $(0,-D''_{\tau''(2)}+1)\in Pr'''$ is a proper prong-matching in \Cref{lm:IIIa_nonhyp_induction_reduce_ell}.
    
     Now if $a_2=a_1$, then we can only be in case (4b), and have $b_1=b_2=2$, $b_3=4$. We may W.L.O.G. assume that $\tau=\Id$ on $\overline{X}$. Since we do not have a proper prong-matching, we should have $(0,1)\in Pr$. In this case, we have $UR^{-2}UR^4UR^{-2}U\cdot \overline{X}(1,0)=\overline{X'''}(1,0)$, where $\overline{X'''}=X^B_{\III a}((2,1),1,2,(123),(1,2,1),[(1,0)])$. We can send back the poles on the bottom level differently by considering $UR^2UR^4UR^2U\cdot \overline{X'''}(0,1)=\overline{X''''}(-2,1)$ where $\overline{X''''}$ only differs from $\overline{X}$ by the prong-matching equivalence class. On $\overline{X''''}$ we can find the proper prong-matching satisfying \Cref{lm:IIIa_nonhyp_induction_reduce_ell}. This completes our proof.   
\end{proof}

In the inductive step of the proof of \Cref{lm:IIIa_nonhyp_top_level_modification}, we need to compare different Type IIIa boundaries $\overline{X'}$ with the smallest residueless pole on the top level of $\overline{X}$ sent to the bottom level. The following lemma ensures that if \Cref{lm:IIIa_nonhyp_top_level_modification} holds on $\overline{X'}$, then we can easily modify the prong-matching and the $C_{\tau(\ell_1)}$ of the pushed pole whenever the pseudo-index remains unchanged.

\begin{lemma} \label{cor:IIIa_nonhyp_top_level_add_twoprongs}
   Let $\overline{X}=X^B_{\III a}((2,1),1,n-2,\tau,{\bf C},[(0,v)])$ and $\overline{X'}=X^B_{\III a}((2,1),1,n-2,\tau,{\bf C'},[(0,v')])$ be Type IIIa boundaries of some stratum of B-signature $\mu^\fR=(a_1,a_2\mid -b_1\mid\dots -b_{n-2}\mid-1^2)$ having the non-hyperelliptic top levels. In addition, $\tau'(1)=\tau(1)=1$. Assume that item (ii) of \Cref{lm:IIIa_nonhyp_top_level_modification} holds for both $\overline{X}$ and $\overline{X'}$. Then $\overline{X'}$ is in $\mathcal{A}(\overline{X})$ if $\delta((1),\emptyset)\neq a_1 $, or $v,v'\neq 0\pmod{a_1}$; and $C_{1}-v\equiv C'_{1}-v'\pmod{\delta}$.
         
\end{lemma}

\begin{proof}
   By the assumption that (ii) of \Cref{lm:IIIa_nonhyp_top_level_modification} holds on $\overline{X}$ and $\overline{X'}$, we may assume that $\tau=\tau'=\Id$. Let $\delta^{\top} \coloneqq \delta((1),\emptyset)= \gcd(a_1,\{b_{i}\}_{i=2,\dots,n-2})$. Notice that $\delta=\gcd(\delta^{\top},a_2)$.  If some pole $p_{i}$ on the top level of $\overline{X}$ and $\overline{X'}$ is a double pole, then by assumption, we have $b_{1}=2$. In this case, we always have $C_{1}=C'_{1}=1$ and $v\equiv v' \pmod \delta$. In addition, $\delta=\delta^{\top}$ and thus we should have $v\equiv v'\pmod{\delta^\top}$. In this case, $\overline{X'}$ is in $\mathcal{A}(\overline{X})$ by (ii) of \Cref{lm:IIIa_nonhyp_top_level_modification}. So we can reduce to the case when all the poles on the top level are non-double poles. If $b_{1}=\delta$ and $v\equiv v' \pmod{\delta^\top}$, then it is also immediate that $C_{1}=C'_{1}$ and we are done by (ii) of \Cref{lm:IIIa_nonhyp_top_level_modification}.
 
 Since $(C_1-v)-(C'_1-v')\equiv 0 \pmod{\delta}$, we have $(C_1-v)-(C'_1-v')=k_1 \delta^{\top}+k_2 a_2$ for some integers $k_1,k_2$. We claim that we can in general assume $k_1=0$. Indeed, we consider the equality in modulo $a_1$: $$(C_1-v)-(C'_1-v')\equiv k_1 \delta^{\top}+k_2 a_2 \pmod{a_1}.$$ This means that if $v+\widetilde{k_1}\delta^\top\neq 0\pmod{a_1}$ and $v'-(k_1-\widetilde{k_1})\delta^\top\neq 0 \pmod{a_1}$ for some integer $\widetilde{k_1}$, then by changing $v$ resp. $v'$ to $v+\widetilde{k_1}\delta^\top$ resp. $v'-(k_1-\widetilde{k_1})\delta^\top$, we reduce to the case $k_1=0$. This is only impossible if $a_1=2\delta^\top$ and $v\equiv v'\equiv \delta^\top\pmod{a_1}$. This exceptional case will be dealt later.

Assuming that $k_1=0$ as above. Now we will show that if $n-2\leq 3$ or $a_1\leq \sum_{i=2}^{n-2}(b_i-1)-2$, then there are angle assignments ${\bf C}''=(C_1,C''_2,\dots,C''_{n-2})''$ resp. ${\bf C}'''=(C'_1,C'''_2,\dots, C'''_{n-2})$ such that :
 \begin{itemize}
     \item $C''_i=C_i'''$ for all $i\geq 2$ and $d_{j-1}''< v\leq v'<d_{j}''\pmod{a_1}$; or
     \item $C''_i=C_i'''=C_i$ for all $i\geq 2$, except $C_j''=C_j'''+1$ and $C_k''=C_k'''-1$ for some $2\leq j< k \leq n-2$. We have also $d_{j-1}''< v<d_{j}''\pmod{a_1}$ and $d_{k-1}'''< v'<d_{k}'''\pmod{a_1}$.
 \end{itemize}
  Then, we can modify $\overline{X}$ (and $\overline{X'}$) by the angle assignments ${\bf C}''$ (and ${\bf C}'''$ respectively), and have $UR^2U\cdot \overline{X}(0,v)=\overline{Y}(D_1-1,v+1)$ resp. $UR^2U\cdot \overline{X'}(0,v')=\overline{Y}(D'_1-1,v'+1)$, where $\overline{Y}=X^B_{\III a}((2,1),0,n-2,\tau,{\bf C},Pr'')$. The case for $n-2=2$ is trivial on our assumption of $v,v'$ and $\delta^\top$. If $n-2=3$, then we can assume there is at most one of $v$ and $v'$ equivalent to $d_1 \pmod{a_1}$. First, in the case of $a_2-b_1=2$, either $\delta^\top=1$ (one can observe that for $a_1=3,4,5$) or $a_1\geq 3\cdot 2$. By shifting $v$ and $v'$ simultaneously by $\pm \delta^\top$ (which is $\leq 2$), we are done. Second, in the case of $a_2-b_1=3$, then we can easily adjust $C_2$ by one. The procedure according to the situations are as follows and we are done with setting appropriate ${\bf C}''$ and ${\bf C}'''$:
 \begin{itemize}
     \item If $0< v< b_1-1\pmod{a_1}$ and $b_1-1\leq v'<0\pmod{a_1}$, then we can set $C''_1=b_1-1$ and $C'''_1=b_1-2$.
     \item If $0<v\leq v' < b_1-1$, then we can set $C''_1=C'''_1=b_1-1$.
     \item If $b_1-1\leq v\leq v'<a_1$, then we can set $C''_1=C'''_1=b_1-2$.
 \end{itemize}
  We now give the general procedure for $a_1\leq \sum_{i=2}^{n-2}(b_i-1)-2$. Let $N$ be the largest number $1\leq N\leq n-3$ such that $\sum_{i=1}^N(b_{1+i}-1)\leq a_1-(n-3-N)$. Then we can first modify $\bf C$ such that 
 $$C_{1+i}=\begin{cases}
     b_{1+i}-1 & \mbox{ for } 1\leq i\leq N,\\
     1+a_1-(n-3-N)-\sum_{j=1}^N(b_{1+j}-1) &\mbox{ for } i=N+1,\\
     1&\mbox{ for } N+2\leq i\leq n-3.
 \end{cases} $$
We remark that either $N+1<n-2$ or $C_{n-2}\leq b_{n-2}-3$ by our assumption on $a_1$. If $v\equiv d_j\pmod{a_1}$, then we can decrease $C_2$ by one while increasing $C_{\max(j+2,N+2)}$ by one, in order to make $v$ not equal to any $d_i$. Notice that after the change, by the remark, we still have $C_{n-2}\leq b_{n-2}-2$. Then, by assuming $d_{j-1}<v<d_{j}$, the remaining procedure according to cases are as follows:
\begin{itemize}
    \item If $d_{k-1}<v'\leq d_k\pmod{a_1}$ for some $k\geq j$ such that $C_{k+1}\leq b_{k+1}-2$, then we can set ${\bf C}''={\bf C}$. And we set ${\bf C}'''={\bf C}$ except $C'''_{j+1}=C_{j+1}-1$ and $C'''_{k+1}=C_{k+1}+1$.
    \item If $d_{k-1}< v'\leq d_{k}\pmod{a_1}$, for some $k\geq j$ and $C_{k+1}= b_{k+1}-1$, then we can decrease $C_{k+1}$ by one and increase $C_{n-2}$ by one. Then the procedure of last item applies.
\end{itemize}
Now we have already give the proof of the lemma for the case where $k_1$ can be reduced to $0$, and $n-2\leq 3$ or $a_1\leq \sum_{i=2}^{n-2}(b_i-1)-2$.

 There are some cases still not yet covered, namely the cases $a_1= \sum_{i=2}^{n-2}(b_i-1)-1$ (i.e. $a_2-b_1=n-2$) or $a_1= \sum_{i=2}^{n-2}(b_i-1)$ (i.e. $a_2-b_1=n-3$) with $n-2>3$. Since $a_2\geq a_1$, if we substitute the conditions into the inequality, we obtain:
 \begin{align}
    2(n-3)+2 &\geq \sum_{i=3}^{n-2}b_i +(b_{2}-b_{1})\Rightarrow 2\cdot \frac{n-2}{n-4}\geq \frac{\sum_{i=3}^{n-2}b_i}{n-4}+\frac{b_{2}-b_{1}}{n-4} \\
    2(n-3)&\geq \sum_{i=3}^{n-2}b_i +(b_{2}-b_{1})\Rightarrow 2\cdot \frac{n-3}{n-4}\geq \frac{\sum_{i=3}^{n-2}b_i}{n-4}+\frac{b_{2}-b_{1}}{n-4}.
 \end{align}
 Notice that $\frac{\sum_{i=3}^{n-2}b_i}{n-4}$ is the average of the pole orders of poles $p_3,p_4,\dots,p_{n-2}$. Since we assume that there is no double pole on the top level of $\overline{X}$ and $\overline{X'}$, the average order should be at least $3$ and thus $n-2\leq 6$ resp. $n-2\leq 4$ for (8.20) resp. (8.21). For (8.21), there is a unique case where $b_1=b_2=b_3=b_4=a_2=3$, but this does not satisfy the item (ii) of \Cref{lm:IIIa_nonhyp_top_level_modification}. Hence, we only need to consider the following explicit possibilities for (8.20):
 \begin{itemize}
     \item[(1)] $n-2=4$ with $b_{4}=4$: $(b_4,b_3,b_2,b_1,(a_2,a_1))=(4,4,4,4,(8,8)),(4,4,3,3,(7,7)),(4,3,3,3,(7,6)),$ $(4,3,3,2,(6,6)), (3,3,3,3,(7,5))$ or $(3,3,3,2,(6,5))$.
     \item[(2)] $n-2=4$ with $b_4=3$: $(b_4,b_3,b_2,b_1,(a_2,a_1))=(3,3,3,3,(7,5))$ or $(3,3,3,2,(6,5))$.
     \item[(3)] $n-2=5$: $(b_5,b_4,b_3,b_2,b_1,(a_2,a_1))=(4,3,3,3,3,(8,8)),(3,3,3,3,3,(8,7))$ or $(3,3,3,3,2,(7,7))$.
     \item[(4)] $n-2=6$: $(b_6,b_5,b_4,b_3,b_2,b_1,(a_2,a_1))=(3,3,3,3,3,3,(9,9))$.
 \end{itemize}
 Notice that if we are in the case $a_1= \sum_{i=2}^{n-2}(b_i-1)-1$ (i.e. $a_2-b_1=n-2$) and go through the procedure of setting up ${\bf C},{\bf C}'',{\bf C}'''$ as that of $a_2-b_1>n-2$, then the only obstacle would be $v\equiv\sum_{i=1}^j (b_{1+i}-1)\pmod{a_1}$ for some $1\leq j\leq N-3$ while $\sum_{i=1}^{j+1}(b_{1+i}-1)< v'<0\pmod{a_1}$. This means that if we can shift $v,v'$ simultaneously such that $v\neq\sum_{i=1}^j (b_i-1)\pmod{a_1}$ for any $j$ and $v\leq v'<0\pmod{a_1}$, then we can also find the appropriate ${\bf C}''$ and ${\bf C}'''$. On the list of possibilities for (8.20), the cases with $\delta^\top=1$ are obvious that we can shift $v,v'$. Hence, the only exceptional cases are $(4,4,4,4,(8,8))$ and $(3,3,3,3,3,3,(9,9))$. Notice that $b_1=\delta=\delta^\top$ in these cases, so we only need to consider the cases where $v\neq v'\pmod{\delta^\top}$. One can observe that for such $v,v'$, we can always shift $v,v'$ simultaneously to avoid having bad situation.

 It remains now to deal with the exceptional case for reducing $k_2$ to be $0$, i.e. we have $a_1=2\delta^\top$ with $v,v'\equiv 0\pmod{\delta^\top}$. Notice that the case of $b_{1}=2$ will be trivial because it will imply that $C_1=C'_1$ and we can make $v=v'$. Now, it suffices to show the lemma for the specific situation where $a_1=2\delta^\top$ and $(C'_1,v')=(C_{1}\pm 1,v\pm 1)$, because then we can use the tricks similar to the proof for general cases, by designing the suitable ${\bf C}''$ and ${\bf C}'''$. From the listing of cases $a_2-b_1=n-3$ or $a_2-b_1=n-2$, we know that there is no case satisfying $a_1=2\delta^\top$ plus that we cannot shift $v,v'$. Hence, we can assume that we can apply $UR^2U$ on $\overline{X}$ resp. $\overline{X'}$, reaching some Type IIIa boundaries $\overline{Y}$ resp. $\overline{Y'}$ such that the angle assignments of $\overline{Y}$ and $\overline{Y'}$ coincide. Since we have assumed $(C'_{1},v')=(C_{1}\pm 1,v\pm 1)$, $\overline{Y}$ and $\overline{Y'}$ must coincide as well. This completes the proof.
 
\end{proof}

The following lemma can be proved by similar arguments in \Cref{cor:IIIa_nonhyp_top_level_add_twoprongs}. It will allow us to send back some pole to the top level, in the induction of the proof of \Cref{lm:IIIa_nonhyp_top_level_modification}. 

\begin{lemma} \label{cor:IIIa_nonhyp_top_level_add_one}
Let $\overline{X}=X^B_{\III a}((2,1),1,n-2,\tau,{\bf C},[(0,v)])$ be Type IIIa boundaries of some stratum of B-signature $\mu^\fR=(a_1,a_2\mid -b_1\mid\dots -b_{n-2}\mid-1^2)$ having the non-hyperelliptic top level. Assume that item (ii) of \Cref{lm:IIIa_nonhyp_top_level_modification} holds for both $\overline{X}$. If $\sum_{i=2}^{n-2}(b_{\tau(i)}-1)>a_1> n-3$, and $a_1\neq \delta^\top$ or $v\neq0\pmod{a_1}$, then there exists some $\overline{X'}$ in $\mathcal{A}(\overline{X})$ where $\overline{X}=X^B_{\III a}((2,1),0,n-2,\tau',{\bf C}',Pr')$.
\end{lemma}

\begin{proof}
     If $a_1\neq \delta^\top$, by item(ii) \Cref{lm:IIIa_nonhyp_top_level_modification}, we can always assume $(0,v)\in Pr$ for some $v\neq 0\pmod{a_1}$. Suppose that there exists $(u_0,v_0)\in [(0,v)]$ such that $u_0\equiv c_i\pmod{a_2-b_{\tau(1)}}$ and $d_{j-1}<v_0<d_j\pmod{a_1}$ for some $i,j$. Then $UR^2U\cdot \overline{X}(u_0,v_0)=\overline{X'}(u',v')$, where $\overline{X}=X^B_{\III a}((2,1),0,n-2,\tau',{\bf C}',Pr')$ and we are done. In the lemma, we also have excluded the case where $a_1=\sum_{i=2}^{n-2}(b_i-1)$. Thus, if we do not have any double pole on the top level, then by the argument of adjusting the angle assignments (and shifting $v$) as in the proof of \Cref{cor:IIIa_nonhyp_top_level_add_twoprongs}, we can find some proper angle assignments on the top level such that $c_{k-1}<v<c_{k}\pmod{a_1}$. By applying $UR^2U$ on $\overline{X}(0,v)$ then we are also done. 
    
    The cases not covered by the argument of angle adjustments in \Cref{cor:IIIa_nonhyp_top_level_add_twoprongs} is if there is some double pole on the top level. However, then we will have $\delta^\top\leq 2$ and we can assume $v\equiv -1\pmod{a_1}$ or $-2\pmod{a_1}$. If $v\equiv -1\pmod{a_1}$, then we are fine because $b_{\tau(n-2)}\geq 3$ and we can make $D_{n-2}\geq 2$. For the case $v\equiv -2\pmod{a_1}$, we can modify the positions of the poles such that $b_{\tau(n-2)}=2$ and $\tau(n-3)=n-2$. Then by setting $D_{n-2}\geq 2$, we are also done.

\end{proof}

Finally, we are ready to prove \Cref{lm:IIIa_nonhyp_top_level_modification} by induction on $\ell_2-\ell_1 \geq 1$. Notice first that if we have proved \Cref{lm:IIIa_nonhyp_top_level_modification} for the boundaries $X^B_{\III a}((2,1),0,\ell_2-\ell_1,\tau',{\bf C}',Pr') $ of the stratum of B-signature of the form $\mu'^{\fR'}=(a_1',a_2'\mid b_1'\mid\dots\mid b_{\ell_2-\ell_1}'\mid -1^2)$, then we have also proved the lemma for boundaries $\overline{X}=X^B_{\III a}((h_1,h_2),\ell_1,\ell_2,\tau,{\bf C},Pr)$ of arbitrary stratum of B-signature $\mu^\fR$. Namely, we have the covering map $\rho^*:\mathcal{A}_{h_1,h_2}(\mu^\fR,\mathcal{I}_1,\mathcal{I}_2)\supset \mathcal{A}(\overline{X}) \rightarrow\mathcal{A}(\overline{X'})\subset \mathcal{A}_{h_1',h_2'}(\mu'^{\fR'},\emptyset,\emptyset)$, where $(h_1',h_2')=(2,1)$ if and only if $\kappa_1\geq \kappa_2$. The value of $(h_1',h_2')$ is exactly the reason why we have to consider the relation of $\kappa_1,\kappa_2$ in item (i) of \Cref{lm:IIIa_nonhyp_top_level_modification}. By \Cref{lm:IIIa_nonhyp_induction_reduce_ell}, we can reduce the number $\ell_2-\ell_1$ of poles on the top level component and apply induction hypothesis. In particular, we will use \Cref{cor:IIIa_nonhyp_top_level_add_twoprongs} as a part of induction hypothesis.

\begin{proof}[Proof of \Cref{lm:IIIa_nonhyp_top_level_modification}]
    As we have discussed above, we may use the reduction to strata of simpler B-signature to simplify our arguments in the proof. We may assume that $\overline{X}=X^B_{\III a}((2,1),0,n-2,\tau,{\bf C},Pr)$ is boundary of a stratum of B-signature $\mu^{\fR}=(a_1,a_2\mid b_1\mid\dots\mid b_{n-2}\mid -1^2)$, and $\tau(1)=1$. And we use induction on $n-2>0$. For $n-2=1$, it is trivial because item (i) start with $n-2=2$ while item (ii) is just due to the level rotation.
    
   Now assume $n-2=2$, then by \Cref{lm:TypeB_IIIa_2pole_top_level_C2=1}, we can make $C_2=1$ except for the exceptional case (a3) (which is included in (a1)). In addition, if we are not in the exceptional case (a2), then by level rotation, we can surely find some prong-matching $(u,v)$ such that $0\leq u < c_1\pmod{a_2}$ while $d_1< v \leq 0\pmod{a_1}$. Thus, $UR^{-2}U\cdot \overline{X}{(u,v)}$ will reach the desired $\overline{X'}$ in item (i) of the lemma. Item (ii) follows item (i) by \Cref{cor:IIIa_nonhyp_top_level_add_twoprongs}. Although item (i) does not hold for case (a1), the level rotation action is trivial, and \Cref{lm:TypeB_IIIa_2pole_top_level_C2=1} shows that except for the case (a3), angle assignments can be adjusted freely. Hence, item (ii) also holds for (a1) except (a3). 
   
    Now we assume that $n-2>2$ and we will first prove item (i). By \Cref{lm:IIIa_nonhyp_induction_reduce_ell}, if there is $(u,v)\in Pr$ such that $0\leq u < c_1\pmod{a_2}$ while $d_{k-1}< v \leq d_k\pmod{a_1}$ for some $k\neq 1$, then we are done with item (i). Otherwise, we can obtain another element $\overline{Y}$ in $\mathcal{A}(\overline{X})$ with $n-3$ poles on the top level component such that $p_{1}$ remains on the top level. We can have $\overline{Y}=X^B_{\III a}((2,1),1,n-2,\tau',{\bf C}',Pr')$ or $\overline{Y}=X^B_{\III a}((2,1),0,n-3,\tau',{\bf C}',Pr')$. First, assume the former case. W.L.O.G. we can assume $(0,v')\in Pr'$ with $d'_{k'-1}<v'<d'_{k'}\pmod{a_1}$ for some $k'$. We now assume that item (ii) holds on $\overline{Y}$. If $b_{1}=2$, the trick to prove item (i) is already given in \Cref{lm:IIIa_proper_prong_matching_for_double_pole}. Now we assume $b_{1}>2$. If $a_2-b_k\geq n-4+\frac{b_1}{2}$, then by induction hypothesis, we may modify $\overline{Y}$ such that $C'_{1}\geq D'_{1}$. Since we have assumed that item (ii) holds on $\overline{Y}$, then by \Cref{cor:IIIa_nonhyp_top_level_add_one} (or \Cref{prop:IIIa_hyper}) we can send back $p_k$ to the top level. This is because, we obtain $\overline{Y}$ by sending exactly one pole to the bottom level, so the equatorial arc approaching $\overline{Y}$ will never be $\overline{Y}(0,0)$. Now we can reduce to the case where $C_1\geq D_1$, which by \Cref{lm:IIIa_nonhyp_induction_reduce_ell} we are done with item (i). There are some cases where item (ii) will not hold on $\overline{Y}$. However, if we observe the exceptional cases carefully (see also \Cref{lm:TypeB_IIIa_2pole_top_level_C2=1}), we can make $D'_1-C'_1\leq 2$. And there is a proper prong-matching $(u',v')\in Pr'$ such that $u'\equiv c_i \pmod{a_2-b_k}$ and $0< v'<d_1 \pmod{a_1}$ such that $UR^2U\cdot\overline{Y}(u',v')=\overline{X''}(u'',v'')$, where $\overline{X''}=X^B_{\III a}((2,1),0,n-2,\tau'',{\bf C}'',Pr'')$ and $C''_1\geq D''_1$, so we are done with the case $a_1\leq a_2-b_k$. If now $a_2-b_k<n-3+\frac{b_1}{2}$, i.e. $a_2-b_k<a_1$, then by induction hypothesis, we can reach $\overline{Y'}=X^B_{\III a}((2,1),1,n-3,\tau'',{\bf C}'',[(0,v'')])$ in $\mathcal{A}(\overline{X})$ such that $\tau''(n-2)=1$ and $\tau''(1)=k$. By \Cref{cor:IIIa_nonhyp_top_level_add_one}, we can send $p_k$ back to the top level component. This will be reduced to the other case of $\overline{Y}$ we obtained in the beginning. This will be treated later. 
  
    Now assume that we have $\overline{Y}=X^B_{\III a}((2,1),0,n-3,\tau',{\bf C}',[(0,v')])$. If $\tau'(n-2)\neq 1$, then the top level component $Y_0$ contains $p_{1}$. Now, if item (i) holds on $\overline{Y}$, we can obtain $\overline{Y'}=X^B_{\III a}((2,1),1,n-3,\tau'',{\bf C}'',[(0,v'')])$ such that $\tau''(1)=1$. By \Cref{prop:IIIa_hyper}, we can find proper prong-matching and thus by applying $UR^2U$ move $p_{\tau''(n-2)}$ back to the top level. If item (i) does not hold on $\overline{Y}$, then we may assume $Y$ is in (a1), (a2) or (c). However, since $a_2>a_1-b_{\tau'(n-2)}$, (a1) and (c) cannot happen. Exceptional case (a2) means that $b_1=2$, so \Cref{lm:IIIa_proper_prong_matching_for_double_pole} guarantees item (i) on $\overline{X}$. We now consider the case $\tau'(n-2)=1$ for $\overline{Y}$. W.L.O.G., we can assume $\tau'(1)=2$. By induction hypothesis, we can reach $\overline{Y'}=X^B_{\III a}((2,1),1,n-3,\tau'',{\bf C}'',[(0,v'')])$ in $\mathcal{A}(\overline{X})$ such that $\tau''(1)=2$ and $\tau''(n-2)=1$. By induction hypothesis and \Cref{cor:IIIa_nonhyp_top_level_add_one}, we can move $p_{1}$ back to the top, i.e. we are reduced to the case when $\overline{Y}=X^B_{\III a}((2,1),1,n-2,\tau',{\bf C'},[(0,v')])$ where $\tau'(1)=2$. If $a_1-b_{2}\geq \ell-2+\frac{b_{\ell}}{2}$, then we are done by the argument in last paragraph. We now assume that $a_2-b_{2}< n-3+\frac{b_{1}}{2}$, i.e. $a_2<n-3+\frac{b_1}{2}+b_2$. That is, $$a_1=\sum_{i=1}^{n-2} b_i -a_2 > \sum_{i=1}^{n-2} b_i -(n-3)-\frac{b_{1}}{2}-b_{2}.$$ Since $a_1\leq a_2$, we have $$\sum_{i=1}^{n-2} b_i -(n-3)-\frac{b_{1}}{2}-b_{2} <  n-3+\frac{b_{1}}{2}+b_{2}.$$
    Therefore, $b_{2} > \sum_{i=3}^{n-2}(b_i-2)\geq (n-4)(b_{2}-2)$. Thus $b_{2}<2+\frac{2}{n-5}$. As item (i) for the case where $b_2=2$ (i.e. $b_1=2$ as well) has been handled by \Cref{lm:IIIa_proper_prong_matching_for_double_pole}, we should have $n-2=3$ or $4$. If $n-2 =4$, then $b_1=b_2=b_3=b_4=3$ with $a_2=a_1=6$, and we only consider the case where $v\neq 0,3 \pmod{6}$. By \Cref{lm:TypeB_IIIa_4pole_top_level_ind=1or2}, item (ii) holds already on $\overline{X}$, i.e. we can have $C_1>D_1$ which guarantees item (i) by \Cref{lm:IIIa_nonhyp_induction_reduce_ell}. If $n-2=3$, then $b_2>b_3-2$ and thus $b_2=b_3$ or $b_2=b_3-1$. If $b_2=b_3-1$, then $a_1=b_3-1+\frac{b_1}{2}$ and $a_2=b_3+\frac{b_1}{2}$. That is, $\gcd(a_1,a_2)=1$ and we can find the proper prong-matching on $\overline{X}$ as in \Cref{lm:IIIa_nonhyp_induction_reduce_ell}. If $b_2=b_3$, then $a_1=b_3-1+\frac{b_1}{2}$ and $a_2=b_3+1+\frac{b_1}{2}$. So we have $\gcd(a_1,a_2)\leq 2$ and again we can find a proper prong-matching to apply \Cref{lm:IIIa_nonhyp_induction_reduce_ell}.

    Now we will prove the second item. Assume that $n-2>2$. By the first item to $\overline{X}$, we can obtain and obtain $\overline{Y}=X^B_{\III a}((2,1),1,n-2,\tau',{\bf C'},[(0,v')])$ in $\mathcal{A}(\overline{X})$, where $\tau'(1)=1$ and $v'\neq 0$. We also have $\operatorname{Ind}(\overline{Y})\equiv C'_1-v' \pmod \delta$. Then if item (ii) holds on $\overline{Y}$, then by \Cref{cor:IIIa_nonhyp_top_level_add_twoprongs} we are done. Now we may assume that $\overline{Y}$ falls into the exceptional cases (a3) or (b) or (c), we only need to consider the following cases of $\overline{X}$:
    \begin{enumerate}
        \item[(1)] $n-2=3$ and $b_1=2$ and $b_1=b_2=a_1$ while $a_2=a_1+2$;
        \item[(2)] $n-2=3$ and $b_1\geq 3$ and $b_2=b_3=a_1$ while $a_2=a_1+b_1$;
        \item[(3)] $n-2=4$ and $b_1=b_2=b_3=b_4=3$ with $a_1=a_2=6$ ($b_1=2$ will not fall into (b) because $3+2<6$ which means that it is not in our reduction);
        \item[(4)] $n-2=5$ and $b_1=b_2=b_3=b_4=b_5=3$ with $a_2=a_1+3=9$
        \item[(5)] $n-2=5$ and $b_1=b_2=b_3=b_4=3$, $b_5=2$ with $a_2=a_1+2=8$
    \end{enumerate}

    In case (1), we have $\delta=\gcd(a_1,a_2)$. And by level rotation, we can change the prong-matching by $2$. Notice that by the value of $a_2$ and $b_2,b_3$, we should always have $C_2,C_3>1$. This means that by considering an appropriate prong-matching, we can adjust the angle assignments for $p_2$ and $p_3$ freely. Furthermore, we can send $p_2$ to the bottom level component containing $q_1$, and obtain a top level satisfying item (ii) of the lemma. Hence, we can send $p_2$ back in another position. This directly shows that we can modify $\overline{X}$.

    In case (2), we can assume that $b_2=b_3=a_1\geq 4$ because otherwise the top level of $\overline{X}$ will be already the exceptional case (b). Assume that the prong-matching equivalence class on $\overline{Y}$ is $[(0,\lfloor a_1/2\rfloor)]$ (notice that $\lfloor a_1/2\rfloor\geq 2 $). By \Cref{lm:TypeB_IIIa_2pole_top_level_C2=1}, we can make $D'_2$ or $D'_3$ to be $a_1-1$. W.L.O.G., we assume $D'_2=a_1-1$. Then $UR^2U\cdot \overline{Y}(0,\lfloor a_1/2\rfloor)=\overline{X'}(D'_1-1,\lfloor a_1/2\rfloor+1)$, where we send back $p_1$ to the top level. If $D'_1>1$, then by applying $UR^{-2}U$ on $\overline{X'}(D_1'-2,\lfloor a_1/2\rfloor+2)$ we can send again $p_1$ to the bottom level but the prong-matching on the new boundary $\overline{Y'}$ is non-exceptional. Otherwise, if $a_1$ is even, then we can consider $\overline{X'}(D_1',\lfloor a_1/2\rfloor)$. The only problematic situation is $b_1=3$ and $a_2$ is odd. Then $\gcd(a_1,a_2)\leq 3$, so we can change the prong-matching by $\pm 3$. If $a_2 \neq 6$, then we are done. It is obvious that $a_2 \neq 6$ because the problematic situation is under the condition that $a_2$ is odd. 

    In case (3), if $v\equiv 0,3\pmod 6$, then this case is the exceptional case (c). If $v\not\equiv 0,3 \pmod 6$, then \Cref{lm:TypeB_IIIa_4pole_top_level_ind=1or2} ensures that we can modify the top level.

    In case (4) and (5), on $\overline{X}$, only one $C_i$ ($i\neq 1$) is one. In case (4), we can rotate the prong-matching on by $+1$ or $-1$ and change the prong-matching by $3$ to avoid $v'\equiv 0,3 \pmod{6}$ on $\overline{Y}$. In case (5), $\gcd(a_1,a_2)=2$, so we can change the prong-matching on $\overline{X}$ as case (4) and avoid obtaining exceptional $\overline{Y}$.

\end{proof}

The following proposition is the variant of item (ii) of \Cref{lm:IIIa_nonhyp_top_level_modification} on the connectedness of the whole equatorial net. We require that the bottom level component containing $q_2$ does not contain any marked residueless pole. In return, we can drop the exceptional cases in \Cref{lm:IIIa_nonhyp_top_level_modification}.    

\begin{proposition}\label{prop:IIIa_nonhyp_empty_2nd_tail_top_level_modification}   
    Let $\cC$ be a non-hyperelliptic component of a stratum of B-signature and let $$\overline{X}=X^B_{\III a}((2,1),\ell,n-2,\tau,{\bf C},[(0,v)])\in \partial\cC$$ be a Type IIIa boundary has the non-hyperelliptic top level and $\kappa_1\geq \kappa_2$. In addition, $C_{\tau(\ell)}=b_{\tau(\ell)}-1$. Then there exists $$\overline{X'}=X^B_{\III a}((2,1),\ell,n-2,\tau',{\bf C}',[(0,v')])\in \partial\cC$$ for any $\tau'$, ${\bf C}'$ and $v'$ such that
    \begin{itemize}
        \item $\overline{X'}$ has the non-hyperelliptic top level; 
        \item $v'=v\pmod{\delta(\mathcal{I}_1,\emptyset)}$, where $\mathcal{I}_1$ is the set of bottom level residueless poles;
        \item $\tau'(i)=\tau(i)$ and $C'_{\tau(i)}=C_{\tau(i)}$ for any $i\leq\ell$.
    \end{itemize}
\end{proposition}

\begin{proof}
  It suffices to just consider the exceptional cases that are covered by the proposition but not covered by \Cref{lm:IIIa_nonhyp_top_level_modification}. The cases are:
\begin{itemize}
    \item[(a3)] $\ell=n-4$,  $\kappa_1=\kappa_2=b_{\tau(n-3)}=b_{\tau(n-2)}$ and $v\equiv \pm\lfloor \kappa_1/2\rfloor \pmod{\kappa_1}$;
    \item[(b)] $\ell=n-5$, $\{\kappa_1,\kappa_2\}=\{3,6\}$ and $b_{\tau(n-4)}=b_{\tau(n-3)}=b_{\tau(n-2)}=3$;
    \item[(c)] $\ell=n-6$, $\kappa_1=\kappa_2=6$ and $b_{\tau(n-5)}=b_{\tau(n-4)}=b_{\tau(n-3)}=b_{\tau(n-2)}=3$ with $v\equiv 0,3 \pmod{6}$.
\end{itemize}

 In case (a3), according to \Cref{lm:TypeB_IIIa_2pole_top_level_C2=1}, we can always modify the top level such that either $C_{\tau(n-2)}=1$ or $C_{\tau(n-2)}=D_{\tau(2)}+1$ (resp. $C_{\tau(n-2)}=D_{\tau(2)}+2$) if $\kappa_1$ is odd (resp. even). As the top level is non-hyperelliptic with respect to the prong-matching, we can find some prong-matching $(u_0,v_0)$ such that $c_{i-1}+1\leq u_0 < c_i \pmod{\kappa_1}$ for any $i$ while $v_0\equiv d_j\pmod{\kappa_2}$ for some $i,j$. Then $UR^2U\cdot \overline{X}(u_0,v_0)$ will be adjacent to a Type IIIb or IIIc boundary (depending on whether $e_2=1$) having the same set of residueless poles on the top level as $\overline{X}$ . If $e_2=1$, we get a Type IIIc boundary and the two poles on the top level are permutatble, i.e. we can permute the two residueless poles on the top level of the original Type IIIa boundary. Hence, after interchanging the residueless poles on the top level of the original Type IIIa boundary, we can easily reduce $C_{\tau(n-2)}$ to one (see the proof of \Cref{lm:TypeB_IIIa_2pole_top_level_C2=1}). In case $e_2>1$, we have $UR^2U\cdot \overline{X}(u_0,v_0)=\overline{Y}(u_0)$ where $\overline{Y}=X^B_{\III b}((2,1,2,1),n-4,\tau',e_2-1,{\bf C}')$. Unless $\kappa_1=3$, we should have $ 1< u_0+1< C_{\tau(\ell+i)}\pmod{\kappa_1}$ or $0\leq u_0-1< C_{\tau(\ell+i)}\pmod{\kappa_1}$. Then $UR^{-2}U\cdot \overline{Y}(u_0\pm 1)=\overline{X'}(u_0\pm 1,v_0)$ where $\overline{X'}$ is not of case (a3). We can make $C''_{\tau(n-2)}$ to one and change back the prong-matching via some Type IIIb boundary. Now assume $\kappa_1=3$. Notice that as we assume $a_1\leq a_2$, either $e_1=e_2>1$ or there is some other marked residueless pole on the bottom level. If there is no other marked residueless pole, then $UR^{-2e_2+2}UR^{-2C_{\tau(n-4+j)}}UR^{2e_2-2}U\cdot \overline{Y}(1)=\overline{Y'}(u'')$, where $\overline{Y'}=X^B_{\III b}((2,1,2,1),n-4,\tau'\circ (n-3,n-2),e_2-1,{\bf C}')$, i.e. we have interchanged the positions of $p_{\tau(n-2)}$ and poles $p_{\tau(n-3)}$ and we are done by return to a Type IIIa boundary. Assume that there is some marked residueless pole on the bottom level of $\overline{Y}$ (and also $\kappa_1=3$), then as $D_{\tau(n-4)}=1$ by the hypothesis of the proposition, we have $UR^{-2}UR^2U(R^{-2}U)^2R^2U \cdot\overline{Y}(2)=\overline{X'}(0,0)$ where $\overline{X'}$ has the same bottom level as $\overline{X}$ but with the hyperelliptic top level. This implies that we can actually interchange the positions and angle assignments on the top level of $\overline{X'}$, i.e. it induces the moves to interchange the two residueless marked poles on the top level of $\overline{Y}$ and $\overline{X}$, and we are done. 

In case (c), if $e_2>1$, then the argument will be similar to (a3). We can change the prong-matching equivalence class such that it is non-exceptional via some Type IIIb boundary. After we have arranged the angles and position of the residueless poles on the top level, we change back the prong-matching via Type IIIb boundary again. Now we consider $e_2=1$. By changing the choice of $\tau(n-5)$, we can assume that $v\equiv 0\pmod{6}$. Notice that by observation, we can also choose that $C_{\tau(n-5)}=2$ (also $C_{\tau(n-2)}=2$ due to non-hyperellipticity). We have then $UR^2U\cdot \overline{X}(2,-2)=\overline{Z}(1)$, where $\overline{Z}=X^B_{\III c}((2,1),n-6,n-2,\tau,3,{\bf C}')$ with $C'_{\tau(n-2)}=2$ and $C'_{\tau(n-5)}=C'_{\tau(n-4)}=C'_{\tau(n-2)}=1.$ We can simultaneously interchange the angle assignments and the positions of $p_{\tau(n-3)}$ and $p_{\tau(n-4)}$ by considering $UR^{-7}U\cdot \overline{Z}(1)$. In addition, we can purely interchange position of $p_{\tau(n-3)}$ and $p_{\tau(n-4)}$ by considering $UR^2UR^4UR^{-2}U\cdot \overline{Z}(1)$. By combining these moves we can modify the top level of $\overline{Z}$ and hence all possible Type IIIa boundary in (c) are contained in the same component.

    The case (b) is easy because the angle assignments are forced to be identical and by level rotation we can easily change the prong-matching by $3$. The only thing we need to show is we can interchange any two poles. If $\kappa_1=6$, then we can send some pole to the bottom level and reduce to case (a2) and we are done. Otherwise, there has to be some residueless marked pole on the bottom level, so by non-hyperellipticity and \Cref{prop:IIIa_hyper}, we can find a proper prong-matching and push it to the top level and got either case (c) or some non-exceptional case. Then we can change the position and push back the pole from bottom level. Thus, cases (b) can be dealt by other cases.
    
\end{proof}

Now we can also drop the requirement that $\kappa_2> \delta^\top$ or $(0,0)\notin Pr$ in \Cref{cor:IIIa_nonhyp_top_level_add_one}. The following corollary allows us in general maximizing the number of poles on the top level of a Type IIIa boundary without marked residueless pole on the bottom level component containing $q_2$.

\begin{corollary} \label{cor:IIIa_nonhyp_top_level_add_max}
    Let $\cC$ be a non-hyperelliptic component of a stratum of B-signature and $$\overline{X}=X^B_{\III a}((2,1),\ell,n-2,\tau,{\bf C},Pr)\in \partial\cC$$ be a Type IIIa boundary having the non-hyperelliptic top level and with $\ell>0$ such that $C_{\tau(\ell)}=b_{\tau(\ell)}-1$. Assume that $\kappa_2> n-2-\ell$, i.e. the number of poles on the top is not yet maximized. Then there exists $X^B_{\III a}((2,1),\ell-1,n-2,\tau,{\bf C}',Pr')\in \partial\cC$, where $\tau'(i)=\tau(i)$ and $C'_{\tau(i)}=C_{\tau(i)}$ for $i<\ell$. 
\end{corollary}

\begin{proof}
    If $\kappa_2<\sum_{i=\ell+1}^{n-2}(b_{\tau(i)}-1)$ and we can find some prong-matching $(0,v)\in Pr$ for $0<v<\kappa_2$, then the argument in \Cref{cor:IIIa_nonhyp_top_level_add_one} can apply again. Notice that if $\kappa_2=\sum_{i=\ell+1}^{n-2}(b_{\tau(i)}-1)$, then $\kappa_2\geq \kappa_1$. By \Cref{prop:IIIa_hyper}, we can find a proper prong-matching $(u,v)$ and consider $UR^2U\cdot \overline{X}(u,v)$. The only problem now is when $\kappa_2=\delta(\mathcal{I}_1,\emptyset)$ with $v\equiv 0\pmod{\delta(\mathcal{I}_1,\emptyset)}$. If $e_2>1$, then $UR^2U\cdot \overline{X}(0,0)=\overline{Y}(0)$ where $\overline{Y}=X^B_{\III b}((2,1,2,1),\ell-1,\tau,e_2-1,{\bf C})$. As $\kappa_2> n-2-\ell$, there is some $D_{\tau(j)}>1$ on the top level, and $UR^{-2}U\cdot \overline{Y}(c'_j)$ will be adjacent to a Type IIIa boundary where the pole $p_{\tau(\ell)}$ will be sent to the top level. If $e_2=1$, we have $UR^2U\cdot \overline{X}(0,0)=\overline{Z}(0)$ where $\overline{Z}=X^B_{\III c}((2,1),\ell-1,n-2,\tau',C',{\bf C}')$ with $D'=1$. Then again we choose appropriate prong-matching such that only the $u\equiv c'_j\pmod{\kappa_1+b_{\tau(\ell)}}$ where $D_{\tau(\ell+j)}\geq 2$. Then by reversing the operation, we obtain a Type IIIa boundary with $p_{\tau(\ell)}$ sent to the top level. Finally, by \Cref{prop:IIIa_nonhyp_empty_2nd_tail_top_level_modification}, we can adjust the positions of the poles on the new top level and obtain the boundary in the statement.
\end{proof}

\subsection{Non-hyperelliptic components: higher order non-residueless poles}
In this section, we will give the proof of the classification of connected components of a B-signature stratum with $e_2>1$. In the proof of classifying the connected components via finding boundaries of standard forms, the following lemma allows us to reduce the possibilities of the top level by pushing marked residueless pole on the top level to the bottom level.

\begin{lemma}\label{lm:typeB_IIIb_e2big_reduce_top_level}
    Let $\cC$ be a non-hyperelliptic component of a stratum of B-signature and $e_2>1$. Then there exists some Type IIIb boundary $\overline{Y}=X^B_{\III b}((2,1,2,1),\ell,\tau,e_2-1,{\bf C})\in\overline{\cC}$ such that 
    \begin{itemize}
        \item  $\tau(i)=i$ and $C_{i}=b_{i}-1$ for $i\leq \ell$;
        \item  $\ell=n-2$ or $a_1+1>e_2-1+\sum_{i=\ell+2}^{n-2} (b_{i}-1)$, i.e. numerically we cannot move any residueless pole on the top level to the bottom level.
    \end{itemize}
\end{lemma}
\begin{proof}
    By \Cref{cor:IIIb_perfect_1st_tail_empty_2nd_tail}, we can find some Type IIIb boundary $\overline{Y}=X^B_{\III b}((2,1,2,1),\ell,\tau,C,{\bf C})$ such that $\tau(i)=i$ and $C_{i}=b_i-1$ for $i\leq \ell$. 
    We will divide the cases into $e_2>2$ resp. $e_2=2$. Assume now $e_2>2$. If $C>1$ and $C_{\tau(\ell+j)}<b_{\tau(\ell+j)}-1$ for some $j$, then $UR^{-2}U\cdot\overline{Y}(c_{j-1})=\overline{Y'}(u')$ where $\overline{Y'}=X^B_{\III b}((2,1,2,1),\ell,\tau,C-1,{\bf C}')$ where $C_i=C'_i$ except $C'_{\tau(\ell+j)}=C_{\tau(\ell+j)}+1$. Similarly, if $C<e_2-1$ and $C_{\tau(\ell+j)}>1$ for some $j$, then $UR^{2}U\cdot\overline{Y}(c_{j-1}+1)=\overline{Y'}(u')$ where $\overline{Y'}=X^B_{\III b}((2,1,2,1),\ell,\tau,C+1,{\bf C}')$ where $C_i=C'_i$ except $C'_{\tau(\ell+j)}=C_{\tau(\ell+j)}-1$. By combining these two moves, we can modify the angle assignments on the top level. The first move can also send $p_{\ell+1}$ to the bottom level if it is numerically possible. By repeating these until the second condition in the lemma is attained, then are done.
    
    Now we assume $e_2=2$. If the second condition in the lemma is not yet attained, then there is some $C_{\tau(\ell+j)}> 1$ (where $\tau(\ell+j)\neq \ell+1$) on the top level, then $UR^{2}U\cdot\overline{Y}(c_{j-1}+1)=\overline{Z}(u')$ where $\overline{Z}=X^B_{\III c}((2,1),\ell,n-2,\tau,C',{\bf C}')$ with $C_i=C'_i$ except $C'_{\tau(\ell+j)}=C_{\tau(\ell+j)}-1$. We can just apply the moves in the proof of \Cref{lm:IIIc_push_pole_standard_to_bot} to try to push the pole $p_{\ell+1}$ to the lower level. If we obtain back a Type IIIb boundary before we successfully push the pole $p_{\ell}$ to the lower level, then we repeat the procedure again. One just needs to notice that the moves we use will only decrease the angle assignments $C_{\tau(\ell+j)}$ ($\tau(\ell+j)\neq \ell+1$) while increasing $C_{\ell+1}$. Hence, we can keep pushing poles as long as the second condition in the lemma is not yet satisfied and we are done.
    
\end{proof}

Now, we can prove our main result on classification for $e_2>1$.

\begin{proposition}\label{prop:e2_high_standard_boundary}
    Let $\cC$ be a non-hyperelliptic component of a stratum of B-signature and $e_2>1$. Let $\ell\leq n-2$ be the largest number such that $a_1+1\leq e_2-1+\sum_{i=\ell+1}^{n-2}(b_i-1)$. Then $\overline{\cC}$ contains the boundary $\overline{Y}=X^B_{\III b}((2,1,2,1),\ell,n-2,\Id,e_2-1,{\bf C})$, where $C_i=b_i-1$ for $i\leq\ell$ while $C_i=1$ for $i>\ell+1$.
\end{proposition}
 \begin{proof}    
    By \Cref{lm:typeB_IIIb_e2big_reduce_top_level}, we can find some Type IIIb boundary $\overline{Y}=X^B_{\III b}((2,1,2,1),\ell,\tau,e_2-1,{\bf C})$ such that numerically it is not possible to push any residueless pole to the bottom level while on the bottom level $\tau(i)=i$ and $C_{i}=b_i-1$. Now we will divide the cases into following:
    \begin{enumerate}
        \item $a_1+1\leq e_2-1+b_{n-2}-1$; otherwise
        \item $b_1=b_2=\dots=b_{n-2}=2$; 
        \item the remaining cases, i.e. $a_1+1> e_2-1+b_{n-2}-1$ and $b_{n-2}>2$.
    \end{enumerate}

    In case (i), $\overline{Y}$ has at most one marked residueless pole on the top level, thus it is automatically the desired boundary.

    In case (ii), it suffices to show that we can permute the poles on the top level (including $q_2$). Notice that $UR^3U\cdot \overline{Y}(1+\frac{1}{2})=\overline{Y'}(u')$, where $\overline{Y'}=X^B_{\III b}((2,1,2,1),\ell,\tau',e_2-1,{\bf C})$ such that $\tau'=\tau\circ(n-2,n-3,\dots,\ell+1)$. This implies that we can place $q_2$ arbitrarily. In addition, $UR^5U\cdot \overline{Y}(2+\frac{1}{2})=\overline{Y''}(u'')$, where $\overline{Y''}=X^B_{\III b}((2,1,2,1),\ell,\tau'',e_2-1,{\bf C})$ such that $\tau''=\tau\circ(n-2,n-3,\dots,\ell+1)^2\circ (\ell+1,\ell+2)$. By combining the moves, we are done with permutation. 

    In case (iii), by assumption, there is some $D_{\tau(i)}>1$ on the top level. In addition, by the first move in case (ii), we can assume that $D_{\tau(\ell+1)}>1$. Hence, $UR^{-2}U\cdot \overline{Y}(1)=\overline{X}(1,0)$ where $\overline{X}=X^B_{\III a}((2,1),\ell,n-2,\tau,{\bf C}',[(1,0)])$ such that $C'_i=C_i$ except for $i=\tau(\ell+1)$. By \Cref{prop:IIIa_nonhyp_empty_2nd_tail_top_level_modification}, we can so modify $\overline{X}$:
    \begin{itemize}
        \item[(iiia)] If $a_1+1\geq e_2-1+\sum_{i=\ell+2}^{n-2}(b_i-1)+2$, then $D''_i=b_i-1$ for all $i\neq \ell+1$. 
        \item[(iiib)] Otherwise, let $k>1$ be the smallest number such that $b_{\ell+k}>2$. We set $D''_{\ell+k}=b_{\ell+k}-2$ while $D''_{\ell+j}=b_{\ell+j}-1$ for $j\neq k$.
    \end{itemize} In case (iiia), we will set $\tau''=\Id$ while in case (iiib), $\tau''=(\ell+1,\ell+2,\dots,n-2)^{k-1}$. Then $UR^2U\cdot \overline{X}(1,0)=\overline{Y'}(u')$. Then we are done with case (iiia). In case (iiib), by considering $UR^{-3}U\cdot \overline{Y'}(-1)$, we can change $\tau''$ by $(\ell+1,\ell+2,\dots,n-2)^{-1}$. Hence, by applying this procedure by $k-1$ times, we are also done.
\end{proof}

\subsection{Non-hyperelliptic components: a pair of simple poles}

In this subsection, we assume that $(e_1, e_2) = (1,1)$. We begin with the following corollary, which establishes the existence of Type~IIIc boundaries for each non-hyperelliptic component. Recall that we define 
\[
\delta \coloneqq \gcd(a_1, b_1, \dots, b_{n-2}).
\]
The lemma below allows us to start with a Type~IIIa boundary with a pre-determined bottom level on any connected component $\overline{\cC}$.

To state this lemma, we extend the notion of a \emph{truncated equatorial net}. Let $\overline{Z} = X^B_{\III c}((2,1), \ell, n{-}2, \tau, C, {\bf C})$ be a boundary point of some stratum with B-signature $\mu^\fR$. Let $\mathcal{I}_1 = (\tau(1), \tau(2), \dots, \tau(\ell))$. We denote by $\mathcal{A}(\mu^\fR, \mathcal{I}_1, \emptyset)$ the subgraph of $\mathcal{A}(\mu^\fR)$ consisting of edges adjacent to boundaries of the following forms:
\begin{itemize}
    \item $\overline{X'} = X^B_{\III c}((2,1), \ell_1, \ell_2, \tau', C', {\bf C}')$ with $\ell_1 \geq \ell$, such that $\tau'(i) = \tau(i)$ and $C'_{\tau(i)} = C_{\tau(i)}$ for all $i \leq \ell$;
    \item $\overline{X'} = X^B_{\III a}((2,1), \ell_1, \ell_2, \tau', {\bf C}', [(u,v)])$ with $\ell_1 \geq \ell$, such that $\tau'(i) = \tau(i)$ and $C'_{\tau(i)} = C_{\tau(i)}$ for all $i \leq \ell$.
\end{itemize}

Let $\mathcal{A}(\overline{Z})$ denote the connected component of $\mathcal{A}(\mu^\fR, \mathcal{I}_1, \emptyset)$ containing $\overline{Z}$. Recall that the index of a Type~IIIa boundary $\overline{X'}$ relative to the index set $\mathcal{I}_1$ is given by:
\[
\Ind(\overline{X'}, \mathcal{I}_1, \emptyset) \equiv -u - v + \sum_{i > \ell} C_{\tau'(i)} \pmod{\delta(\mathcal{I}_1, \emptyset)},
\]
where
\[
\delta(\mathcal{I}_1, \emptyset) = \gcd\left(a_2 + 1 - e_1 - \sum_{i \in \mathcal{I}_1} b_i, \: \{ b_i \}_{\substack{i \in \underline{n-2} \\ i \notin \mathcal{I}_1}} \right).
\]Similarly, for a Type~IIIc boundary $\overline{Z'}$, we define:
\[
\Ind(\overline{Z'}, \mathcal{I}_1, \emptyset) \equiv C' + \sum_{i > \ell} C'_{\tau'(i)} \pmod{\delta(\mathcal{I}_1, \emptyset)}.
\]A boundary $\overline{X'}$ or $\overline{Z'}$ belongs to $\mathcal{A}(\overline{Z})$ only if
\[
\Ind(\overline{X'}, \mathcal{I}_1, \emptyset) = \Ind(\overline{Z}, \mathcal{I}_1, \emptyset) \quad \text{or} \quad \Ind(\overline{Z'}, \mathcal{I}_1, \emptyset) = \Ind(\overline{Z}, \mathcal{I}_1, \emptyset).
\]

\begin{lemma} \label{lm:B3aexist}
    Let $\cC$ be a non-hyperelliptic component of a stratum with B-signature. For any permutation $\tau^* \in \mathrm{Sym}_{n-2}$ such that $b_{\tau^*(n-2)} > 2$, there exists a boundary 
    \[
    \overline{X} = X^B_{\III a}((2,1), \ell, n{-}2, \tau, \mathbf{C}, \mathrm{Pr}) \in \partial \cC
    \]
    such that $\tau(i) = \tau^*(i)$ and $C_{\tau(i)} = b_{\tau(i)} - 1$ for all $i \leq \ell$. Moreover, $\overline{X}$ has the non-hyperelliptic top level.
\end{lemma}

\begin{proof}
    By \Cref{cor:IIIb_perfect_1st_tail_empty_2nd_tail}, we can find a Type~IIIc boundary 
    \[
    \overline{Z} = X^B_{\III c}((2,1), \ell, n{-}2, \tau, C, \mathbf{C}) \in \partial \cC
    \]
    such that $\tau(i) = \tau^*(i)$ and $C_{\tau(i)} = b_{\tau(i)} - 1$ for all $i \leq \ell$. Let 
    \[
    \overline{Z'} = X^B_{\III c}((2,1), \ell', n{-}2, \tau', C', \mathbf{C}')
    \]
    be a Type~IIIc boundary in $\mathcal{A}(\overline{Z})$ with $\ell'$ maximal such that $\tau'(i) = \tau^*(i)$ and $C'_{\tau'(i)} = b_{\tau'(i)} - 1$ for $i \leq \ell'$. Without loss of generality, we may also assume $\tau'(\ell'+1) = \tau^*(\ell'+1)$ and $C'_{\tau'(\ell'+1)}$ is maximal among all such Type~IIIc boundaries in $\mathcal{A}(\overline{Z})$.

    Observe that the transformation $UR^{-2}U \cdot \overline{Z'}(0)$ must be adjacent either to $\overline{Z'}$ again or to a Type~IIIa boundary. Otherwise, this would contradict the maximality of either $\ell'$ or $C'_{\tau'(\ell'+1)}$. This implies either:
    \[
    0 < C' \leq C'_{\tau'(\ell'+1)} \pmod{\kappa'} \quad \text{or} \quad C' \equiv c'_j \pmod{\kappa'}
    \]
    for some $j$. 

    In the former case, if $D'_{\tau'(\ell'+1)} > 1$, then 
    \[
    UR^{-2}U \cdot \overline{Z'}(C'_{\tau'(\ell'+1)} - C') = \overline{X}(C'_{\tau'(\ell'+1)} - C' + 1, D'_{\tau'(\ell'+1)} - 1),
    \]
    where 
    \[
    \overline{X} = X^B_{\III a}((2,1), \ell', n{-}2, \tau', \mathbf{C}'', \mathrm{Pr}),
    \]
    and $\mathbf{C}''$ agrees with $\mathbf{C}'$ except at index $i = \tau'(\ell'+1)$. Then $\overline{X}$ is the desired Type~IIIa boundary.

    If $D'_{\tau'(\ell'+1)} = 1$, we can increase $C'$ by considering the transformation $UR^{-2\theta - 1}U \cdot \overline{Z'}(0)$ where $\theta = b_{\tau'(\ell'+1)} - C'$. Repeating this process, we eventually reach the case where $C' \equiv c'_j \pmod{\kappa'}$ for some $j$.

    Now suppose again $D'_{\tau'(\ell'+1)} > 1$. We can then construct the desired Type~IIIa boundary as before. However, if $D'_{\tau'(\ell'+1)} = 1$, then $UR^{-2}U \cdot \overline{Z'}(0)$ will be adjacent to a Type~IIIa boundary where $p_{\tau'(\ell'+1)}$ is sent to the bottom level. If the top level is non-hyperelliptic, the proof is complete.

    Suppose instead that the resulting Type~IIIa boundary $\overline{X}$ has a hyperelliptic top level. Then $p_{\tau'(\ell'+2)}$ must be conjugate to (or coincide with) $p_{\tau'(\ell'+j)}$ for some $j$. Due to the symmetry of hyperellipticity, we can rearrange conjugate poles. Thus, without loss of generality, we may assume $b_{\tau'(\ell'+2)} > 2$ and $D'_{\tau'(\ell'+2)} > 1$ (and also $C'_{\tau'(\ell'+j)} > 1$).

    Returning to $\overline{Z'}$, suppose $C' \equiv c'_j \pmod{\kappa'}$. Then the transformation $UR^{-2}U \cdot \overline{Z'}(C'_{\tau'(\ell'+1)})$ is either adjacent to a desired Type~IIIa boundary (since $D'_{\tau'(\ell'+2)} > 1$), or to a Type~IIIc boundary. If it is the latter, then $C'_{\tau'(\ell'+2)}$ increases by one, so we have $C' \equiv c'_j - 1 \pmod{\kappa'}$. This implies that $UR^{-2}U \cdot \overline{Z'}(0)$ becomes adjacent to a Type~IIIc boundary with $\ell'+1$ poles on the bottom level, contradicting the maximality of $\ell'$. Hence, the boundary $\overline{X}$ cannot have the hyperelliptic top level, completing the proof.
\end{proof}

Now we first give the connectedness result on strata whose residueless poles are all double. 

\begin{proposition}\label{prop:Bpair_only_double}
     Let $\cC$ be a non-hyperelliptic component of a stratum of B-signature with $e_1=e_2=1$ and $b_i=2$ for any $i$. Then there exists $\overline{Z}=X^B_{\III c}((2,1),n-3-a_1,n-2,\Id,C,(1,\dots,1))\in\partial\cC$, where 
     \begin{itemize}
         \item $C=1$ or $2$ if $a_1$ is even;
         \item $C=2$ if $a_1$ is odd.
     \end{itemize} 
\end{proposition}
\begin{proof}
   By \Cref{cor:IIIb_perfect_1st_tail_empty_2nd_tail}, we can find some Type IIIc boundary $\overline{Z}=X^B_{\III c}((2,1),n-3-a_1,n-2,\tau,C,(1,\dots,1))\in \partial\cC$  such that $\tau(i)=i$ for all $i\leq n-3-a_1$. Notice that $UR^{-2}U\cdot \overline{Z}(0)$ will be adjacent to a Type IIIa boundary $\overline{X}$ having the hyperelliptic top level with one fixed pole (resp. zero or two fixed pole) if $a_1$ is odd (resp. even). If there is any fixed pole, we can assume that $p_{\tau(n-2)}$ is one fixed pole and hence we started with $\overline{Z}$ with $C=2$. If $C=2$, then we can interchange any two adjacent poles on the top level of $\overline{Z}$, namely by considering $UR^5U\cdot \overline{Z}(c_j+\frac{1}{2})$ (which will interchange the position of $p_{\tau(n-3-a_1+j)}$ and $p_{\tau(n-2-a_1+j)}$.
   
   Assume that we have no fixed pole on the top level of $\overline{X}$, then we can set $C=1$ on $\overline{Z}$. In addition, there are at least two conjugate poles on $\overline{X}$ that are interchangeable. Since the transpositon $(13)$ and the cyclic group of order $a_1+1$ generate $\operatorname{Sym}_{a_1+1}$ (as $a_1$ even and $\gcd(2,a_1+1)=1$), we are done. 
    
\end{proof}

The following proposition characterizes ubiquitous boundary points on connected components with varying topological invariants.

\begin{proposition} \label{prop:Bpair_standard}
    Let $\cC$ be a non-hyperelliptic component of a stratum of B-signature with $e_1 = e_2 = 1$ and index $I$. Then there exists a boundary point 
    \[
    \overline{X} = X^B_{\III a}((2,1), \max(0, n{-}2{-}a_1), n{-}2, \Id, \mathbf{C}, [(0, I')]),
    \]
    where $C_i = b_i - 1$ for all $i \leq \max(0, n{-}2{-}a_1)$, and $I \equiv I' + \max(0, n{-}2{-}a_1) \pmod{\delta}$.
\end{proposition}

\begin{proof}
    In the case where $b_{n-2} = 2$, by \Cref{prop:Bpair_only_double}, we can begin with
    \[
    \overline{Z} = X^B_{\III c}((2,1), n{-}3{-}a_1, n{-}2, \Id, C, \mathbf{1}),
    \]
    where $C = 1$ or $2$. By considering either $UR^{-2}U \cdot \overline{Z}(0)$ (if $C = 1$) or $UR^{2}U \cdot \overline{Z}(\frac{1}{2})$ (if $C = 2$), we obtain the desired boundary point.

    Now assume $b_{n-2} > 2$. Fix permutations $\tau^{(1)}, \tau^{(2)}, \dots, \tau^{(M)} \in \mathrm{Sym}_{n-2}$, where $M = 1 + \max(0, n{-}2{-}a_1)$, defined by:
    \[
    \tau^{(r)}(i) = 
    \begin{cases}
        r & \text{if } i = n{-}1{-}a_1, \\
        i & \text{if } i < r \text{ or } i > n{-}1{-}a_1, \\
        i + 1 & \text{if } r \leq i < n{-}1{-}a_1.
    \end{cases}
    \]
    Note that $\tau^{(M)} = \Id$.

    By \Cref{lm:B3aexist}, \Cref{cor:IIIa_nonhyp_top_level_add_max}, and \Cref{prop:IIIa_nonhyp_empty_2nd_tail_top_level_modification}, there exists a boundary 
    \[
    \overline{X^{(1)}} = X^B_{\III a}((2,1), \max(0, n{-}2{-}a_1), n{-}2, \tau^{(1)}, \mathbf{C}^{(1)}, [(u^{(1)}, v^{(1)})])
    \]
    such that $C^{(1)}_{\tau^{(1)}(i)} = b_{\tau^{(1)}(i)} - 1$ for $i \leq \max(0, n{-}2{-}a_1)$. If $a_1 \geq n{-}2$, then $C^{(1)}_i = b_i - 1$ for all $i$, and we can adjust the prong-matching and permutation via \Cref{prop:IIIa_nonhyp_empty_2nd_tail_top_level_modification} to reach a standard boundary point.

    Now assume $a_1 < n{-}2$. Given index $I$, define $I'$ to be the smallest positive integer such that
    \[
    n{-}2{-}a_1 + I' \equiv I \pmod{\delta}.
    \]
    Let $\delta^{(r)} = \gcd\left(a_1, \{b_{\tau^{(r)}(n{-}2+j)}\}_{j=1}^{a_1}\right)$. Then there exists a linear combination
    \[
    -u^{(1)} - v^{(1)} - (n{-}2{-}a_1) \equiv I' + t_1 \delta^{(1)} + \dots + t_M \delta^{(M)} \pmod{a_1}.
    \]By \Cref{prop:IIIa_nonhyp_empty_2nd_tail_top_level_modification}, we can modify the prong-matching of $\overline{X^{(1)}}$ to $(u^{(2)}, v^{(2)})$ such that
    \[
    -u^{(2)} - v^{(2)} - (n{-}2{-}a_1) \equiv I' + t_2 \delta^{(2)} + \dots + t_M \delta^{(M)} \pmod{a_1}.
    \]
    Let $K$ be the smallest non-negative integer such that $(K, 0) \in [(u^{(2)}, v^{(2)})]$. Then
    \[
    UR^2U \cdot \overline{X^{(1)}}(K, 0) = \overline{Z}(u')
    \]
    where 
    \[
    \overline{Z} = X^B_{\III c}((2,1), \ell', n{-}2, \tau', C', \mathbf{C}')
    \]
    with $\ell' + D' < n{-}2$.

    By the arguments in \Cref{lm:B3aexist} and \Cref{prop:IIIa_nonhyp_empty_2nd_tail_top_level_modification}, we obtain 
    \[
    \overline{X^{(2)}} = X^B_{\III a}((2,1), n{-}2{-}a_1, n{-}2, \tau^{(2)}, \mathbf{C}^{(2)}, [(u^{(2)}, v^{(2)})])
    \]
    such that $D^{(2)}_i = 1$ for all $i$. Repeating this process $M{-}1$ times yields the desired boundary:
    \[
    X^B_{\III a}((2,1), n{-}2{-}a_1, n{-}2, \Id, \mathbf{C}, [(0, I')]),
    \]
    where $C_i = b_i - 1$ for all $i$. Hence, for each index $I$, a unique such boundary point exists.
\end{proof}

Now we are ready to prove \Cref{Cor:Bssc}, which establishes the existence of a flat surface with a multiplicity one saddle connection.

\begin{proof}[Proof of \Cref{Cor:Bssc}]
    Consider a Type I boundary $\overline{X}=X^B_{\I}(n-2,\tau,{\bf C})$. The bottom level component of $\overline{X}$ has a multiplicity one saddle connection joining $z_1$ and $z_2$. Thus, it suffices to show that every non-hyperelliptic component contains such a boundary for suitable choices of $\tau$ and ${\bf C}$.

    If $a_1 \neq a_2$ or $e_1 \neq e_2$, then the component is uniquely non-hyperelliptic, and any choice of $\tau$ and ${\bf C}$ suffices.

    Suppose now that $b_i > 2$ for some $i$. By relabeling the residueless poles, we may assume that $\tau = \Id$ and that the sequence $(b_1, \dots, b_{n-2})$ is in decreasing order. If $(e_1,e_2) \neq (1,1)$, then \Cref{Prop:Bnonhyper} implies there exists a unique non-hyperelliptic component, and we may take ${\bf C} = (1,\dots,1)$.

    Now consider the case $(e_1,e_2) = (1,1)$. Let $\delta = \gcd(b_1, \dots, b_{n-2})$ and let $I$ be the index of a non-hyperelliptic component. There exists a choice of ${\bf C}$ such that $\sum_i C_i \equiv I \pmod{\delta}$. The top level component of $\overline{X}$ is hyperelliptic only if $b_i = b_1$ and $C_i = D_{n-1-i}$ for each $i$, in which case $\sum_i C_i = \frac{(n-2)b_1}{2}$. If $\frac{(n-2)b_1}{2} \geq n-2 + b_1$, then we can choose ${\bf C}'$ such that $\sum_i C'_i = \sum_i C_i - b_1 \equiv \frac{(n-2)b_1}{2} \pmod{b_1}$. Similarly, if $\frac{(n-2)b_1}{2} \leq (n-2)(b_1 - 1) - b_1$, then a non-hyperelliptic boundary with index $I$ exists.

    Therefore, assume that
    \[
        (n-2)(b_1 - 1) - b_1 < \frac{(n-2)b_1}{2} < n-2 + b_1.
    \]
    This inequality implies $(n-2)(b_1 - 2) \leq 2b_1 - 2$, which in turn gives $n - 2 \leq 2$. Thus, we are reduced to the following exceptional cases:
    \begin{itemize}
        \item $\cP(\mu^{\fR}) = \cP(2b \mid -b \mid -b \mid -1^2)$ and $I = b$,
        \item $\cP(\mu^{\fR}) = \cP(2b \mid -2b \mid -1^2)$ and $I = b$,
        \item $\cP(\mu^{\fR}) = \cP(2 \mid -1^2)$ and $I = 1$.
    \end{itemize}
    However, by \Cref{Prop:Bpair}, such non-hyperelliptic components do not exist. Therefore, in all cases, a Type I boundary with a multiplicity one saddle connection exists within each non-hyperelliptic component.
\end{proof}

%% file: 09Csignatures.tex
\section{Strata of C-Signatures} \label{Sec:Csignature}

In this section, we prove \Cref{Prop:Cconnected}, which establishes the connectedness of the strata of C-signatures. Let $\cP(\mu^{\fR})$ be a stratum of C-signature. We denote $\mu^\fR = (a\mid-b_1 \mid \dots \mid -b_{n-3}\mid  -e_1, -e_2, -e_3)$ with $a-(e_1+e_2+e_3+\sum_i b_i)=-2$ throughout this section. Assume also that $0 < a$, $0 < e_1 \leq e_2\leq e_3$, and $2 \leq b_1 \leq \dots \leq b_{n-3}$. Let $q_1$,$q_2$ and $q_3$ denote the non-residueless poles of orders $e_1$,$e_2$ and $e_3$, respectively. 

\subsection{Type III Boundary}

Since any flat surface in $\cP(\mu^{\fR})$ has a unique zero, each saddle connection must join the zero to itself. Hence, this stratum only has a Type III boundary. These can be given by the following combinatorial data:

\begin{itemize}
    \item A tuple $\underline{h} = (h_1, h_2, h_3) \in \{(1,2,3), (1,3,2), (2,3,1)\}$, where $h_3$ indicates the non-residueless pole $q_{h_3}$ on the bottom level;
    \item An integer $0 \leq \ell \leq n-3$, representing the number of residueless poles on the top level component;
    \item A permutation $\tau \in \mathrm{Sym}_{n-3}$ of the residueless poles;
    \item A tuple of positive integers ${\bf C} = (C_1, \dots, C_{n-3})$, where $1 \leq C_i \leq b_i - 1$ for each $i$.
\end{itemize}

Let $\tau_1=\tau|_{\{1,\dots, \ell\}}$ and ${\bf C_1}=(C_{\tau(i)})_{i=1,\dots, \ell}$. Then the top level component is isomorphic to $Z_2(\tau_1,{\bf C_1},e_{h_1},e_{h_2})$. Let $\tau_2(i)=\tau(i+\ell)$ for $i=1,\dots, n-3-\ell$ and ${\bf C_2}=(C_{\tau(i)})_{i=\ell+1,\dots, n-3}$. Then the bottom level component is isomorphic to $Z_2(\tau_2,{\bf C_2},\kappa+1,e_{h_3})$ where $\kappa=a-e_{h_3}+1-\sum_{i=\ell+1}^{n-3} b_{\tau(i)}$. We denote this boundary by 
\[
X^C_{\III}(\underline{h}, \ell, \tau, {\bf C}).
\]
The level graph and separatrix diagram are illustrated in \Cref{fig:C_type_III_graph}.

\begin{figure}[h!]
    \centering
    \includegraphics[width=0.8\textwidth]{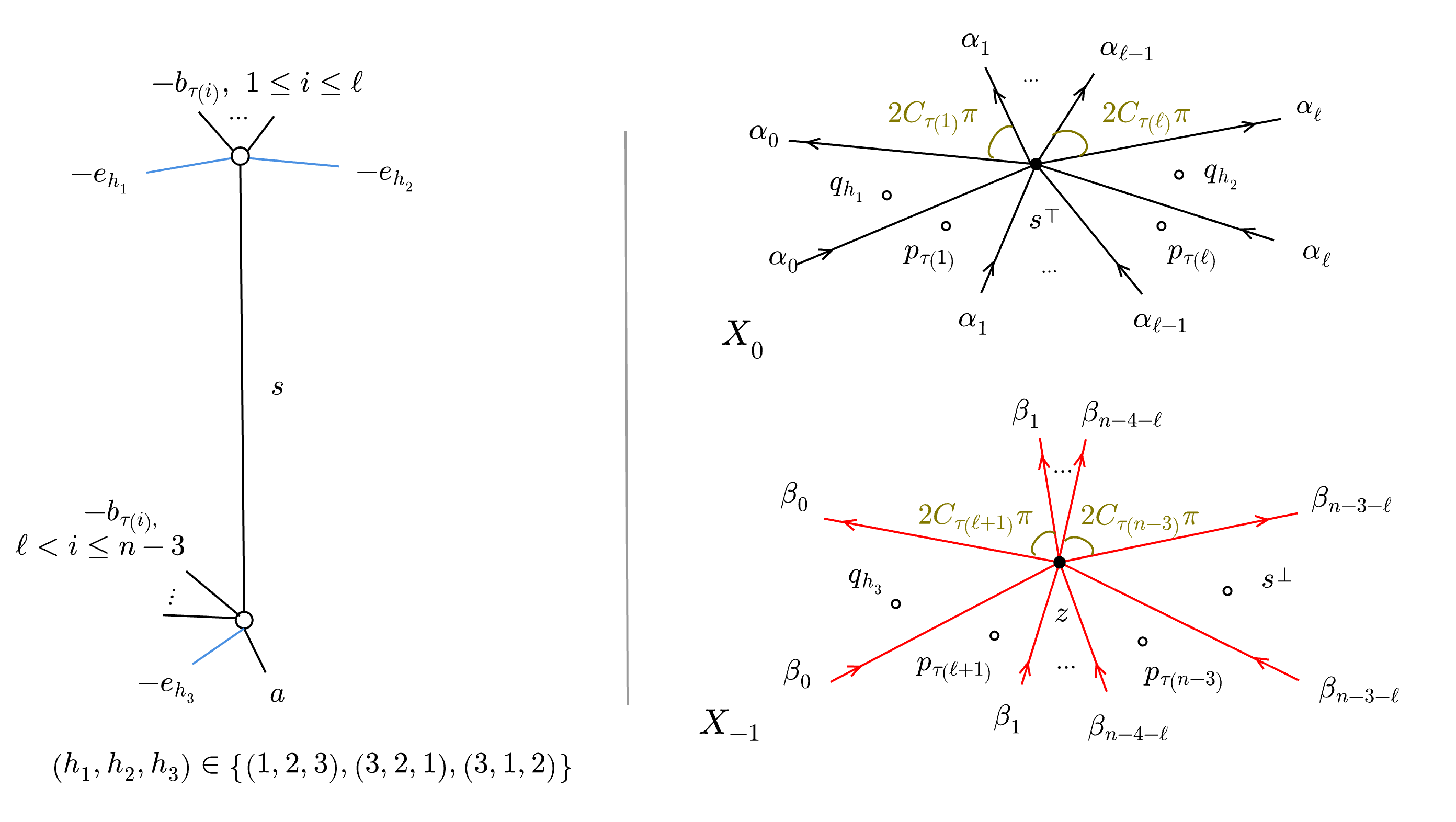}
    \caption{Level graph and separatrix diagram of a Type III boundary in the strata of C-signatures}
    \label{fig:C_type_III_graph}
\end{figure}

In \Cref{fig:C_type_III_pm}, we label the outgoing prongs at $s^\top$ and the incoming prongs at $s^\bot$. More precisely, the outgoing prongs at $s^\top$ are labeled $v^+_1, \dots, v^+_\kappa$ clockwise, with $v^+_1$ aligned to the direction of $\alpha_0$. The incoming prongs at $s^\bot$ correspond to the outgoing prongs at $z$ directed into the polar domain of $s^\bot$. We label them $v^-_1, \dots, v^-_\kappa$ clockwise at $z$, with $v^-_1$ nearest to the saddle connection $\beta_0$. A prong-matching $\boldsymbol{\sigma}$ is determined by $\boldsymbol{\sigma}(v^-_\kappa)=v^+_u$, and we identify $\boldsymbol{\sigma}$ with an element $u \in \ZZ/\kappa\ZZ$.

\begin{figure}[h!]
    \centering
    \includegraphics[width=0.8\textwidth]{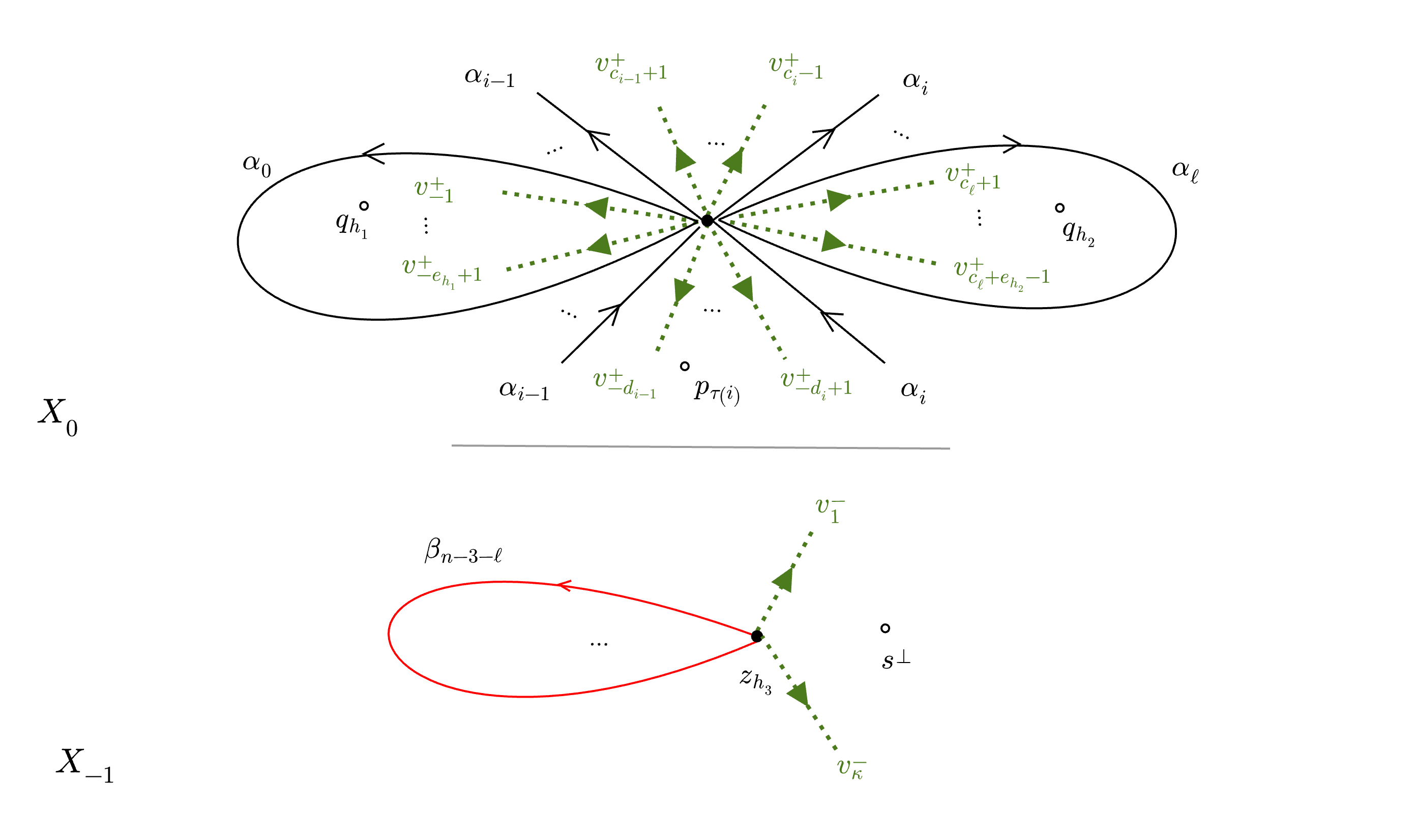}
    \caption{Prong labeling for strata of C-signatures}
    \label{fig:C_type_III_pm}
\end{figure}

\subsection{Plumbing construction and equatorial half-arcs}

Let $\overline{X} = X^C_{\III}(\underline{h}, \ell, \tau, {\bf C})$ be a Type III boundary. By plumbing construction with $t\in \RR_+$ and a prong-matching $u \in \ZZ/\kappa\ZZ$, we obtain a flat surface $\overline{X}_t(u)$ on an equatorial arc adjacent to $\overline{X}$. We denote this equatorial half-arc by $\overline{X}(u)$. Similarly, we obtain another equatorial half-arc by taking $t\in \RR_-$, denoted by $\overline{X}(u-\frac{1}{2})$. The number of equatorial arcs adjacent to $\overline{X}$ is equal to $2\kappa$. We describe the transformation $U$ for some equatorial half-arcs in $\cP(\mu^{\fR})$, which will be used frequently, in the later proofs. Other cases can be obtain in a similar way. 

\begin{proposition}\label{Prop:Cmove}
Let $\overline{X} = X^C_{\III}((1,2,3), \ell, \Id, {\bf C})$ and $0 \leq u < c_\ell + e_2$. Then $U \cdot \overline{X}(u) = \overline{X'}(u')$ where $\overline{X'}$ and $u'$ are given by the following:

\begin{enumerate}
    \item If $c_\ell \leq u < c_\ell + e_2$, then $\overline{X'} = X^C_{\III}((3,2,1), n - 3 - \ell, \tau, {\bf C})$, with
    \[
        u' = (c_\ell + e_2 - 1) - (u - c_\ell),
    \]
    and
    \[
        \tau(i) =
        \begin{cases*}
            \ell + i & if $1 \leq i \leq n - 3 - \ell$, \\
            i - (n - 3 - \ell) & otherwise.
        \end{cases*}
    \]

    \item If $c_{j-1} \leq u < c_j$ for $j \leq \ell$, then $\overline{X'} = X^C_{\III}((3,2,1), n - 2 - j, \tau, {\bf C'})$, with
    \[
        u' = (-e_3 - \sum_{k = \ell+1}^{n-3} D_k) - (u - c_{j-1}),
    \]
    where
    \[
        \tau(i) =
        \begin{cases}
            \ell + i & \text{if } 1 \leq i \leq n - 3 - \ell, \\
            i - (n - 3 - \ell) + j & \text{if } n - 3 - \ell < i \leq n - 2 - j, \\
            i - (n - 2 - j) & \text{otherwise},
        \end{cases}
    \]
    and
    \[
        C'_i =
        \begin{cases*}
            c_j - u & if $i = j$, \\
            C_i & otherwise.
        \end{cases*}
    \]
\end{enumerate}
\end{proposition}

\subsection{Connectedness of strata}

We now prove \Cref{Prop:Cconnected}, showing that the strata of C-signatures are connected.

\begin{lemma}\label{lm:TypeC_III_ex}
Every connected component $\cC$ of $\cP(\mu^{\fR})$ contains a boundary of the form $X^C_{\III}((1,2,3), n-3, \tau, {\bf C})$.
\end{lemma}

\begin{proof}
Let $\overline{X} = X^C_{\III}((h_1, h_2, h_3), \ell, \tau, {\bf C}) \in \partial\overline{\cC}$. If $h_3 \neq 3$, then $U \cdot \overline{X}(0)$ is adjacent to a desired boundary. So assume that $h_3 = 3$. Then now $U\cdot \overline{X}(0)$ is adjacent to $X^C_{\III}((3,h_2,h_1), n-3, \tau, {\bf C})$, reducing to the case $h_3\neq 3$. 
\end{proof}

\begin{proposition}\label{prop:C_standard}
Every connected component $\cC$ of $\cP(\mu^{\fR})$ contains the standard boundary $\overline{X} = X^C_{\III}((1,2,3), n-3, \Id, (1,\dots,1))$.
\end{proposition}

\begin{proof}
By \Cref{lm:TypeC_III_ex}, $\cC$ contains an element of the form $X^C_{\III}((1,2,3), n-3, \tau, {\bf C})$ for some $\tau$ and ${\bf C}$. It remains to show that $\cC$ does not depend on these parameters. We find the following boundary of $\cC$ for each step:    
    \begin{itemize}
        \item[(1)] $X^C_{\III}((1,2,3), n-3, \Id, {\bf C})$.
        \item[(2)] $X^C_{\III}((1,2,3), n-3, \Id, {\bf 1})$.
    \end{itemize}

For the first step, it is sufficient to prove the case $\tau=(j,j+1)$ for each $1 \leq j \leq n-4$. Let $\overline{X'} = X^C_{\III}((1,2,3), n-3, \tau \circ (j,j+1), {\bf C})$. Then $UR^{2D_j}U \cdot \overline{X}(c_{j-1})$ and $UR^{2D_j}U \cdot \overline{X'}(c_j)$, thus $\overline{X'}\in \partial\overline{\cC}$. 
    
For the second step, suppose $C_k > 1$ for some $k$. Then $UR^{2(C_k - 1)}U \cdot \overline{X}(c_k - 1)$ is adjacent to a boundary $\overline{X'}$ with $C'_k = 1$ and other data is identical to $\overline{X}$. Repeating this process yields the standard form.
\end{proof}

This completes the proof of the connectedness of the strata of C-signatures. We now conclude with the following corollary.

\begin{proof}[Proof of \Cref{Cor:Cssc}]
Assume $p = q_3$. The connected stratum $\cP(\mu^{\fR})$ contains the boundary $\overline{X} = X^C_{\III}((1,2,3), n-3, \Id, (1, \dots, 1))$. The bottom level component contains a multiplicity one saddle connection bounding the polar domain of $p$. The top level component is hyperelliptic only if $e_1 = e_2$ and all $b_i = 2$.
\end{proof}

%% file: 10Dsignatures.tex
\section{Strata of D-signatures} \label{Sec:Dsignature}

In this section, we prove \Cref{Prop:Dhyper}--\ref{Prop:Dnonhyper}. Let $\cP(\mu^{\fR})$ be a stratum of D-signature. We denote $\mu^\fR = (a\mid-b_1 \mid \dots \mid -b_{n-4}\mid  -e_1, -e_2\mid -e_3,-e_4)$ with $a-(e_1+e_2+e_3+e_4+\sum_i b_i)=-2$ throughout this section. Assume also that $0< e_1\leq e_2, 0<e_3\leq e_4$ and $b_i\geq 2$. Let $q_j$ denote the non-residueless poles of orders $e_j$, $j=1,2,3,4$. 

\subsection{Type III boundary}

Since any flat surface in $\cP(\mu^{\fR})$ has a unique zero, each saddle connection is joining the zero to itself. So the stratum only has Type III boundaries. In the case of D-signatures, there are three types of Type III boundaries, distinguished by the type of level graphs. We can list them as follows (the combinatorial data in \Cref{tab:TypeDIIIabc_comb}) :

\begin{itemize}
    \item[(IIIa)] : if all $q_j$ are contained in the top level;    
    \item[(IIIb)] : if only one of the $q_j$ is contained in the bottom level components;
    \item[(IIIc)] : if exactly a pair of non-residueless poles, either $(q_1,q_2)$ or $(q_3,q_4)$, is contained in the bottom level component.
\end{itemize}

\begin{table}[h!]
    \centering
    \begin{tabular}{|c|p{12cm}|}
    \hline
        \textbf{Type} & \textbf{Combinatorial data}\\ \hline
         IIIa& A tuple $\underline{h}=(h_1,h_2)$ indicating the indices of the non-residueless pole contained in $X_0^{1}$ (the component containing $q_1$) resp. $X_0^{2}$.  \\ \cline{2-2}
         & Integers $\ell_1,\ell_2$, where $0\leq \ell_1\leq \ell_2\leq n-4$ such that $\ell_1$ resp. $\ell_2-\ell_1$ are the number of residueless poles contained in the component $X_0^1$ resp. $X_0^2$.\\ \cline{2-2}
         & A permutation $\tau\in \operatorname{Sym}_{n-4}$ to indicate the indices of residueless poles.\\ \cline{2-2}
         & A tuple of integers ${\bf C}=(C_1,\dots,C_{n-4})$, where $1\leq C_i\leq b_{i}-1$.\\ \cline{2-2}
         & A prong-matching equivalence class $[(u,v)]\in \ZZ/\kappa_1\ZZ\times \ZZ/\kappa_2\ZZ$ modulo the diagonal rotation.\\ \hline
        IIIb &  A tuple $\underline{h}=(h_1,h_2,h_3,h_4)$ indicating the indices of non-residueless poles, where $h_2$ indicate the index of the non-residueless pole on bottom level.\\ \cline{2-2}
        & Integers $\ell_1,\ell_2$, where $0\leq \ell_1\leq \ell_2\leq n-4$ such that $\ell_1$ resp. $\ell_2-\ell_1$ are the number of residueless poles lying between $q_{h_3}$ and $q_{h_1}$ resp. between $q_{h_1}$ and $q_{h_4}$.\\ \cline{2-2}
        & A permutation $\tau\in \operatorname{Sym}_{n-4}$ to indicate the indices of residueless poles.\\ \cline{2-2}
        & An integer $C$, where $1\leq C \leq e_{h_1} -1 $. \\ \cline{2-2}
        & A tuple of integers ${\bf C}=(C_1,\dots,C_{n-4})$, where $1\leq C_i\leq b_i-1$ .\\ \hline
        IIIc & A tuple $\underline{h}=(h_1,h_2,h_3,h_4)$ indicating the indices of non-residueless poles, where $h_3,h_4$ indicate the indices of the non-residueless pole on bottom level.\\ \cline{2-2}
        & Integers $\ell_1,\ell_2$, where $0\leq \ell_1\leq \ell_2\leq n-4$ such that $\ell_1$ resp. $\ell_2-\ell_1$ are the number of residueless poles lying on $Z_0$ resp. on the $Z_{-1}$ between $q_{h_3}$ and $s^\bot$.\\ \cline{2-2}
        & A permutation $\tau\in \operatorname{Sym}_{n-4}$ to indicate the indices of residueless poles.\\ \cline{2-2}
        & An integer $C$, where $1\leq C \leq \kappa$. \\ \cline{2-2}
        & A tuple of integers ${\bf C}=(C_1,\dots,C_{n-4})$, where $1\leq C_i\leq b_i-1$ .\\ \hline
    \end{tabular}
    \caption{Combinatorial data of boundaries}
    \label{tab:TypeDIIIabc_comb}
\end{table}

For a Type IIIa boundary, let $\tau_1=\tau|_{\{1,\dots, \ell_1\}}$, $\tau_2(i)=\tau(i+\ell_1)$ for $i=1,\dots, \ell_2-\ell_1$ and $\tau_3(i)=\tau(i+\ell_2)$ for $i=1,\dots,n-4-\ell_2$. Also, let ${\bf C_1}=(C_{\tau(i)})_{i=1,\dots, \ell_1}$, ${\bf C_2}=(C_{\tau(i)})_{i=\ell_1+1,\dots, \ell_2}$ and ${\bf C_3}=(C_{\tau(i)})_{i=\ell_2+1,\dots, n-4}$. There are two top level components, isomorphic to $Z_2(\tau_1,{\bf C_1},e_1,e_{h_1})$ and $Z_2(\tau_2,{\bf C_2},e_{h_2},e_2)$. The bottom level component is isomorphic to $Z_2(\tau_3,{\bf C_3},\kappa_1+1,\kappa_2+1)$ where $\kappa_1=e_1+e_{h_1}-1+\sum_{i=1}^{\ell_1} b_{\tau(i)}$ and $\kappa_2=e_2+e_{h_2}-1+\sum_{i=\ell_1+1}^{\ell_2} b_{\tau(i)}$. We denote this boundary by 
\[
X^D_{\III a} (\underline{h},\ell_1,\ell_2,\tau,{\bf C},[(u,v)]).
\]

For a Type IIIb boundary, the top level component contains $\ell_2+1$ residueless poles, one of which is $q_{h_1}$. Let $\tau_1(i)=\tau(i)$ for $i=1,\dots, \ell_1$ and $\tau_1(i)=\tau(i-1)$ for $i=\ell_1+2,\dots,n-3-\ell_2$ and $\tau_1(\ell_1+1)=q_{h_1}$. Also, let ${\bf C_1}=(C_{\tau(1)},\dots, C_{\tau(\ell_1)},C,C_{\tau(\ell_1+1)},\dots,C_{\tau(\ell_2)})$. Then the top level component is isomorphic to $Z_2(\tau_1,{\bf C_1},e_{h_3},e_{h_4})$. Let $\tau_2(i)=\tau(i+\ell_2)$ for $i=1,\dots, n-4-\ell_2$ and ${\bf C_2}=(C_{\tau(i)})_{i=\ell_2+1,\dots, n-4}$. Then the bottom level component is isomorphic to $Z_2(\tau_2,{\bf C_3},e_{h_2},\kappa+1)$ where $\kappa=e_{h_1}+e_{h_3}+e_{h_2}-1+\sum_{i=1}^{\ell_2} b_{\tau(i)}$. We denote this boundary by 
\[
X^D_{\III b} (\underline{h},\ell_1,\ell_2,\tau,C,{\bf C}).
\]

For a Type IIIc boundary, let $\tau_1(i)=\tau(i)$ for $i=1,\dots, \ell_1$ for $i=1,\dots,\ell_1$. Also, let ${\bf C_1}=(C_{\tau(1)},\dots, C_{\tau(\ell_1)})$. Then the top level component is isomorphic to $Z_2(\tau_1,{\bf C_1},e_{h_1},e_{h_2})$. The bottom level component contains $n-3-\ell_1$ residueless poles, one of which is the pole of order $\kappa+1\coloneqq e_{h_1}+e_{h_2}+\sum_{i=1}^{\ell_1}b_{\tau(i)}$ at the node $s^\bot$. Let $\tau_2(i)=\tau(i+\ell_1)$ for $i=1,\dots, \ell_2-\ell_1$, $\tau_2(i)=\tau(i+\ell_1-1)$ for $i=\ell_2-\ell_1+2,\dots, n-3-\ell_1$ and $\tau_2(\ell_2-\ell_1+1)=s^\bot$. Also let ${\bf C_2}=(C_{\tau(\ell_1+1)},\dots, C_{\tau(\ell_2)},C,C_{\tau(\ell_2+1)},\dots,C_{\tau(n-4)})$. Then the bottom level component is isomorphic to $Z_2(\tau_2,{\bf C_2},e_{h_3},e_{h_4})$. We denote this boundary by 
\[ 
X^D_{\III c} (\underline{h},\ell_1,\ell_2,\tau,C,{\bf C}).
\]

The choices of $h_1,h_2,h_3,h_4$ have already taken into account of permutations. The level graphs and separatrix diagrams of corresponding boundaries are depicted in \Cref{tab:TypeD_IIIabc}. We use the alphabet $s$ to label the node, and with the symbols $\top,\bot$ to label the nodal poles and zeros. The enhancements of the level graphs are denoted by $\kappa$. 

\begin{table}[h!]
    \centering
    \begin{tabular}{|c|c|}
    \hline
      Type   & Level graph/ Separatrix diagram \\ \hline
       IIIa  & \centered{  
\resizebox{14cm}{6.2cm}{\includegraphics[]{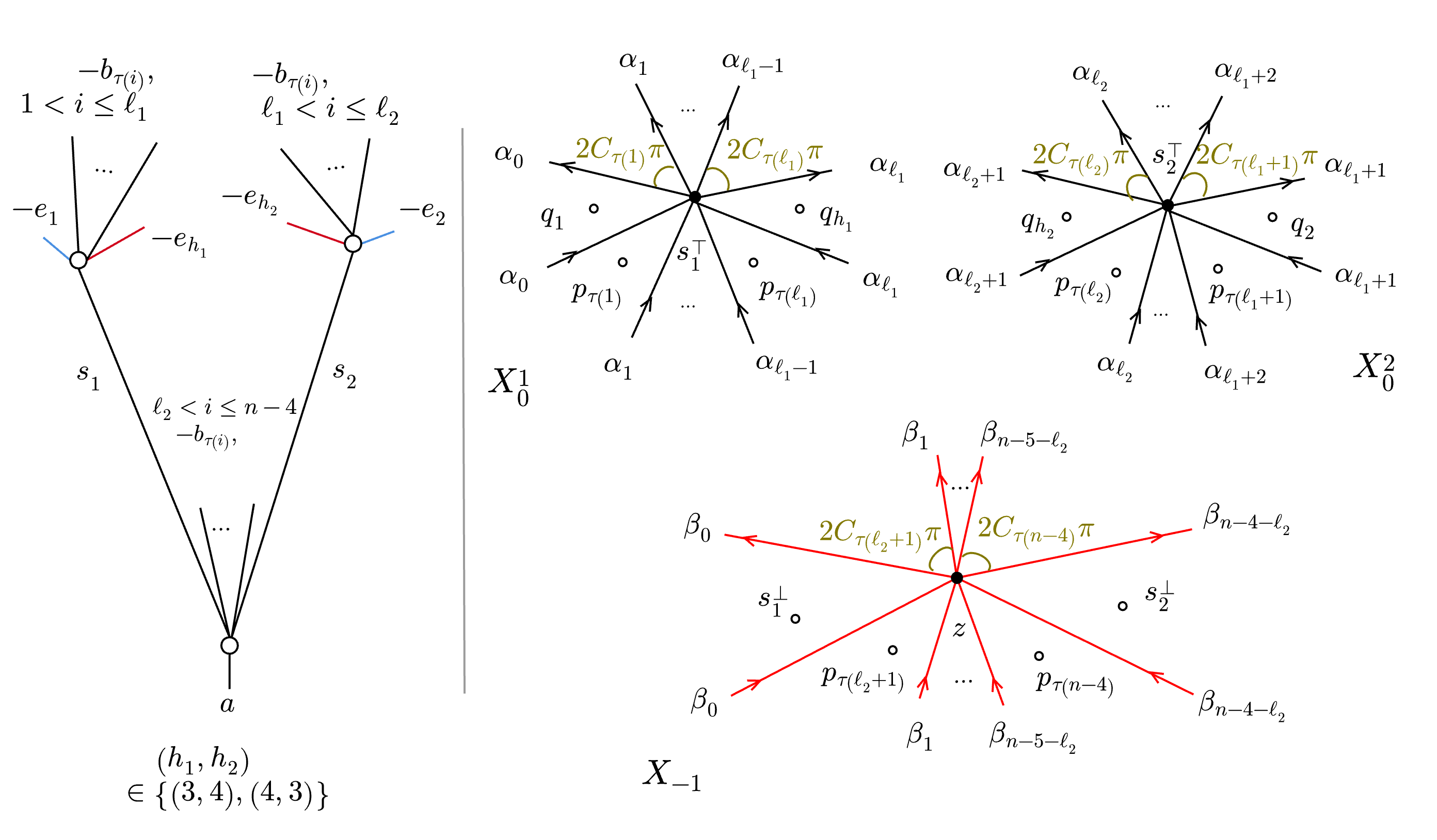}}
 }  \\ \hline

 IIIb & \centered{  
\resizebox{14cm}{6.2cm}{\includegraphics[]{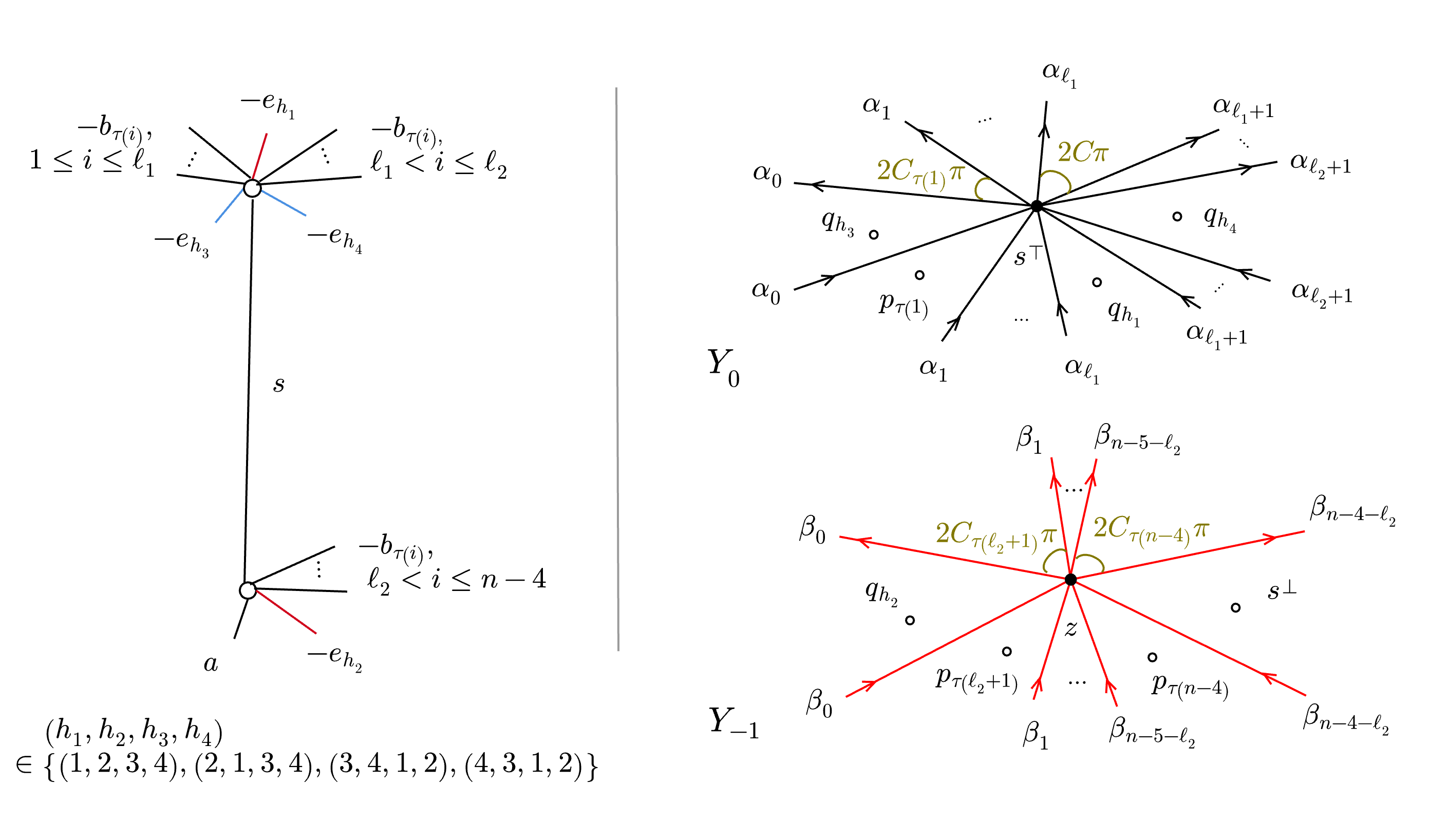}}
 }  \\ \hline

 IIIc & \centered{  
\resizebox{14cm}{6.2cm}{\includegraphics[]{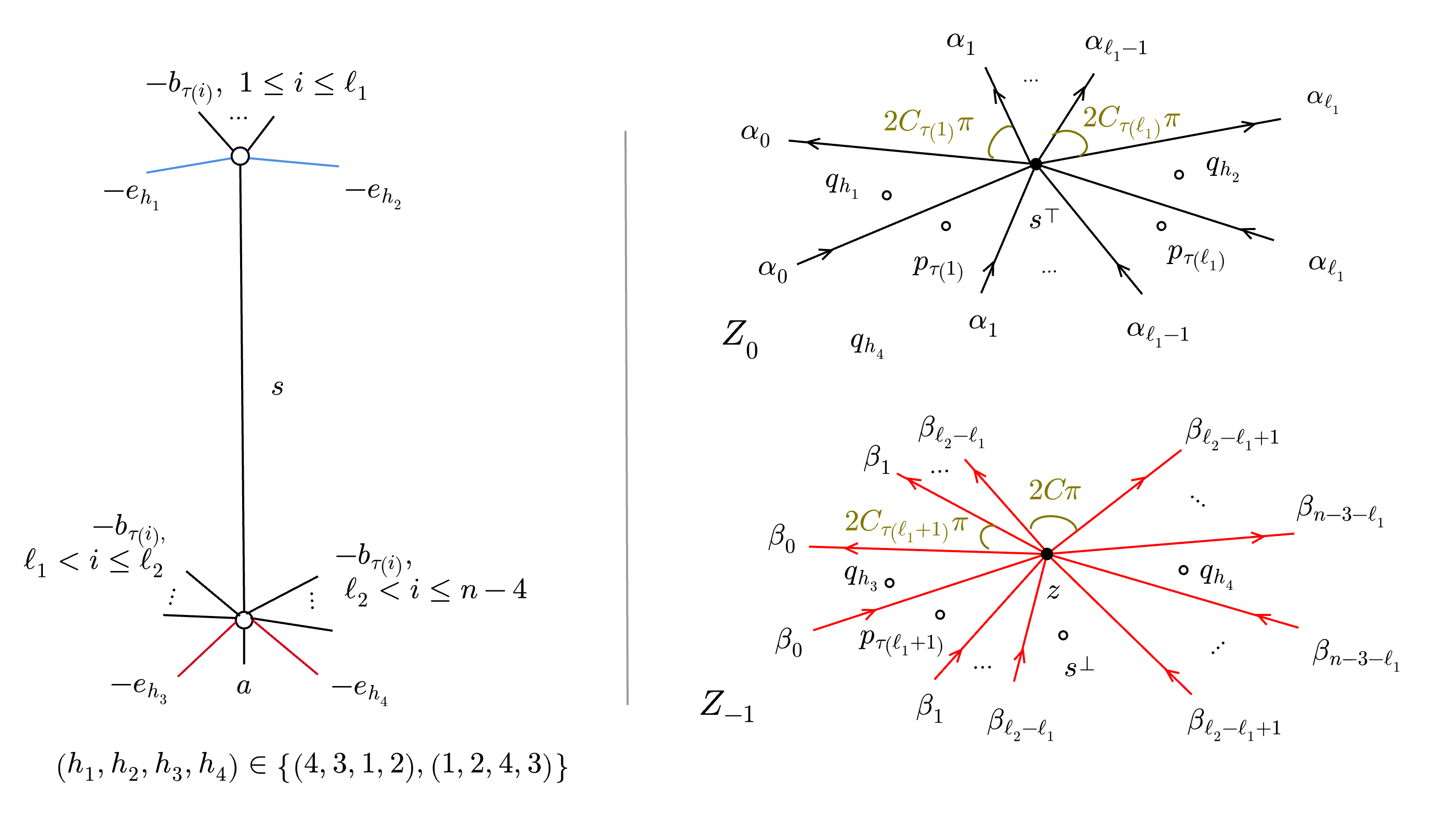}}
 }  \\ \hline
       
    \end{tabular}
    \caption{The boundaries of strata of D-signatures}
    \label{tab:TypeD_IIIabc}
\end{table}

In \Cref{tab:TypeD_pm}, the labeling of outgoing prongs at $s^\top$ and the incoming prongs at $s^\bot$ is depicted. 

\begin{table}[]
    \centering
    \begin{tabular}{|c|c|}
    \hline
      Type   & Prongs-labeling\\ \hline
       IIIa  & \centered{  
\resizebox{14cm}{7cm}{\includegraphics[]{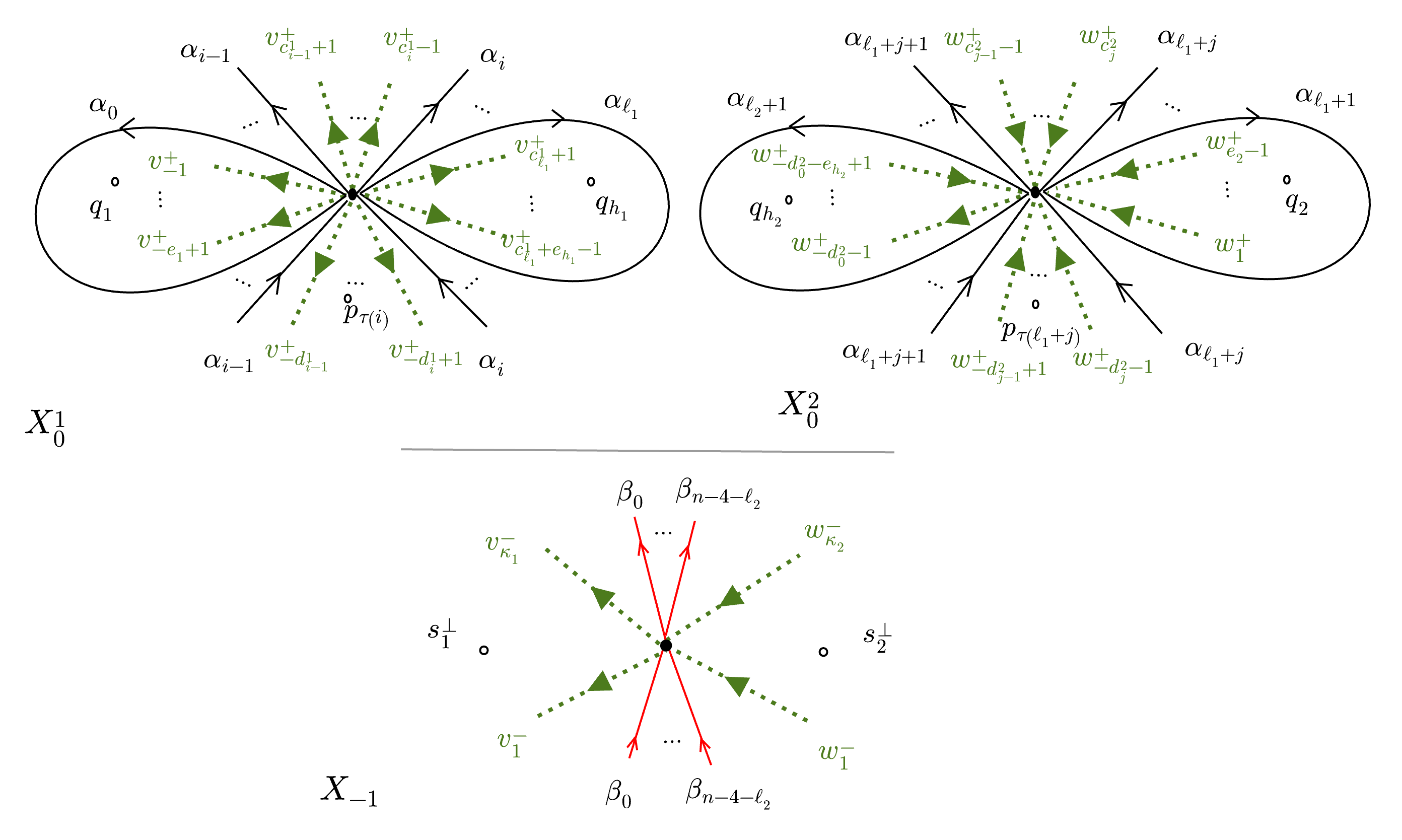}}
 }  \\ \hline
 IIIb & \centered{  
\resizebox{14cm}{7.3cm}{\includegraphics[]{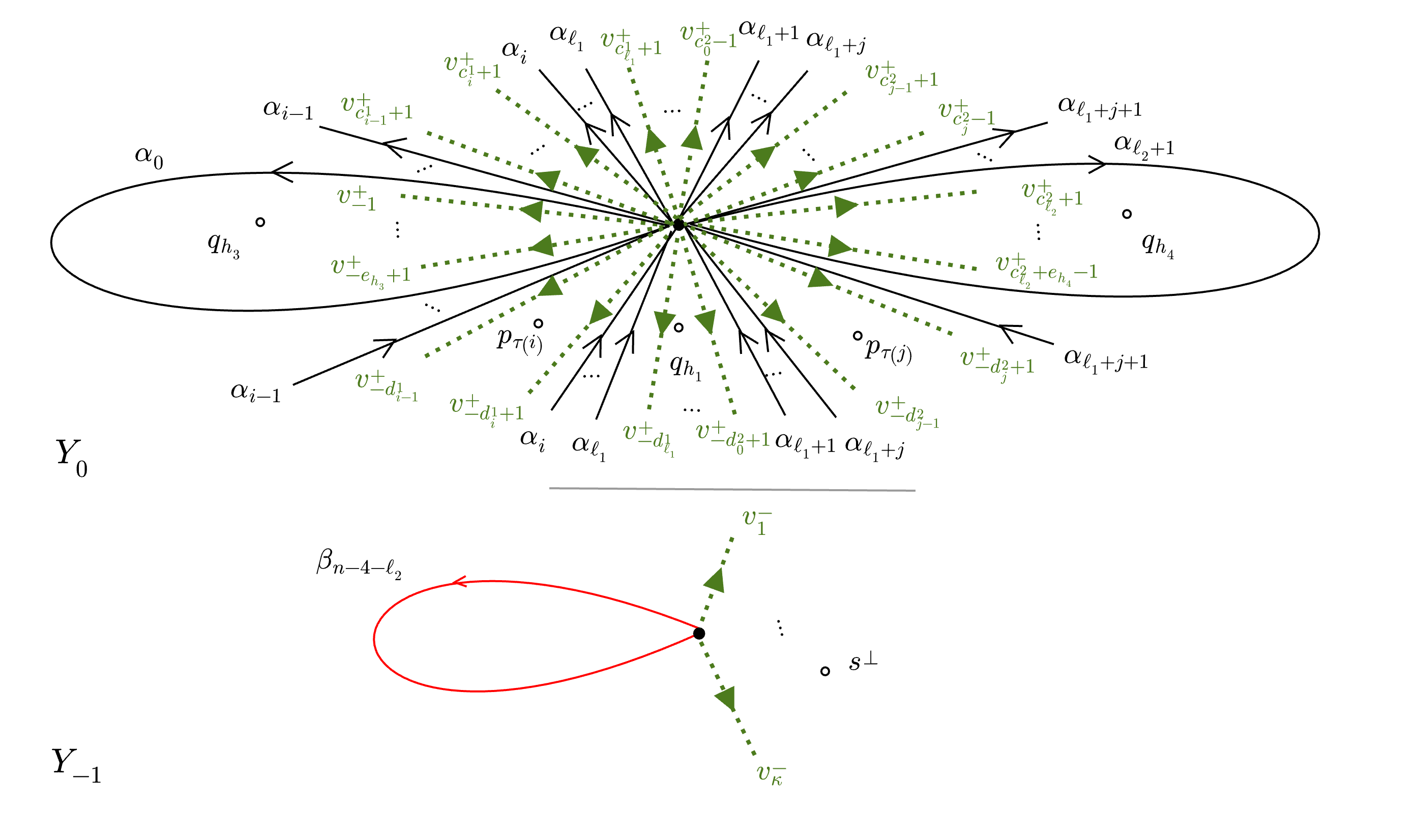}}
 }  \\ \hline
 IIIc & \centered{  
\resizebox{14cm}{6.3cm}{\includegraphics[]{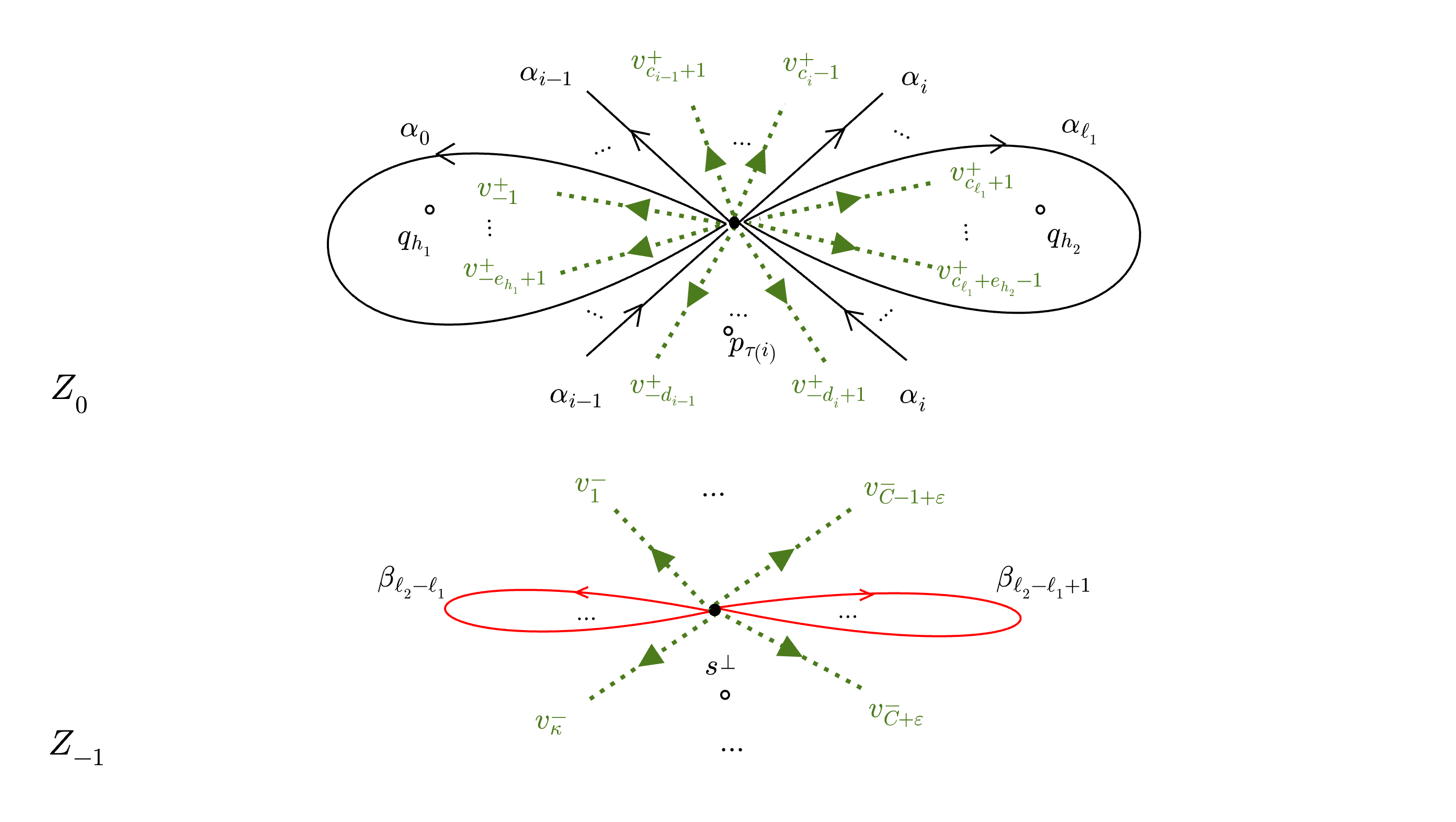}}
 }  \\ \hline      
    \end{tabular}
    \caption{Prongs-labeling of Strata of D-signatures}
    \label{tab:TypeD_pm}
\end{table}

For Type IIIa, on the top level, we label the outgoing prongs at $s_1^\top$ by $v^+_i$, $i\in \ZZ/\kappa_1\ZZ$ clockwise so that $v^+_0$ is aligned with $\alpha_0$. Also, the incoming prongs at $s_2^\top$ are labeled by $w^+_j$, $j\in\ZZ/\kappa_2\ZZ$, counter-clockwise so that $w^+_0$ is aligned with $\alpha_{\ell_2+1}$. At $s_1^\top$, the outgoing prongs $v^+_{c^1_0},v^+_{c^1_1},\dots,v^+_{c^1_{\ell_1}}$ overlap with saddle connections in the same direction while the outgoing prongs $v^+_{-d^1_0},v^+_{-d^1_1},\dots,v^+_{-d^1_{\ell_1}}$ are the prongs that are closest to the saddle connections counter-clockwise, where
\begin{align*}
    c^1_i=\sum_{k=1}^i C_{\tau(k)}\quad
    d^1_i=e_1+\sum_{k=1}^i D_{\tau(k)}.
\end{align*}
Similarly, at $s_2^\top$, the outgoing prongs $w^+_{-d^2_0},w^+_{-d^2_1},\dots,w^+_{-d^2_{\ell_2-\ell_1}}$ overlap with saddle connections in the same direction while the outgoing prongs $w^+_{c^2_0},w^+_{c^2_1},\dots,w^+_{c^2_{\ell_2-\ell_1}}$ are the prongs that are closest to the saddle connections counter-clockwise, 
\begin{align*}
    c^2_j=e_2+\sum_{k=\ell_1+1}^{\ell_1+j} C_{\tau(k)}\quad
    d^2_j=\sum_{k=\ell_1+1}^{\ell_1+j} D_{\tau(k)}.
\end{align*}
On the bottom level, note that the incoming prongs at $s_1^\bot$ is corresponding to the outgoing prongs at $z$ towards the polar domain of $s_1^\bot$. We label them by $v^-_1,\dots, v^-_{\kappa_1}$ clockwise at $z_{h_1}$ so that $v^-_1$ is closest to the saddle connection $\beta_0$. Similarly, we label the incoming prongs at $z$ towards the polar domain of $s_2^\bot$ by $w^-_1,\dots, w^-_{\kappa_2}$ counter-clockwise at $z_{h_2}$ so that $w^-_1$ is closest to the saddle connection $\beta_{n-4-\ell_2}$. Then a prong-matching $\boldsymbol{\sigma}$ is uniquely determined by the images of $\boldsymbol{\sigma}(v^-_{1})=v^+_{u}$ and $\boldsymbol{\sigma}(w^-_{\kappa_2})=w^+_{v}$. We identify $\boldsymbol{\sigma}$ with an element $(u,v)\in \ZZ/ \kappa_1 \ZZ\times \ZZ/ \kappa_2 \ZZ$. 

For Type IIIb boundaries, on the top level, we label the outgoing prongs at $s^\top$ by $v^+_i$, $i\in\ZZ/\kappa\ZZ$, so that $v^+_0$ is aligned with the saddle connection $\alpha_0$. The outgoing prongs $v^+_{c^1_0},v^+_{c^1_1},\dots,v^+_{c^1_{\ell_1}}$ and $v^+_{c^2_0},v^+_{c^2_2},\dots,v^+_{c^2_{\ell_2-\ell_1}}$ overlap with saddle connections in the same direction, while the outgoing prongs $v^+_{-d^1_0},v^+_{-d^1_1},\dots,v^+_{-d^1_{\ell_1}}$ and $v^+_{-d^2_0},v^+_{-d^2_1},\dots,v^+_{-d^2_{\ell_2-\ell_1}}$ are the prongs that are closest to the saddle connections counter-clockwise, where
\begin{align*}
    c^1_i&=\sum_{k=1}^i C_{\tau(k)} \quad
    d^1_i=e_1+\sum_{k=1}^i D_{\tau(k)},\\
    c^2_j&=C+\sum_{k=1}^{\ell_1+j} C_{\tau(k)} \quad
    d^2_j=e_1+D+\sum_{k=1}^{\ell_1+j} D_{\tau(k)}.
\end{align*}
On the bottom level, the incoming prongs at the nodal poles $s^\bot$ can be represented by outgoing prongs at $z$ towards the nodal polar domain. We label the outgoing prongs in the nodal polar domain clockwise by $v^-_1,\dots,v^-_{\kappa}$. The prong $v^-_1$ is the closest outgoing prong to the saddle connection $\beta_{n-4-\ell_2}$ clockwise at $z$. A prong-matching $\boldsymbol{\sigma}$ is uniquely determined by $\boldsymbol{\sigma}(v^-_\kappa)=v^+_u$, so we identify $\boldsymbol{\sigma}$ with an element $u \in \ZZ/\kappa\ZZ$.

For Type IIIc boundaries, on the top level, the prongs at $s^\top$ are labeled similarly as that of Type IIIb boundary. The outgoing prongs $v^+_{c_0},\dots, v^+_{c_{\ell_1}}$ at $s^\top$ overlap with the saddle connections, where $c_i=\sum_{k=1}^i C_{\tau(k)}$. On the bottom level, the incoming prongs at the nodal poles $s^\bot$ can be represented by outgoing prongs at $z$ towards the nodal polar domain. We label the outgoing prongs in the nodal polar domain at $z$ clockwise by $v^-_1,\dots,v^-_{\kappa}$. The prong $v^-_1$ is the outgoing prong at $z$ towards $s^\bot$ which is the outgoing prong next to $\beta_{\ell_2-\ell_1}$ clockwise. Depending on the rescaling of the bottom level, there are $C-1+\varepsilon$ outgoing prongs between $\beta_{\ell_1}$ and $\beta_{\ell_1+1}$ clockwise, where $$\varepsilon=\begin{cases}
    1 &\mbox{ if }\arg(\alpha_0/\beta_0)\in(0,\pi]\\
    0 &\mbox{ if }\arg(\alpha_0/\beta_0)\in(\pi,2\pi].
\end{cases}$$ A prong-matching $\boldsymbol{\sigma}$ is uniquely determined by $\boldsymbol{\sigma}(v^-_\kappa)=v^+_u$, so we identify $\boldsymbol{\sigma}$ with an element $u \in \ZZ/\kappa\ZZ$.

\subsection{Plumbing construction and equatorial half-arcs}

In this subsection, we describe the equatorial half-arcs adjacent to each boundary. We describe the transformation $U$ for some equatorial half-arcs in $\cP(\mu^{\fR})$, which will be used frequently, in the later proofs. Other cases can be obtain in a similar way. 

\paragraph{\bf Type IIIa boundary}

Let $\overline{X}$ be a Type IIIa boundary. By plumbing construction with $t\in \RR_+$ and a prong-matching $(u,v)\in Pr$, we can obtain a flat surface $\overline{X}_t(u,v)$ on an equatorial half-arc adjacent to $\overline{X}$. We denote this equatorial half-arc by $\overline{X}(u,v)$. Similarly, we obtain another equatorial half-arc by taking $t\in \RR_-$, denoted by $\overline{X}(u-\frac{1}{2},v+\frac{1}{2})$. There are $2 \operatorname{lcm} (\kappa_1,\kappa_2)$ equatorial arcs adjacent to $\overline{X}$.

\begin{proposition} \label{Prop:DIIIamove}
Let $\overline{X}=X^D_{\III a} ((4,3),0,n-4,\Id,{\bf C},[(u,v)])$ be a Type IIIa boundary. Assume $0< u < e_4$ and $c^2_{j} < v \leq c^2_{j+1}$ for some $j$. Then $U\cdot \overline{X}(u,v)=\overline{X'}(u')$ for $\overline{X'}=X^D_{\III b} ((4,3,1,2)),0,n-4-j,\tau,C',{\bf C'})$ where

\begin{itemize}
    \item A permutation $\tau$ given by
    \begin{equation*}
        \tau(i) = \begin{cases}
                j+2-i &\text{if $1\leq i\leq j+1$} \\
                i &\text{otherwise}.
                \end{cases}
    \end{equation*}
    \item A positive integer $C'=u.$
    \item A tuple ${\bf C'}$ given by
    \begin{equation*}
        C'_i = \begin{cases}
                v-c^2_j &\text{if $i=j+1$ }  \\
                C_i &\text{otherwise}.
                \end{cases}
    \end{equation*}
    \item $u'=u+v+d^2_j$.
\end{itemize}
\end{proposition}

\paragraph{\bf Type IIIb boundary}

Let $\overline{X}$ be a Type IIIb boundary. By plumbing construction with $t\in \RR_+$ and a prong-matching $u\in \ZZ/ \kappa \ZZ$, we can obtain a flat surface $\overline{X}_t(u)$ on an equatorial half-arc denoted by $\overline{X}(u)$. Similarly, by taking $t\in \RR_-$, we obtain $\overline{X}(u-\frac{1}{2})$. There are $2\kappa$ equatorial arcs adjacent to $\overline{X}$. 

\begin{proposition} \label{Prop:DIIIbmove}
Let $\overline{X}=X^D_{\III b} ((4,3,1,2),\ell_1,\ell_2,\Id,C,{\bf C})$ be a Type IIIb boundary. 

(1) Assume $ -e_1 < u < 0$. Then $U\cdot \overline{X}(u)=\overline{X'}(u')$ for $\overline{X'}=X^D_{\III b} ((1,2,3,4)),n-4-\ell_2,n-4-(\ell_2-\ell_1),\tau,-u,{\bf C})$ where
\begin{itemize}
    \item A permutation $\tau$ given by
    \begin{equation*}
        \tau(i) = \begin{cases}
            i+\ell_2 &\text{if $1\leq i\leq n-4-\ell_2$}  \\
            i-(n-4-\ell_2) &\text{if $n-4-\ell_2+1\leq i\leq n-4-\ell_2+\ell_1$ }  \\
            (\ell_1+n-3)-i &\text{otherwise}.
                \end{cases}
    \end{equation*}
    \item $u'=(c')^2_{n-4-(\ell_2-\ell_1)}+C$.
\end{itemize}

(2) Assume $c^2_{\ell_2} < u < c^2_{\ell_2}+e_2$. Then $U\cdot \overline{X}(u)=\overline{X'}(u')$ for $\overline{X'}=X^D_{\III b} ((2,1,4,3)),\ell_2-\ell_1,n-4-\ell_1,\tau,u-c^2_{\ell_2},{\bf C})$ where
\begin{itemize}
    \item A permutation $\tau$ given by
    \begin{equation*}
        \tau(i) = \begin{cases}
            i+\ell_1 &\text{if $1\leq i\leq \ell_2-\ell_1$}  \\
            (n-3+\ell_2-\ell_1)-i &\text{if $ \ell_2-\ell_1+1 \leq i\leq n-4-\ell_1$ }  \\
            i+\ell_1-(n-4) &\text{otherwise}.
                \end{cases}
    \end{equation*}
    \item $u'=-C$.
\end{itemize}
\end{proposition}

\paragraph{\bf Type IIIc boundary}

Let $\overline{X}$ be a Type IIIc boundary. Similarly to Type IIIb boundary, there are $2\kappa$ equatorial arcs adjacent to $\overline{X}$, denoted by $\overline{X}(u)$ or $\overline{X}(u-\frac{1}{2})$ for each $u\in \ZZ/\kappa\ZZ$.

\begin{proposition} \label{Prop:DIIIcmove}
Let $\overline{X}=X^D_{\III c} ((4,3,1,2),\ell_1,\ell_2,\Id,C,{\bf C})$ be a Type IIIc boundary. Then $U\cdot \overline{X}(u)$ is given by the following:

(1) Assume $-e_4 < u < 0$. Suppose that $c_{j-1} < u+C \leq c_j$ for some $1\leq j\leq \ell_1$ or $u+C\leq 0$ and $j=0$. Then $\overline{X'}=X^D_{\III b} (4,3,1,2),\ell_2-\ell_1,j+n-4-\ell_2,\tau,C',{\bf C'})$ where

\begin{itemize}
    \item A permutation $\tau$ given by
    \begin{equation*}
        \tau(i) = \begin{cases}
                i+\ell_1 &\text{if $ 1\leq i\leq \ell_2-\ell_1$}  \\
                i+\ell_1-j &\text{if $\ell_2-\ell_1+1\leq i\leq \ell_2-\ell_1+j$} \\
                i+\ell_1-j &\text{if $\ell_2-\ell_1+j+1\leq i\leq n-4-(\ell_1-j)$} \\
                -i+(n-4+j+1) &\text{otherwise}. 
                \end{cases}
    \end{equation*}
    \item A tuple ${\bf C'}$ given by
    \begin{equation*}
        C'_i = \begin{cases}
                u+C-c_{j-1} &\text{if $i=j$} \\
                C_i &\text{otherwise}.
                \end{cases}
    \end{equation*}
    \item $u'=-(d^2)'_{\ell_2-\ell_1+j}$.
\end{itemize}

(2) Assume $-e_4 < u < 0$ and $\kappa-e_4<u+C<\kappa+u$ or $\kappa-d_j < u+C \leq \kappa-d_{j-1}$ for some $j$. Then $\overline{X'}=X^D_{\III b} (4,3,1,2),\ell_2-\ell_1,j+n-4-\ell_2,\tau,C',{\bf C'})$ where

\begin{itemize}
    \item A permutation $\tau$ given by
    \begin{equation*}
        \tau(i) = \begin{cases}
                i+\ell_1 &\text{if $ 1\leq i\leq \ell_2-\ell_1$}  \\
                i+\ell_1-j &\text{if $\ell_2-\ell_1+1\leq i\leq \ell_2-\ell_1+j$} \\
                i+\ell_1-j &\text{if $\ell_2-\ell_1+j+1\leq i\leq n-4-(\ell_1-j)$} \\
                -i+(n-4+j+1) &\text{otherwise}. 
                \end{cases}
    \end{equation*}
    \item A tuple ${\bf C'}$ given by
    \begin{equation*}
        C'_i = \begin{cases}
                u+C+d_j-\kappa &\text{if $i=j$} \\
                C_i &\text{otherwise}.
                \end{cases}
    \end{equation*}
    \item $u'=(c^2)'_{\ell_2-\ell_1+j}$.
\end{itemize}

(3) Assume $c_j \leq u < c_{j+1}$ and $c_{k-1} < u+C \leq c_k$ for some $j<k$. Then $\overline{X'}=X^D_{\III c} (1,2,4,3),n-4-\ell_1+(k-j),n-4-\ell_1+k,\tau,C',{\bf C'})$ where

\begin{itemize}
    \item A permutation $\tau$ given by
    \begin{equation*}
        \tau(i) = \begin{cases}
                i+\ell_1 &\text{if $ 1\leq i\leq \ell_2-\ell_1$}  \\
                i+(j-\ell_2) &\text{if $\ell_2+1\leq i\leq \ell_2+(k-j)$} \\
                i+\ell_1-k &\text{if $\ell_2+(k-j)+1 \leq i\leq n-4-(\ell_1-k)$} \\
                i-(n-4-(\ell_1-k)) &\text{if $n-4-(\ell_1-k)+1 \leq i\leq n-4-(\ell_1-k)+j$} \\
                i+\ell_1-(n-4) &\text{otherwise}. 
                \end{cases}
    \end{equation*}
    \item A tuple ${\bf C'}$ given by
    \begin{equation*}
        C'_i = \begin{cases}
                c_{j+1}-u &\text{if $i=j+1$} \\
                u+C-c_{k-1} &\text{if $i=k$} \\
                C_i &\text{otherwise}.
                \end{cases}
    \end{equation*}
    \item $C'=c_k-c_j+\sum_{i=\ell_1+1}^{n-4} b_i$.
    \item $u'=-(d'_{\ell_2-\ell_1}+u-c_j)$.
\end{itemize}

(4) Assume $c_j \leq u < c_{j+1}$ and $\kappa-d_k < u+C \leq \kappa-d_{k-1}$ for some $j<k$. Then $\overline{X'}=X^D_{\III c} (1,2,4,3),n-4-\ell_1+(k-j),n-4-\ell_1+k,\tau,C',{\bf C'})$ where

\begin{itemize}
    \item A permutation $\tau$ given by
    \begin{equation*}
        \tau(i) = \begin{cases}
                i+\ell_1 &\text{if $ 1\leq i\leq \ell_2-\ell_1$}  \\
                i+\ell_1-j &\text{if $\ell_2-\ell_1+1\leq i\leq \ell_2-\ell_1+j$} \\
                i+\ell_1-j &\text{if $\ell_2-\ell_1+j+1\leq i\leq n-4-(\ell_1-j)$} \\
                -i+(n-4+j+1) &\text{otherwise}. 
                \end{cases}
    \end{equation*}
    \item A tuple ${\bf C'}$ given by
    \begin{equation*}
        C'_i = \begin{cases}
                c_{j+1}-u &\text{if $i=j+1$} \\
                C_k+u+C-(\kappa-d_k) &\text{if $i=k$} \\
                C_i &\text{otherwise}.
                \end{cases}
    \end{equation*}
    \item $C'=c_k-c_j+\sum_{i=\ell_1+1}^{\ell_2} b_i$.
    \item $u'=-(d^2)'_{\ell_2-\ell_1+j}$.
\end{itemize}
\end{proposition}

\begin{proposition}
    Let $\overline{X}=X^D_{\III c} ((4,3,1,2),\ell_1,\ell_2,\Id,C,{\bf C})$ be a Type IIIc boundary. Assume $-e_4 < u \leq 0$ and $c_j < u+C \leq c_{j+1}$ for some $0\leq j< \ell_1$. Then $U\cdot \overline{X}(u-\frac{1}{2})=\overline{X'}(u'+\frac{1}{2},v'-\frac{1}{2})$ where $\overline{X'}=X^D_{\III a} (4,3),\ell_2-\ell_1,n-4-j,\tau,{\bf C'},[(u',v')])$ is given by the following:

\begin{itemize}
    \item A permutation $\tau$ given by
    \begin{equation*}
        \tau(i) = \begin{cases}
                i+\ell_1 &\text{if $ 1\leq i\leq \ell_2-\ell_1$}  \\
                \ell_2+1-i &\text{if $\ell_2-\ell_1+1\leq i\leq \ell_2-j$} \\
                i+j &\text{if $\ell_2-j+1\leq i\leq n-4-j$} \\
                i-(n-4-j) &\text{otherwise}. 
                \end{cases}
    \end{equation*}
    \item A tuple ${\bf C'}$ given by
    \begin{equation*}
        C'_i = \begin{cases}
                c_j+1-(u+C) &\text{if $i=j+1$} \\
                C_i &\text{otherwise}.
                \end{cases}
    \end{equation*}
    \item $(u',v')=(-u,u+C-c_j)$.
\end{itemize}

\end{proposition}

Before classifying the connected components of $\cP(\mu^\fR)$, we need to show that there exists a certain Type IIIc boundary for each connected component.

\begin{lemma}\label{lm:TypeD_IIIc_ex}
   Any connected component $\cC$ of $\cP(\mu^{\fR})$ contains some Type IIIc boundary of the form $$X^D_{\III c} ((4,3,1,2)),\ell_1,\ell_2,\tau,C,{\bf C}).$$ 
\end{lemma}
\begin{proof}
    By shrinking a collection of saddle connection, we can obtain some Type III boundary $\overline{X}$ of $\cC$. 
    
    First, suppose that $\overline{X}$ is a Type IIIa boundary $\overline{X}=X^D_{\III a} (\underline{h},\ell_1,\ell_2,\tau,{\bf C},Pr)$. If $\underline{h}=(4,3)$, then consider a prong-matching $(0,v)\in Pr$. A flat surface in $U\cdot \overline{X}(\frac{1}{2},v-\frac{1}{2})$ contains a saddle connection $\gamma$ bounding the polar domain of $q_4$. By shrinking the collection of saddle connection parallel to $\gamma$, we can reduce to the case of a Type IIIb or Type IIIc boundary. The same argument applies to the case $\underline{h}=(3,4)$.
    
    Second, suppose that $\overline{X}$ is a Type IIIb boundary $X^D_{\III b} (\underline{h},\ell_1,\ell_2,\tau,{\bf C})$. If $\underline{h}=(3,4,1,2)$, then $q_4\in X_0$ and $q_3\in X_{-1}$. Choose a prong-matching $u=\sum_{k=1}^{\ell_1} C_k+1$ so that $v^+_u$ is lying in the polar domain of $q_4$. Then $U\cdot \overline{X}(u)$ is adjacent to a desired Type IIIc boundary. The proof for $\underline{h}=(4,3,1,2)$ is symmetrical to above. 

    Finally, we are reduced to the Type III boundaries. Suppose that $\overline{X}=X^D_{\III c} ((1,2,4,3)),\ell_1,\ell_2,\tau,C,{\bf C})$. If $c^1_{\ell_1}\leq C <c^1_{\ell_1}+e_2$, then $U\cdot \overline{X}(0)$ is adjacent to a Type IIIb boundary, so we are done by the previous paragraph. Otherwise, $U\cdot \overline{X}(0)$ is adjacent to a Type IIIc boundary of the desired form.
\end{proof} 

\subsection{Hyperelliptic components} \label{subsec:Dhyper}

In this subsection, we classify the hyperelliptic components of a stratum $\cP(\mu^\fR)$ of D-signature. This stratum can have hyperelliptic components only if $e_1=e_2$ and $e_3=e_4$. In this case, a ramification profile $\Sigma$ on $\mu^\fR$ interchanges two pairs $(q_1,q_2)$ and $(q_3,q_4)$. Also it has two fixed points, one of which is the unique zero. So there is one fixed pole if and only if $n$ is even. If a boundary is contained in a hyperelliptic component, it also as an involution which is a degeneration of $\Sigma$. For a Type IIIc boundary $\overline{Z}=X^D_{\III c} ((1,2,3,4),\ell_1,\ell_2,\tau,C,{\bf C})$, the fixed pole is contained in the top level component if exists. Due to the symmetry given by heprelliptic involutions, we have the following equalities:

\begin{itemize}
    \item $\ell_2-\ell_1=n-4-\ell_2$
    \item $(C_{\tau(i)},D_{\tau(i)})=(D_{\tau(\ell_1+1-i)},C_{\tau(\ell_1+1-i)})$ for $1\leq i \leq \ell_1$
    \item $(C_{\tau(\ell_1+i)},D_{\tau(\ell_1+i)})= (D_{\tau(n-3-i)},C_{\tau(n-3-i)})$ for $1\leq i \leq \ell_2-\ell_1$
    \item $C=D$.
\end{itemize}

For each ramification profile $\Sigma$, it is easy to construct a Type IIIc boundary with ramification profile $\Sigma$. That is, there exists at least one hyperelliptic component for each ramification profile. Now we prove \Cref{Prop:Dhyper}, by proving the uniqueness of such connected component. 

\begin{proof}[Proof of \Cref{Prop:Dhyper}]
    Let $\Sigma$ be a ramification profile of $\cP(\mu^{\fR})$. If $\overline{X}=X^D_{\III c} ((4,3,1,2),\ell_1,\ell_2,\tau,C,{\bf C})$ is a boundary of $\cC$ obtained by \Cref{lm:TypeD_IIIc_ex}, then $C=\kappa-d_{\ell_1}$. So by \Cref{Prop:DIIIcmove} with $u=0$, we obtain $X^D_{\III c} ((1,2,4,3),n-4,n-4,\tau',C',{\bf C'})$. By the same argument again, we can reduce to the case $\overline{X}=X^D_{\III c} ((4,3,1,2),n-4,n-4,\tau,\frac{a}{2}-e_1,{\bf C})$. By relabeling the poles, we may assume that $\Sigma(i)=n-3-i$. Now $\cC$ is determined by $\tau$ and ${\bf C}$ satisfying $\tau(j)=\tau(n-3-j)$ and $C_j+C_{n-3-j}=b_j$ for each $j$. In order to show that $\cC$ is unique, it remains to show that $\cC$ also contains the following boundaries:
     
     \begin{itemize}
        \item[(1)] $X^D_{\III c} ((4,3,1,2),n-4,n-4,\Id,\frac{a}{2}-e_1,{\bf C'})$ for any ${\bf C'}$ satisfying $C'_j=D'_{n-3-j}$ for each $j$. 
        \item[(2)] $X^D_{\III c} ((4,3,1,2),n-4,n-4,\tau\circ(j,n-3-j),\frac{a}{2}-e_1,{\bf C})$ for each $1\leq j\leq \frac{n-4}{2}$.
        \item[(3)] $X^D_{\III c} ((4,3,1,2),n-4,n-4,\tau\circ(j,j+1)(n-3-j,n-4-j),\frac{a}{2}-e_1,{\bf C})$ for each $1\leq j< \frac{n-4}{2}$.
    \end{itemize}
    We can use (4) of $\Cref{Prop:DIIIcmove}$ to show each of above stpes. For the first step, it reduces to the case when $C_i=C'_i+1$, $C_{n-3-i}=C'_{n-3-i}-1$ for some $i$ and $C_j=C'_j$ for $j\neq i$. Note that $UR^2U\cdot\overline{X}(c_i)$ is adjacent to $X^D_{\III c} ((4,3,1,2),n-4,n-4,\Id,\frac{a}{2}-e_1,{\bf C'})$, a desired boundary. For the second step, note that $UR^{-2C_j}U\cdot \overline{X}(c_{j-1})$ and $UR^{2C_j}U\cdot \overline{X'}(-d_{j-1})$ are adjacent to the same boundary, where $\overline{X'}$ is the desired boundary. For the last step, note that $UR^{-2C_{j+1}}U\cdot \overline{X}(c_j)$ and $UR^{-2C_{j+1}}U\cdot \overline{X'}(c_{j+1})$ are adjacent to the same boundary, where $\overline{X'}$ is the desired boundary. 
\end{proof}

\subsection{Non-hyperelliptic components}

Now we classify the non-hyperelliptic components of $\cP(\mu^{\fR})$, proving \Cref{Prop:Dspin} - \Cref{Prop:Dnonhyper}. We first show the existence of a special collection of Type IIIc boundaries in non-hyperelliptic components.

\begin{lemma}\label{lm:non-hyp_D_IIIc_char}
    Any non-hyperelliptic component $\cC$ of $\cP(\mu^{\fR})$ has some Type IIIc boundary $$X^D_{\III c} (4,3,1,2),\ell_1,\ell_2,\tau,C,{\bf C})$$ such that 
    \begin{equation}\label{eq:non-hyp}
      C< e_4+ c_{\ell_1}.  
    \end{equation}
    Moreover, if $e_4>1$, then we can further assume that $\ell_1=\ell_2$. Also in this case, $\cC$ has some Type IIIb boundary of the form $X^D_{\III b} ((4,3,1,2),0,\ell_2,\tau',C',{\bf C'})$.    
\end{lemma}

\begin{proof}
    If $e_4>1$, we show first that the existence of any Type IIIb boundary is equivalent to the existence of the desired Type IIIc boundary. Let $\overline{Y}=X^D_{\III b} ((4,3,1,2),\ell'_1,\ell'_2,\tau',C',{\bf C'})$. Then by (2) of \Cref{Prop:DIIIbmove}, $U\cdot \overline{Y}(-(d')^2_{\ell_2})$ is adjacent to a Type IIIc boundary satisfying \Cref{eq:non-hyp}. Conversely, if $\overline{X}=X^D_{\III c} (4,3,1,2),\ell_1,\ell_2,\tau,C,{\bf C})$ satisfies \Cref{eq:non-hyp}, then there exists $u\in \ZZ/\kappa\ZZ$ such that $-e_4<u<0$ and $u+C\leq c_{\ell_1}$. By (1) of \Cref{Prop:DIIIcmove}, $U\cdot \overline{X}(u)$ is adjacent to a Type IIIb boundary. 

    For any $e_4>0$, we claim that the existence of the desired Type IIIc boundary is also equivalent to the existence of Type IIIc boundary of the form $X^D_{\III c} (1,2,4,3),\ell'_1,\ell'_2,\tau',C',{\bf C'})$, which contains $q_1,q_2$ on the top level component, satisfying $C'>e_2+c'_{\ell'_1}$. That is, we can ignore the data $\underline{h}$ in this proof. If such $\overline{X'}$ exists and $e_2>1$, then we have a Type IIIb boundary $\overline{Y'}=X^D_{\III b} ((2,1,4,3),\ell''_1,\ell''_2,\tau'',C'',{\bf C''})$. Then $U\cdot \overline{Y'}(-1)$ is adjacent to another Type IIIb boundary $\overline{Y}$ as in the previous paragraph, so the proof is complete. If $(e_1,e_2)=(1,1)$, then $U\cdot \overline{Y'}(-1)$ is adjacent to the desired Type IIIc boundary $\overline{X}$. 

    By \Cref{lm:TypeD_IIIc_ex}, $\cC$ has some Type IIIc boundary $\overline{X}=X^D_{\III c} ((4,3,1,2),\ell_1,\ell_2,\tau,C,{\bf C})$. First, assume that $C > e_3+c_{\ell_1}$. If $e_4>1$, then there exists $-e_4\leq u<0$ such that $c_{\ell_1}+e_3<u+C\leq \kappa-e_4$. By (2) of \Cref{Prop:DIIIcmove}, we obtain a Type IIIb boundary by $U\cdot \overline{X}(u)$, so the proof is complete. 
    
    Now assume that $C\leq e_3+c_{\ell_1}$. If $e_3<e_4$, then we have \Cref{eq:non-hyp}. So may assume that $e_3=e_4$ and $C=e_4+c_{\ell_1}$. In particular, since $e_3>2$, we can also reduce to the case $C=d_{\ell_1}=e_4+\sum_{i=1}^{\ell_1} D_{\tau(i)}$. That is, $C=D=d_{\ell_1}=e_4+c_{\ell_1}$. Consider $U\cdot \overline{X}(0)$. This is adjacent to a Type IIIc boundary containing all residueless poles on the top level component. So we can further assume that $\ell_1=\ell_2=n-4$. In particular, the bottom level component $X_{-1}$ of $\overline{X}$ does not contain any residueless pole and satisfies $C=D=\frac{1}{2}\sum_i b_i$, thus hyperelliptic. Since $\cC$ is non-hyperelliptic, the top level component $X_0$ of $\overline{X}$ must not be hyperelliptic. If $C_i=D_{n-3-i}$ for each $i$, then the top level component is hyperelliptic. So we can find smallest integer $k$ such that $C_k\neq D_{n-3-k}$ or $D_k\neq C_{n-3-k}$. In particular, $b_i=b_{n-3-i}$ for each $i=1,\dots, k-1$. Suppose that $C_k>D_{n-3-k}$. Other cases will follow symmetrically. Then $U\cdot \overline{X}(c_{k-1}+D_{n-3-k})$ is adjacent to a Type IIIc boundary $\overline{X'}$ such that the top level component contains $p_k,\dots, p_{n-4-k}$ and $c'_{\ell'_1}=C-\sum_{i=1}^k b_{n-3-i} =\frac{1}{2}\left(\sum_{i=k}^{n-4-k}b_i - b_{n-3-k}\right)<\frac{1}{2}\sum_{i=k}^{n-4-k}b_i$. So this satisfies \Cref{eq:non-hyp}. 
\end{proof}

As a consequence, we can obtain a certain Type IIIa boundary by the following

\begin{corollary} \label{cor:non_hyp_IIIa_ex}
    Suppose that $e_4>1$. Any non-hyperelliptic component $\cC$ of $\cP(\mu^{\fR})$ contains some Type IIIa boundary of the form $X^D_{\III a} ((4,3),0,n-4,\tau,{\bf C},Pr)$.
\end{corollary}

\begin{proof}
    By \Cref{lm:non-hyp_D_IIIc_char}, we have some Type IIIb boundary of the form $\overline{X}=X^D_{\III b} ((4,3,1,2),0,\ell_2,\tau,C,{\bf C})$. By (2) of \Cref{Prop:DIIIbmove}, $U\cdot \overline{X}(C)$ is adjacent to a Type IIIa boundary of the form $X^D_{\III a} ((4,3),0,n-4,\tau,{\bf C},Pr)$. 
\end{proof}

Note that $X^D_{\III a} ((4,3),0,n-4,\tau,{\bf C},Pr)$ is determined by combinatorial data $\tau$, ${\bf C}$ and $Pr$. For the rest of this section, we navigate the boundary and control these data to classify the non-hyperelliptic components.

\paragraph{\bf No pair of simple poles}
    Assume $(e_1,e_2),(e_3,e_4)\neq(1,1)$. For any non-hyperelliptic component of $\cP(\mu^{\fR})$, we can find a certain boundary. Let ${\bf 1}$ denote the tuple $(1,\dots, 1)$. More precisely, we have

\begin{proposition} \label{prop:TypeD_nonhyp_no_sim_pair}
    Assume $(e_1,e_2),(e_3,e_4) \neq (1,1)$. Any non-hyperelliptic component contains the boundary $X^D_{\III a} ((4,3),0,n-4,\Id,{\bf 1},[(0,0)])$.
\end{proposition}

\begin{proof}
    By assumption, we have $e_2,e_4>1$. By \Cref{cor:non_hyp_IIIa_ex}, $\cC$ has a boundary of the form $X^D_{\III a} ((4,3),0,n-4,\tau,{\bf C},Pr)$. We find the following boundary of $\cC$ for each step:
    \begin{itemize}
        \item[(1)] $X^D_{\III a} ((4,3),0,n-4,\tau,{\bf C},[(0,1)])$. 
        \item[(2)] $X^D_{\III a} ((4,3),0,n-4,\tau,{\bf 1},[(0,1)])$.
        \item[(3)] $X^D_{\III a} ((4,3),0,n-4,\Id,{\bf 1},[(0,1)])$.
    \end{itemize}
    
    For the first step, let $\overline{X}=X^D_{\III a} ((4,3),0,n-4,\tau,{\bf C},Pr)\in \partial\overline{\cC}$. Choose $(0,w)\in Pr$. By level rotation action, we may assume that $0< w\leq \kappa_1$. We use induction on $w>0$. If $w=1$, then nothing to prove. Suppose that $(u,v)\in Pr$ for some $0<u<e_4$ and $0<v<e_2$. Let $\overline{X'}=X^D_{\III a} ((4,3),0,n-4,\tau,{\bf C},[(u,v-1)])$. Then $R^2U\cdot \overline{X}(u,v)=U\cdot \overline{X'}(u,v-1)$. Since $u+(v-1)=w-1$, the induction applies to $\overline{X'}$. So it is sufficient to find such $(u,v)\in Pr$. Assume the contrary that whenever $(u,v)\in Pr$ and $0<v<e_2$, then $e_4\leq u \leq \kappa_1$. Then $e_2-1\leq \kappa_1-e_4+1=e_1\leq e_2$. That is, $e_2=e_1$ or $e_2=e_1+1$. Moreover, if $e_2=e_1+1$, then we must have $(1,\kappa_1)\in Pr$. By level rotation action, we have $(0,1)\in Pr$. So assume that $e_1=e_2$. In this case, we have either $(1,\kappa_1)\in Pr$ or $(1,\kappa_1-1)\in Pr$. If $(1,\kappa_1)$, then we again have $(0,1)\in Pr$. Therefore we are reduced to the case $(1,\kappa_1-1)\in Pr$. In this case, we have $(1,-1)\in Pr$. By relabeling the poles, we may assume that $\tau=\Id$. If $e_4>2$, let $\overline{X'}=X^D_{\III a} ((4,3),0,n-4,\Id,{\bf C},[(2,-1)])$. Then $R^2UR^{-2(D_1+1)}U\cdot\overline{X}(1,-1)=UR^{-2(D_1+1)}U\cdot\overline{X'}(2,-1)$. Finally, suppose that $e_4=2$. If $D_1>1$, then we consider $UR^{2(D_1-1)}U\cdot \overline{X}(1,-1)$ and reduce to $D_1=1$. Let $\overline{X'}=X^D_{\III a} ((4,3),0,n-4,\Id,{\bf C},[(0,-1)])$. Then $UR^{-2D_1-1}U\cdot \overline{X}(1,-1)=R^4U\cdot\overline{X'}(0,-1)$. This completes the first step. 

    For the second step, let $\overline{X}=X^D_{\III a} ((4,3),0,n-4,\tau,{\bf C},[(0,1)])\in \partial\overline{\cC}$ and assume that $C_k>1$ for some $k$. It is sufficient to prove that $\overline{X'}=X^D_{\III a} ((4,3),0,n-4,\tau,{\bf C'},Pr)\in \partial\overline{\cC}$ where $C'_k=C_k-1$ and $C'_i=C_i$ for any $i\neq k$. By repeating this, we can reach the case $C_i=1$ for all $i$. By relabeling the poles, we may assume that $\tau=\Id$. Also, we may assume that $(1,c^2_{k-1})\in Pr$. Then $R^2U\cdot \overline{X} (1,c^2_{k-1}) = U\cdot \overline{X'}(1,c^2_{k-1})$, so $\overline{X'}\in \partial\overline{\cC}$ and this completes the second step.

    For the last step, let $\overline{X}=X^D_{\III a} ((4,3),0,n-4,\tau,{\bf 1},[(0,0)])\in \partial\overline{\cC}$. We may assume that $\tau=\Id$ and it is sufficient to prove that $\overline{X'}=X^D_{\III a} ((4,3),0,n-4,(k,k+1),{\bf 1},Pr')\in \partial\overline{\cC}$ for each $k$, since $\operatorname{Sym}_{n-4}$ is generated by transpositions $(k,k+1)$, $k=1,\dots, n-5$. Note that $c^2_i=i$ for each $i$. Again, by the first step, we may assume that $(1,k-1)\in Pr$ and $(1,k)\in Pr'$. Then $UR^{2(b_k-1)}U\cdot \overline{X} (1,k-1)$ and $UR^{2(b_{k+1}-1)}U\cdot \overline{X'}$ are adjacent to the same boundary. So $\overline{X'}\in \partial{\cC}$ and this completes the proof. 
\end{proof}

\paragraph{\bf One pair of simple poles}
Now we assume $(e_1,e_2)=(1,1)$ and $e_4>1$. Recall that the flat surfaces in $\cP(\mu^\fR)$ has a deformation invariant called {\em index}. This is an integer $I$ satisfying $1\leq I \leq \delta=\gcd(e_3,e_4,b_1,\dots, b_{n-4})$. The index of a connected component $\cC$ is the index of flat surfaces contained in $\cC$. We can compute the index of $\cC$ from the combinatorial data of its boundaries. For example, we have the following

\begin{lemma} \label{lm:index_D_IIIa}
     Assume $(e_1,e_2)=(1,1)$ and $e_4>1$. The index of a connected component $\cC$ containing the boundary $\overline{X}=X^D_{\III a} ((4,3),\ell_1,\ell_2,\tau,{\bf C},[(0,v)])$ is equal to 
    \begin{equation*}
          v+\sum_{i=\ell_1+1}^{\ell_2} C_{\tau(i)} \pmod \delta.
    \end{equation*}
\end{lemma} 

\begin{proof}
    Consider a flat surface in the equatorial arc of $U\cdot\overline{X}(0,v)$. The index is given by the angle formed by two saddle connections $\alpha_0$ and $\alpha_{\ell_2+1}$, each bounding the polar domains of $q_1$ and $q_2$. The angle between $\alpha_0$ and $\beta_0$ is equal to $0$, the angle between $\beta_0$ and $\beta_{n-4-\ell_2}$ is equal to $2\pi \sum_{i=\ell_1+1}^{\ell_2} C_{\tau(i)}$, and the angle between $\beta_{n-4-\ell_2}$ and $\alpha_{\ell_2+1}$ is equal to $2\pi v$. Their sum is equal to $2\pi (v+\sum_{i=\ell_1+1}^{\ell_2} C_{\tau(i)})$, gives the formula for the index of $\cC$.
\end{proof}

In particular, if $\ell_1=\ell_2=0$, then the index is $v \pmod \delta$. For each $I$, the boundary $X^D_{\III a} ((4,3),0,n-4,\Id,{\bf 1},[(0,I)])$ is contained in the connected component with index $I$. If $n-4>0$, then this proves that there exists at least one non-hyperelliptic component with index $I$. Conversely, we have the following

\begin{proposition} \label{prop:TypeD_index}
    Assume $(e_1,e_2)=(1,1)$ and $e_4>1$. Any non-hyperelliptic component $\cC$ with index $I$ contains the boundary $X^D_{\III a} ((4,3),0,n-4,\Id,{\bf 1},[(0,I)])$.
\end{proposition}

\begin{proof} 
    By \Cref{cor:non_hyp_IIIa_ex}, $\cC$ has a boundary of the form $X^D_{\III a} ((4,3),0,n-4,\tau,{\bf C},Pr)$. Recall that $1\leq I \leq \delta=\gcd(e_3,e_4,b_1,\dots, b_{n-4})$. We find the following boundary of $\cC$ for each step:    
    \begin{itemize}
        \item[(1)] $X^D_{\III a} ((4,3),0,n-4,\tau,{\bf C},[(0,I)])$.
        \item[(2)] $X^D_{\III a} ((4,3),0,n-4,\tau,{\bf 1},[(0,I)])$.
        \item[(3)] $X^D_{\III a} ((4,3),0,n-4,\Id,{\bf 1},[(0,I)])$.
    \end{itemize}
    
    For the first step, choose $(0,v)\in Pr$. By \Cref{lm:index_D_IIIa}, $v\equiv I\pmod \delta$. It is sufficient to prove that $X^D_{\III a} ((4,3),0,n-4,\tau,{\bf C},[(0,v-d)])\in \partial\overline{\cC}$. By relabeling the poles, we may assume $\tau=\Id$. In this case, $\kappa_1=e_4$ and $\kappa_2=e_3+\sum_{i=1}^{n-4} b_i$. If $n-4=0$, then $\delta=\gcd(\kappa_1,\kappa_2)$ and we have $(0,v-d)\in [(0,v)]$ by level rotation action. So assume $n-4>0$. Denote $r_k=\sum_{i=1}^k b_i$ for $k=1,\dots, n-4$. Then $$\delta=\gcd(\kappa_1,\kappa_2,r_1,\dots, r_{n-4})$$ and it is sufficient to prove $X^D_{\III a} ((4,3),0,n-4,\tau,{\bf C},[(0,v-r_k)])\in \partial\overline{\cC}$ for each $k$. 
    
    Choose a prong-matching $(u,c^2_k)\in Pr$. Then $u+c^2_k\equiv I \pmod \delta$. Given this condition, we have $q\coloneqq \kappa_1/\gcd(\kappa_1,r_k)$ possible values for $u\in \ZZ/\kappa_1\ZZ$. If $u\neq 0$, then $UR^{2\kappa_1+2}U\cdot \overline{X}(u,c^2_k)$ is adjacent to another Type IIIa boundary $X^D_{\III a} ((4,3),0,n-4,\Id,{\bf C},[(u,-d^2_k+1)])$. Since $c^2_k+d^2_k=\sum_{i=1}^k b_i$, we have $(r_k,v)\in [(u,-d^2_k+1)]$. By repeating this, we can either obtain all $q$ values for $u$, or we reach the case $(0,c^2_k)\in Pr$. In any case, we can conclude that all $q$ cases are contained in the same connected component, completing the first step. 

    For the second step, let $\overline{X}=X^D_{\III a} ((4,3),0,n-4,\tau,{\bf C},[(0,I)])\in \partial\overline{\cC}$. We may assume $\tau=\Id$ and $C_k>1$ for some $k$. It is sufficient to prove that $X^D_{\III a} ((4,3),0,n-4,\Id,{\bf C'},Pr)\in \partial\overline{\cC}$ where $C'_k=C_k-1$ and $C'_i=C_i$ for any $i\neq k$. Choose a prong-matching $(u,c^2_{k-1})\in Pr$ such that $c^2_{k-1}<v<c^2_k$. If $u\neq 0$, then $R^2U\cdot \overline{X}(u,v)=U\cdot \overline{X'}(u,v)$. So assume $u=0$, $v=c^2_{k-1}+1$ and $C_k=2$. Then $UR^{2(\kappa_1+2)}U\cdot \overline{X}(-1,v+1)$ is adjacent to $\overline{X'}$.     
    
    For the last step, let $\overline{X}=X^D_{\III a} ((4,3),0,n-4,\tau,{\bf 1},[(0,I)])\in \partial\overline{\cC}$. We may assume $\tau=\Id$. It is sufficient to prove that $\overline{X'}=X^D_{\III a} ((4,3),0,n-4,(k,k+1),{\bf 1},[(0,I)])\in \partial\overline{\cC}$ for each $k=1,\dots, n-5$. Suppose that there exists $(u,k)\in Pr$ for some $u\neq 0,1$. Then $R^{2(b_k-1)}U\cdot \overline{X}(u,k)$ and $R^{2(b_k-1)} U \cdot \overline{X'}(u-1,k+1)$ are adjacent to the same boundary, thus the proof is complete. So assume that $(u,k)\in Pr$ for $u=0$ or $u=1$. 

    Recall that $\kappa_1=e_4$. We divide the proof into the following cases:
    \begin{enumerate}
        \item $\delta<\kappa_1$
        \item $\delta=\kappa_1>2$
        \item $\delta=\kappa_1=2$, $b_i>2$ for some $i$
        \item $\delta=\kappa_1=2$, $b_i=2$ for each $i$
    \end{enumerate}
    
    Note that $\kappa_1=e_4$. If $\delta<\kappa_1$, then by the first step, we have $(u+d,k)\in Pr$. Also, $u+d\leq 1+d\leq \kappa_1$. If $u+\delta=\kappa_1$, then $u=1$ and $\delta=\kappa_1-1$. Since $d\mid \kappa_1$, we have $\kappa_1=2$ and $\delta=1$. In this case, by the first step, we can take a prong-matching $(1,k+1)$ for $\overline{X'}$. Then $R^{2(b_k-1)}U\cdot \overline{X}(1,k)$ and $R^{2(b_k-1)} U \cdot \overline{X'}(1,k+1)$ are adjacent to the same boundary, thus the proof is complete. 
    
    If $\delta=\kappa_1>2$, then automatically we have $b_i\geq d>2$ for each $i$. By the second step, we can take any other ${\bf C}$ for $\overline{X}$ and $\overline{X'}$. Since $b_k>2$ we can set $C_k=2$ and $C_i=1$ for any $i\neq k$. If $u=0$ and $(0,k)\in Pr$, then $(-1,k+1),(-2,k+2)\in Pr$. Then $R^{2(b_k-2)}U\cdot \overline{X}(-1,k+1)$ and $R^{2(b_k-2)} U \cdot \overline{X'}(-2,k+2)$ are adjacent to the same boundary. If $u=1$ and $(1,k)\in Pr$, then $(-1,k+2)\in Pr$. So $R^{2(b_{k+1}-1)}U\cdot \overline{X}(-1,k+2)$ and $R^{2(b_{k+1}-1)} U \cdot \overline{X'}(1,k)$ are adjacent to the same boundary, thus the proof is complete.

    If $\delta=\kappa_1=2$, then we also have $e_3=2$ and $2\mid b_i$ for each $i$. By the first step, we may assume that $(0,0)\in Pr$ or $(0,1)\in Pr$. First, assume that $(0,0)\in Pr$. The proof for the case $(0,1)\in Pr$ is similar to this case. Then also by the first step, we can choose any prong-matching $(1,v)$ for odd $v$. In this case, let $\overline{X'}=X^D_{\III a} ((4,3),0,n-4,\tau,{\bf 1},[(0,0)])$ with $\tau=(k,k+2)$ for odd $k$. Then $R^{2(b_k-1)}U\cdot \overline{X}(1,k)$ and $R^{2(b_k-1)} U \cdot \overline{X'}(1,k+2)$ are adjacent to the same boundary. Similarly for $\tau=(k,k+2)(k+1,k+3)$, $R^{2(b_k+b_{k-1}-2)}U \cdot \overline{X}(1,k)$ and $R^{2(b_k+b_{k-1}-2)} U \cdot \overline{X'}(1,k+2)$ are adjacent to the same boundary. Therefore, $\overline{X'}\in \partial\overline{\cC}$ for any $\tau=(k,k+2)$. Note that $\operatorname{Sym}_{n-4}$ is generated by $(k,k+2)$ for each $k$ and $(1,2)$. So it remains to take care of the case $\tau=(1,2)$. If $b_2>2$, then by the second step, we may take $C_2=2$ and $C_i=1$ for each $i\neq 2$. Then $c^2_1=1$ and $c^2_2=3$ are both odd. Thus $R^{2(b_1-1)}U\cdot \overline{X}(1,1)$ and $R^{2(b_1-1)} U \cdot \overline{X'}(1,3)$ are adjacent to the same boundary. Similarly, if $b_1>2$, then by the same argument applied to $\overline{X'}$ instead of $\overline{X}$, we obtain the same result. If $b_1=b_2=2$, then $UR^{-2}UR^{-4} UR^3UR^2 U\cdot\overline{X}(\frac{1}{2},\frac{3}{2}) = \overline{X'}(\frac{1}{2},-\frac{1}{2})$. 
\end{proof}

Note that $X^D_{\III a} ((4,3),0,n-4,\Id,{\bf 1},[(0,I)])$ is not hyperelliptic unless $n-4=0$, $e_3=e_4$ and $I=\delta$. In the exceptional case, $\cP(\mu^\fR)=\cP(2b \mid -b^2 \mid-1^2)$ for some $b$ and $I=b$.

\paragraph{\bf Two pairs of simple poles}
Finally, we assume $(e_1,e_2)=(e_3,e_4)=(1,1)$. If all $b_i$ are even, recall that the flat surfaces in $\cP(\mu^\fR)$ has a deformation invariant called {\em spin parity}. The spin parity of a connected component $\cC$ is the spin parity of flat surfaces contained in $\cC$. We can compute the spin parity of $\cC$ from the combinatorial data of its boundaries. For example, we have the following

\begin{lemma} \label{lm:spin_D_IIIc}
    Assume $(e_1,e_2)=(e_3,e_4)=(1,1)$ and all $b_i$ are even. The spin parity of a connected component $\cC$ containing the boundary $\overline{X}=X^D_{\III c} ((4,3,1,2),\ell_1,\ell_2,\tau,C,{\bf C})$ is equal to 
    \begin{equation*}
          C+ \sum_{i=1}^{n-4} C_i \pmod 2.
    \end{equation*}
\end{lemma}

\begin{proof}
    The angle between $q_4$ and $q_3$ is equal to $2\pi \sum_{i=1}^{\ell_1}C_{\tau(i)}$. The angle between $q_1$ and $q_2$ is equal to $2\pi \left(C+ \sum_{i=\ell_1+1}^{n-4} C_{\tau(i)} \right)$. Consider a multi-scale differential of genus two obtained by gluing two pairs of half-infinite cylinders. 
    
    These two angles corresponds to the indices $I_1,I_2$ of the cross curves $\gamma_1,\gamma_2$ of two cylinders obtained by gluing two pairs of half-infinite cylinders, respectively, multiplied by $2\pi$. The core curves $\delta_1,\delta_2$ of the cylinders have index zero, and four curves form a symplectic basis of this genus two nodal curve. So the spin parity is equal to $(I_1+1)(0+1)+(I_2+1)(0+1) \equiv I_1+I_2=C+ \sum_{i=1}^{n-4} C_i \pmod 2$.
\end{proof}

In particular, if $\cC$ contains $X^D_{\III c} ((4,3,1,2),n-4,n-4,\tau,C,{\bf 1})$, then the spin parity of $\cC$ is equal to $n+C\pmod 2$. So each spin parity can be realized by the boundaries of this form with $C=1$ and $C=2$. Next, we prove the converse that each non-hyperelliptic component contains a boundary of this form. 

\begin{lemma} \label{lm:DIIIcspin}
    Assume $(e_1,e_2)=(e_3,e_4)=(1,1)$. Any non-hyperelliptic component $\cC$ of $\cP(\mu^\fR)$ contains a boundary $X^D_{\III c} ((4,3,1,2),n-4,n-4,\Id,C,{\bf 1})$ such that $C=1$ or $C=2$. 
\end{lemma}

\begin{proof}
    By \Cref{lm:non-hyp_D_IIIc_char}, $\cC$ contains a Type IIIc boundary $\overline{X}=X^D_{\III c} ((4,3,1,2),\ell_1,\ell_2,\tau,C,{\bf C})$. We find the following boundary of $\cC$ for each step.    
    \begin{itemize}
        \item[(1)] $X^D_{\III c} ((4,3,1,2),n-4,n-4,\tau,C,{\bf C})$.
        \item[(2)] $X^D_{\III c} ((4,3,1,2),n-4,n-4,\tau,C',{\bf C})$ with $C'\leq 2$. Furthermore, $C'\leq 1$ if $b_i>1$ for some $i$. 
        \item[(3)] $X^D_{\III c} ((4,3,1,2),n-4,n-4,\Id,C',{\bf C})$ with $C'\leq 2$.
        \item[(4)] $X^D_{\III c} ((4,3,1,2),n-4,n-4,\Id,C',{\bf 1})$ with $C'\leq 2$.
    \end{itemize}
    
    For the first step, suppose that $\sum_{i=1}^{k-1} C_i\leq C<\sum_{i=1}^k C_i$ for some $i$. Then $U\cdot \overline{X}(-\frac{1}{2})$ is adjacent to a Type IIIa boundary $\overline{Y}$ of the form $X^D_{\III a} ((4,3),0,\ell_2,\tau',{\bf C'},Pr)$. Since $\kappa_1=1$, by level rotation action, we can obtain $(0,-1)\in Pr$. Then $U\cdot \overline{Y}(\frac{1}{2},-\frac{1}{2})$ gives a desired boundary in the first step. 

    For the second step, we assume that $\overline{X}=X^D_{\III c} ((4,3,1,2),n-4,n-4,\tau,C,{\bf C})$. If $C\leq 2$, then we are done. Assume that $C>2$ and use induction on $C$. By relabeling the poles, we may assume $\tau=\Id$. Let $k$ be the smallest index such that $C-1< \sum_{i=1}^k C_i$. By assumption, $k>1$. If $\sum_{i=1}^{k-1} C_i < C-1$, then $UR^2 U\cdot \overline{X}(-\frac{1}{2})$ is adjacent to $X^D_{\III c} ((4,3,1,2),n-4,n-4,\tau,C-1,{\bf C'})$ where $C'_k=C_k-1$ and $C'_i=C_i$ for any $i\neq k$. So the induction hypothesis applies. Now assume that $C-1=\sum_{i=1}^{k-1} C_i$. If $C>b_1$, then $C-D_1\geq C_1$. If $C-D_1=\sum_{i=1}^j C_i$, then we take $UR^{-2}U\cdot \overline{X}(-\frac{1}{2})$ instead, reducing to the case $C-D_1=1+\sum_{i=1}^j C_i$. So we may assume that $\sum_{i=1}^j C_i<C-D_1<\sum_{i=1}^{j+1} C_i$ for some $j$. Then $UR^2U \overline{X}(-D_1-\frac{1}{2})$ is adjacent to a boundary of the form $X^D_{\III c} ((4,3,1,2),n-4,n-4,\tau,C-b_1,{\bf C'})$. So the induction hypothesis applies. Now it reduces to the case $C=\sum_{i=1}^{k-1} C_i +1 \leq b_1$. Then $0<C-C_1\leq D_1$ and $UR^2U\overline{X}(C_1-C-\frac{1}{2})$ is adjacent to a Type IIIc boundary $\overline{X'}=X^D_{\III c} ((4,3,1,2),n-4,n-4,\tau,C',{\bf C'})$ such that $D'=a-C'=b_1-C<D_1$. Then $UR^{2(b_1-C)}U\cdot\overline{X'}(-\frac{1}{2})$ is adjacent to a desired element with $C'\leq 2$. Assume further that $b_i>2$ for some $i$ and $C=2$. Let $k$ be the smallest index such that $b_k>2$. If $k=1$, then $C-1<C_1$ or $C-2<D_1$, so by above argument, we can decrease $C$ by one. So assume that $b_1=2$ and $k>1$. Also, suppose that $C_k>1$. The case $D_k>1$ follows by the symmetry. Then $\sum_{i=1}^{k-1}C_i=k-1$ and $\sum_{i=1}^k C_i\geq k+1$. By the above argument, for each even $v$, we can find some $X^D_{\III c} ((4,3,1,2),n-4,n-4,\tau,v,{\bf C'})\partial\overline{\cC}$. If $k$ is odd, then take $v=k+1$. Since $\sum_{i=1}^{k-1}C_i< k < \sum_{i=1}^k C_i$, we can decrease $v$ by one. By reducing $v$ by $b_1=2$ repeatedly, we obtain the desired boundary. If $k$ is even, then take $v=k+2$. If $k+1 < \sum_{i=1}^k C_i$, then we are done by the same argument. If $C_k=2$ and $k+1=\sum_{i=1}^k C_i$. Since $D_1=1$, we can increase $v$ by one and obtain the odd $k$ case. 
    
    For the third step, we assume that $\overline{X'}=X^D_{\III c} ((4,3,1,2),n-4,n-4,\Id, C,{\bf C})$ with $C\leq 2$. If $C=1$, for any $1\leq k\leq n-4$, $UR^{2d_k}UR^{2(d_{n-4}-d_k)}UR^{-2D_k}U\cdot \overline{X}(-\frac{1}{2})$ is adjacent to $X^D_{\III c} ((4,3,1,2),n-4,n-4,\tau_k, 1,{\bf C})$ for \begin{equation*}
        \tau(i) = \begin{cases}
                i+k &\text{if $ 1\leq i\leq n-4-k$}  \\
                (k+1)-(i-(n-4-k)) &\text{otherwise}. 
                \end{cases}
    \end{equation*}

    Note that $(k,k+1)=\tau_{k-1}^{-1}\circ \tau_1^{-2}\circ \tau_2\circ \tau_{k-1}$. Therefore, $\tau_k$ for each $k$ generate $\operatorname{Sym}_{n-4}$ and therefore $\overline{X}$ and $\overline{X'}$ are contained in the same connected component. If all $b_i=2$ and $C=2$, then for any $2\leq k\leq n-4$, we can obtain $X^D_{\III c} ((4,3,1,2),n-4,n-4,\tau_k, 2,{\bf C})$ by the similar argument above. So it remains to deal with the case $\overline{X'}=X^D_{\III c} ((4,3,1,2),n-4,n-4,(1,2), 2,{\bf C})$. We have $UR^{-2}UR^4 URU\cdot\overline{X}(1)=\overline{X'}(-\frac{7}{2})$, so $\overline{X'}\in \partial\overline{\cC}$. 
    
    For the final step, let $\overline{X}=X^D_{\III c} ((4,3,1,2),n-4,n-4,\Id,C',{\bf C})$ with $C'\leq 2$. We use the induction on $\sum_i C_i \geq n-4$. If $\sum_i C_i=n-4$, then $C_i=1$ for each $i$ and the proof is complete. So assume that $\sum_i C_i>n-4$ and $C_i>1$ for some $i$. By relabeling the poles, we may assume that $C_1>1$. If $C'=1$, then let $\overline{X'}=X^D_{\III c} ((4,3,1,2),n-4,n-4,\Id,2,{\bf C'})$ where $C'_1=C_1-1$ and $C'_i=C_i$ for each $i\neq 1$. Then $R^2U\cdot \overline{X}(-\frac{5}{2}) = U\cdot \overline{X'}(-\frac{3}{2})$. So the induction applies to $\overline{X'}$. If $C'=2$, then let $\overline{X'}=X^D_{\III c} ((4,3,1,2),n-4,n-4,\Id,1,{\bf C'})$ where $C'_1=C_1-1$ and $C'_i=C_i$ for each $i\neq 1$. Then $R^{-5}U\cdot \overline{X}(-\frac{3}{2})=U\cdot \overline{X'}(1)$. Again the induction applies to $\overline{X'}$. Note that the parity of $C'+\sum_i C_i$ is remained invariant throughout this last step. 
\end{proof}

\begin{proposition} \label{prop:TypeD_spin}
    Any non-hyperelliptic component $\cC$ of $\cP(\mu^{\fR})$ contains a Type IIIc boundary $X^D_{\III c} ((4,3,1,2),n-4,n-4,\Id,C,{\bf 1})$ for $C=1$ or $C=2$. If some $b_i$ is odd, then both boundaries are contained in the same non-hyperelliptic component. 
\end{proposition}

\begin{proof}
    The first statement is immediate by \Cref{lm:DIIIcspin}. If some $b_i$ is odd, by relabeling the poles, we may assume that $b_1$ is odd. Let $\overline{X}=X^D_{\III c} ((4,3,1,2),n-4,n-4,\Id,1,{\bf 1})$. Then $URUR^{2b_1-1}U\cdot \overline{X}(-\frac{1}{2})$ is adjacent to $X^D_{\III c} ((4,3,1,2),n-4,n-4,\Id,1,{\bf C})$ where $C_1=b_1-1$ and $C_i=1$ for $i\neq 1$. The last step in the proof of \Cref{lm:DIIIcspin} preserves the parity $C+\sum_iC_i$. Since $n-4+1\not\equiv n-3+(b_1-1)+1\pmod 2$, $\cC$ also contains $X^D_{\III c} ((4,3,1,2),n-4,n-4,\Id,2,{\bf 1})$. 
\end{proof}

Note that if all $b_i$ are even, the boundary $\overline{X}=X^D_{\III c} ((4,3,1,2),n-4,n-4,\Id,C,{\bf 1})$ with $C=1$ and $C=2$ are contained in two distinct connected components by \Cref{lm:spin_D_IIIc}. If some $b_i>2$, then the top level component $X_0$ of $\overline{X}$ is non-hyperelliptic, so indeed they are non-hyperelliptic components. If all $b_i=2$, then $X_0$ is hyperelliptic and $\kappa=2(n-4)+1$. The bottom level component $X_{-1}$ of $\overline{X}$ is hyperelliptic if only if $C=\frac{\kappa+1}{2}=n-3$. So there are two non-hyperelliptic components distinguished by spin parity, except the following cases:
\begin{itemize}
    \item If $\mu^{\fR}=(2\mid-1^2\mid-1^2)$, then we must have $C=1$ since $1\leq C\leq \kappa-1=1$. However, $X^D_{\III c} ((4,3,1,2),0,0,\Id,1,{\bf 1})$ is contained in a hyperelliptic component. Thus $\cP(2\mid-1^2\mid-1^2)$ has {\em no} non-hyperelliptic component.
    \item If $\mu^{\fR}=(4\mid-2\mid-1^2,-1^2)$, then we can have $C=1,2$ since $\kappa=3$. However, $X^D_{\III c} ((4,3,1,2),1,1,\Id,2,{\bf 1})$ is contained in a hyperelliptic component. Thus $C=1$ gives a {\em unique} non-hyperelliptic component, which is of even parity. 
\end{itemize}

This completes the classification of connected components of strata of D-signatures. Now we can prove \Cref{Cor:Dssc}. 

\begin{proof}[Proof of \Cref{Cor:Dssc}] 
    If $e_2>1$, then there exists a unique non-hyperelliptic component and it contains the boundary $X^D_{\III b} (2,1,4,3),0,n-4,\Id,1,{\bf 1})$. The bottom level component contains a multiplicity one saddle connection bounding the polar domain of $q_1$. If $b_i>2$ for some $i$, $e_2>2$ or $e_3<e_4$, then the top level component is non-hyperelliptic. 

    If $(e_1,e_2)=(1,1)$, then any non-hyperelliptic component contains a Type IIIc boundary $X^D_{\III c} ((4,3,1,2),n-4,0,\Id,C,{\bf 1})$ for some $C$. The bottom level component has a pair of multiplicity two saddle connections bounding the polar domains of $q_1$ and $q_2$. If $b_i>2$ for some $i$ or $e_3<e_4$, then the top level component is non-hyperelliptic. 
\end{proof}

%% file: 00Main.bbl
\begin{thebibliography}{10}

\bibitem{bainbridge2018compactification}
{\sc M.~Bainbridge, D.~Chen, Q.~Gendron, S.~Grushevsky, and M.~M{\"o}ller}, {\em Compactification of strata of abelian differentials}, Duke Mathematical Journal, 167 (2018).

\bibitem{bainbridge2019moduli}
\leavevmode\vrule height 2pt depth -1.6pt width 23pt, {\em The moduli space of multi-scale differentials}, arXiv preprint arXiv:1910.13492,  (2019).

\bibitem{benirschke2022boundary}
{\sc F.~Benirschke}, {\em The boundary of linear subvarieties}, Journal of the European Mathematical Society, 25 (2022), pp.~4521--4582.

\bibitem{benirschke2022equations}
{\sc F.~Benirschke, B.~Dozier, and S.~Grushevsky}, {\em Equations of linear subvarieties of strata of differentials}, Geometry \& Topology, 26 (2022), pp.~2773--2830.

\bibitem{boissy2015connected}
{\sc C.~Boissy}, {\em Connected components of the strata of the moduli space of meromorphic differentials}, Commentarii Mathematici Helvetici, 90 (2015), pp.~255--286.

\bibitem{chan2021tropical}
{\sc M.~Chan, S.~Galatius, and S.~Payne}, {\em Tropical curves, graph complexes, and top weight cohomology of $\mathcal{M}_g$}, Journal of the American Mathematical Society, 34 (2021), pp.~565--594.

\bibitem{chan2022topology}
\leavevmode\vrule height 2pt depth -1.6pt width 23pt, {\em Topology of moduli spaces of tropical curves with marked points}, Facets of algebraic geometry, 1 (2022), pp.~77--131.

\bibitem{chen2019principal}
{\sc D.~Chen and Q.~Chen}, {\em Principal boundary of moduli spaces of abelian and quadratic differentials}, in Annales de l'Institut Fourier, vol.~69, 2019, pp.~81--118.

\bibitem{chengendron2022towards}
{\sc D.~Chen and Q.~Gendron}, {\em Towards a classification of connected components of the strata of k-differentials}, Documenta mathematica Journal der Deutschen Mathematiker-Vereinigung, 27 (2022), pp.~1031--1100.

\bibitem{costantini2022chern}
{\sc M.~Costantini, M.~M{\"o}ller, and J.~Zachhuber}, {\em The chern classes and euler characteristic of the moduli spaces of abelian differentials}, in Forum of Mathematics, Pi, vol.~10, Cambridge University Press, 2022, p.~e16.

\bibitem{eskin2003moduli}
{\sc A.~Eskin, H.~Masur, and A.~Zorich}, {\em Moduli spaces of abelian differentials: the principal boundary, counting problems, and the siegel-veech constants}, Publications Math{\'e}matiques de l'IH{\'E}S, 97 (2003), pp.~61--179.

\bibitem{kontsevich2003connected}
{\sc M.~Kontsevich and A.~Zorich}, {\em Connected components of the moduli spaces of abelian differentials with prescribed singularities}, Inventiones mathematicae, 153 (2003), pp.~631--678.

\bibitem{lee2023connected}
{\sc M.~Lee}, {\em Connected components of strata of residueless meromorphic differentials}, Geometriae Dedicata, 218 (2024).

\bibitem{lee2023one}
{\sc M.~Lee and G.~Tahar}, {\em One-dimensional strata of residueless meromorphic differentials}, arXiv preprint arXiv:2310.13128,  (2023).

\bibitem{payne2013boundary}
{\sc S.~Payne}, {\em Boundary complexes and weight filtrations}, Michigan Mathematical Journal, 62 (2013), pp.~293--322.

\bibitem{Tah}
{\sc G.~Tahar}, {\em Counting saddle connections in flat surfaces with poles of higher order}, Geometriae Dedicata, 196 (2018), pp.~145--186.

\end{thebibliography}
